\documentclass[11pt,twoside]{article}
\usepackage{amsmath,amssymb,amsthm,eucal, mathrsfs, mathtools}
\usepackage{xcolor}
\usepackage[all]{xy}

\usepackage[australian]{babel}

\usepackage{slashed}

\usepackage[shortlabels]{enumitem}
\setlist{topsep=4pt plus 2pt minus 2pt,partopsep=0pt,itemsep=2pt plus 2pt minus 2pt,parsep=0.5\parskip}
\setenumerate[1]{label=(\arabic*),font=\upshape}

\title{Analytic index theory and spectral flow in real Hilbert $C^*$-modules}
\author{Chris Bourne${}^{\dag,\,\spadesuit}$, Alan Carey${}^\ddag$, Koen van den Dungen${}^\S$,  Adam Rennie\thanks{email: 
\texttt{cbourne@nagoya-u.jp, alan.carey@anu.edu.au, 
kdungen@uni-bonn.de, renniea@uow.edu.au}
}
\\[2mm]
{\small ${}^\dag$Institute of Liberal Arts and Sciences, Graduate School of Mathematics,}\\ 
{\small Nagoya University, Furo-cho, Chikusa-ku, Nagoya, 464-8601, Japan}\\ [2pt]
{\small $^{\spadesuit}$RIKEN iTHEMS, 2-1 Hirosawa, Wako, Saitama 351-0198, Japan} \\ [2mm]
{\small ${}^\ddag$Mathematical Sciences Institute, Australian National University,}\\
{\small Canberra ACT, 0200, Australia}\\[2mm]
{\small ${}^\S$Mathematisches Institut, Universit\"{a}t Bonn,}\\
{\small Endenicher Allee 60, D-53115 Bonn, Germany}\\[2mm]
{\small ${}^*$School of Mathematics and Physics, University of Wollongong,}\\
{\small Wollongong, 2522, Australia}
}

\advance\topmargin by -\headheight
\advance\topmargin by -\headsep
\evensidemargin=\oddsidemargin
\usepackage{fancyhdr}
\makeatletter
\def\section{\@startsection{section}{1}{\z@}{-3.5ex plus -1ex minus
  -.2ex}{2.3ex plus .2ex}{\large\bf}}
\def\subsection{\@startsection{subsection}{2}{\z@}{-3.25ex plus -1ex
  minus -.2ex}{1.5ex plus .2ex}{\normalsize\bf}}
\makeatother

\let\OLDthebibliography\thebibliography
\renewcommand\thebibliography[1]{
  \addcontentsline{toc}{section}{\refname}
  \OLDthebibliography{#1}
  \setlength{\parskip}{0pt}
  \setlength{\itemsep}{0pt plus 0.3ex}
}

\numberwithin{equation}{section} %% needs `amsmath' package

\theoremstyle{plain} %% needs `amsmath' package
\newtheorem{thm}{Theorem}[section]
\newtheorem{lemma}[thm]{Lemma}
\newtheorem{prop}[thm]{Proposition}
\newtheorem{cor}[thm]{Corollary}

\theoremstyle{definition} %% needs `amsmath' package
\newtheorem{defn}[thm]{Definition}
\newtheorem*{defn*}{Definition}

\newtheorem{assump}[thm]{Assumption}
\newtheorem{notation}[thm]{Notation}

\newtheorem{rmk}[thm]{Remark}
\newtheorem{remark}[thm]{Remark}
\newtheorem{example}[thm]{Example}
\newtheorem*{example*}{Example}
\newtheorem{remarks}[thm]{Remarks}

\DeclareMathOperator{\Coker}{Coker} %% cokernel
\DeclareMathOperator{\Ker}{Ker} %% kernel
\DeclareMathOperator{\Dom}{Dom}   %% domain of an operator
\DeclareMathOperator{\End}{End}   %% endomorphism algebra
\DeclareMathOperator{\Reg}{Reg}   %% regular operators
\DeclareMathOperator{\Id}{Id}     %% identity map
\DeclareMathOperator{\Index}{Index}     %% Fred index
\DeclareMathOperator{\Ind}{Ind}     %% Fred index again
\DeclareMathOperator{\ind}{ind}     %% Fred index again again
\DeclareMathOperator{\relind}{rel-ind}     %% relative index
\newcommand{\relindH}{\operatorname{rel-ind}^{\calH}}     %% relative index on H
\DeclareMathOperator{\spec}{spec}   %% spectrum
\DeclareMathOperator{\supp}{supp} %% support
\DeclareMathOperator{\Tr}{Tr}     %% operator trace
\DeclareMathOperator{\exc}{exc}  %% excision map
\DeclareMathOperator{\ev}{ev} %% evaluation map
\newcommand{\A}{\mathcal{A}}  %% an algebra
\newcommand{\B}{\mathcal{B}}  %% another algebra
\newcommand{\C}{\mathbb{C}}   %% complex numbers
\newcommand{\N}{\mathbb{N}}   %% natural numbers
\newcommand{\ox}{\otimes}     %% tensor product
\newcommand{\hox}{\, \hat\otimes \,} %% Z_2-graded tensor product
\newcommand{\R}{\mathbb{R}}   %% real numbers
\newcommand{\cS}{\mathcal{S}} %%esss
\newcommand{\Z}{\mathbb{Z}} %%integers
\newcommand{\stroke}{\mathbin|}     %% (for `\pair' and such)
\newcommand{\sa}{\textnormal{sa}}
\newcommand{\sk}{\textnormal{sk}}
\newcommand{\shuffle}{\Sigma} 
\newcommand{\Roe}{\mathfrak{R}}
\newcommand{\Kubota}{\Psi}
\newcommand{\Cayley}{\mathfrak{C}}
\newcommand{\altBott}{\wt{\beta}}
\newcommand{\exterior}{{\textstyle\bigwedge^{\!*}}}

\newcommand*{\KO}{K\!O}
\newcommand*{\KR}{K\!R}
\newcommand*{\DK}{D\!K}
\newcommand*{\KK}{K\!K}
\newcommand*{\KKR}{K\!K\!R}
\newcommand*{\KKO}{K\!K\!O}

\def\pairL_#1(#2|#3){{}_{#1}(#2\stroke#3)} %% hermitian pairing _B(s|t)
\def\pairR(#1|#2)_#3{(#1\stroke#2)_{#3}} %% hermitian pairing (s|t)_A
\def\scal<#1|#2>{\langle#1\stroke#2\rangle} %% scalar product <y|z>

\newbox\ncintdbox \newbox\ncinttbox %% noncommutative integral symbols
	\setbox0=\hbox{$-$}
	\setbox2=\hbox{$\displaystyle\int$}
	\setbox\ncintdbox=\hbox{\rlap{\hbox
		to \wd2{\hskip-.125em \box2\relax\hfil}}\box0\kern.1em}
	\setbox0=\hbox{$\vcenter{\hrule width 4pt}$}
	\setbox2=\hbox{$\textstyle\int$}
	\setbox\ncinttbox=\hbox{\rlap{\hbox
		to \wd2{\hskip-.175em \box2\relax\hfil}}\box0\kern.1em}

\renewcommand{\epsilon}{\varepsilon}

\def\calC{\mathcal{C}}

\def\calK{\mathcal{K}}
\def\calB{\mathcal{B}}
\def\calH{\mathcal{H}}

\def\calA{\mathcal{A}}

\def\calM{\mathcal{M}}
\def\calJ{\mathcal{J}}

\def\calQ{\mathcal{Q}}
\def\calV{\mathcal{V}}

\def\calU{\mathcal{U}}

\newcommand{\ol}{\overline}

\theoremstyle{definition}

\DeclareMathOperator{\Mult}{Mult}

\DeclareMathOperator{\Sf}{sf}

\newcommand{\Cl}{{\C\ell}} %% Clifford algebra

\newcommand{\wh}{\ensuremath{\widehat}}
\newcommand{\wt}{\ensuremath{\widetilde}}
\DeclareMathOperator{\Ran}{Ran}

\newcommand{\rs}{{\mathfrak r}}

\newcommand{\frakr}{{\mathfrak r}}

\newcommand{\one}{{\bf 1}}
\DeclareFontEncoding{FMS}{}{}
\DeclareFontSubstitution{FMS}{futm}{m}{n}
\DeclareFontEncoding{FMX}{}{}
\DeclareFontSubstitution{FMX}{futm}{m}{n}
\DeclareSymbolFont{NEWsymbols}{FMS}{futm}{m}{n}%
\DeclareSymbolFont{NEWlargesymbols}{FMX}{futm}{m}{n}%
\DeclareMathDelimiter{\llbracket}{\mathopen}{NEWsymbols}{153}{NEWlargesymbols}{133}
\DeclareMathDelimiter{\rrbracket}{\mathclose}{NEWsymbols}{154}{NEWlargesymbols}{134}
\newcommand{\class}[1]{\big\llbracket #1 \big\rrbracket}

\definecolor{MyBlue}{cmyk}{1,0.13,0,0.63}
\definecolor{MyGreen}{cmyk}{0.91,0,0.88,0.52}
\newcommand{\mylinkcolor}{MyBlue}
\newcommand{\mycitecolor}{MyGreen}
\newcommand{\myurlcolor}{black}

\usepackage{hyperref}
\hypersetup{%
  bookmarksnumbered=true,bookmarksopen=false,%
  plainpages=false,% necessary to prevent duplicate page identifiers
  linktocpage=true,%
  colorlinks=true,breaklinks=true,%
  linkcolor=\mylinkcolor,citecolor=\mycitecolor,urlcolor=\myurlcolor,%
  pdfpagelayout=OneColumn,%
  pageanchor=true,%
}

\newcommand{\vin}{\rotatebox[origin=c]{-90}{$\in$}}

\begin{document}

\maketitle

\begin{abstract}

We consider the analytic index and spectral flow of  Fredholm operators on Hilbert $C^*$-modules. 
Our spaces and algebras are equipped with a real structure, so the analytic index and spectral flow 
takes value in the real $K$-theory group of a $\sigma$-unital $C^*$-algebra.
We use Van Daele $K$-theory, which allows us to treat the eight real $K$-theory groups and the two complex groups  on an
equal footing. 
We provide a general definition of the analytic index for Clifford anti-linear and skew-adjoint Fredholm operators as well as self-adjoint 
and odd Fredholm operators.
Our definition of spectral flow and its basic properties are 
valid for Wahl-continuous paths of Fredholm operators on a real Hilbert $C^*$-module.
We also provide an analytic approach to the spectral flow as a decomposition into a finite sum of relative indices. 
Furthermore, we prove a real version of the Robbin--Salamon theorem, relating the spectral flow to a Fredholm index.
Our description of the index and spectral flow relies on various isomorphisms between Kasparov's $\KKR$-theory and Van Daele $K$-theory, 
which we systematically describe in the Appendix. 
\end{abstract}

\begin{small}
{Keywords: Index theory, spectral flow, Hilbert $C^*$-modules, Clifford algebra, $K$-theory, $\KK$-theory}
\end{small}

\tableofcontents

\parindent=0.0in
\parskip=0.05in

\section{Introduction}

The celebrated paper of Atiyah and Singer~\cite{AS69} provides a comprehensive study on the homotopy theory of the 
space of skew-adjoint Fredholm operators on real or complex Hilbert spaces which are \emph{$C\ell_{0,k}$-anti-linear} (i.e., they anti-commute with the generators 
of an ungraded Clifford algebra $C\ell_{0,k}$ represented on the Hilbert space). 
The skew-adjoint Fredholm index then labels the connected components of the space of such Fredholm operators. 
See  \cite{AS5, LM} for applications to index theory and geometry. 
Motivated by recent applications to topological phases of matter \cite{AMZ, DSB}, the paper \cite{BCLR} (by some of the authors and Matthias Lesch) considered 
the fundamental group and \emph{spectral flow} of $C\ell_{0,k}$-anti-linear skew-adjoint Fredholm operators on real Hilbert spaces.

For Fredholm operators on complex spaces, the extension of index theory and spectral flow from Hilbert spaces to Hilbert $C^*$-modules has already been extensively studied (for a non-exhaustive list, see e.g.~\cite{MisFom, Troitskii, Kasimov, Mi, Wu, DZ96, DZ98, LP, BJS03, Joachim, Wahl07, vdDungen24, NSWSf, NSW2}). 
In this paper, it is our goal to present a general Fredholm theory on \emph{real} Hilbert $C^*$-modules. 
We will provide a comprehensive description of both the Fredholm index and the spectral flow for $C\ell_{0,k}$-anti-linear skew-adjoint Fredholm operators on real Hilbert $C^*$-modules over some real $C^*$-algebra $A$. 

The skew-adjoint Fredholm index of Atiyah and Singer, defined for $C\ell_{0,k-1}$-anti-linear skew-adjoint Fredholm operators on a real Hilbert space, takes values in the real $K$-theory group $\KO^{-k}$ of a point. 
An elegant description of this index can be obtained via the Atiyah--Bott--Shapiro (ABS) isomorphism~\cite{ABS}. 
Indeed, the kernel $\Ker(F)$ of the Fredholm operator $F$ is a finite-dimensional $C\ell_{0,k-1}$-module, and 
thus defines an element in the Grothendieck group $\calM_{0,k-1}$ of (ungraded) finite-dimensional representations of $C\ell_{0,k-1}$. 
The equivalence class of $\Ker(F)$ (modulo $C\ell_{0,k}$-modules) then provides an element in $\calM_{0,k-1}/\calM_{0,k}$. 
Under the ABS isomorphism $\calM_{0,k-1}/\calM_{0,k} \cong \KO^{-k}$, the equivalence class of $\Ker(F)$ then corresponds to the Fredholm index of $F$ in $\KO^{-k}$. 

Similarly, the skew-adjoint spectral flow from \cite{BCLR} can also be described explicitly in terms of the Atiyah--Bott--Shapiro picture of real $K$-theory, and it generalises both the ordinary $\Z$-valued spectral flow as well as the $\Z_2$-valued spectral flow \cite{CPSB, DSB}.

In order to generalise the Fredholm theory from real Hilbert spaces to real Hilbert $C^*$-modules, 
it might then be tempting to first try to generalise the ABS isomorphism from Hilbert spaces to $C^*$-modules. 
Unfortunately, the analogous ABS construction for $C^*$-modules (simply replacing finite-dimensional vector spaces by finite projective $C^*$-modules over $A$) does not agree with the real $K$-theory group $\KO_{k}(A)$. 
For the record, we provide a simple counterexample in Appendix \ref{sec:ABS}.
In the $C^*$-module setting, we therefore need to move away from the ABS picture, and take a different approach.

First, instead of real spaces, we will work with Real spaces (with capitalised R), 
which are complex spaces equipped with \emph{real structures} 
(one may move between the real and Real pictures by complexifying or restricting to elements invariant under the real structure),
 and consequently we replace real $\KO$-groups by Real $\KR$-groups. 
In analogy to \cite{AS69, BCLR}, we then consider $\Cl_{r,s}$-anti-linear skew-adjoint Fredholm operators on a Real Hilbert $C^*$-module over a Real $C^*$-algebra $A$. 
We will show that such Fredholm operators naturally define an element in Kasparov's Real $\KK$-theory group $\KKR(\Cl_{r,s},A)$ (for a brief introduction to Real $\KK$-theory, we refer to Appendix \ref{sec:RealKK}). 

We would like to define the Fredholm index and spectral flow taking values in the Real $K$-theory groups $\KR_k(A)$ (where $k$ depends on the choice of $r,s$). These $K$-theory groups are defined in terms of the $k$-fold suspension of $A$ as $\KR_k(A) := \KR_0(S^kA)$. 
However, we will find it much more convenient to avoid these suspensions and instead use the already present Clifford algebra $\Cl_{r,s}$ to label the degree of the $K$-theory group. 
To be able to do this, we need a description of $K$-theory which allows for $\Z_2$-graded $C^*$-algebras. 
Our receptacle for the Fredholm index and spectral flow will therefore be Van Daele's $K$-theory groups \cite{vanDaele1,vanDaele2}, denoted $\DK(A)$ (for an overview of the definition and properties of Van Daele $K$-theory, we refer to Appendix \ref{sec:appendix_DK}). 
For an ungraded $C^*$-algebra $A$, the Van Daele $K$-theory groups are related to the (usual) Real $K$-theory groups via the isomorphisms 
\(
\Upsilon_A^{r+1,s}: \DK(A \otimes \Cl_{r+1,s}) \xrightarrow{\simeq} \KR_{s-r}(A) 
\)
(see Appendix \ref{subsec:appendix_DK_KR_isos}).

Our description of the Fredholm index and spectral flow in Van Daele $K$-theory will also rely heavily on various isomorphisms between Van Daele $K$-theory and Kasparov's Real $\KK$-theory developed in~\cite{Roe04, Kubota15a, BKR}. 
 Adapting these isomorphisms to the setting with Clifford symmetries requires some care. One reason is that $\KK$-theory is bivariant and 
we have the isomorphism $\Sigma_{\KK}^{r,s}: \KK( A \hox \Cl_{s,r}, B) \xrightarrow{\simeq} \KK(A, B \hox \Cl_{r,s} )$. An analogous `Clifford shuffle' 
map in $\DK$-theory is more subtle (see \S\ref{subsec:Clifford_Fredholm_Shuffle}).
We will provide a detailed study of these Clifford-decorated isomorphisms between $\KKR$-theory and $\DK$-theory in Appendix \ref{sec:appendix_KK_DK_isos}. 
In particular, we carefully show that these isomorphisms are compatible with the notions of Bott periodicity, Morita invariance, and other maps within $\KK$-theory and $\DK$-theory.

Throughout this paper, we will be   explicit by including the notation for all isomorphisms in and between $\KK$-theory and $\DK$-theory. 
Although this makes the notation somewhat more cumbersome, we believe this may help the reader to better understand our definitions and constructions. 
Moreover, whenever possible, we will always work with concrete isomorphisms applied to explicit representatives of equivalence classes,  
ensuring that our constructions can be useful in explicit computations. 

Let us now provide a brief overview of the main contents of this paper. 

\paragraph{The Fredholm index}
The well-known theory of (bounded or unbounded) Fredholm operators on Hilbert $C^*$-modules is adapted from the complex setting to the Real setting in {\S}\ref{sec:Fredholm_Real}.
We  show that a $\Cl_{r,s}$-anti-linear skew-adjoint Fredholm operator yields a well-defined class in $\KKR$-theory, and 
 establish the technical preliminaries to systematically incorporate Clifford symmetries in both $\KKR$- and $\DK$-theory.

In {\S}\ref{sec:Fredholm_index}, we then provide our definition of the \emph{Fredholm index}, 
which is defined for (bounded or unbounded) $\Cl_{r,s}$-anti-linear skew-adjoint Fredholm operators, 
and takes values in the Van Daele $K$-theory group $\DK(A \otimes \Cl_{r+1,s+1})$. 
We show that the Fredholm index factors through the $\KKR$-class of the Fredholm operator and the isomorphism 
between $\KKR$-theory and $\DK$-theory. Thus, in a generalised sense, we can understand 
the Fredholm index on Real $C^*$-modules as providing  the link between $\KKR$-theory and Van Daele $K$-theory.
Furthermore, we show that our Fredholm index agrees with the (even or odd) index on complex $C^*$-modules \cite{Wahl07}, as well as with the skew-adjoint Fredholm index on real Hilbert spaces via the ABS isomorphism \cite{BCLR}. 

\paragraph{The relative index}
{\S}\ref{sec:Van_Daele_skew_unitaries} first gives a presentation of (relative) Van Daele $K$-theory in terms of $\Cl_{r,s}$-anti-linear skew-adjoint unitaries 
(see also~\cite[\S3.3]{AMZ}). 
We use this presentation in {\S}\ref{sec:relative_index} to define the \emph{relative index} for a Fredholm pair of such skew-adjoint unitaries. 
Again, we show that the relative index factors through a $\KKR$-class constructed from the Fredholm pair. 
Moreover, we show that our $\DK$-valued relative index agrees with the (even or odd) relative index on complex $C^*$-modules, as well as with the relative index on real Hilbert spaces considered in~\cite{BCLR}. 

\paragraph{The spectral flow}
In order to introduce the spectral flow of a path of Fredholm operators, we first study families of Fredholm operators in {\S}\ref{sec:families}. 
Given a compact space $\Omega$ and a family $\{D_\omega\}_{\omega\in\Omega}$ of (densely-defined, self-adjoint or skew-adjoint) regular Fredholm operators on a $C^*$-module $Y_B$, we can consider the \emph{family operator} $D_\bullet$ on $C(\Omega,Y_B)_{C(\Omega,B)}$. 
To study such families, we describe the \emph{Wahl topology} on $C^*$-modules. 
This topology was introduced by Wahl for (unbounded) Fredholm operators on Hilbert spaces \cite{Wahl08}. It is weaker than the gap topology, but still strong enough to be suitable for Fredholm theory. 
We show that $D_\bullet$ is \emph{regular and Fredholm} if and only if the family $\{D_\omega\}_{\omega\in\Omega}$ is \emph{Wahl-continuous}. Thus, the Wahl topology provides the appropriate topology for the study of families of Fredholm operators, and in fact it is the weakest topology under which the spectral flow can be defined. 

In {\S}\ref{sec:spectral_flow}, we then provide the (abstract) definition of the $\DK$-valued spectral flow for a Wahl-continuous path of $\Cl_{r,s}$-anti-linear skew-adjoint Fredholm operators with invertible endpoints, by simply combining our $\DK$-valued Fredholm index with Bott periodicity. 
We prove that the spectral flow satisfies the properties of homotopy invariance, additivity under direct sum, additivity under concatenation of paths, as well as a normalisation property relating the spectral flow to the relative index of a Fredholm pair of skew-adjoint unitaries. 
We also show that our spectral flow agrees with the (even or odd) spectral flow on complex $C^*$-modules \cite{Wahl07}. 

In {\S}\ref{sec:analytic_spectral_flow}, we provide an analytic approach to our spectral flow, in analogy to the work of J. Phillips on complex Hilbert spaces  \cite{P1,P2}. 
The latter is constructed from a finite partition of the path of Fredholm operators and yields a finite sum of relative indices of projections. 
In our Real setting, we will replace the relative index of projections by our $\DK$-valued relative index of skew-adjoint unitaries (as in \cite{BCLR}). 
In fact, we provide two versions of the analytic spectral flow. 
First, our Phillips formula is analogous to \cite{P2} and works under the assumption that the path of Fredholm operators is Riesz-continuous. 
Second, our Wahl formula is an adaptation of the approach by Wahl \cite{Wahl07} on complex $C^*$-modules, which works   for Wahl-continuous paths, but requires the assumption of locally trivialising families. 
Our Phillips formula also allows us to prove that the $\DK$-valued spectral flow agrees with the skew-adjoint spectral flow on real Hilbert spaces from \cite{BCLR}. 

\paragraph{`Index = spectral flow'}
The close relation between the Fredholm index and the spectral flow is well-known and has been used already in the work of Atiyah--Patodi--Singer \cite[\S7]{APS76}. 
The Robbin--Salamon Theorem~\cite{RS} shows that the spectral flow of a family of (unbounded) self-adjoint Fredholm operators $\{D_t\}$ on a complex Hilbert space 
equals the Fredholm index of the operator $\partial_t + D_\bullet$. 
In {\S}\ref{sec:Robbin-Salamon}, we prove a generalisation of the Robbin--Salamon Theorem in the setting of Real $C^*$-modules, which relates the $\DK$-valued spectral flow of a path $\{T_t\}_{t\in[0,1]}$ to the $\DK$-valued Fredholm index of an operator $(T_\bullet + \partial_t)_+$. 

\paragraph{Examples from physics}
One of our motivations for the present paper was the discovery of families of Hamiltonians in 
condensed matter physics that fail to be gap continuous but are Wahl-continuous \cite[\S4.5.1]{Thiang21}. 
Furthermore,  
our setting of $\Cl_{r,s}$-anti-linear skew-adjoint Fredholm operators is a 
  natural framework for the study of Hamiltonians with free fermionic symmetries as investigated by Zirnbauer et al.~\cite{KZ16, AMZ}. 
In {\S}\ref{sec:Physics_examples}, we then consider some applications of our $\DK$-valued spectral flow  
to the study of topological phases of matter. 
By introducing an auxiliary algebra of observables $A$, we have an interpretation 
of a relative index of gapped Hamiltonians as a spectral flow taking values in the Van Daele $K$-theory of $A$. 
Hamiltonians on (discrete) half spaces that are translation invariant parallel to the boundary also give rise to loops of 
Fredholm operators on a  $C^*$-module over an algebra describing the boundary, which also has a spectral flow description. 
In many cases of interest, this spectral flow can be directly related to the boundary map of a boundary-free (bulk) system  and so provides 
a  general spectral flow interpretation of the bulk-boundary correspondence.

\paragraph{Appendix} 
As mentioned above, we explain in Appendix \ref{sec:ABS} that the naive extension of the ABS isomorphism to Real Hilbert $C^*$-modules does not hold. 
The mechanisms that underpin our results are Real Kasparov theory, Van Daele $K$-theory, and how these 
two theories are related.  Appendices \ref{sec:RealKK}, \ref{sec:appendix_DK}, and \ref{sec:appendix_KK_DK_isos}, 
respectively review $\KKR$-theory, $\DK$-theory, and their compatibility with each other. Appendix \ref{appendix:isos} gives  
a summary of the maps used throughout the paper (see also the Notation and conventions section below).

\subsection*{Acknowledgements}
C.B. is supported by a KAKENHI Grant-in-Aid (24K06756) from the Japan Society for the Promotion of Science (JSPS) 
and thanks ANU, UOW and the University of Bonn  for hospitality during the production of this work.
A.C. and A.R. were supported by the ARC Discovery grant DP220101196. A.R. thanks the Institute for Advanced Research, Nagoya University, 
for support and hospitality.
K.v.d.D. was supported by the Hausdorff Center for Mathematics, funded by the Deutsche Forschungsgemeinschaft (DFG, German Research Foundation) under Germany's Excellence Strategy -- EXC-2047/1 -- 390685813.

\subsection*{Notation and conventions}

\begin{itemize}
  \item All linear spaces (e.g., Hilbert spaces, $C^*$-algebras, Hilbert $C^*$-modules) are \emph{Real}, which means they are complex and come equipped with a 
  \emph{real structure} $\frakr$, an anti-linear involution. 
  The fixed points under this action will define a linear space over $\R$. 
  When necessary, we write $\rs_Y$ to denote a real structure on the space $Y$. We say $y \in Y$ is Real if $y= y^\frakr$. The real structure on a $C^*$-algebra $A$ satisfies 
  $(ab)^\frakr = a^\frakr b^\frakr $ for all $a,b\in A$. 
  
   \item All $C^*$-algebras will be $\sigma$-unital. We will  consider both $\Z_2$-graded and ungraded $C^*$-algebras, where 
    $B$ generally denotes a (possibly) $\Z_2$-graded $C^*$-algebra, and $A$ an ungraded $C^*$-algebra. 
    If $B$ is $\Z_2$-graded, then  $B=B^0\oplus B^1$, $B^iB^j\subset B^{i+j\!\bmod2}$, and $b\in B$ is called even (resp. odd) if $b\in B^0$ (resp. $B^1$). 
    A graded $*$-homomorphism $\phi:B\to C$ of $\Z_2$-graded $C^*$-algebras satisfies $\phi(B^i)\subset C^i$, $i=0,1$. We also write $\deg(b)=i$ if $b\in B^i$.
  The multiplier algebra of $B$ is denoted $\Mult(B)$, and we denote the Calkin algebra of $B$ by $\calQ_B := \Mult(B)/B$.   

\item The $\Z_2$-graded tensor product of the $\Z_2$-graded $C^*$-algebras $B,C$ is the usual (spatial) tensor product with operations defined on homogeneous elements by
\begin{align}
(a\,\hat\otimes\, c)(b\,\hat\otimes\, d)
&:= (-1)^{\deg(b)\deg( c )}\,(ab \,\hat\otimes\, cd)\nonumber\\ 
(a\,\hat\otimes\, c)^*
&=(-1)^{\deg(a)\deg(c)}(a^*\,\hat\otimes\, c^*)\nonumber\\
\deg(a\,\hat\otimes\, c)
&=\deg(a)+\deg(c),\qquad a,\,b\in B,\ c,\,d\in C,
\label{eq:zed-two-tensors}
\end{align}
see~\cite[\S2.6]{Kasparov80}.
In particular, if $a=a^*$ and $c=c^*$ are both odd, then $a\hat\otimes c$ is skew-adjoint.
  
    \item Clifford algebras $\Cl_{r,s}$ are regarded as graded or ungraded Real algebras depending on context, where the grading is always the one that makes the generators odd. We use the convention 
  \[
     \Cl_{r,s} = \mathrm{span}_\C \Big\{ e_1,\ldots, e_r, f_1, \ldots, f_s \,\Big| \, 
         \begin{matrix}  e_j = e_j^* = e_j^\rs &  \text{for all } j=1,\ldots, r, \\   f_k = -f_k^* = f_k^\rs & \text{for all } k=1,\ldots, s \end{matrix}  \Big\}
  \]
  for ungraded Clifford algebras and 
  \[
     \Cl_{r,s} = \mathrm{span}_\C  \Big\{ \gamma_1,\ldots, \gamma_r, \rho_1, \ldots, \rho_s \, \Big| \, 
        \begin{matrix} \gamma_j= \gamma_j^*  = \gamma_j^\rs & \text{for all } j=1,\ldots, r, \\  \rho_k = -\rho_k^* = \rho_k^\rs & \text{for all }  k=1,\ldots, s \end{matrix} \Big\}
  \]
  for graded Clifford algebras. The real subalgebra of $\rs$-invariant elements   is denoted by $Cl_{r,s}$.
  
  \item In the complex setting, we will denote by $\Cl_{r+s}$ the complex Clifford algebra $\Cl_{r,s}$ without any real structure. In this case, 
  the distinction between self-adjoint and skew-adjoint generators becomes irrelevant, but we sometimes still  emphasise what generators 
  are being used (e.g. writing $\Cl_{1+1}$ to denote $\Cl_{2}$ with a preferred self-adjoint and skew-adjoint generator).
  
  \item We also make use of the $\Z_2$-graded exterior algebra $\exterior \C^n$ as a complex vector space, 
  \[
     \exterior \C^n =  \exterior (\R\otimes \C)^n, \qquad 
     (v_1 \wedge \cdots \wedge v_m)^\frakr = \ol{v_1} \wedge \cdots \wedge \ol{v_m}, \qquad 
     \End\big( \exterior \C^n \big) \cong \Cl_{n,n}
  \]
  and the last isomorphism is $\Z_2$-graded. In particular, we will frequently use 
  $\End( \exterior \C ) \cong \Cl_{1,1} \cong \mathrm{span}_\C\{ \gamma, \rho \}$, which we concretely realise as 
  \[
     \exterior \C \cong \C^2, \qquad \End\big( \exterior \C \big)  \cong M_2(\C), \qquad 
     \gamma = \begin{pmatrix} 0 & 1 \\ 1& 0 \end{pmatrix}, \quad \rho =  \begin{pmatrix} 0 & -1 \\ 1& 0 \end{pmatrix} .
  \]
  Here the real structure on $\C^2$ and $M_2(\C)$ is given by component-wise complex conjugation.

  \item Occasionally we also use the Pauli matrices given by 
  \[
      \sigma_1 = \begin{pmatrix} 0 & 1 \\ 1 & 0 \end{pmatrix}, \qquad \sigma_2 = \begin{pmatrix} 0 & -i \\ i & 0 \end{pmatrix},  
      \qquad -i\sigma_2 = \begin{pmatrix} 0 & -1 \\ 1 & 0 \end{pmatrix}, 
      \qquad \sigma_3 = \begin{pmatrix} 1 & 0 \\ 0 & -1 \end{pmatrix}.
  \]

 \item The $C^*$-module $Y_B=(Y^0\oplus Y^1)_B$ denotes a generic possibly $\Z_2$-graded and countably generated 
 (Real) Hilbert $C^*$-module over $B=B^0\oplus B^1$. In particular, $Y^jB^k\subset Y^{j+k\!\bmod2}$. 
 The sets $\End_B(Y)$ and $\End_B^0(Y)$ denote the 
  adjointable and compact endomorphisms on $Y_B$ respectively. We also denote by 
  $\calQ_B(Y) := \calQ_{\End_B^0(Y)} = \End_B(Y)/ \End_B^0(Y)$ the Calkin algebra, and we let $q_Y \colon \End_B(Y) \to \calQ_B(Y)$ denote the corresponding quotient map. 
  If the context is clear, we will simply write $q$ for the quotient map to $\calQ_B(Y)$.
  The set of (densely-defined) regular operators on $Y_B$ is denoted by $\Reg_B(Y)$. 
  Each of the algebras $\End_B^0(Y)$, $\End_B(Y)$ and $\Reg_B(Y)$ inherit a real structure and $\Z_2$-grading, and so we may speak of Real or  
  even/odd operators on $Y_B$.
  
  \item If the $C^*$-module $Y_B$ comes with a ($\Z_2$-graded) Clifford representation 
  $\Cl_{r,s} \to \End_B(Y)$, then
  $\End_B^{r,s}(Y)$ denotes the \emph{Clifford anti-linear} adjointable operators 
  (i.e., all operators in $\End_B(Y)$ which anti-commute with the generators of $\Cl_{r,s}$). Similarly, $\Reg_B^{r,s}(Y)$ 
  denotes  the (densely-defined) \emph{Clifford anti-linear} regular operators on $Y_B$.   

  We will often assume that 
  the $\Cl_{r,s}$-representation is \emph{ample}, which means that none of the generators acts compactly on $Y_B$, 
  $\Cl_{r,s} \cap \End_A^0(X) = \{0\}$ (in particular, $Y_B$ cannot be finitely generated). In this case, we also obtain a non-zero graded $*$-homomorphism $\Cl_{r,s} \to \calQ_B(Y)$. 

 \item We usually denote by $X_A$ an \emph{ungraded} $C^*$-module over an ungraded $C^*$-algebra $A$. 
Typically, we assume that $X_A$ is \emph{full}, which means that the closed linear span of $\{ \scal<x_1|x_2> \mid x_1,x_2\in X \}$ equals $A$. 
 Usually the $C^*$-module $X_A$ comes equipped with a Clifford representation $\Cl_{r,s} \to \End_A(X)$, and 
we define $\End_A^{r,s}(X)$ and $\Reg_A^{r,s}(X)$ analogously to $Y_B$ above.

\item Suppose that $Y_B$ carries a $\Z_2$-graded representation of $\Cl_{s,r}$ (note the reversed $(s,r)$). 
Then taking a minimal (rank one) even projection $P^{r,s}\in \Cl_{s,r} \hox \Cl_{r,s} \subset \End_{B}(Y) \hox \Cl_{r,s}$ 
and odd self-adjoint unitary $e \in \End_B^{s,r}(Y)$, $\sigma_e^{r,s}$ denotes the map 
  \begin{align*}
  &\sigma_e^{r,s}: \Reg_B^{s,r}(Y) \to    \Reg_B(Y) \hox \Cl_{r,s}   
  &\sigma_e^{r,s}(x) =  P^{r,s}(x \hox 1_{r,s}) + (\one -P^{r,s})(e \hox 1_{r,s}) 
\end{align*}
with $1_{r,s} \in \Cl_{r,s}$ the algebraic unit. 
We observe that $P_{r,s}\Cl_{r+s,r+s}\cong\overline{S}$ where $S$ is the irreducible spinor representation of $\Cl_{r+s,r+s}$.

In the case that $X_A$ is an ungraded module with $\Cl_{r,s} \to \End_A(X)$ an ungraded representation,  
we instead consider a basepoint skew-adjoint unitary $J \in \End_A^{r,s}(X)$ and define the 
map $\sigma_{J \otimes \rho}^{r,s}: \Reg_A^{r,s}(X) \otimes \Cl_{0,1}  \to    \Reg_A(X) \otimes \Cl_{r,s+1} $ such that 
\[
  \sigma_{J \otimes \rho}^{r,s}(T\otimes \rho) =  P^{r,s}(T \otimes \rho) + (\one -P^{r,s})(J \otimes \rho)  ,
\]

where we identify $\rho$ with  $\rho_1 \in \Cl_{r,s+1}$.

\item For the special case of the standard  $C^*$-module, we let $\calH$ be a separable infinite-dimensional Hilbert space, and $\hat{\calH}\cong \calH\oplus\calH$ a $\Z_2$-graded Hilbert space. We write $\hat{\calH}_B = \hat{\calH} \hox B$ and $\calH_A = \calH \otimes A$ 
for the $\Z_2$-graded and ungraded standard modules respectively. 
As $A$ is ungraded, $\hat{\calH}_A \cong \calH_A \otimes \exterior{\C}$.  For all countably generated modules $Y_B$, there is a unitary isomorphism $Y_B\oplus\hat{\calH}_B\to\hat{\calH}_B$, called a (Kasparov) stabilisation map \cite{KasparovStinespring, Kasparov80}.

  \item Given a Banach space $Y$, we let $SY$ denote the space 
  \[
      S Y := C_0\big( (0,1) , Y \big) = \big\{ f \in C( [0,1], Y ) \mid f(0) = f(1) = 0 \big\} .
  \]
  If $Y$ is Real, we define the real structure on $SY$ by $f^{\frakr_{SY}} (t) =  f(t)^{\frakr_Y}$. 
  If $Y_B$ is a $C^*$-module, then so is $S Y_{SB}$. We will freely use the isomorphisms 
  $S(B \hox \Cl_{r,s} ) \cong SB \hox \Cl_{r,s}$ and $\End_{SB}^0(SY) \cong S \End_B^0(Y)$.
\end{itemize}

For the convenience of the reader, Appendix \ref{appendix:isos} provides a list with our notation of various isomorphisms in and between $\KKR$-theory and $\DK$-theory.

\section{Fredholm operators on Real Hilbert modules} \label{sec:Fredholm_Real}

In this section, we consider Fredholm operators on (possibly $\Z_2$-graded) Real $C^*$-modules 
$Y_B$ and their connection to Kasparov's Real $\KKR$-theory. 
For a brief review of $\KKR$-theory, we refer to Appendix \ref{sec:RealKK}.

In the complex setting, bounded Fredholm operators on complex $C^*$-modules were studied, in all generality, by Exel~\cite{Exel} (see also \cite[{\S}4]{BGB}).
In the real setting, real Fredholm operators with closed range were described by Schr\"{o}der~\cite{S}. 
In this section, we adapt the approach of Wahl \cite{Wahl07} on complex $C^*$-modules to study unbounded Fredholm operators on Real $C^*$-modules.

\subsection{Fredholm operators on Hilbert modules} \label{subsec:Fredholm_intro}

Let $Y_B$ be a countably generated and (possibly) $\Z_2$-graded $C^*$-module over a (possibly $\Z_2$-graded) $C^*$-algebra $B$. 

\begin{defn}
\label{defn:fred}
An adjointable operator $T:Y^1_B\to Y^2_B$ between Hilbert $B$-modules is called \emph{Fredholm} if there exists an adjointable operator
$S:Y^2_B \to Y^1_B$ such that $TS- \one_{Y^2} \in \End_B^0(Y^2)$ and $ST- \one_{Y^1} \in\End^0_B(Y^1)$.
For $T \in \End_B(Y)$, this means that $T$ is Fredholm if 
$q_Y(T) \in \calQ_B(Y)= \End_B(Y)/ \End_B^0(Y)$ is invertible.

A (densely-defined) regular operator $D\in\Reg_B(Y)$ is called \emph{Fredholm} if 
the bounded transform $F_D := D(\one + D^*D)^{-1/2} \in \End_B(Y)$ is Fredholm. 
\end{defn}

Following Wahl \cite{Wahl07}, we recall below that any self-adjoint regular Fredholm operator gives rise to a (bounded) Kasparov $\C$-$B$-module 
and hence to a class in $\KK(\C,B)$, and we will explain how this approach can be adapted to skew-adjoint operators as well. 

\begin{prop}[{\cite[Proposition 2.1, Corollary 2.2]{Wahl07}}] \label{prop:Fredholm_properties}
For a self-adjoint operator $D \in \Reg_B(Y)$, the following statements are equivalent:
\begin{enumerate}[label=(\roman*)]
  \item $D$ is Fredholm.
  \item There is an $\varepsilon >0$ such that for all $\varphi \in C_c(-\varepsilon, \varepsilon)$, $\varphi(D) \in \End_B^0(Y)$.
  \item There is an $\varepsilon >0$ such that for any continuous function $\chi:\R \to \R$ such that $\chi |_{(-\infty, -\varepsilon]} = -1$ 
  and $\chi |_{[\varepsilon, \infty)} = +1$, $\chi(D)^2 - \one  \in \End_B^0(Y)$.
\end{enumerate}
Analogously, for a \emph{skew-adjoint} operator $T \in \Reg_B(Y)$, the following statements are equivalent:
\begin{enumerate}[label=(\roman*')]
  \item $T$ is Fredholm.
  \item There is an $\varepsilon >0$ such that for all $\varphi \in C_c(-i\varepsilon, i\varepsilon)$, $\varphi(T) \in \End_B^0(Y)$.
  \item There is an $\varepsilon >0$ such that for any continuous function $\chi:i\R \to i\R$ such that $\chi |_{(-i\infty, -i\varepsilon]} = -i$ 
  and $\chi |_{[i\varepsilon, i\infty)} = +i$, $\chi(T)^2 + \one  \in \End_B^0(Y)$. 
\end{enumerate}
\end{prop}

\begin{defn} \label{def:Reg_fn}
Let $D \in \Reg_B(Y)$ be self-adjoint. 
We say that a smooth odd non-decreasing function $\chi: \R \to \R$ is a \emph{normalising function} for 
$D$ if $\chi'(0)>0$, $\lim_{x\to \infty}\chi(x) = 1$ and $\chi(D)^2 - \one \in \End^0_B(Y)$.

If $T \in \Reg_B(Y)$ is skew-adjoint, then a smooth odd non-decreasing function $\chi: i\R \to i\R$ is a \emph{normalising function} for 
$T$ if $\chi$ is strictly increasing in a neighbourhood of $0$, $\lim_{ix\to \infty}\chi(ix) = i$ and $\chi(T)^2 + \one \in \End^0_B(Y)$.
\end{defn}

\begin{remark}
Normalising functions are guaranteed to exist for Fredholm operators by Proposition \ref{prop:Fredholm_properties}. Furthermore, if $Y_B$ is a 
Real $C^*$-module and $T = T^\rs$, then $\chi(T)^\rs = \chi(T)$.
\end{remark}

\begin{lemma}[{cf. \cite[Lemma 2.7]{Wahl07}}] \label{lemma:bdd_unbdd_Fred_condition}
\begin{enumerate}
  \item A skew-adjoint operator $F \in \End_B(Y)$ with $\|F\|\leq 1$ is Fredholm if  and only if
  \[
     \big\| \one + q_Y(F)^2 \big\|_{\calQ_B(Y)} < 1.
  \]
  \item A skew-adjoint operator $T \in \Reg_B(Y)$ is Fredholm if and only if
  \[
     \big\| q_Y\big((\one + T)^{-1}\big) \big\|_{\calQ_B(Y)} < 1.
  \]
\end{enumerate}
\end{lemma}
\begin{proof}
(1) 
Since $F^2 \leq 0$ and $\|F\| \leq 1$, the condition $\big\| \one + q_Y(F)^2 \big\|_{\calQ_B(Y)} < 1$ holds if and only if $q_Y(F)$ 
is invertible, which by definition means $F$ is Fredholm. 

(2)
The operator $T$ is Fredholm if and only if $q_Y(F_T) \in \calQ_B(Y)$ is invertible, where we recall that $F_T =T (\one - T^2)^{-1/2}$. 
We compute 
$
 \one +  F_T^2  = (\one-T^2)^{-1} = \big((\one + T)(\one-T)\big)^{-1}.
$
Using part (1), it follows $q_Y(F_T)$ is invertible if and only if 
$\big\| q_Y\big((\one + T)^{-1}\big) \big\|_{\calQ_B(Y)}^2 = \big\| \one +  q_Y(F_T)^2 \big\|_{\calQ_B(Y)} < 1$.
\end{proof}

\begin{prop}[{\cite[Lemma 2.2]{Joachim}, \cite[Proposition 2.14]{vdDungen19}}] \label{prop:Fred_to_KasMod}
Let $D \in\Reg_B(Y)$ be a Real odd self-adjoint Fredholm operator. 
For any normalising function $\chi$ for $D$, the triple 
$\big( \C, \, Y_B, \, \chi(D) \big)$ is a Real Kasparov module. 
The $\KK$-class $[D] := [\chi(D)] \in \KKR(\C, B)$ is independent of the choice of normalising function.
\end{prop}

\subsection{Clifford anti-linear Fredholm operators} \label{subsec:Clifford_Fredholm}

Proposition \ref{prop:Fred_to_KasMod} shows that an odd Real self-adjoint Fredholm operator defines a class in the Real $K$-theory group $\KKR(\C,B) \cong \KR_0(B)$. 
We shall more generally be interested in describing elements of the higher $\KR$-theory groups $\KKR(\Cl_{r,s},B) \cong \KR_{r-s}(B)$ in terms of Fredholm operators. 
For this purpose, we consider Fredholm operators which anti-commute with a given representation of the Clifford algebra $\Cl_{r,s}$. 
In the Real setting, it will be more convenient to work with skew-adjoint (instead of self-adjoint) operators. 

So let $X_A$ be a full countably generated and ungraded Real Hilbert $C^*$-module over an ungraded Real $C^*$-algebra $A$. 
  In addition we assume that there is an ungraded and ample representation $\Cl_{r,s} \to \End_A(X)$ 
  (meaning that $\Cl_{r,s} \cap \End_A^0(X) =\{0\}$) 
  with ungraded generators $\{e_1,\ldots, e_r, f_1,\ldots, f_s\}$. 
  We define the \emph{Clifford anti-linear} adjointable operators
  \[
    \End_A^{r,s}(X) = \Big\{ S \in \End_A(X) \, \Big| \, \begin{matrix}  Se_j = -e_j S &  \text{for all } j=1,\ldots, r, \\   S f_k = -f_k S & \text{for all } k=1,\ldots, s \end{matrix}  \Big\},
  \]
  and similarly the \emph{Clifford anti-linear} regular operators
  \[
    \Reg_A^{r,s}(X) = \Big\{ D \in \End_A(X) \, \Big| \, \begin{matrix}  e_j\cdot\Dom D \subset \Dom D \text{ and } D e_j = -e_j D &  \text{for all } j=1,\ldots, r, \\ f_k\cdot\Dom D \subset \Dom D \text{ and }   D f_k = -f_k D & \text{for all } k=1,\ldots, s \end{matrix}  \Big\}.
  \]

Proposition \ref{prop:Fred_to_KasMod} can be adapted straightforwardly to the setting of Clifford anti-linear skew-adjoint Fredholm operators, and we thus obtain:
\begin{cor} \label{cor:skew_Fred_to_KasMod}
Let $T \in \Reg_A^{r,s}(X)$ be a (bounded or unbounded) Real skew-adjoint Fredholm operator on $X_A$. 
For any normalising function $\chi$ for $T$, the triple 
\[
    \Big( \Cl_{s+1, r}, \  X_A \otimes \exterior \C, \   \chi(T) \otimes \rho  \Big)
\]
is a Real Kasparov module, 
where $\big\{ \one \otimes \gamma, f_1 \otimes \rho, \ldots, f_s \otimes \rho, e_1 \otimes \rho, \ldots, e_r \otimes \rho \big\}$
are the generators of the graded $\Cl_{s+1, r}$ representation. 
The resulting $\KK$-class 
\[
\class{T} := \big[ \chi(T) \otimes \rho \big] \in \KKR( \Cl_{s+1, r}, A) 
\]
is independent of the choice of normalising function. 
\end{cor}
\begin{proof}
Observe that, since $\chi$ is an odd function, $\chi(T)$ also anti-commutes with the Clifford generators $\{e_1,\ldots, e_r, f_1,\ldots, f_s\}$. The statement then follows from Proposition \ref{prop:Fred_to_KasMod}. 
\end{proof}

For a $\Z_2$-graded $C^*$-module $Y_B$ with a $\Z_2$-graded representation $\Cl_{r,s} \to \End_B(Y)$, we can analogously define 
$\End_B^{r,s}(Y)$ and $\Reg_B^{r,s}(Y)$.

\subsection{The Clifford shuffle map for Clifford anti-linear operators} \label{subsec:Clifford_Fredholm_Shuffle}

Let us fix a $\Z_2$-graded $C^*$-module $Y_B$ with a $\Z_2$-graded Clifford representation $\Cl_{s,r} \to \End_B(Y)$ 
(note the reversed $(s,r)$). 
In order to define a Van Daele $K$-theory index for Clifford anti-linear Fredholm operators, we need a systematic procedure to 
relate Clifford anti-linear operators in $\End_B^{s,r}(Y)$ to the tensor product $\End_B(Y) \hox \Cl_{r,s}$. 
Note that in particular $\Cl_{s,r} \hox \Cl_{r,s} \cong M_{2^{r+s}} \subset \End_{B}(Y) \hox \Cl_{r,s}$.

\begin{notation} \label{notation:P^rs_and_sigma_e^rs}
Let $Y_B$ be a $\Z_2$-graded Hilbert $C^*$-module with $\Cl_{s,r} \to \End_B(Y)$ a $\Z_2$-graded Clifford representation. 
\begin{itemize}
  \item $P^{r,s}\in \Cl_{s,r} \hox \Cl_{r,s} \subset \End_B(Y) \hox \Cl_{r,s}$ denotes a minimal (rank one) even projection. 
  \item If $e \in \End_B^{s,r}(Y)$ is an odd self-adjoint unitary (OSU), we define 
  \begin{align} \label{eq:sigma_e-rs_defn}
  &\sigma_e^{r,s}: \Reg_B^{s,r}(Y) \to    \Reg_B(Y) \hox \Cl_{r,s}   
  &\sigma_e^{r,s}(x) =  P^{r,s}(x \hox 1_{r,s}) + (\one -P^{r,s})(e \hox 1_{r,s}) .
\end{align}
\end{itemize}
\end{notation}

Note that $\sigma_e^{r,s}(e) = e \hox 1_{r,s}$.
If  $\End^{s,r}_B(Y)$ has no odd self-adjoint unitaries, then we can  consider
$\End^{s,r}_B(Y)\hox\Cl_{1,1}=\End^{s,r}_B(Y\hox \exterior{\C})$, which contains $\one \hox\gamma$.  
This procedure is consistent with the definition of Van Daele $K$-theory for algebras which are not balanced graded: see Appendix \ref{sec:appendix_DK}.

\begin{rmk}
\label{rmk:AMZ-proj}
An explicit construction of a minimal even projection $P^{r,s}$ is given by \cite[\S3.3]{AMZ}. 
Writing the generators of $\Cl_{s,r} \subset \End_B(Y)$ and $\Cl_{r,s}$ as $\tilde{\gamma}_j,\, \tilde{\rho}_k$ and  $\gamma_j, \,\rho_k$ respectively, we set
\[
    2p_j - \one = \tilde{\rho}_j \hox  {\gamma}_j,  \quad \qquad 2q_k - \one = \tilde{\gamma}_k \hox  {\rho}_k, \qquad 
    j=1,\ldots, r, \, k = 1,\ldots, s.
\]
and define 
$P^{r,s} = p_1\ldots p_r q_1 \ldots q_s \in\Cl_{s,r} \hox \Cl_{r,s} \subset \End_B(Y) \hox \Cl_{r,s}$.
\end{rmk}

The following lemma, which is a  special case of the general results in \cite[\S3.3]{AMZ}, 
guarantees that we can map odd self-adjoint unitaries in $\End_B^{s,r}(Y)$ to odd self-adjoint unitaries 
in $\End_B(Y) \hox \Cl_{r,s}$ without loss of topological information. Odd self-adjoint unitaries in 
$\End_B(Y) \hox \Cl_{r,s}$ can then be  used to construct Van Daele $K$-theory classes. 
See Appendix \ref{sec:appendix_DK} for an overview of Van Daele $K$-theory.

\begin{lemma}[{\cite[Lemma 4, 5, 6]{AMZ}}] \label{lem:AMZ-OSU-iso}
Let $B$ be a $\Z_2$-graded algebra and $\Cl_{s,r} \to \Mult(B) = \End_B(B)$ a $\Z_2$-graded Clifford representation. 
Fix also a basepoint odd self-adjoint unitary $e \in \End_B^{s,r}(B)$. 
\begin{enumerate}
  \item The map 
  \[
      \End_B^{s,r}(B) \ni x \mapsto P^{r,s}(x \hox 1_{r,s}) \in P^{r,s}( \Mult(B) \hox \Cl_{r,s} )P^{r,s}
  \]
  gives a bijective correspondence between odd self-adjoint unitiaries.
  \item There is a $\Z_2$-graded isomorphism $M_{2^{r+s}} \big( P^{r,s}( \Mult(B) \hox \Cl_{r,s} ) P^{r,s} \big) \cong \Mult(B) \hox \Cl_{r,s}$.
  \item Let $x \in \End_B^{s,r}(B)$ be an odd self-adjoint unitary with $x - e \in B$. 
  Then applying the maps from parts (1) and (2), there is a well-defined Van Daele $K$-theory class in $\DK(B \hox \Cl_{r,s})$, which is represented by 
  \[
        \big[ \sigma_{e}^{r,s} (x) \big] - \big[ e \hox 1_{r,s} \big] = \big[ P^{r,s}(x \hox 1_{r,s}) + (\one -P^{r,s})(e \hox 1_{r,s}) \big] - \big[ e \hox 1_{r,s} \big]  
        \in \DK( B \hox \Cl_{r,s} ).
  \]
 \end{enumerate}
 \end{lemma}

\begin{cor} \label{cor:sigma_also_gives_DK_hom}
Let $B$ be a $\Z_2$-graded algebra, $\Cl_{s,r} \to \Mult(B)$ a $\Z_2$-graded Clifford representation,  
and $e \in \End_B^{r,s}(B)$ an odd self-adjoint unitary. 
There is a well-defined homomorphism of basepointed Van Daele $K$-theory groups  
$(\sigma_e^{r,s})_\ast : \DK_e( B ) \to \DK_{e\hox  1_{r,s}}(  B \hox \Cl_{r,s} )$.
\end{cor}

Lemma \ref{lem:AMZ-OSU-iso} and Corollary \ref{cor:sigma_also_gives_DK_hom} suggest that one could define an abelian 
group $G^{s,r}(B)$ using homotopy classes of odd-self-adjoint unitaries in $\End^{s,r}_{M_n(B)}(M_n(B))$, and that the induced map
$(\sigma_e^{r,s})_\ast:G^{s,r}(B)\to \DK_{e\hox  1_{r,s}}(  B \hox \Cl_{r,s} )$ would then be an isomorphism. We do not pursue this line of thought here, but
see~\cite[Theorem 3]{AMZ} 
for a closely related construction.

The map $\sigma_e^{r,s}\colon \Reg_B^{s,r}(Y) \to    \Reg_B(Y) \hox \Cl_{r,s}$ is also related to 
the Clifford shuffle isomorphisms in $\KKR$-theory,   $\shuffle_{\KK}^{r,s}:\KKR(\Cl_{s,r},B) \xrightarrow{\simeq} \KKR(\C,B\hox\Cl_{r,s})$, 
see Eq. \eqref{eq:Clifford_shuffle} in the appendix. We will consider this map in the $\Z_2$-graded self-adjoint and 
ungraded skew-adjoint setting.

\begin{lemma}
\label{lem:cliff-shuffle}
\begin{enumerate}
  \item Let $Y_B$ be a $\Z_2$-graded Hilbert $C^*$-module with $\Cl_{s,r} \to \End_B(Y)$ a $\Z_2$-graded Clifford representation. 
If  $D\in\Reg^{s,r}_B(Y)$ is odd, self-adjoint, and Fredholm, then for any normalising function $\chi$ for $D$, 
$\big( \Cl_{s,r}, \, Y_B, \, \chi(D) \big)$ is a Kasparov module and for any basepoint self-adjoint unitary $e \in \End_B^{s,r}(Y)$, 
\begin{align*}
\shuffle^{r,s}_{\KK}\big( \big[ \big( \Cl_{s,r}, \, Y_B, \, \chi(D) \big) \big] \big) &=
    \big[ \big(\C, \,  (Y \hox \Cl_{r,s})_{B \hox \Cl_{r,s}}, \, \sigma^{r,s}_e( \chi(D) ) \big) \big]  \\
  &= \big[ \big(\C, \, P^{r,s}(Y \hox \Cl_{r,s})_{B \hox \Cl_{r,s}}, \, P^{r,s}( \chi(D) \hox 1_{r,s} ) \big) \big].
\end{align*}
 \item Let $X_A$ be an ungraded module with $\Cl_{r,s} \to \End_A(X)$ an ungraded Clifford representation. 
 If  $T\in \Reg^{r,s}_A(X)$ is skew-adjoint, Real, and Fredholm, then for any basepoint skew-adjoint unitary 
 $J \in \End_A^{r,s}(X)$, 
 \begin{align*}
    \shuffle^{r,s+1}_{\KK}\big( \class{T} \big)  & =  \big[ \big(\C, \,  (X_A\ox\Cl_{r,s+1})_{A\otimes \Cl_{r,s+1}}, \, \sigma^{r,s}_{J \otimes \rho}( \chi(T)\ox\rho )\big) \big] \\
      &= \big[ \big(\C, \, P^{r,s} (X_A\ox\Cl_{r,s+1})_{A\otimes \Cl_{r,s+1}}, \,  P^{r,s}( \chi(T) \otimes \rho)\big) \big],
 \end{align*}
where $\class{T} \in \KKR(\Cl_{s+1,r}, A)$ is  from Corollary \ref{cor:skew_Fred_to_KasMod} and $\chi$ is any normalising function for $T$. 
\end{enumerate}
\end{lemma}

\begin{proof}
(1) We let $F = \chi(D)$ for brevity. The proof that $\big( \Cl_{s,r}, \, Y_B, \, F \big)$ is a Kasparov module is
analogous to Corollary \ref{cor:skew_Fred_to_KasMod}.
Following the definition of $\shuffle_{\KK}^{r,s}$ from Eq. \eqref{eq:Clifford_shuffle}, 
the  external product of $\big(\Cl_{s,r},\, Y_B, \, F\big)$ with $\big(\Cl_{r,s}, \, \Cl_{r,s}, \, 0\big)$, which represents 
$\mathrm{Id}_{\Cl_{r,s}} \in \KKR(\Cl_{r,s}, \Cl_{r,s} )$, yields
\[
 \big( \Cl_{r+s,r+s}, \, Y_B\hox \Cl_{r,s}, \, F \hox 1_{r,s} \big).
\]
Taking the Kasparov product with the Morita equivalence provided by the (dual of the) spinor representation $S$, 
$\big(\C, \, \bar{S}_{\Cl_{r+s,r+s}}, \, 0)= \big(\C, \, P^{r,s}\Cl_{r+s,r+s}, \,0 \big)$, yields
\[
\big(\C, \, P^{r,s}(Y \hox \Cl_{r,s})_{B \hox \Cl_{r,s}}, \, P^{r,s}(F \hox 1_{r,s}) \big) .
\]
At the level of $\KK$-classes, 
\begin{align*}
  &\big[ \big(\C, \, P^{r,s}(Y \hox \Cl_{r,s})_{B \hox \Cl_{r,s}}, \, P^{r,s}(F \hox 1_{r,s}) \big)  \big] \\
  &\qquad = \big[\big(\C, \, (Y \hox \Cl_{r,s})_{B \hox \Cl_{r,s}}, \, P^{r,s}(F \hox 1_{r,s}) + (\one - P^{r,s} ) (e \hox 1_{r,s} ) \big) \big] 
\end{align*}
as  $\big( \C, \, (\one -P^{r,s}) ( Y \hox \Cl_{r,s} )_{B \hox \Cl_{r,s} }, \,  (\one - P^{r,s} ) (e \hox 1_{r,s} ) \big)$ is a degenerate Kasparov module, see Definition \ref{defn:real-kasmod}.

(2) We use the fact that $\shuffle^{r,s+1}_{\KK} = \shuffle^{r,s}_{\KK} \circ \shuffle^{0,1}_{\KK}$ to apply the shuffle map in stages. 
To apply $\shuffle^{0,1}_{\KK}$, we let $P^{1,1}=\frac{1}{2}(\one+\gamma\hox \tilde{\rho})$, which acts on $\exterior{\C}\hox\Cl_{0,1}$ and with $\tilde{\rho}$  
the generator of $\Cl_{0,1}$. 
As $P^{1,1}$ commutes with $\rho \hox 1_{0,1}$,  there is a unitary isomorphism $U:P^{1,1}(\exterior{\C}\hox\Cl_{0,1})\to \Cl_{0,1}$ of right $\Cl_{0,1}$ modules such that
\[
UP^{1,1}(\rho\hox 1_{0,1})P^{1,1}U^*=\tilde{\rho}.
\]
Changing notation for the generator of $\Cl_{0,1}$ from $\tilde{\rho}$ to $\rho$ yields 
\[
  \shuffle^{0,1}_{\KK}\big( \class{T} \big)  = \big[ \big(\Cl_{s,r}, \, (X_A \otimes \Cl_{0,1})_{A \otimes \Cl_{0,1 }} , \,  \chi(T)\ox\rho \big) \big].
\]
We now use part (1) to apply $\shuffle^{r,s}_{\KK}$ to this class, where we obtain  
\begin{align*}
 \shuffle^{r,s}_{\KK}\circ  \shuffle^{0,1}_{\KK}\big( \class{T} \big) 
 &= \big[ \big( \C, \, P^{r,s} ( X_A \otimes \Cl_{0,1} \hox \Cl_{r,s} )_{A \otimes \Cl_{0,1} \hox \Cl_{r,s}}, \, P^{r,s}( \chi(T) \otimes \rho \hox 1_{r,s}) \big) \big] \\
 &= \big[ \big( \C, \,  (X \otimes \Cl_{r,s+1} )_{A \otimes \Cl_{r,s+1}} , \, P^{r,s}( \chi(T) \otimes \rho )+ (\one -P^{r,s})( J \otimes \rho) \big) \big] ,
 \end{align*}
where we have applied the isomorphism  $\Cl_{0,1} \hox \Cl_{r,s} \cong \Cl_{r,s+1}$, written $\rho$ for $\rho_1= \rho \hox 1_{r,s}$,  
and added a degenerate Kasparov module.
\end{proof}

The Clifford shuffle map $\sigma_{e}^{r,s}$ also behaves well with respect to the boundary map in $\DK$-theory (see \S\ref{subsec:DK_bdry}).

\begin{lemma} \label{lemma:sigma^rs_commutes_with_boundary}
Let $Y_B$ be a $\Z_2$-graded Hilbert $C^*$-module with $\Cl_{s,r} \to \End_B(Y)$ a $\Z_2$-graded and ample Clifford representation. 
If $x \in \End_B^{s,r}(Y)$ is odd and self-adjoint with $\one -x^2 \in \End_B^0(Y)$, then for any basepoint odd self-adjoint unitary $e \in \End_B^{s,r}(Y)$, 
\begin{equation} \label{eq:shuffle_boundary_commutes}
     \delta \circ (\sigma_{q(e)}^{r,s})_\ast \big( [q(x)] - [q(e)] \big) = (\sigma_{\one \hox \gamma}^{r,s})_\ast \circ \delta \big( [q(x)] - [q(e)] \big) 
     \in \DK\big( \End_B^0(Y) \hox \Cl_{r+1,s} \big),
\end{equation}
where the boundary map on the left-hand side is $\delta\colon \DK\big( \calQ_B(Y) \hox \Cl_{r,s}  \big) \to \DK\big(\End^0_B(Y) \hox \Cl_{r+1,s}\big)$ 
and the boundary map on the right-hand side is
$\delta \colon \DK( \calQ_B(Y) ) \to \DK\big( \End_B^0(Y ) \hox \Cl_{1,0} \big)$.
\end{lemma}
\begin{proof}
Let $\big\{\tilde{\gamma}_1 , \ldots, \tilde{\gamma}_s,  \tilde{\rho}_1\ , \ldots, \tilde{\rho}_r  \big\}$ 
denote the $\Cl_{s,r}$-generators in $\End_B(Y)$.
Starting with the right-hand side of Eq. \eqref{eq:shuffle_boundary_commutes}, applying Lemma \ref{lemma:delta_formula}, and simplifying by the Taylor expansion,
\begin{align*}
  \delta\big( [q(x)] - [q(e)] \big) &= \big[ -\exp( \pi  {x} \hox \gamma )(\one \hox \gamma) \big] - \big[ -\exp( \pi e \hox \gamma )(\one \hox \gamma) \big] \\
    &= \big[ -(\cos( \pi  {x} ) \hox 1_{1,0} - \sin( \pi {x} ) \hox \gamma)(\one \hox \gamma) \big] - [ -(\cos(\pi) + \sin(\pi)(e \hox \gamma) ) (\one \hox \gamma)] \\
    &= \big[ -\cos( \pi {x} ) \hox \gamma - \sin( \pi {x} )\hox 1_{1,0} \big] - [  \one \hox \gamma] .
\end{align*}
We see that $-\cos( \pi  {x} ) \hox \gamma - \sin( \pi {x} )\hox 1_{1,0}$ and $\one \hox \gamma$  are $\Cl_{s,r}$-anti-linear with respect to 
$\tilde{\gamma}_j\hox 1_{1,0}$ and $\tilde{\rho}_k \hox 1_{1,0}$ for $j=1,\ldots, s, \, k=1,\ldots, r$. 
In particular,  
$(\sigma_{\one \hox \gamma}^{r,s})_\ast \circ \delta \big( [q(x)] - [q(e)] \big)$ is well-defined.

For the left-hand side of Eq. \eqref{eq:shuffle_boundary_commutes}, we  use
that $P^{r,s}$  commutes with $x \hox 1_{r,s} \hox \gamma$ and $e \hox 1_{r,s}\hox \gamma$. Then 
\begin{align*}
   \delta \circ (\sigma_{q(e)}^{r,s})_\ast \big( [q(x)] - [q(e)] \big)  &=
    \big[ -e^{ \pi P^{r,s}(x \hox 1_{r,s} \hox \gamma) + \pi( \one - P^{r,s})(e \hox 1_{r,s} \hox \gamma) }(\one \hox 1_{r,s} \hox \gamma) \big] \\
   &\hspace{6cm}   - \big[ -e^{\pi e \hox 1_{r,s} \hox \gamma }(\one \hox 1_{r,s} \hox \gamma) \big] \\
  &\hspace{-4.8cm}= \big[ - P^{r,s}\big( \cos( \pi x) \hox 1_{r,s} ) \hox \gamma + \sin( \pi x )\hox 1_{r,s}  \hox 1_{1,0} \big) - (\one - P^{r,s})(\one \hox 1_{r,s} \hox \gamma) \big] \
    - \big[ \one \hox 1_{r,s} \hox \gamma \big] \\
  &\hspace{-4.8cm}= \big[ - P^{r,s}\big( \cos( \pi x)  \hox \gamma + \sin( \pi x )\hox 1_{r+1,s}   \big) - (\one - P^{r,s})(\one \hox \gamma) \big] - \big[ \one   \hox \gamma \big] \\
  &\hspace{-4.8cm}= (\sigma_{\one \hox \gamma}^{r,s})_\ast \circ \delta \big( [q(x)] - [q(e)] \big),
\end{align*}
where we have identified  $\Cl_{r,s} \hox \Cl_{1,0}  \cong\Cl_{r+1, s}$ and written $\gamma$ for $\gamma_1 =   1_{r,s} \hox \gamma$. 
\end{proof}

\section{Index theory on Real Hilbert modules} \label{sec:Fredholm_index}

In this section, we will define the Fredholm index of a Real skew-adjoint Fredholm operator, taking values in Van Daele $K$-theory (or $\DK$-theory, for short). 
An introduction to Van Daele $K$-theory is given in Appendix \ref{sec:appendix_DK}, 
and various concrete isomorphisms between Kasparov's Real $\KKR$-theory and Van Daele's $\DK$-theory are presented in Appendix \ref{sec:appendix_KK_DK_isos}.
As in \S\ref{subsec:Clifford_Fredholm}, we are mainly interested in Real skew-adjoint Fredholm operators which anti-commute with the generators of a Clifford representation.

\subsection{The complex analytic index in \texorpdfstring{$\DK$}{DK}-theory} \label{subsec:Complex_index_and_DK}

We first recall the construction of the Fredholm index in complex $K$-theory. 
For a Fredholm operator on a complex Hilbert space, one defines $\Index_0(T) := \dim\Ker(T) - \dim\Coker(T) \in \Z$. 
Using the identification $K_0(\C) \simeq \Z$, we may rewrite the index as a formal difference of projections
\[
\Index_0(T) := \big[ P_{\Ker(T)} \big] - \big[ P_{\Coker(T)} \big] \in K_0(\C) .
\]
Generalising the index to complex $C^*$-modules is complicated by the fact that the kernel and cokernel of a Fredholm operator on a 
$C^*$-module might fail to be complemented (in which case, the projections onto the kernel or cokernel do not exist). 

\paragraph{Even index}
The work of Mingo \cite{Mi} and Exel \cite{Exel} shows that, given an adjointable Fredholm operator 
$T$ on a $C^*$-module $X_A$, there exists a suitable compact perturbation $\widetilde{T}$ of $T$ on $X_A\oplus A^n$ for some $n$ such that 
$\widetilde{T}$ has closed range. Having closed range ensures that the kernel and cokernel of $\widetilde{T}$ are complemented, and one may then define the Fredholm index as 
\begin{equation}
\Index_0(T) := \big[ P_{\Ker(\widetilde{T})} \big] - \big[ P_{\Coker(\widetilde{T})} \big] \in K_0(A) ,
\label{eq:index-proj}
\end{equation}
which turns out to be independent of the choice of perturbation $\tilde{T}$ and so is well-defined. 

An alternative definition which avoids the technicalities above is to observe that the image $q(T)\in \calQ_A(X)$ of $T$ in the Calkin algebra is invertible. By taking the unitary phase $q(T)|q(T)|^{-1}$ we obtain a class $[q(T)]\in K_1(\calQ_A(X))$. Applying the boundary map $\delta:K_1(\calQ_A(X))\to K_0(\End^0_A(X))$ and the Morita invariance 
$\calM_X^K:K_0(\End^0_A(X)) \xrightarrow{\simeq} K_0(A)$ yields a class
\begin{equation}
\Index_0(T)=\calM_X^K\circ\delta([q(T)]).
\label{eq:index-scheme}
\end{equation}
Since we may assume without loss of generality that $T$ has closed range, a standard argument, e.g. \cite[Proposition 4.8.10]{HR}, 
shows that \eqref{eq:index-proj} and \eqref{eq:index-scheme} agree.
Below we interpret \eqref{eq:index-scheme} in Van Daele $K$-theory for complex algebras, and how \eqref{eq:index-scheme} also yields the odd index in complex and Van Daele $K$-theory.

\paragraph{The even  index in \texorpdfstring{$\DK$}{DK}-theory}
To define the index of $T:X_A\to X_A$ in Van Daele $K$-theory, described by classes of odd self-adjoint unitaries, we 
need to consider $\Z_2$-graded spaces and operators. 
We take the $\Z_2$-graded module   $X_A \otimes \exterior{\C}$ and 
$\tilde{T} = \tfrac12(T+T^*) \otimes \gamma + \tfrac12(T- T^*) \otimes \rho \in \End_A(X_A) \otimes \Cl_{1+1}$, 
which is odd and self-adjoint. We also note that by the isomorphism $\Cl_{1+1} \cong M_2(\C)$ (with $\gamma \cong \sigma_1$ and $\rho \cong -i\sigma_2$),
\[
\tilde{T} = \tfrac12(T+T^*) \otimes \sigma_1 + \tfrac12(T- T^*)\otimes (-i\sigma_2) = \begin{pmatrix} 0 & T^* \\ T & 0 \end{pmatrix}.
\]
We now follow a  procedure analogous to \eqref{eq:index-scheme} to define $\Index_{2,1}(T) \in DK(A \otimes \Cl_{2+1})$. The image of $\tilde{T}$ in the Calkin algebra
$q(\tilde{T}) \in \calQ_A(X) \otimes \Cl_{1+1}$ is invertible, self-adjoint, odd, and so defines an element 
$[q(\tilde{T}) ] - [\one \otimes \gamma] \in DK( \calQ_A(X) \otimes \Cl_{1+1})$
(where $[q(\tilde{T}) ]$ is represented by the odd self-adjoint unitary $q(\tilde{T}) \big|q(\tilde{T})\big|^{-1}$).

We now compose with the boundary map $\delta: \DK( \calQ_A(X) \otimes \Cl_{1+1}) \to \DK( \End_A^0(X) \otimes \Cl_{2+1})$ (Lemma \ref{lem:other-bdry})
and the Morita invariance $\calM^{\DK}_X: \DK\big( \End_A^0(X)  \otimes \Cl_{2+1} \big) \xrightarrow{\simeq} \DK\big( A  \otimes \Cl_{2+1} \big)$ 
(cf. Appendix \ref{sec:DK_Morita})
to define 
\[
   \Index_{2,1} ( \tilde{T} ) := \calM^{\DK}_X \circ \delta \big( [q(\tilde{T}) ] - [\one \otimes \gamma]  \big) \in \DK(A \otimes \Cl_{2+1} ).
\]

To show the $\DK$-valued index represents the same information as our $K_0(A)$-valued index, we recall the natural isomorphisms
from $\Upsilon_A^{r,s} : \DK(A \otimes \Cl_{r,s}) \to K_{1+s-r}(A)$ from  Appendix \ref{subsec:appendix_DK_KR_isos}.

\begin{lemma} 
Let $T \in \End_A(X)$ be Fredholm. Then 
the isomorphism $\Upsilon_A^{2,1} : \DK(A \otimes \Cl_{2,1} ) \to K_0(A)$ is such that 
\[
   \Upsilon_A^{2,1} \big(  \Index_{2,1} ( \tilde{T} ) \big) = \Index_0(T) .
\]
\end{lemma}
\begin{proof}
We first note that $q(T) \in \calQ_A(X)$ is homotopic to $q(T)|q(T)|^{-1}$, which will not affect 
$\Index_0(T)$ or $\Index_{2,1}(\tilde{T}))$. So we can assume without loss of generality that $T$ is an essential unitary, 
$\one - T^*T, \one -TT^* \in \End_A^0(X)$. Passing to matrices and taking a finite rank perturbation if 
necessary, we can furthermore reduce to the case that $T$ is a partial isometry with $T^*T$ and $TT^*$ projections in $\one + \End_A^0(X)$. 
With these simplifications, the complex $K$-theory index can be easily computed:
\[
   \Index_0(T)  = \calM^K_X \big( \big[ \one - T^*T \big] - [ \one - TT^* ]  \big) \in K_0(A),
\]
 see~\cite[\S17.3.12]{WeggeOlsen} for example.
Considering the Van Daele index and $\tilde{T}$, a similar computation as in the proof of Lemma \ref{lemma:sigma^rs_commutes_with_boundary}   gives that 
\[
   \Index_{2,1} ( \tilde{T} )=  \calM^{\DK}_X \big( \big[ -\cos\big(\pi \tilde{T} \big) \hox \gamma_2 - \sin\big(\pi  \tilde{T} \big)  \hox \one \big] - \big[ \one \hox \gamma_2 \big] \big) 
    \in \DK( A \otimes \Cl_{2,1} ).
\]
Because $T$ is a partial isometry, 
\[
   -\cos\big(\pi \tilde{T} \big)  = \begin{pmatrix} \one - 2T^*T & 0 \\ 0 & \one - 2TT^* \end{pmatrix}, \qquad  -\sin \big(\pi \tilde{T} \big) = 0
\]
so that 
\[
   \Index_{2,1} ( \tilde{T} )= \calM^{\DK}_X \left( \left[ \begin{pmatrix} \one - 2T^*T & 0 \\ 0 & \one - 2TT^* \end{pmatrix} \hox \gamma_2 \right]  - \big[ \one \hox \gamma_2 \big]  \right).
\]
Next we note that $\Upsilon_A^{2,1} = \Upsilon_A^{1,0} \circ (\calM_{1,1}^{\DK})^{-1}$ with 
$\calM_{1,1}^{\DK}: \DK( A \otimes \Cl_{1,0} ) \xrightarrow{\simeq} \DK( A \otimes \Cl_{2,1} )$ the Clifford stability of $\DK$ 
from Lemma \ref{lemma:DK_Clifford_stability}, where 
\begin{align*}
   (\calM^{\DK}_{1,1})^{-1} \left( \left[ \begin{pmatrix} \one - 2T^*T & 0 \\ 0 & \one - 2TT^* \end{pmatrix} \hox \gamma_2 \right]  - \big[ \one \hox \gamma_2 \big]  \right) 
     = \big[ (\one - 2T^*T) \otimes \gamma \big] - \big[ (\one - 2TT^*) \otimes \gamma \big] .
\end{align*}
Finally, we use the compatibility of the isomorphism $\Upsilon^{r,s}$ with Morita invariance (Lemma \ref{lemma:Upsilon_Morita_commute_complex})  
and apply $\Upsilon^{1,0}_{\End_A^0(X)}$, where $\Upsilon^{1,0}_{\End_A^0(X)} \big([ (2p-\one) \otimes \gamma ] - [ (2q-\one)\otimes \gamma] \big) = [p]-[q]$. Therefore 
\begin{align*}
    \Upsilon_A^{2,1} \big(  \Index_{2,1} ( \tilde{T} ) \big) &= \Upsilon_A^{1,0} \circ (\calM_{1,1}^{\DK})^{-1} \circ 
     \calM_X^{K} \left( \left[ \begin{pmatrix} \one - 2T^*T & 0 \\ 0 & \one - 2TT^* \end{pmatrix} \hox \gamma_2 \right]  - \big[ \one \hox \gamma_2 \big]  \right)  \\
     &=  \calM_X^{K} \circ \Upsilon_{\End_A^0(X)}^{1,0} \big( \big[ (\one - 2T^*T) \otimes \gamma \big] - \big[ (\one - 2TT^*) \otimes \gamma \big] \big) \\
      &= \calM^K_X \big(  \big[ \one - T^*T \big] - \big[ \one - TT^* \big] \big) = \Index_0(T).  \qedhere
\end{align*}
\end{proof}

\paragraph{Odd index}

Although less common in the literature (largely due to the fact that $K_1(\C)$ is trivial),   one may similarly define an \emph{odd index} 
of self-adjoint and regular Fredholm operators  $D:X_A\to X_A$. 
Given a normalising function $\chi$ for $D$, we find that
$q\big( \chi(D) \big) \in \calQ_A(X)$ is a self-adjoint unitary and the $K$-theory class 
$\big[ \tfrac{1}{2}( \one + q( \chi(D) ) ) \big] \in K_0( \calQ_A(X))$ is well-defined. Analogous to the even index, 
the  odd $K$-theory index can be defined by the boundary map and Morita invariance, 
\[
   \Index_1(D) = \calM^K_X  \circ \delta \big( \big[ \tfrac{1}{2}( \one + q( \chi(D) ) ) \big]  \big) 
     = \calM^K_X \big( \big[ \exp( i \pi (\chi(D) + 1) ) \big] \big) \in K_1(A),
\]
see~\cite[Proposition 17.5.6]{Blackadar}, for example.

To describe this index in Van Daele $K$-theory, we again need to consider graded modules and operators. 
Also, as preparation for the real indices, we will consider skew-adjoint Fredholm operators. Namely, 
we take $(X_A \otimes \Cl_{0+1})_{A \otimes \Cl_{0+1}}$ and $iD \otimes \rho$, which is self-adjoint and Fredholm. 
Then $q( i \chi( D))  \otimes \rho \in \calQ_A(X) \otimes \Cl_{0+1}$  
is an odd self-adjoint unitary and 
\begin{align} \label{eq:complex_DK_odd_index}
   \Index_{1,1} (i D) &:= \calM^{\DK}_X \circ \delta \big( \big[ q( i \chi( D))  \otimes \rho \big] - [ \one \otimes i \rho] \big) \nonumber \\
      &\hspace{-1cm}=  \calM^{\DK}_X \left( \left[ \begin{pmatrix} 0 & -\exp( -i  \pi \chi(D) ) \\ -\exp( i \pi \chi(D)) & 0 \end{pmatrix} \right] 
       - \left[ \begin{pmatrix} 0 & \one \\ \one & 0 \end{pmatrix} \right] \right)  \in \DK(A \otimes \Cl_{1+1}) ,
\end{align}
see the proof of Lemma \ref{lem:index-sinh-cosh} below.
Applying the isomorphism $\Upsilon_A^{1,1} : \DK(A \otimes \Cl_{1+1}) \to K_{1}(A)$ from Proposition \ref{prop:iso_DK_odd}, 
which commutes with Morita invariance (Lemma \ref{lemma:Upsilon_Morita_commute_complex}), we have that 
\begin{align} \label{eq:K_1-DK-index-equality}
   \Upsilon_A^{1,1} \big( \Index_{1,1}(iD) \big) 
   &= \calM^K_X \circ \Upsilon^{1,1}_{\End_A^0(X)} \left( \left[ \begin{pmatrix} 0 & -\exp( -i \pi \chi(D) ) \\ -\exp( i \pi \chi(D)) & 0 \end{pmatrix} \right] 
       - \left[ \begin{pmatrix} 0 & \one \\ \one  & 0 \end{pmatrix} \right] \right) \nonumber  \\
    &= \calM^K_X \big( \big[  -\exp( i \pi \chi( D) ) \big] \big)
    = \calM^K_X \big( \big[  \exp( i \pi (\chi( D) + 1) ) \big] \big)  
    =  \Index_1(D).
\end{align}

We summarise this subsection.

\begin{prop} \label{prop:complex_index}
\begin{enumerate}
 \item Let $T:X_A\to X_A$ be a Fredholm operator  with closed range. Then 
\[
   \Index_0(T) = \calM^K_X \circ \delta \big( \big[ q(T) \big] \big) = \big[ \Ker(T) \big] - \big[ \Ker(T^*) \big]   =  \Upsilon_A^{2,1} \big( \Index_{2,1}( \tilde{T} )\big) \in K_0(A).
\]
\item Let $D:X_A\to X_A$ be a self-adjoint regular Fredholm operator. Then 
\[
   \Index_1(D) =  \calM^K_X  \circ \delta \big( \big[ \tfrac{1}{2}( \one + q( \chi(D) ) ) \big]  \big) = \Upsilon_A^{1,1} \big( \Index_{1,1}(iD) \big) \in K_1( A).
\]
\end{enumerate}
\end{prop}

In the next subsection, we will generalise the prescription
\[
\Index(T)=\calM\circ\delta[q(T)]
\]
to accommodate Clifford anti-symmetries in a systematic way.

\subsection{The real analytic index} \label{subsec:Real_index_defn}

 We now consider the Real setting, which allows for eight different indices taking values in the Real $K$-theory groups $\KR_p(A)$ ($p=0,\ldots,7$). 
In order to accommodate these eight versions of the Fredholm index, 
we will consider Fredholm operators which anti-commute with a Clifford representation, 
and we will replace $\KR$-theory by Van Daele $K$-theory ($\DK$-theory) as the receptacle for our Fredholm index.
Due to the impossibility of extending the ABS isomorphism to  $C^*$-modules (cf. Appendix \ref{sec:ABS}), 
the index is defined in terms of boundary maps. 

\paragraph{Settings and assumptions}
Here and in the sequel, $X_A$  is a full countably generated and ungraded Real Hilbert $C^*$-module over 
an ungraded Real $C^*$-algebra $A$. 
  We also assume that there is an ungraded and ample representation $\pi:\Cl_{r,s} \to \End_A(X)$, where the adjective ``ample'' means that  
  $\Cl_{r,s} \cap \End_A^0(X) =\{0\}$, with   generators $\{e_1,\ldots, e_r, f_1,\ldots, f_s\}$. Using the ampleness assumption, we abuse notation by denoting the generators of the representation $q\circ\pi:\Cl_{r,s} \to \calQ_A(X)$ by the same symbols.
  
  We further assume that there exists a basepoint skew-adjoint unitary $J\in\End^{r,s}_A(X) \setminus \End_A^0(X)$. 
  This assumption is equivalent to the representation $\pi$ extending to an ample representation $\tilde{\pi}:\Cl_{r,s+1}\to\End_A(X)$, where $\tilde{\pi}(\rho_{s+1})=J$.
  
  Let $\rho$ denote the odd generator of $\Cl_{0,1}$. Then the map 
  $T \mapsto T \otimes \rho$ gives   a one-to-one equivalence between elements in $\Reg^{r,s}_A(X)$ 
  and elements in $\Reg^{s,r}_{A \otimes \Cl_{0,1} }(  X \otimes \Cl_{0,1} )$, where the ($\Z_2$-graded) $\Cl_{s,r}$-representation 
  is generated by $\{ f_1 \otimes \rho, \ldots, f_s \otimes \rho, e_1 \otimes \rho, \ldots, e_r \otimes \rho\}$.
  Following the Clifford shuffle construction from \S\ref{subsec:Clifford_Fredholm_Shuffle}, we take the even minimal projection 
  \begin{equation}  \label{eq:ExplicitP^rs}
    P^{r,s} = \frac{1}{2^{r+s}} \prod_{j=1}^r \big( \one + e_j \otimes \rho\gamma_j\big) \prod_{k=1}^s \big( \one + f_k \otimes \rho \rho_{1+k} \big) 
     \in    \End_A(X)\otimes \Cl_{r,s+1},
  \end{equation}
  where we identify $\rho \cong \rho_1 \in \Cl_{r,s+1}$. Then recalling Lemma \ref{lem:AMZ-OSU-iso}, the map 
  \begin{equation} \label{eq:skew_sigma^{r,s}}    
    \Reg^{r,s}_A(X) \otimes \Cl_{0,1} \ni T \otimes \rho  \xmapsto{\sigma_{J \otimes \rho}^{r,s}} P^{r,s} (T \otimes \rho) + (\one - P^{r,s})(J \otimes \rho)  
    \in \Reg_A(X) \otimes \Cl_{r,s+1}
  \end{equation}
  sends odd self-adjoint unitaries to odd-self adjoint unitaries in a way that is compatible with $\DK$-theory. 
  Taking the quotient map to the Calkin algebra, 
\[
   q_X \otimes \one: \End_{A}(X) \otimes \Cl_{r,s+1} \to \calQ_A(X) \otimes \Cl_{r,s+1} , \qquad 
   (q_X\otimes \one)( P^{r,s} ) = P^{r,s}.
\]
Therefore, if $T \in \Reg_A^{r,s}(X)$ is skew-adjoint and Fredholm, $q_X(\chi(T)) \otimes \rho$ is a $\Cl_{s,r}$-anti-linear odd self-adjoint unitary. 
By Lemma \ref{lem:AMZ-OSU-iso} we obtain a well-defined  Van Daele  $K$-theory class 
\begin{align*}
    &\big[ (q_X \otimes \one) \circ  \sigma_{J\otimes \rho}^{r,s}\big( \chi(T) \otimes \rho \big) \big] - \big[  q_X(J) \otimes \rho \big] \\
   &\quad = \big[ P^{r,s}( q_X( \chi(T) ) \otimes \rho ) + (\one - P^{r,s})( q_X(J) \otimes \rho) \big]  
   - \big[ q_X(J) \otimes \rho \big]  \in \DK(\calQ_A(X) \otimes \Cl_{r,s+1} ).
\end{align*}
Hence we can apply the boundary map $\delta: \DK(\calQ_A(X) \otimes \Cl_{r,s+1} ) \to \DK(\End_A^0(X) \otimes \Cl_{r+1,s+1} )$  associated to the short exact sequence 
\[
  0 \to \End_A^0(X) \otimes \Cl_{r,s+1} \to \End_A(X) \otimes \Cl_{r,s+1} \to \calQ_A(X) \otimes \Cl_{r,s+1} \to 0.
\]

\begin{defn} \label{defn:index-of-Fred}
Let $T \in \Reg^{r,s}_A(X)$ be a (bounded or unbounded) Real skew-adjoint Fredholm operator on $X_A$, $\chi$  a normalising function for $T$, 
and $J \in \End_A^{r,s}(X)$ a skew-adjoint unitary in $\End_A^{r,s}(X)$.
We define the \emph{Fredholm index} of $T$ as  
\begin{align*}
   \Index_{r+1,s+1}(T) &=\calM_{X}^{\DK}\circ\delta^{\DK} 
   \big( \big[ (q_X \otimes \one) \circ  \sigma_{J\otimes \rho}^{r,s}\big( \chi(T) \otimes \rho \big) \big] - \big[  q_X(J) \otimes \rho \big] \big) 
   \in  \DK(A \otimes \Cl_{r+1,s+1} )
\end{align*}
where $\sigma^{r,s}_{J\otimes \rho}: \Reg_A^{r,s}(X)\otimes \Cl_{0,1} \to \Reg_{A}(X) \otimes \Cl_{r,s+1}$ 
is from Eq. \eqref{eq:skew_sigma^{r,s}}, $q_X: \End_A(X) \to \calQ_A(X)$ is the quotient by compacts, and 
$\calM_X^{\DK}: \DK\big( \End_A^0(X) \otimes \Cl_{r+1,s+1} \big) \xrightarrow{\simeq} \DK( A \otimes \Cl_{r+1,s+1})$ 
denotes Morita invariance 
(see Appendix \ref{sec:DK_Morita}). 
\end{defn}

We start by giving a first computation of the index, which also shows that $\Index_{r+1,s+1}(T)$ does not depend on the choice 
of basepoint skew-adjoint unitary $J \in \End_A^{r,s}(X)$.

\begin{lemma}
\label{lem:index-sinh-cosh}
Let $T \in \Reg^{r,s}_A(X)$ be a (bounded or unbounded) Real skew-adjoint Fredholm operator on $X_A$, and let $\chi$ be a normalising function for $T$. 
Then $ \cosh( \pi \chi(T) ) \otimes \gamma +\sinh( \pi \chi(T) ) \otimes \rho$ is an odd self-adjoint unitary in $\End_A^0(X)^{\sim} \otimes \Cl_{1,1}$ that 
is  $\Cl_{s,r}$-anti-linear with respect to the generators $\{ f_1\otimes \rho, \ldots, f_s \otimes \rho, e_1\otimes \rho, \ldots, e_r\otimes \rho \}$, and 
$ \Index_{r+1,s+1}( T )  \in \DK( A \otimes \Cl_{r+1,s+1} )$ is such that 
\begin{align*}
  \Index_{r+1,s+1}( T ) &= \calM_X^{\DK} \big( \big[ \sigma^{r,s}_{\one \otimes \gamma} \big( - \cosh( \pi \chi(T) ) \otimes \gamma - \sinh( \pi \chi(T) ) \otimes \rho \big) \big] - \big[ \one \otimes \gamma \big]  \big) \\
  &\hspace{-2.5cm} =  \calM_{X}^{\DK} \big( \big[ -P^{r,s}(\cosh(\pi \chi(T)) \otimes \gamma + \sinh( \pi \chi(T)) \otimes \rho ) - (\one - P^{r,s})(\one \otimes \gamma) \big] 
   - \big[ \one \otimes \gamma \big]  \big),
\end{align*}
where $\sigma^{r,s}_{\one \otimes \gamma}: \End_A^{s,r}(X \otimes \Cl_{1,1} ) \to \End_A(X) \otimes \Cl_{r+1,s+1}$ is the map from 
Eq. \eqref{eq:sigma_e-rs_defn}.
 Identifying $\End_A^0(X) \otimes \Cl_{1,1} \cong \End_A^0(X) \otimes  M_2(\C)$ with 
 $\gamma \cong \sigma_1$ and $\rho \cong -i\sigma_2$, 
we can write
\begin{align*}
    \Index_{r+1,s+1}( T ) &= \calM_{X}^{\DK} \left(   \left[  \sigma^{r,s}_{\one \otimes \sigma_1} \begin{pmatrix} 0 & - \exp(-\pi \chi(T) ) \\ - \exp(\pi \chi(T)) & 0 \end{pmatrix} \right] - 
    \left[ \begin{pmatrix} 0 & \one \\ \one & 0 \end{pmatrix} \right] \right) \\
     &\hspace{-2cm}=\calM_{X}^{\DK} \left(   \left[ -P^{r,s} \begin{pmatrix} 0 & \exp(-\pi \chi(T) )  \\ \exp(\pi \chi(T)) & 0 \end{pmatrix} - (\one - P^{r,s})\begin{pmatrix} 0 & \one \\ \one & 0 \end{pmatrix} \right] - 
          \left[ \begin{pmatrix} 0 & \one \\ \one & 0 \end{pmatrix} \right]  \right).
\end{align*}
\end{lemma}
\begin{proof}
We write $F := \chi(T)$ for brevity.
By Lemma \ref{lemma:sigma^rs_commutes_with_boundary}, 
\[
   \delta \circ ( \sigma^{r,s}_{q_X(J) \otimes \rho} )_\ast \big(  [ q_X(F) \otimes \rho ] -  [ q_X(J) \otimes \rho ] \big) = 
   (\sigma_{\one \otimes \gamma})_\ast \circ \delta \big(  [ q_X(F) \otimes \rho ] -  [ q_X(J) \otimes \rho ] \big)
\]
and it suffices to work with the simpler class $[ q_X(F) \otimes \rho ] -  [ q_X(J) \otimes \rho ] \in \DK\big( \calQ_A(X) \otimes \Cl_{0,1} \big)$.  
Applying the boundary map and   Lemma \ref{lemma:delta_formula},  
\begin{align*}
   \delta\big( [ q_X(F) \otimes \rho ] -  [ q_X(J) \otimes \rho ] \big) 
   &= \big[ -\exp\big( \pi (F \otimes \rho\gamma) \big) (\one \otimes \gamma) \big] - \big[ - \exp\big( \pi  (J\ox\rho \gamma) \big) (\one \otimes \gamma)  \big] \\
   &= \big[ -\exp\big( \pi (F \otimes \rho\gamma) \big) (\one \otimes \gamma) \big] -  [   (\one \otimes \gamma)   ]
\end{align*}
Using that $(F \otimes \rho\gamma)^{2n} = F^{2n} \otimes \one$ and 
$(F \otimes \rho\gamma)^{2n+1} = F^{2n+1} \otimes \rho\gamma$, we can simplify
\begin{align*}
   \exp\big( \pi (F \otimes \rho\gamma) \big) &= \big( \cosh(\pi F) \otimes 1_{1,1} + \sinh( \pi F) \otimes \rho\gamma \big)(\one \otimes \gamma) 
   = \cosh( \pi F) \otimes \gamma + \sinh(\pi F) \otimes \rho.
\end{align*}
Thus we can represent 
\[
    \delta\big( [ q_X(F) \otimes \rho ] -  [ q_X(J) \otimes \rho ] \big) 
    = \big[ - \cosh(\pi F) \otimes \gamma - \sinh(\pi F) \otimes \rho \big] - [ \one \otimes \gamma],
\]
where $ - \cosh( \pi \chi(T) ) \otimes \gamma - \sinh( \pi \chi(T) ) \otimes \rho$ and $\one \otimes \gamma$ are odd self-adjoint unitaries 
in $\End_A^0(X)^{\sim} \otimes \Cl_{1,1}$ that anti-commute
with the generators $\{ f_1\otimes \rho, \ldots, f_s \otimes \rho, e_1\otimes \rho, \ldots, e_r\otimes \rho \}$ of a $\Z_2$-graded $\Cl_{s,r}$-representation. 
Hence 
\begin{align*}
   \Index_{r+1,s+1}( T )  &= \calM_{X}^{\DK} \circ  (\sigma_{\one \otimes \gamma}^{r,s})_\ast \circ \delta \big(  [ q_X(F) \otimes \rho ] -  [ q_X(J) \otimes \rho ] \big)  \\
   &= \calM_X^{\DK}   \big( \big[ \sigma^{r,s}_{\one \otimes \gamma} \big( - \cosh( \pi \chi(T) ) \otimes \gamma - \sinh( \pi \chi(T) ) \otimes \rho \big) \big] - \big[ \one \otimes \gamma \big]  \big). \qedhere
\end{align*}
\end{proof}

\begin{example}
If $r=s=0$, the map $\sigma_{\one\otimes \gamma}^{r,s}$ is not necessary and we can directly write 
\begin{align*}
  \Index_{1,1}(T) &= \calM_X^{\DK} \big( \big[    - \cosh( \pi \chi(T) ) \otimes \gamma - \sinh( \pi \chi(T) ) \otimes \rho   \big] - \big[ \one \otimes \gamma \big]  \big) \\
    &=  \calM^{\DK}_X \left( \left[ \begin{pmatrix} 0 & -\exp(-\pi \chi(T) )  \\ -\exp(\pi \chi(T)) & 0 \end{pmatrix}\right] 
       - \left[ \begin{pmatrix} 0 & \one \\ \one & 0 \end{pmatrix} \right] \right)  \in \DK(A \otimes \Cl_{1+1}) .
\end{align*}
If $X_A$ is complex, $T= iD$, then recalling  Eq. \eqref{eq:complex_DK_odd_index} and Proposition \ref{prop:complex_index} we recover the complex odd index. 
The Clifford shuffle map $\sigma^{r,s}_{\one\otimes \gamma}$ is used to access the other degrees.
\end{example}

The Clifford shuffle map is responsible for much of the notation, but can not in general be removed. 
The next example shows the effect of the Clifford shuffle when the Clifford representation ``decouples'' from the other data as much as possible.

\begin{example}
Let $r,s$ be such that the spinor representation $S$ of $\Cl_{r,s}$ is $\Z_2$-graded. 
We then take $Y = S \otimes \calH_A$ as a $\Z_2$-graded $C^*$-module, where $A$ is still ungraded, 
and $\Cl_{r,s}$ acts by the spinor representation on $S$ tensored by the identity on $\calH_A$.
 Suppose further that $T=\Gamma\ox\tilde{T}$ where $\Gamma$ is the grading operator of $S$ and $\tilde{T} \in \Reg_A(\calH_A)$ 
 is skew-adjoint and Fredholm. 
 Then $T \in \Reg_A^{r,s}(Y)$ is skew-adjoint and Fredholm.
 In this particular example, 
 \begin{align*}
 &P^{r,s}(Y_A \ox\Cl_{r,s+1}) = \calH_A\ox\Cl_{0,1}\hox\overline{S}, 
 &&P^{r,s}( \Gamma \otimes \tilde{T} \otimes \rho) = \tilde{T} \otimes \rho \hox \one_{\ol{S}} ,
 \end{align*}
  and 
 the Clifford 
 information decouples from the rest of the data. 
 Taking $J \in \End_A(\calH)$ a skew-adjoint unitary, 
 the index can then be computed to be 
\begin{align*}
     \Index_{r+1,s+1}(T) &= \calM^{\DK}_X \circ \delta \big( \big[ q_{\calH_A}( \chi(T) ) \otimes \rho \hox \one_{\ol{S}}  \big] - \big[ q_{\calH_A}(J) \otimes \rho \hox \one_{\ol{S}} \big] \big) \\
       &= \calM^{\DK}_X  \Big( \big[ - \big(\cosh( \pi \chi(T) ) \otimes \gamma + \sinh( \pi \chi(T) ) \otimes \rho\big) \hox \one_{\ol{S}} \big] - \big[ \one_{\calH_A} \otimes \gamma \hox \one_{\ol{S}} \big] \Big).
\end{align*}
\end{example}

\subsection{The index, stabilisation, and \texorpdfstring{$\KKR$}{KKR}-theory}
\label{subsec:ISK}

Recall from Corollary \ref{cor:skew_Fred_to_KasMod} that any Real skew-adjoint Fredholm operator $T \in \Reg^{r,s}_A(X)$ also 
gives a Real Kasparov class $\class{T} := \big[ \chi(T) \otimes \rho \big] \in \KKR( \Cl_{s+1, r}, A)$. 
We will prove that, if $T$ and $T'$ are homotopic in Kasparov theory (i.e., $\class{T} = \class{T'}$), then $T$ and $T'$ have the same Fredholm index. 
It then follows in particular that the Fredholm index is homotopy-invariant. 

First we must address the problem that two such Fredholm operators $T$ and $T'$ may be defined on different $C^*$-modules. In particular $q(T)$ and $q(T')$ will lie in different Calkin algebras, and we will only be able to make meaningful comparison after applying boundary and Morita maps. To get around this issue, we  
 use Kasparov's stabilisation theorem in order to compare operators on the standard $C^*$-module $\calH_A$. 

To understand the analytic index in terms of the $\KKR$-classes $\class{T} = \class{T'} \in \KKR( \Cl_{s+1, r}, A)$, we also 
need to understand the passage between $\KKR$-theory and $\DK$-theory. 
The first and simplest connection is given by the  functorial extension of the quotient map $q:\End_B(\hat\calH_B )\to\calQ_B(\hat\calH_B)$.

\begin{lemma}[{\cite[{\S}2]{Roe04}}] \label{lemma:Phi}
Let $B$ be a $\Z_2$-graded and $\sigma$-unital Real $C^{*}$-algebra. Then for any basepoint 
odd self-adjoint unitary $e \in \Mult( B \hox \calK)$, there is a natural isomorphism $\Phi_B: \KKR(\C, B) \to \DK_e( \calQ_{B \hox \calK} )$ satisfying
\[
\Phi_B\big( [ (\C, \hat{\calH} \hox B, S) ] \big) = [q(S)] \in \DK_e( \calQ_{B \hox \calK} ).
\]
\end{lemma}
Abusing notation, given a basepoint odd self-adjoint unitary $e \in \Mult(B \hox \calK)$, we also denote by $e$ the basepoint 
odd self-adjoint unitary in $\calQ_{B \hox \calK}$.
\begin{proof}
By Kasparov stabilisation,  any class in $\KKR(\C, B)$ can be represented by a Kasparov module 
$\big( \C, \, \hat{\calH}_B, \, S \big)$, where $S=S^* \in \Mult(B \hox \calK)$ is odd and such that  $\one - S^2 \in B \hox \calK$. 
Then $q(S) \in \calQ_{B \hox \calK}$ is an odd self-adjoint unitary, and we claim that the map 
$\Phi_B\big( [ (\C, \hat{\calH} \hox B, S) ] \big) = [q(S)] \in \DK_e( \calQ_{B \hox \calK} )$ is well-defined. 
Indeed, given a homotopy of normalised Kasparov modules $\big( \C, \, \hat{\calH}_B \ox C([0,1]), \, S_\bullet \big)$, 
we obtain a homotopy $q(S_\bullet)$ of odd self-adjoint unitaries and so $[q(S_0)]=[q(S_1)]$.

The map is surjective as  any odd self-adjoint unitary $x \in \calQ_{B \hox \calK}$ will define a Kasparov module $( \C, \, \hat{\calH} \hox B, \tilde{x} )$ 
with $\tilde{x} \in \Mult(B \hox \calK)$ a lift of $x$. 
For injectivity, any trivial odd self-adjoint unitary $x \in \calQ_{B \hox \calK}$ will be stably homotopic to an odd self-adjoint unitary $x_0$ such that 
there exists a lift $\tilde{x}_0 \in \Mult(B \hox \calK)$ which is also an odd self-adjoint unitary. Such a homotopy will then induce a homotopy from 
$(\C, \, \hat{\calH} \hox B , \, \tilde{x} )$ to $(\C, \, \hat{\calH} \hox B , \, \tilde{x}_0 )$, a degenerate Kasparov module.
\end{proof}
\begin{rmk}
\label{rmk:Phi-revealed}
Using the relative rather than base-pointed picture of Van Daele theory, 
the isomorphism $\Phi_B$ sends a Kasparov class $\big[ (\C, \hat{\calH} \hox B, S) \big] \in \KKR(\C,B)$ to the Van Daele class $[q(S)] - [e] \in \DK( \calQ_{B \hox \calK} )$. 
\end{rmk}

Using the Clifford shuffle isomorphism $\shuffle^{r,s+1}_{\KK}: KKR( \Cl_{s+1, r}, A) \xrightarrow{ \simeq} \KKR( \C, A \otimes \Cl_{r, s+1} ) $ 
from Eq. \eqref{eq:Clifford_shuffle}, we adapt the isomorphism $\Phi_B$ to our setting of interest.

\begin{defn} \label{def:Indexed_KK_isomorphisms}
Let $A$ be an ungraded $\sigma$-unital Real $C^*$-algebra.
We define 
\[
\Phi_A^{r,s+1}:\KKR(\Cl_{s+1,r},A)\to \DK(\calQ_{A\ox \calK}  \otimes \Cl_{r,s+1})
\] 
as the composition 
  of isomorphisms
\begin{align*}
  &\Phi_A^{r,s+1}: \KKR( \Cl_{s+1, r}, A) \xrightarrow{ \shuffle^{r,s+1}_{\KK} } \KKR( \C, A \otimes \Cl_{r, s+1} ) \xrightarrow{\Phi_{A \otimes \Cl_{r, s+1} } } 
      \DK( \calQ_{A  \otimes \calK} \otimes \Cl_{r,s+1} ) ,
\end{align*}
where $\shuffle^{r,s+1}_{\KK}$ is from Eq. \eqref{eq:Clifford_shuffle} (see also Lemma \ref{lem:cliff-shuffle}) and 
 we have used that $\DK( \calQ_{A \otimes  \Cl_{r,s+1} \otimes \calK} ) \cong \DK( \calQ_{A  \otimes \calK} \otimes \Cl_{r,s+1} )$. 
\end{defn}

\begin{lemma} \label{lemma:Calkin_class}
Let $F \in \End_A^{r,s}(X)$ be Real and skew-adjoint, such that $\one + F^2 \in \End_A^0(X)$, and let $\class{F} := [F\otimes\rho] \in \KKR(\Cl_{s+1,r},A)$ denote the resulting $\KK$-class (cf.\ Corollary \ref{cor:skew_Fred_to_KasMod}).
Then, for any skew-adjoint unitary $J \in \End_A^{r,s}(\calH_A)$ and stabilisation unitary $U \colon X_A \oplus \calH_A \xrightarrow{ \simeq } \calH_A$,
\begin{align*}
  \Phi_A^{r,s+1}\big( \class{F} \big) 
    &= \big[ P^{r,s} \big( q_{\calH_A}( U(F \oplus J)U^*) \otimes \rho  \big) + (\one -P^{r,s})( q_{\calH_A}(J) \otimes \rho ) \big] - \big[ q_{\calH_A}( J) \otimes \rho  \big] \\
    &= \big[ (q_{\calH_A} \otimes \one) \circ \sigma_{J\otimes \rho}^{r,s}( U(F\oplus J)U^* \otimes \rho) \big] - \big[ q_{\calH_A}( J) \otimes \rho \big] \in \DK\big( \calQ_A(\calH_A) \otimes \Cl_{r,s+1} \big). 
\end{align*}

\end{lemma}
\begin{proof}
We first choose a stabilisation unitary $U \colon X_A \oplus \calH_A \xrightarrow{ \simeq } \calH_A$,
an ample representation $\pi = U( \pi_{X_A} \oplus \pi_{\calH_A} )U^*: \Cl_{r,s} \to \End_A(\calH_A)$ that commutes with the projection $P_X \colon \calH_A \to X_A$ 
and  a basepoint skew-adjoint unitary $J \in \End_A^{r,s}(\calH_A)$.
Since $J\otimes\rho$ gives rise to a degenerate Kasparov module, the class $\class{F} := [F \otimes\rho] \in \KKR(\Cl_{s+1,r},A)$ can equivalently be described by 
\[
  \Big[ \Big( \Cl_{s+1, r}, \ X_A \otimes \exterior \C, \ 
  F \otimes \rho \Big) \Big]= \Big[ \Big( \Cl_{s+1, r}, \ \calH_A \otimes \exterior \C, \ 
  \big( U (F\oplus J) U^* \big) \otimes \rho \Big) \Big] .
\]
To apply $\Phi_A^{r,s+1}$, we need to apply the Clifford shuffle $\shuffle_{\KK}^{r,s+1}$. Lemma \ref{lem:cliff-shuffle} yields
\begin{align*}
\shuffle_{\KK}^{r,s+1}(\class{F})&=\shuffle_{\KK}^{r,s+1}\Big[ \Big( \Cl_{s+1, r}, \ \calH_A \otimes \exterior \C, \ 
  \big( U (F\oplus J) U^* \big) \otimes \rho \Big) \Big]\\
   &= \Big[ \Big( \C, \,  \calH_A \otimes \Cl_{r,s+1} , \, 
    P^{r,s}\big(U(F \oplus J)U^* \ox \rho_1 \big) + (\one - P^{r,s})( J \otimes \rho_1) \Big) \Big]
\end{align*}
as an element in $ \KKR(\C, A \otimes \Cl_{r,s+1} )$. 
The   skew-adjoint unitary $J \in \End_A^{r,s}(\calH_A)$ gives a basepoint odd self-adjoint unitary $J \otimes \rho \in \End_A(\calH_A) \otimes \Cl_{r,s+1}$, 
so applying $\Phi_{A \otimes \Cl_{r,s+1} }$  
\begin{align*}
  \Phi_A^{r,s+1}\big( \class{F} \big) 
    &= \big[ P^{r,s} \big( q_{\calH_A}( U(F \oplus J)U^*) \otimes \rho  \big) + (\one -P^{r,s})( q_{\calH_A}(J) \otimes \rho ) \big] - \big[ q_{\calH_A}( J) \otimes \rho  \big]. \qedhere
\end{align*}
\end{proof}

Using $\Phi^{r,s+1}_A$ from Definition \ref{def:Indexed_KK_isomorphisms}, we also recall  the {\em Roe isomorphism} (cf. Appendix \ref{subsec:bdd_KK_to_DK})
\[
\Roe_A^{r+1,s+1} = \calM^{\DK}_{\calH_A} \circ \delta \circ \Phi_A^{r,s+1} \colon \KKR(\Cl_{s+1,r}, A) \xrightarrow{\simeq} \DK( A \otimes \Cl_{r+1,s+1} ) .
\]

\begin{thm} \label{thm:index-of-Fred}
Let $T \in \Reg^{r,s}_A(X)$ be a (bounded or unbounded) Real skew-adjoint Fredholm operator.  
Then 
 \[
 \Index_{r+1,s+1}(T) = \Roe_A^{r+1,s+1}(\class{T})=  \calM^{\DK}_{\calH_A} \circ \delta    \circ \Phi_A^{r,s+1}(\class{T}) .
 \]
\end{thm}
\begin{proof}
We write $F := \chi(T)$ for brevity. 
Using Lemma \ref{lemma:Calkin_class} and Lemma \ref{lemma:sigma^rs_commutes_with_boundary}, we have that 
\begin{align*}
  \Roe^{r+1,s+1}_A( \class{T} ) &= \calM_{\calH_A}^{\DK} \circ \delta \circ \Phi_A^{r,s+1}( \class{T}) \\
   &=   \calM_{\calH_A}^{\DK} \circ \delta \circ ( \sigma_{q_{\calH_A}(J) \otimes \rho}^{r,s} )_\ast \big( \big[ q_{\calH_A}( U(F\oplus J)U^* ) \otimes \rho \big] - [ q_{\calH_A}( J) \otimes \rho ] \big) \\
   &=  \calM_{\calH_A}^{\DK} \circ ( \sigma_{\one \otimes \gamma}^{r,s} )_\ast \circ \delta \big( \big[ q_{\calH_A}( U(F\oplus J)U^* ) \otimes \rho \big] - [ q_{\calH_A}( J) \otimes \rho ] \big) .
\end{align*}
Computing analogously to Lemma \ref{lem:index-sinh-cosh}, 
\begin{align*}
  \delta \circ \Phi^{r,s}( \class{T} ) &= \big[ \sigma^{r,s}_{\one \otimes \gamma}\big( \cosh( \pi U(F \oplus J)U^*) \otimes \gamma  - \sinh( \pi U(F \oplus J)U^*) \otimes \rho \big)  \big] 
     - \big[ \one \otimes \gamma \big] \\
  &\hspace{-1.7cm}= \big[ - P^{r,s}\big( \cosh( \pi U(F \oplus J)U^*) \otimes \gamma  - \sinh( \pi U(F \oplus J)U^*) \otimes \rho \big) - (\one - P^{r,s} )(\one \otimes \gamma) \big] - \big[ \one \otimes \gamma \big] 
\end{align*}
and we can simplify 
\begin{align*}
   & P^{r,s}\big( \cosh( \pi U(F \oplus J)U^*) \otimes \gamma  - \sinh( \pi U(F \oplus J)U^*) \otimes \rho \big)  \\
   &\qquad = (U \otimes 1_{r+1,s+1}) P^{r,s} \Big( \big( \cosh(\pi F) \otimes \gamma + \sinh(\pi F) \otimes \rho \big) \oplus (\one \otimes \gamma) \Big) (U^* \otimes 1_{r+1,s+1}).
\end{align*}
Denote by $\nu \colon X_A \hookrightarrow \calH_A$ the inclusion corresponding to the stabilisation $U \colon X_A\oplus\calH_A \xrightarrow{\simeq} \calH_A$. 
Because the Clifford representation $\Cl_{r,s} \to \End_A(\calH_A)$ commutes with the projection $P_X: \calH_A \to X_A$, we 
have that $P_X P^{r,s}_{\calH_A} P_X = P^{r,s}_{X}$ recovers the  minimal projection in $\End_A(X) \otimes \Cl_{r+1,s+1}$ 
used to define $\Index_{r+1,s+1}(T)$. 
Therefore  
composing with the Morita invariance 
\begin{align*}
  &\Roe_A^{r+1,s+1}(\class{T}) \\
 & = \calM_{\calH_A}^{\DK} \Big( \Big[- \mathrm{Ad}_{U \otimes 1 } \Big( P^{r,s}_{\calH_A} \big(  \cosh(\pi F) \otimes \gamma + \sinh(\pi F) \otimes \rho \big) \oplus (\one \otimes \gamma) 
     - (\one - P^{r,s}_{\calH_A}) \big( \one \otimes \gamma )\Big) \Big] \\
   &\qquad   \qquad \qquad  -   \big[ \one \otimes \gamma \big] \Big) \\
 &  = \calM_{\calH_A}^{\DK}  \circ (\mathrm{Ad}_{\nu\otimes 1})_\ast \Big( \big[  -P^{r,s}_{X_A}   \big( \cosh(\pi F) \otimes \gamma + \sinh(\pi F) \otimes \rho \big) - (\one - P^{r,s}_{X_A} )( \one \otimes \gamma) \big] 
 - [ \one \otimes \gamma ] \Big) \\
 &  =  \calM_{X}^{\DK} \Big( \big[ - P^{r,s}_{X_A}   \big( \cosh(\pi F) \otimes \gamma + \sinh(\pi F) \otimes \rho \big) - (\one - P^{r,s}_{X_A} )( \one \otimes \gamma) \big] 
 - [ \one \otimes \gamma ] \Big) \\
  &  = \Index_{r+1,s+1}(T),
\end{align*}
where we have used part (1) of Lemma \ref{lemma:DK_Morita_compat_new} adapted to the setting where 
$X_A$ is ungraded but $X_A \otimes \Cl_{r+1,s+1}$ is $\Z_2$-graded. The final equality is Lemma \ref{lem:index-sinh-cosh}.
\end{proof}

Since the index factors through $\KKR$, we can immediately conclude its basic stability properties. 
To consider homotopy invariance of the index, we use the Wahl topology of operators on $C^*$-modules,  
which is introduced in \S\ref{sec:families} below (see Corollary \ref{cor:Wahl_family_to_KK_htpy} in particular).

\begin{cor}
\label{cor:index-properties}
\begin{enumerate}
 \item Let $[0,1]\ni t\mapsto T_t \in \Reg^{r,s}_A(X)$ be a Wahl-continuous path of (bounded or unbounded) Real skew-adjoint Fredholm operators on $X_A$ 
 (see \S\ref{sec:families}). Let $K\in\End_A^{r,s}(X) \cap \End_A^0(X)$ be a compact endomorphism. Then 
\[
  \Index_{r+1,s+1}(T_0)=\Index_{r+1,s+1}(T_1) = \Index_{r+1,s+1}(T_0+K).
\]
 \item If $S, T \in \Reg^{r,s}_A(X)$ are (bounded or unbounded) Real skew-adjoint Fredholm operators, then 
 $S \oplus T \in \Reg^{r,s}_A(X\oplus X)$ and 
 \[
   \Index_{r+1,s+1}\big( S \oplus T \big) =  \Index_{r+1,s+1}(S) +  \Index_{r+1,s+1}(T).
 \]  
 with addition in $\DK(A \otimes \Cl_{r+1,s+1})$.
\end{enumerate}
\end{cor}

In the special case that $T \in \Reg_A^{r,s}(X)$ is skew-adjoint and has compact resolvent, $(\one + T)^{-1} \in \End_A^0(X)$, we 
can also define an equivalent expression for the analytic index via the Cayley transform, $  U_T = (T-\one)(T+\one)^{-1}$, 
which is such that $\one - U_T \in \End_A^0(X)$ (see Appendix \ref{subsec:Cayley}).  An advantage of the Cayley transform 
is that we can work directly with the (unbounded) operator $T$ and do not have to take a normalising function. 
The use of the Cayley transform as a passage between (unbounded) $\KK$-theory and $K$-theory was previously considered 
in \cite{BKR}. We review this work and adapt it to the Clifford anti-linear setting in Appendix \ref{subsec:Cayley}.

\begin{prop} \label{prop:DK_index_Cayley}
Let $T \in \Reg^{r,s}_A(X)$ be a Real skew-adjoint operator with compact resolvent.  
Then  $U_T = (T-\one)(T+\one)^{-1} \in \End_A^0(X)$ is a Real unitary with 
$\one - U_T \in \End_A^0(X)$ and such that 
$\tfrac12 (U_T + U_T^*) \otimes \gamma + \tfrac12 ( U_T - U_T^* ) \otimes \rho$ is Clifford anti-linear 
with respect to a graded $\Cl_{s,r}$-representation. Identifying $\Cl_{1,1} \cong M_2(\C)$ with $\gamma \cong \sigma_1$ and $\rho \cong -i\sigma_2$, we have that 
\[
\Index_{r+1,s+1}(T) =  
\calM_{X}^{\DK} \left( \left[ P^{r,s} \begin{pmatrix} 0 &  U_T^* \\  U_T & 0 \end{pmatrix}  + (\one -P^{r,s})  \begin{pmatrix} 0 & \one \\ \one & 0 \end{pmatrix}\right] 
- \left[\begin{pmatrix} 0 & \one \\ \one & 0 \end{pmatrix} \right] \right) .
\]
\end{prop}
\begin{proof}
The Van Daele class for $U_T$ is well-defined by Proposition \ref{prop:CliffordCayley_computation} and 
the proof of Theorem \ref{thm:Roe_and_Cayley_comparison} gives an explicit homotopy 
from $-e^{\pi \chi(T) }$ to $U_T$ within the unitaries in $\End_A^0(X)^\sim$.
\end{proof}

\subsection{A \texorpdfstring{$\Z_2$}{Z2}-graded index}

For completeness we also define a $\DK$-valued analytic index in 
the setting of $\Z_2$-graded (Real) $C^*$-modules over $\Z_2$-graded $C^*$-algebras $B$ that can be more 
general than $B = A \otimes \Cl_{r,s}$. An advantage of working with generic graded algebras is that 
we do not need to employ the Clifford shuffle maps that appear in the  Clifford anti-linear setting. 
The proof of the following is analogous to Theorem \ref{thm:index-of-Fred}.

\begin{thm} \label{thm:Z_2-graded_Fred_index}
Let $B$ be a $\Z_2$-graded Real $C^*$-algebra and $Y_B$ a countably generated full $C^*$-module. If $D$ is a Real 
odd self-adjoint regular Fredholm operator on $Y_B$, then for any normalising function $\chi$ of $D$,
\begin{align*}
    \wh{\Index}_{1,0}(D) &= \calM_Y^{\DK} \big( \big[ -\cos\big(\pi \chi(D) \big) \hox \gamma - \sin\big(\pi \chi(D) \big)  \hox 1_{1,0} \big] - \big[ \one \hox \gamma \big] \big) \\
    &=\calM^{\DK}\circ\delta\circ\Phi_B \big( \big[( \C, \, Y_B , \, D) \big] \big)\\
    &=\Roe_B( [D] )
    \in DK( B \hox \Cl_{1,0} )
\end{align*}
is well-defined. Here $\Roe_B$ is the $\Z_2$-graded  Roe isomorphism from Theorem \ref{thm:graded_Roe_and_Kubota},
\[
  \mathfrak{R}_B:  \KKR( \C, B) \xrightarrow{\Phi_B} \DK( \calQ_{B \hox \calK} ) \xrightarrow{\delta} \DK( B \hox \calK \hox \Cl_{1,0}) \xrightarrow{\calM^{\DK} } \DK( B \hox \Cl_{1,0}).
\]
\end{thm}

In the special case of a \emph{complex} $\Z_2$-graded $C^*$-module over an \emph{ungraded} algebra, we can 
relate the $\Z_2$-graded indices with those we have previously studied.

\begin{prop} \label{prop:Graded_ungraded_index_comparison}
Let $T = iD \in \Reg^{1,0}_A(\calH_A)$ be a skew-adjoint Fredholm operator on the complex standard $C^*$-module $\calH_A$ with an (ungraded) 
$\Cl_{1,0}$-representation generated by $e \in \End_A(\calH_A)$. 
Consider $\hat{\calH}_A \cong \calH_A^+ \oplus \calH_A^- \cong \calH^+_A \otimes \C^2$ 
as a $\Z_2$-graded $C^*$-module, graded by the self-adjoint unitary $e$. Then $D = \begin{pmatrix} 0 & D_- \\ D_+ & 0    \end{pmatrix} $ is an odd 
self-adjoint Fredholm operator on $\calH^+_A \otimes \C^2$ and
\[
   \Index_{2,1}( iD ) = \calM_{1,1}^{\DK}  \left(  \wh{\Index}_{1,0}(D) \right) \in \DK( A \otimes \Cl_{2,1} )
\]
with $\calM_{1,1}^{\DK}: \DK( A \otimes \Cl_{1,0} ) \xrightarrow{\simeq} \DK( A \otimes \Cl_{2,1} )$ the Clifford stability
of $\DK$-theory (Lemma \ref{lemma:DK_Clifford_stability}).
\end{prop}
\begin{proof}
Letting $\hat{\calH}_A \cong \calH^+_A \otimes \C^2 \cong \calH_A \otimes \C^2$, we have that $\End_A(\hat{\calH}_A) \cong \End_A(\calH_A)\otimes \Cl_{1,1}$ with 
analogous identities for $\calQ_A(\hat{\calH}_A)$ and $\End_A^0(\hat{\calH}_A)$. The isomorphism 
$\calM_{1,1}^{\DK}$ is such that the diagram 
\begin{equation} \label{eq:graded_to_ungraded_diagram}
  \xymatrix{
    \DK\big( \calQ_A( \hat{\calH}_A )\big) \ar[rr]^{\delta}  \ar[d]^{\simeq}  & & \DK\big( \End_A^0(\hat{\calH}_A) \hox \Cl_{1,0} \big)  \ar[rr]^{\calM_{\hat{\calH}_A}^{\DK}} \ar[d]^{\simeq} 
    & & \DK(A \otimes \Cl_{1,0} ) \ar[d]^{\calM_{1,1}^{\DK}} \\
    \DK\big( \calQ_A( \calH_A) \otimes \Cl_{1,1} \big) \ar[rr]^{\delta} & & \DK\big( \End_A^0(\calH_A) \otimes \Cl_{2,1} \big) \ar[rr]^{\calM_{\calH_A}^{\DK}} &  &
    \DK( A \otimes \Cl_{2,1} )
  }
\end{equation}
commutes. The class $\wh{\Index}_{1,0}(D)$ comes from applying the top row of the above diagram to the Van Daele class 
\begin{equation} \label{eq:graded_Phi_matrix_representative}
   \left[  \begin{pmatrix} 0 & q(\chi(D)_-) \\ q(\chi(D)_+)  & 0    \end{pmatrix} \right] - \left[ \begin{pmatrix} 0 & \one \\ \one & 0 \end{pmatrix} \right] 
   \in \\DK\big( \calQ_A( \hat{\calH}_A )\big) , \qquad \chi(D)_\pm = \chi(D) \tfrac{1}{2}( \one \pm e).
\end{equation}
For the ungraded index, we let $U: \calH_A \otimes \C^2  \xrightarrow{\simeq} \calH_A$ be an (ungraded) stabilisation 
map and take $J = U (\one \otimes -i\sigma_2 )U^*$ a basepoint skew-adjoint unitary that will anti-commute with the $\Cl_{1,0}$ generator.  
The skew-adjoint index 
$\Index_{2,1}(iD)$ comes from applying the bottom row of the diagram \eqref{eq:graded_to_ungraded_diagram} to the 
Van Daele class 
$\big[ P^{1,0}( q( \chi(iD) ) \otimes \rho) + (\one -P^{1,0}) q(J) \otimes \rho) \big] - [q(J) \otimes \rho] \in \DK\big( \calQ_A(\calH_A) \otimes \Cl_{1,1}\big)$ 
with  $P^{1,0} =  \tfrac{1}{2}(\one+ e \otimes \gamma\rho)$. 
Using the isomorphism $\Cl_{1,1} \cong M_2(\C)$, $P^{1,0} \cong \mathrm{diag}\big( \tfrac{1}{2}(\one + e), \, \tfrac{1}{2}(\one- e) \big)$ and 
\begin{align*}
    P^{1,0}( q( \chi(iD) ) \otimes \rho) 
    \cong \begin{pmatrix} 0 & -iq(\chi(D)) \tfrac{1}{2}(\one - e) \\ iq(\chi(D)) \tfrac{1}{2}(\one + e) \end{pmatrix} = 
    \begin{pmatrix} 0 & -iq(\chi(D)_-) \\ i q(\chi(D)_+) & 0 \end{pmatrix}.
\end{align*}
Using the 
unitary $U: \calH_A \otimes \C^2  \xrightarrow{\simeq} \calH_A$, we can simplify
$P^{1,0}\big( \calQ_A(\calH_A) \otimes \Cl_{1,1}\big)P^{1,0} \cong \calQ_A(\calH_A) \otimes \Cl_{1,1} \cong \calQ_A\big( \hat{\calH}_A \big)$,  
which induces the map 
\[
  \big[ P^{1,0}( q(\chi(iD)) \otimes \rho) + (\one - P^{1,0})(q(J) \otimes \rho) \big] - \big[ q(J) \otimes \rho\big]  \xmapsto{\simeq}   
  \left[  \begin{pmatrix} 0 & -iq(\chi(D)_-) \\ i q(\chi(D)_+) & 0 \end{pmatrix} \right] - \left[ \begin{pmatrix} 0 & \one \\ \one & 0 \end{pmatrix} \right] 
\]
that represents the inverse of the left vertical arrow from Eq. \eqref{eq:graded_to_ungraded_diagram}. Because we are 
in complex spaces, the factor of $i$ will not change the $\DK$ class (more precisely, we can exchange $\sigma_2 \leftrightarrow \sigma_1$). 
Hence the image is the same as 
the Van Daele class in Eq. \eqref{eq:graded_Phi_matrix_representative}.
  Because the bottom row of 
\eqref{eq:graded_to_ungraded_diagram} computes $\Index_{2,1}(T)$ and the diagram commutes, the result follows.
\end{proof}

\begin{cor} 
If $\hat{\calH}_A$ is the graded standard complex $C^*$-module and  $D \in \Reg_A( \hat{\calH}_A)$ is self-adjoint, odd, and Fredholm, 
then 
\[
  \Upsilon_A^{1,0} \big( \wh{\Index}_{1,0}(D) \big) = \Index_0\big( D_+ ), \qquad D \simeq \begin{pmatrix} 0 & D_- \\ D_+ & 0 \end{pmatrix}
\]
\end{cor}
\begin{proof}
For $A$ ungraded $\Upsilon_A^{1,0} = \Upsilon_{A}^{2,1} \circ \calM_{1,1}^{\DK}$, so the result follows from 
Propositions \ref{prop:Graded_ungraded_index_comparison} and \ref{prop:complex_index}.
\end{proof}

As we have done in the Clifford anti-linear setting, we also give a description of the $\Z_2$-graded Fredholm index in 
the case where $D$ has compact resolvent. We use the excision map $\exc_Y:\DK(\End_B(Y),\calQ_B(Y))\to  \DK(\End^0_B(Y))$.

\begin{prop} [cf. Theorem \ref{thm:Graded_Cayley_iso}]
Let $Y_B$ be a countably generated full $C^*$-module with a $\Z_2$-graded $\Cl_{1,0}$-representation 
generated by $\gamma \in \End_B(Y)$. If  $D \in \Reg_B^{1,0}(Y)$ is Real, odd, self-adjoint, and 
Fredholm with   $(D + i)^{-1} \in \End_B^0(Y)$, 
then $\calC_\gamma(D) = \gamma( D + \gamma)(D - \gamma)^{-1} \in \End_B(Y)$ is an odd-self 
adjoint unitary such that $\calC_\gamma(D) - \gamma \in \End_B^0(D)$ and the index 
\[
  \wh{\Index}( D)  =  \calM_Y^{\DK} \circ \exc_Y \big( \big[ \calC_\gamma(D) \big] - [\gamma] \big) \in \DK( B )
\]
is well-defined.
\end{prop}

Note that $\wh{\Index}( D)$ and $\wh{\Index}_{1,0}( D)$ have  different assumptions on the odd self-adjoint 
Fredholm operator $D$. The former requires a Clifford generator and defines a class in $\KKR(\Cl_{1,0}, B)$, 
whereas the latter does not include additional Clifford data and defines a class in $\KKR(\C, B)$.  While the domain and ranges of these indices are different, 
 Proposition \ref{prop:Kubota_Roe_Cayley_compat} in the Appendix gives that 
\[
   \beta_{\DK} \Big( \wh{\Index}(D) \Big) = \wh{\Index}_{1,0} \Big( \altBott_{\KK}\big( [ D] \big) \Big),
\]
where $\beta_{DK}: \DK(B) \xrightarrow{\simeq} \DK(SB \hox \Cl_{1,0})$ and 
$\altBott_{KK}: \KKR( \Cl_{1,0}, B) \xrightarrow{ \simeq} \KKR( \C, SB )$ are the Bott isomorphisms in Van Daele and 
$\KK$-theory (Theorem \ref{thm:DK_Bott_isos} and Corollary \ref{cor:KK_Bott_alternative}).

\subsection{Comparison with Real Hilbert spaces} \label{subsec:Index_Real_Hilbert_Space}

Let us fix a separable Real Hilbert space $\calH$ with an ample $\Cl_{r,s}$-representation  in $\calB(\calH)$ generated 
by $\{e_1,\ldots, e_r, f_1,\ldots, f_s\}$. 
We can always add more Clifford generators  by considering 
$\calH^{\oplus 2}$ with $e_{r+1} = \sigma_1$ and $f_{s+1} = -i\sigma_2$. We will therefore assume 
that $r,s \geq 1$. 
 The set $\Reg^{r,s}(\calH)$ denotes the closed operators on $\calH$ that anti-commute with 
the Clifford generators.  In~\cite{AS69}, Atiyah and Singer studied the index theory of the space of skew-adjoint 
Fredhom operators in $\Reg^{0,k}(\calH)$.\footnote{See~\cite{BCLR} for the case of general $(r,s)$.} 
If $T \in \Reg^{r,s}(\calH)$
is skew-adjoint 
and Fredholm, then $\Ker(T)$ is a (finite-dimensional) $\Cl_{r,s}$-module and we can use the 
Atiyah--Bott--Shapiro isomorphism \cite{ABS} to define 
\[
   \Index_{1+s-r}^\mathrm{ABS}( T ) = \big[ \Ker(T) \big] \in \calM_{r,s} / \calM_{r,s+1} \cong \KO_{1+s-r}(\R),
\]
where $ \calM_{r,s} $   denotes the Grothendieck group of finite-dimensional  $\Cl_{r,s}$-modules (with addition by direct sum).  

The skew-adjoint Fredholm operator $T \in \Reg^{r,s}(\calH)$ also determines a Kasparov module 
and class $\class{T} = [\chi(T) \otimes \rho] \in \KKR( \Cl_{s+1, r} , \C)$ by Corollary \ref{cor:skew_Fred_to_KasMod}, where 
\[
   \KKR( \Cl_{s+1,r}, \C) \ni \class{T}
   \mapsto  \Index_{1+s-r}^\mathrm{ABS}( T ) \in \KO_{1+s-r}(\R)
\]
is an isomorphism of groups (see~\cite[Chapter 2]{S} for example). Using that the 
$\DK$-valued index $\Index_{r+1,s+1}(T) \in \DK(\Cl_{r+1,s+1})$ can also be described as 
the isomorphism $\Roe_\C^{r+1,s+1}$ applied to  $\class{T} $, 
there is an induced isomorphism $\xi_\C^{r+1,s+1}: \DK( \Cl_{r+1,s+1} ) \xrightarrow{ \simeq} \KO_{1+s-r}(\R)$ such that  
\[
  \xymatrix{ 
      \KKR( \Cl_{s+1, r}, \C) \ar[rrr]^{\Roe_\C^{r+1,s+1}}  \ar[drrr]_{\Index_{1+s-r}^\mathrm{ABS}} &  & & \DK( \Cl_{r+1,s+1} ) \ar[d]^{\xi_\C^{r+1,s+1} } \\ & & & \KO_{1+s-r}(\R)
  }
\]
commutes. 
The map $\xi_\C^{r+1,s+1}$ is the same as the isomorphism $\Upsilon_\C^{r+1,s+1}$ considered in Appendix \ref{subsec:appendix_DK_KR_isos}, 
which we can check directly.
In degree $0$ with $T$ Fredholm (and not skew-adjoint), both $\Index_{2,1}$ and the ABS index compute the usual 
Fredholm index, as we have seen 
in {\S}\ref{subsec:Complex_index_and_DK}. This also shows the case of degree $4$, where $\KO_4(\R) \simeq \KO_0(\mathbb{H} )$-valued 
indices are described by the Fredholm index on a quaternionic Hilbert space.  
Lastly the groups $\KO_1(\R)$ and $\KO_2(\R)$ are $\Z_2$ and any two isomorphisms of $\Z_2$ must agree.

\begin{remark}[Numerical indices]
In addition to Fredholm operators on Hilbert spaces and the ABS index, 
we can also define numerical indices via the Kasparov product.
If $\big( A \otimes \Cl_{0,d}, \, \calH, \, G \big)$ is a Real Fredholm module 
with class  $[G] \in \KKR(A \otimes \Cl_{0,d}, \C)$,  the Kasparov 
product induces a map $(\cdot) \otimes_A [G]: \KKR( \Cl_{s+1,r}, A) \to \KKR( \Cl_{s+1,r+d}, \C)$. 
We can then define a homomorphism on Van Daele $K$-theory by
\[
    \phi_G: \DK(A \otimes \Cl_{r+1,s+1}) \to \DK( \Cl_{r+d+1,s+1} ), \quad 
    \phi_G\big( \Index_{r+1,s+1}(F) \big) := \Roe_\C^{r+d+1,s+1}\big( \class{F} \otimes_A [G] \big).
\]
We can then apply $\xi_\C^{r+d+1,s+1}: \DK( \Cl_{r+d+1,s+1}) \to \KO_{s+1-r-d}(\R)$  to obtain a 
numerical index pairing of $\Index_{r+1,s+1}(F)$ with the Fredholm module $[G]$. In practice, we often need 
unbounded representatives of $\class{F}$ and $[G]$ to obtain explicit representatives of the Kasparov product 
whose numerical index has  geometric or physical meaning. Pairings of $\DK$-theory with cyclic cohomology 
have also been considered in~\cite{Kellendonk19}.
\end{remark}

\section{Van Daele \texorpdfstring{$K$}{K}-theory via skew-adjoint unitaries} \label{sec:Van_Daele_skew_unitaries}

In order to define the relative index in $\DK$-theory in {\S}\ref{sec:relative_index}, 
we use the Clifford shuffle map from \S\ref{subsec:Clifford_Fredholm_Shuffle} to give a presentation of $\DK$-theory via Clifford anti-linear skew-adjoint unitaries. 
The approach to $K$-theory via ungraded and skew-adjoint unitaries goes back to the work of Wood~\cite{Wood}.

\begin{defn}
\label{defn:complex-structure}
Let $A$ be a $\sigma$-unital and ungraded Real $C^{\ast}$-algebra  such that $\Mult(A)$ contains the generators 
$\{e_1,\ldots, e_r,  f_1,\ldots, f_s\}$ of an ungraded $\Cl_{r,s}$-representation. 
The set 
\[
   \calJ_{A}^{r,s} = \big\{ J \in \Mult(A) \,\mid \, J^\frakr = J, \, J^* = -J = J^{-1}, \, Je_j = -e_j J, \, J f_k, = -f_k J \text{ for all } j,k\big\}.
\]
denotes Real skew-adjoint  unitaries in $\Mult(A)$ that are $\Cl_{r,s}$-anti-linear.
\end{defn}

\begin{lemma}[{cf. \cite[Proposition 5.2.6]{WeggeOlsen}}]  \label{lemma:close_J_are_unitary_equiv}
 If $J_0, J_1 \in  \calJ_{A}^{r,s}$ are such that $\|J_0 - J_1\| < 2$, then $J_0$ and $J_1$ are unitarily 
 equivalent, and $J_0 \sim_h J_1$ in $\calJ_{A}^{r,s}$.
\end{lemma}
\begin{proof}
We let $Z_{J_0, J_1} = \frac{1}{2}(J_0 + J_1)$, which is skew-adjoint and anti-commutes with the Clifford generators. Also, $J_0Z=ZJ_1$, $J_1Z=ZJ_0$. Then 
\begin{align*}
   \one + Z_{J_0,J_1}^2 &= \frac{1}{4}( 2+ J_0J_1 + J_1J_0) = \frac{1}{4} (J_1 - J_0) (J_0 - J_1)
\end{align*}
and hence 
\[
   \big\| \one + Z^2_{J_0, J_1} \big\| = \frac{1}{4} \| J_0 - J_1\|^2 < 1 
\]
since $\| J_1 - J_0\| < 2$ by assumption. Therefore    $Z_{J_0, J_1}$ is invertible. 
As $Z^*Z$ commutes with $J_0$ and $J_1$, we find that  $w_{J_0, J_1} = Z_{J_0, J_1} \big|Z_{J_0, J_1}\big|^{-1}$ is a skew-adjoint unitary 
anti-commuting with the Clifford generators, and $w_{J_0, J_1} J_0 w_{J_0, J_1}^* = J_1$. 

Now fix $J_0$ and consider the set $\calJ_{J_0} = \{ J \in \calJ_{A}^{r,s} \mid \| J_0 - J\| < 2\}$. 
For any $J \in \calJ_{J_0}$ we define  $Z_J := \frac{1}{2} (J_0 + J)$ and $w_J := Z_J |Z_J|^{-1}$. 
Since the map $Z \mapsto Z|Z|^{-1}$ is norm-continuous for invertible $Z$, we see that 
the map $\calJ_{J_0} \ni J \mapsto w_J \in \calJ_{A}^{r,s}$ is norm-continuous. 

We now consider  $Z_t = J_0 + \frac{t}{2} (J - J_0)$, $t\in [0,1]$,  which is a Clifford anti-linear straight-line homotopy 
  from  $Z_0 = J_0$ to $Z_1 = \frac{1}{2} (J_0 + J) = Z_J$. 
We see that for any $t \in [0,1]$, 
\begin{align*}
   \big\| \one + Z_t^2 \big\| &= \big\| \tfrac{t}{2} \big( J_0 J + JJ_0 + 2 + \tfrac{t}{2} (J - J_0)^2  \big)\big\| 
   = \big\| \tfrac{t}{2} \big(  - (J - J_0)^2 + \tfrac{t}{2} (J - J_0)^2 \big) \big\|  \\
   &= \big| \tfrac{ t}{2} ( 1 - \tfrac{t}{2} ) \big| \,  \big\| (J - J_0)^* (J - J_0) \big\| \leq \tfrac{1}{4} \| J - J_0 \|^2,
\end{align*}
so $Z_t$ is invertible if $\|J - J_0 \| < 2$. Putting this together, 
for any  $J \in \calJ_{A}^{r,s}$ such that $\| J_0 - J\| < 2$, there is a homotopy 
$[0,1] \ni t \mapsto w_{t} J_0 w_t^* \in \calJ_A^{r,s}$ from $J_0$ to $J$, where $w_t = Z_t |Z_t|^{-1} \in \calJ_{A}^{r,s}$ .
\end{proof}

Any (ungraded) skew-adjoint unitary $J \in \calJ_{A}^{r,s}$ gives an odd self-adjoint unitary $J \otimes \rho \in \Mult(A) \otimes \Cl_{0,1}$ 
that is $\Cl_{s,r}$-anti-linear with respect to the generators $\{f_1\otimes \rho,\ldots, f_s \otimes \rho, e_1\otimes \rho,\ldots, e_r\otimes \rho\}$.
Fixing a basepoint $J_\mathrm{ref} \in \calJ_{A}^{r,s}$, we can apply Lemma \ref{lem:AMZ-OSU-iso} to obtain an odd self-adjoint unitary 
\begin{equation} \label{eq:sigma-rs_J_map}
    \sigma_{J_\mathrm{ref}\otimes \rho}^{r,s}( J\otimes \rho) = P^{r,s} ( J \otimes \rho) + (\one - P^{r,s})(J_\mathrm{ref} \otimes \rho) \in \Mult(A) \otimes \Cl_{r,s+1}
\end{equation}
Note that for $J_0, J_1 \in \calJ_{A}^{r,s}$, $\big\| \sigma_{J_\mathrm{ref}\otimes \rho}^{r,s}( J_1\otimes \rho)   - \sigma_{J_\mathrm{ref}\otimes \rho}^{r,s}( J_0\otimes \rho)   \big\| \leq \| J_1 - J_0 \|$.
To define a Van Daele $K$-theory class in $A \otimes \Cl_{r,s+1}$, we consider pairs of elements and excision.

\begin{lemma}[cf. Lemma \ref{lem:relative_DK}] \label{lem:Fred_pair_via_rel_DK}
If $J_0, J_1 \in \calJ_{A}^{r,s}$ are such that $\big\|q(J_1) - q(J_0)\big\|_{\calQ_A}<2$, then 
for any basepoint $J_\mathrm{ref} \in \calJ_{A}^{r,s}$, the map \eqref{eq:sigma-rs_J_map}
gives a well-defined relative class 
\[
    \big[ \sigma_{J_\mathrm{ref}\otimes \rho}^{r,s}( J_0\otimes \rho)  \big] - \big[  \sigma_{J_\mathrm{ref}\otimes \rho}^{r,s}( J_1\otimes \rho)   \big] 
       \in \DK \big( \Mult(A)\otimes   \Cl_{r,s+1}, \calQ_A \otimes  \Cl_{r,s+1} \big) .
\]
To be precise, there exists a $\tilde{J}_1 \in \calJ_{A}^{r,s}$ such that $\tilde{J}_1 \sim_{h} J_1$ in $\calJ_{A}^{r,s}$,  $\tilde{J}_1 - J_0 \in A$, and  
\[
    \big[ \sigma_{J_\mathrm{ref}\otimes \rho}^{r,s}( J_0\otimes \rho)  \big] - \big[  \sigma_{J_\mathrm{ref}\otimes \rho}^{r,s}( J_1\otimes \rho)   \big] 
 =   \big[ \sigma_{J_\mathrm{ref}\otimes \rho}^{r,s}( J_0\otimes \rho)  \big] - \big[  \sigma_{J_\mathrm{ref}\otimes \rho}^{r,s}( \tilde{J}_1\otimes \rho)   \big]  . 
\]
The excision isomorphism (see Eq.\ \eqref{eq:DK_excision}) then yields a well-defined element
\[
      \exc\big(   \big[ \sigma_{J_\mathrm{ref}\otimes \rho}^{r,s}( J_0\otimes \rho)  \big] - \big[  \sigma_{J_\mathrm{ref}\otimes \rho}^{r,s}( J_1\otimes \rho)   \big] \big) \in \DK( A \otimes \Cl_{r,s+1} ).
\]
\end{lemma} 

If we take $J_\mathrm{ref} = J_0$, for example, then $\sigma_{J_0 \otimes \rho}^{r,s}( J_0\otimes \rho)  = J_0 \otimes \rho$.
When $r,s =0$, we simply take $ \exc\big(  [J_0 \otimes \rho] - [ J_1 \otimes \rho] \big) \in \DK( A \otimes \Cl_{0,1} )$ and 
do not require   the basepoint $J_\mathrm{ref}$ and map $\sigma^{r,s}_{J_\mathrm{ref}}$.

\subsection{Liftings of skew-adjoint unitaries}

Let $A$ be a unital $C^*$-algebra with $I \subset A$ a closed two-sided ideal. 
Fix ungraded Clifford generators  $\{e_1,\ldots, e_r, f_1,\ldots, f_s\} \subset A \setminus I$, so that 
we have generators of ungraded Clifford representations in both $A$ and $A/ I$. Hence the sets 
$\calJ_{A}^{r,s}$ and $\calJ_{A/I}^{r,s}$ of Clifford anti-linear skew-adjoint unitaries are well-defined.

For the construction of the analytic spectral flow in {\S}\ref{sec:analytic_spectral_flow}, 
it will be useful to know when a skew-adjoint unitary $J \in  \calJ_{A/I}^{r,s}$ has a 
lift $\wt{J} \in  \calJ_{A}^{r,s}$.

\begin{lemma} \label{lemma:homotopic_to_lifting}
Let $J_0, J_1 \in \calJ_{A/I}^{r,s}$ and 
suppose that $J_0 \in A/ I$ can be lifted to a skew-adjoint unitary  $\wt{J}_0 \in \calJ_{A}^{r,s}$. 
\begin{enumerate}
  \item If $\|J_0 - J_1\|_{A/I} < 2$, then $J_1$ can also be lifted to a skew-adjoint unitary $\wt{J}_1 \in \calJ_{A}^{r,s}$.
  \item If $J_0 \sim_h J_1$ in $\calJ_{A/I}^{r,s}$, then  $J_1$ can also be lifted to a skew-adjoint unitary $\wt{J}_1 \in \calJ_{A}^{r,s}$.
\end{enumerate}

\end{lemma}
\begin{proof}
(1) By Lemma \ref{lemma:close_J_are_unitary_equiv} there is a norm-continuous path $[0,1] \ni t \mapsto J_t = w_t J_0 w_t^* \in \calJ_{A/I}^{r,s}$
that connects $J_0$ with $J_1 = w_1 J_0 w_1^*$. In particular,  $w_0 = J_0$ and $w_1 = \mathrm{sgn}\big( \tfrac{ 1}{2} (J_0 + J_1) \big)$.
Then $t \mapsto u_t = -J_0 w_t$ is a path of (Real) unitaries from $u_0 = \one$ to $u_1 = -J_0 w_1$. 
Hence $u_t$ lifts to a path $\wt{u}_t \in A$ in the connected component of the identity. 

Because $w_t$ and $J_0$ anti-commute with the Clifford generators, $u_t$ commutes with the Clifford generators. Thus each $u_t$ lies in the fixed point algebra $(A/I)^G$, where the finite group $G$ is generated by the automorphisms 
$\mathrm{Ad}_{e_j}$ and $\mathrm{Ad}_{f_k}$ for all $j=1,\ldots, r$ and $k = 1,\ldots, s$. Hence we may assume that $\wt{u}_t\in A^G$, and so commutes with the Clifford generators as well.

We now consider the lift $\wt{J}_0 \in \calJ_{A}^{r,s}$ of $J_0$ and define 
$\wt{J}_1 = (\wt{J}_0 \wt{u}_1) \wt{J}_0 ( \wt{J}_0 \wt{u}_1 )^*$, which by construction is 
 a skew-adjoint unitary such that
\[
  q( \wt{J}_1 ) = J_0 u_1 J_0 ( J_0 u_1)^* = J_0 (-J_0) w_1 J_0 ( J_0 (-J_0) w_1)^* = w_1 J_0 w_1^* = J_1 .
\]
Furthermore, $\wt{J}_1$ anti-commutes with the Clifford generators, and so 
is indeed a lift in $\calJ_{A}^{r,s}$.

(2) We take a partition $0=t_0 < t_1<\cdots < t_n = 1$ such that $\|J_{t_j} - J_{t_{j-1}} \|_{A/I} < 2$ and repeatedly  apply part (1).
\end{proof}

\section{The relative index of skew-adjoint unitaries} \label{sec:relative_index}

In this section we introduce a (Van Daele) $K$-theoretic relative index for a pair of skew-adjoint unitaries 
$J_0, J_1 \in \End_A^{r,s}(X)$ satisfying $\big\| q(J_0) - q(J_1) \big\|_{\calQ_A(X)} < 2$. 
The index is a generalisation of the relative index of skew-adjoint unitaries on real Hilbert spaces \cite[\S4]{BCLR}, 
which in turn is an adaptation of   the well-known relative index of a Fredholm pair of projections, first considered 
 in the setting of pairs of subspaces of Banach spaces by Kato \cite[{\S}IV.4.1]{Kato} and then rediscovered and 
popularised in the Hilbert space setting by Avron--Seiler--Simon  \cite[\S4]{ASS}.

Given two projections $P$ and $Q$ on a (complex) Hilbert space $\calH$ such that $\big\| q(P) - q(Q) \big\|_{\calQ(\calH)} < 1$, 
it follows that $Q\colon\Ran(P)\to\Ran(Q)$ is Fredholm \cite{BCPRSW}, and we may define the \emph{relative index} of $(P,Q)$ as
\[
\relind(P,Q) := \Index\big( Q\colon\Ran(P)\to\Ran(Q) \big) \in K_0(\C) \simeq \Z .
\]

For two projections on the standard complex $C^*$-module $\calH_A$, a $K_0(A)$-valued generalisation of this index 
was given by Wahl~\cite[\S3]{Wahl07} and Ng--Sutradhar--Wang~\cite[\S4]{NSW2}. Briefly, if $P, Q \in \End_A(\calH_A)$ are
projections with $\big\| q(P) -q(Q) \big\|_{\calQ_A(\calH)} < 1$,
then $QP$ is Fredholm in the sense of Exel \cite{Exel}, and the elements 
$q(PQP) \in \calQ_{A}(P \calH_A)$ and $q(QPQ) \in \calQ_A(Q \calH_A)$ are invertible. Taking 
isometries $R,S \in M_2\big( \End_A(\calH_A)\big)$ such that $RR^*= Q \oplus \one$ and $SS^* = P \oplus \one$, 
$q\big( R^*( QP \oplus \one) S \big) \in M_2\big( \calQ_{A}(\calH_A) \big)$ is invertible, and we can define 
\[
   \relind(P,Q) :=  \calM_{\calH_A}^K \circ  \delta \big( \big[ q\big( R^*( QP \oplus \one) S \big) \big] \big), \qquad 
   \delta: K_1\big( \calQ_A(\calH) \big) \xrightarrow{\simeq} K_0\big( \End_A^0(\calH) \big).
\]
When $P, Q \in \End_A^0(\calH_A)$, this definition simplifies to 
$\calM_{\calH_A}^K \big( [P] - [Q] \big) \in K_0(A)$~\cite[Proposition 4.7]{NSW2}. 

Wahl axiomatically characterises a $K_0(A)$-valued relative index for projections $P, Q \in \End_A(\calH_A)$ with $P-Q \in \End_A^0(\calH_A)$. 
Because the index is uniquely characterised by its properties~\cite[Proposition 3.11]{Wahl07}, Wahl's index coincides 
with $\relind(P,Q)$ when $P-Q \in \End_A^0(\calH_A)$. Uniqueness of the relative index also implies that we can write 
\begin{equation} \label{eq:relind_in_relK}
\relind(P,Q) = \calM_X^K \circ \exc_X \big( [P] - [Q] \big) ,
\end{equation}
where $\big\| q(P) -q(Q) \big\|_{\calQ_A(\calH)} < 1$ and we apply Lemma \ref{lem:relative_DK} (adapted to the 
setting of complex $K$-theory).
Wahl also defines a $K_1(A)$-valued 
\emph{odd} relative index \cite[\S8]{Wahl07} (which is necessarily trivial on complex Hilbert spaces, since $K_1(\C)=0$).

In this section, we will introduce an analogous definition for the relative index on \emph{Real} $C^*$-modules, replacing projections by skew-adjoint unitaries, 
generalising the relative index defined on real Hilbert space introduced in \cite[\S4]{BCLR}. 
We show that our relative index agrees with the case of complex $C^*$-modules as described by Wahl \cite[\S3 \& \S8]{Wahl07} 
and Ng--Sutradhar--Wang~\cite{NSW2}.
Similarly to the construction of the analytic spectral flow by Phillips \cite{P2}, we will use our 
relative index of skew-adjoint unitaries to define the analytic spectral flow on Real $C^*$-modules in {\S}\ref{sec:analytic_spectral_flow}.

\subsection{Definition and properties of the relative index} \label{subsec:relative_index_defn}

As before, let $X_A$ be a full countably generated and ungraded Real Hilbert $C^*$-module over 
an ungraded Real $C^*$-algebra $A$, 
equipped with an ungraded and ample representation $\Cl_{r,s} \to \End_A(X)$ 
generated by $\{e_1, \ldots, e_r, f_1,\ldots, f_s\}$. We also recall the minimal even projection 
$P^{r,s}$ from Eq. \eqref{eq:ExplicitP^rs}.

\begin{defn}[Relative index] \label{def:Ind_of_pair_of_complex_structures_DK}
We say that two skew-adjoint unitaries $J_0, J_1 \in \End_A^{r,s}(X)$ form a \emph{Fredholm pair} if $\| q(J_0) - q(J_1) \|_{\calQ_A(X)} < 2$. 
Given a Fredholm pair $(J_0, J_1)$ and basepoint skew-adjoint unitary $J_\mathrm{ref} \in \End_A^{r,s}(X)$, 
we define $\relind_{r,s+1}(J_0, J_1) \in \DK(A \otimes \Cl_{r,s+1}),$ as the Van Daele class
\begin{align*}
      \relind_{r,s+1}(J_0, J_1) &=  \calM_{X}^{\DK} \circ \exc_{X} \big(  \big[ \sigma_{J_\mathrm{ref}\otimes \rho}^{r,s}( J_0\otimes \rho)  \big] - \big[  \sigma_{J_\mathrm{ref}\otimes \rho}^{r,s}( J_1\otimes \rho)   \big]  \big) \\
      &\hspace{-3cm} =
      \calM_{X}^{\DK} \circ \exc_{X} \big(  \big[P^{r,s}( J_0 \otimes \rho_1 ) + (\one - P^{r,s})(J_\mathrm{ref} \otimes \rho_1) \big] 
      - \big[P^{r,s}( J_1 \otimes \rho_1 ) + (\one - P^{r,s})(J_\mathrm{ref} \otimes \rho_1) \big]  \big) ,
\end{align*}
where $\big[ \sigma_{J_\mathrm{ref}\otimes \rho}^{r,s}( J_0\otimes \rho)  \big] - \big[  \sigma_{J_\mathrm{ref}\otimes \rho}^{r,s}( J_1\otimes \rho)   \big] \in \DK\big( \End_{A}(X) \otimes \Cl_{r,s+1} , \calQ_{A}(X)\otimes \Cl_{r,s+1} \big)$ is the relative 
Van Daele class from Lemma \ref{lem:Fred_pair_via_rel_DK}, 
\[
   \exc_X:  \DK\big( \End_A(X) \otimes \Cl_{r,s+1}, \calQ_A(X) \otimes \Cl_{r,s+1} \big) \xrightarrow{\simeq} \DK\big( \End_A^0(X) \otimes \Cl_{r,s+1}\big) 
 \]
is the 
excision isomorphism (Eq.\ \eqref{eq:DK_excision}), and 
$\calM_{X}^{\DK}: \DK\big( \End_A^0(X) \otimes \Cl_{r,s+1} \big) \xrightarrow{\simeq} \DK( A \otimes \Cl_{r,s+1})$ is the Morita invariance isomorphism 
(Eq.\ \eqref{eq:DK_Morita}). When $r=s=0$, we directly define the relative index
$\relind_{0,1}(J_0, J_1) = \calM_{X}^{\DK} \circ \exc_{X} \big( [ J_0 \otimes \rho] - [ J_1 \otimes \rho] \big) \in \DK(A \otimes \Cl_{0,1})$.
\end{defn}

We can take $J_0$ or $J_1$ as the basepoint skew-adjoint unitary. Taking $J_1$ for example, 
\[
  \relind_{r,s+1}(J_0, J_1) = \calM_X^{\DK} \circ \exc_X\big(  \big[ P^{r,s} (J_0 \otimes \rho) + (\one - P^{r,s})(J_1 \otimes \rho) \big] - [ J_1 \otimes \rho ] \big).
\]
We assemble some basic properties of the relative index.

\begin{lemma} \label{lemma:addition_for_Ind_J0J1}
\begin{enumerate}
  \item Let $J_0, J_1 \in \End_A^{r,s}(X)$ be a Fredholm pair of skew-adjoint unitaries. Then  
  \[\relind_{r,s+1}(J_1, J_0) = - \relind_{r,s+1}(J_0, J_1) \in \DK( A\otimes \Cl_{r, s+1}  ).
  \]
  \item If $J_0, J_1, J_2 \in \End_A^{r,s}(X)$ are skew-adjoint unitaries such that both
  $\|q(J_0 - J_1)\|_{\calQ_A(X)} < 1$ and $\|q(J_1 - J_2)\|_{\calQ_A(X)} < 1$, then 
  \[
     \relind_{r,s+1}(J_0, J_2) = \relind_{r,s+1}(J_0, J_1) + \relind_{r,s+1}(J_1, J_2).
  \]
    \item If $J_0, J_1 \in \End_A^{r,s}(X)$ are skew-adjoint unitaries such that $\| J_0 - J_1\| < 2$, then the relative index $\relind_{r,s+1}(J_0, J_1) $ is trivial.
    \item Let $\{J_{0,\lambda}\}_{\lambda\in [0,1]}$ and $\{J_{1,\lambda}\}_{\lambda\in [0,1]}$ be two norm-continuous paths of skew-adjoint unitaries in 
$\End_A^{r,s}(X)$ such that $\big\| q(J_{0,\lambda}) - q( J_{1,\lambda}) \big\|_{\calQ_{A}(X)} < 2$ for all $\lambda \in [0,1]$. Then  
$\relind_{r,s+1}(J_{0,\lambda}, J_{1,\lambda} )$ is well-defined and constant for all $\lambda \in [0,1]$.
\end{enumerate}
\end{lemma}
\begin{proof}
We write $[J \otimes \rho] =   \big[ \sigma_{J_\mathrm{ref}\otimes \rho}^{r,s}( J \otimes \rho)  \big] $ for brevity.
Part (1) is immediate from the definition. 
For part (2), we have seen in Lemma \ref{lem:Fred_pair_via_rel_DK} that all three of $[J_1\ox\rho]-[J_0\ox\rho]$, 
$[J_2\ox\rho]-[J_1\ox\rho]$, $[J_2\ox\rho]-[J_0\ox\rho]$ give well-defined classes in the relative Van Daele group
$\DK\big(\End_{A}(X)\otimes\Cl_{r,s+1},\calQ_{A}(X)\otimes \Cl_{r,s+1}\big)$. Then we  clearly have 
\[
[J_0\ox\rho]-[J_2\ox\rho]=[J_0\ox\rho]-[J_1\ox\rho]+[J_1\ox\rho]-[J_2\ox\rho].
\]
Part (3) follows from Lemma \ref{lemma:close_J_are_unitary_equiv}. 
For part (4), the conditions on the path 
ensure that we obtain  homotopies of odd self-adjoint unitaries $J_{0,\bullet} \otimes \rho$ 
and $J_{1,\bullet} \otimes \rho$ that will leave the relative $\DK$-class constant.
\end{proof}

We remark that $\relind_{r,s+1}(J_0, J_1) $ also possesses a stronger homotopy invariance property (Corollary \ref{cor:Ind(J0,J1)_htpy_invariant}), 
but its proof first requires a connection of the relative index with $\KKR$-theory.

\subsection{The relative index as a Kasparov class}

We will show here that the relative index of a Fredholm pair $(J_0,J_1)$ can be described as the image of a Kasparov class 
under the Kubota isomorphism $\Kubota_A^{r,s+1}: \KKR(\Cl_{s+1,r}, SA) \to \DK( A \otimes \Cl_{r,s+1} )$ 
described in Appendix \ref{subsec:bdd_KK_to_DK} (Definition \ref{def:Indexed_KK_isomorphisms}), 
where $SA$ denotes the suspension of $A$. 
In particular, this implies that the relative index inherits the properties of $\KK$-theory.

\begin{lemma} \label{lemma:Fredholm_pair_straight_line_path}
Let $J_0, J_1 \in \End_A^{r,s}(X)$ be a Fredholm pair of skew-adjoint unitaries. Define the family 
$\{F_t\}_{t\in [0,1]} \subset \End^{r,s}_A(X)$ by $F_t  = (1-t)J_0 + t J_1$.
Then $F_\bullet$ is a skew-adjoint 
Fredholm operator on $SX_{SA}$. 
Consequently, the triple 
\[
  \Big( \Cl_{s+1, r}, \,  SX_{SA} \otimes \exterior \C, \,  \chi(F_\bullet) \otimes \rho \Big)
\]
is a Real Kasparov module with Clifford generators 
  $ \big\{  \one \otimes \gamma, \, f_1 \otimes \rho, \ldots, f_s \otimes \rho, e_1 \otimes \rho, \ldots, e_r \otimes \rho \big\}$.
\end{lemma}
\begin{proof}
Clearly $t \mapsto F_t$ is norm-continuous and for any  $t \in [0,1]$,
we have that 
\begin{align*}
  \big\| \one + F_t^2 \big\| &= \big\|\one + \big( (1-t)J_0 + tJ_1 \big)^2 \big\| = |t(1-t) | \big\|2+J_0J_1 + J_1J_0 \big\| \\
  &= |t(1-t) | \,  \big\| (J_0 -J_1)(J_1-J_0) \big\| = |t(1-t) |  \, \big\| J_0 -J_1 \big\|^2.
\end{align*}
Because $\big|t (t-1)\big| \leq \frac{1}{4}$ with equality at $t=1/2$, 
 if $(J_0, J_1)$ is a Fredholm pair, then $F_t$ is invertible in $\End_A(X)$ for all $t \in [0,1/2) \cup (1/2, 1]$ since $\Vert J_0-J_1\Vert^2\leq 4$. 
 Furthermore, $\big\| \one + q(F_t)^2 \big\|_{\calQ_A(X)} < 1$ and $F_t$ is Fredholm for all $t\in[0,1]$ by Lemma \ref{lemma:bdd_unbdd_Fred_condition}. Hence $\sup_{t\in(0,1)}\big\| \one + q(F_t)^2 \big\|_{\calQ_A(X)} < 1$ and so
  $\big\| \one + q( F_\bullet)^2 \big\|_{\calQ_{SA}(SX)}<1$. 
This shows, by Lemma \ref{lemma:bdd_unbdd_Fred_condition}, that  $F_\bullet$ is Fredholm on $SX_{SA}$. 
The Kasparov module is then obtained from Corollary \ref{cor:skew_Fred_to_KasMod}. 
\end{proof}

\begin{remark} \label{rk:rewriting_Ind(J0,J1)}
  If $J_0, J_1 \in \End_A^{r,s}(X)$ are a Fredholm pair of skew-adjoint unitaries with $J_0 - J_1 \in \End_A^0(X)$, then 
  the operator $F_\bullet \otimes \rho$ from Lemma \ref{lemma:Fredholm_pair_straight_line_path} is not only Fredholm but in fact already defines a Kasparov module (i.e., we do not need to consider a normalising function). 
 For a general Fredholm pair of skew-adjoint unitaries $(J_0, J_1)$, 
 we can always replace $J_1$ by a skew-adjoint unitary $\tilde{J}_1 \in \End_A^{r,s}(X)$ 
  such that $\tilde{J}_1 - J_0 \in \End_A^0(X)$. 
  Indeed, $\tilde{J}_1$ is obtained from $J_1$ by Lemma \ref{lem:Fred_pair_via_rel_DK}, and the homotopy from $J_1$ to $\tilde{J}_1$ also yields a homotopy from $\{F_t := (1-t)J_0+tJ_1\}_{t\in[0,1]}$ to $\{\tilde{F}_t := (1-t)J_0+t\tilde{J}_1\}_{t\in[0,1]}$. 
  Hence the $\KKR$-class $\big[ \chi(F_\bullet) \otimes \rho \big]$ of Lemma \ref{lemma:Fredholm_pair_straight_line_path} can equivalently be described as $\big[ \tilde{F}_\bullet \otimes \rho \big]$. 
\end{remark}

\begin{prop} \label{prop:KK_DK_ComplexStrucIndex_agrees}
Let $J_0, J_1 \in \End_{A}^{r,s}(X)$ be a Fredholm pair of skew-adjoint unitaries.
Consider the family $\{F_t\}_{t\in [0,1]} \subset \End_A^{r,s}(X)$ given by $F_t  = (1-t)J_0 + t J_1$.
Then the Kubota isomorphism $\Kubota_A^{r,s+1}:\KKR(\Cl_{s+1,r},SA)\to \DK(A\ox\Cl_{r,s+1})$ 
maps the $\KK$-class $\class{F_\bullet} := [\chi(F_\bullet)\otimes\rho] \in \KKR(\Cl_{s+1,r}, SA)$ to the relative index of $(J_1,J_0)$ in $\DK( A \otimes \Cl_{r,s+1} )$, 
\[
\Kubota_A^{r,s+1} \big( \class{F_\bullet} \big) = \relind_{r,s+1}(J_1,J_0) = -\relind_{r,s+1}(J_0,J_1) .
\]
\end{prop}
\begin{proof}
We write $[J \otimes \rho] =   \big[ \sigma_{J_\mathrm{ref}\otimes \rho}^{r,s}( J \otimes \rho)  \big] $ for brevity.
Given a Fredholm pair of skew-adjoint unitaries, the relative index is given by 
\[
  \relind_{r,s+1}(J_1, J_0) = \calM_X^{\DK} \circ \exc_X \big( [J_1 \otimes \rho] - [J_0 \otimes \rho] \big)
   = \calM_X^{\DK} \circ \exc_X\big( [\tilde{J}_1 \otimes \rho] - [J_0 \otimes \rho] \big),
\]
where $\tilde{J}_1 = v J_1v^*$ is a skew-adjoint unitary homotopic to $J_1$  such that  $q(\tilde{J}_1) = q(J_0)$ 
(see Lemma \ref{lem:Fred_pair_via_rel_DK}). By  Remark \ref{rk:rewriting_Ind(J0,J1)}, 
the $\KK$-class $\class{F_\bullet}$ can equivalently be described by the Kasparov module 
\[
 \Big( \Cl_{s+1, r}, \,  SX_{SA} \otimes \exterior \C, \,   \tilde{F}_\bullet \otimes \rho \Big) , \qquad \tilde{F}_t = (1-t)J_0 + t\tilde{J}_1.
\]
The latter satisfies the hypotheses of Proposition \ref{prop:concrete_Kubota_computation}, 
which we can apply to obtain 
\begin{align*}
     \Kubota_A^{r,s+1}   \big( [\tilde{F}_\bullet \otimes \rho] \big) 
     &= \calM^{\DK}_{X} \circ \exc_{X} \big( [  \tilde{F}_1  \otimes \rho ] - [  \tilde{F}_0  \otimes \rho ] \big) \\
       &= \calM_{X}^{\DK} \circ \exc_{X}\big( [ \tilde{J}_1 \otimes \rho ] - [ J_0 \otimes \rho ] \big) 
       = \relind_{r,s+1}(J_1, J_0) .  \qedhere
\end{align*}
\end{proof}

We can therefore use the homotopy invariance of $\KK$-theory to obtain a stronger homotopy invariance 
of  $\relind_{r,s+1}(J_0, J_1)$.

\begin{cor} \label{cor:Ind(J0,J1)_htpy_invariant}
Let $\{J_{0,\lambda}\}_{\lambda\in [0,1]}$ and $\{J_{1,\lambda}\}_{\lambda\in [0,1]}$ be strongly continuous paths of skew-adjoint unitaries in 
$\End_A^{r,s}(X)$ such that $\big\| q(J_{0,\bullet}) - q( J_{1,\bullet}) \big\|_{\calQ_{C([0,1],A)}(C([0,1],X))} < 2$. Then  
$\relind_{r,s+1}(J_{0,\lambda}, J_{1,\lambda} )$ is well-defined and constant for all $\lambda \in [0,1]$.
\end{cor}
\begin{proof}
The conditions ensure that 
\[
   \{F_{t,\lambda}\}_{(t,\lambda) \in [0,1]\times[0,1]} \subset \End_A^{r,s}(X), \qquad 
   F_{t,\lambda} = J_{0,\lambda} + t(J_{1,\lambda} - J_{0,\lambda} )
\]
gives a  homotopy of Kasparov modules in $\KKR(\Cl_{s+1,r}, SA)$, \cite[\S17.2]{Blackadar}. Indeed, 
$(t,\lambda) \mapsto F_{t,\lambda}$ is strongly continuous and 
\begin{align*}
  \big\| \one + q(F_{t,\bullet})^2 \big\|_{\calQ_{C([0,1],A)}(C([0,1],X))}  &= |t(1-t)| \, \big\| 2 + q(J_{0,\bullet}J_{1,\bullet}) + q(J_{1,\bullet}J_{0,\bullet}) \big\|_{\calQ_{C([0,1],A))}(C([0,1],X)} \\
  &= |t(1-t)| \, \big\| q(J_{1,\bullet}) - q(J_{0,\bullet}) \big\|_{\calQ_{C([0,1],A)}(C([0,1],X))}^2 < 1
\end{align*}
for all $t\in [0,1]$. Thus we get a homotopy of Kasparov modules by Corollary \ref{cor:Wahl_family_to_KK_htpy} below, and the statement follows from Proposition \ref{prop:KK_DK_ComplexStrucIndex_agrees}. 
\end{proof}

\subsection{A  \texorpdfstring{$\Z_2$}{Z2}-graded relative index}

As we have done with the Real Fredhom index (Theorem \ref{thm:Z_2-graded_Fred_index}), we can also define 
a relative index for odd self-adjoint unitaries on $\Z_2$-graded $C^*$-modules. This assumption puts us directly in the 
setting of Van Daele $K$-theory, which is defined by odd self-adjoint unitaries, and  the Clifford shuffle 
map   $\sigma^{r,s}_{J_\mathrm{ref}\otimes \rho}$ is not required.

\begin{defn}
Let $Y_B$ be a $\Z_2$-graded $C^*$-module and $G_0, G_1 \in \End_B(Y)$ odd self-adjoint Real unitaries. We say that 
$(G_0, G_1)$ are a Fredholm pair if $\big\| q(G_0) - q(G_1) \big\|_{\calQ_B(Y) } < 2$ and define 
\[
   \wh{\relind}(G_0, G_1) = \calM_Y^{\DK} \circ \exc_Y \big( [G_0] - [G_1] \big) \in \DK( B ).
\]
\end{defn}

The index is well-defined by Lemma \ref{lem:relative_DK}.
The analogue of Lemma \ref{lemma:Fredholm_pair_straight_line_path} and 
Proposition \ref{prop:KK_DK_ComplexStrucIndex_agrees} in this setting is the following (the proof is the same with 
minor adjustments).

\begin{prop} \label{prop:Z_2-rel-index-properties}
Let $Y_B$ be a $\Z_2$-graded full $C^*$-module and $G_0, G_1 \in \End_B(Y)$ a Fredholm pair of odd self-adjoint Real unitaries.
\begin{enumerate}
  \item The triple 
  \[
    \big( \C, \, SY_{SB}, \, \chi( G_\bullet ) \big), \qquad G_t = (1-t)G_0 + t G_1
  \]
  is a Real Kasparov module with class $[G_\bullet] \in \KKR(\C, SB)$.
  \item The $\Z_2$-graded Kubota isomorphism $\Psi_B: \KKR(\C, SB) \xrightarrow{\simeq} \DK( B )$ 
  from Theorem \ref{thm:graded_Roe_and_Kubota} is such that 
  \[
    \Psi_B\big( [G_\bullet] \big) = \wh{\relind}(G_1, G_0) = -\wh{\relind}(G_0, G_1) . 
  \]
\end{enumerate}
\end{prop}

In the special case of $\Z_2$-graded \emph{complex} $C^*$-modules over ungraded $C^*$-algebras, we 
can relate the  $\Z_2$-graded and non-graded relative index.

\begin{lemma} \label{lem:Rel_ind_graded_to_ungraded}
Let $\calH_A$ be the standard ungraded  complex $C^*$-module and 
$J_0= i (2P_0 - \one), \, J_1= i(2P_1 - \one)  \in \End_A^{1,0}(\calH_A)$  a Fredholm pair 
of skew-adjoint unitaries.
Using the self-adjoint Clifford generator $e \in \End_A(\calH_A)$, consider 
$\calH_A \cong \calH_A^+ \oplus \calH_A^- \cong \calH_A^+ \otimes \exterior{\C}$ 
as a $\Z_2$-graded $C^*$-module graded by $e$. Then  $2P_0-\one, 2P_1-\one \in   \End_{A}(\calH_A) \otimes \Cl_{1,1} $
is a Fredholm pair of odd self-adjoint unitaries such that 
\[
    \relind_{1,1}( J_0, J_1) =   \wh{ \relind} (2P_0-\one, 2P_1-\one) \in \DK( A \otimes \Cl_{1,1} ).
\]
\end{lemma}
Note that because $A$ is ungraded, the range of the graded relative index is $\DK( A \otimes \Cl_{1,1} )$.
\begin{proof}
The proof is similar to Proposition \ref{prop:Graded_ungraded_index_comparison}. 
Using the isomorphism $\Cl_{1,1} \simeq M_2(\C)$, the $\Z_2$-graded relative index can be written 
\[
   \wh{ \relind} (2P_0-\one, 2P_1-\one) =  \calM^{\DK}_{\calH_A} \circ \exc_{\calH_A} \left(   \left[  \begin{pmatrix} 0 & U_0^* \\  U_0 & 0    \end{pmatrix} \right] - 
     \left[ \begin{pmatrix} 0 & U_1^* \\ U_1 & 0 \end{pmatrix} \right]  \right)
   \in \DK( A \otimes \Cl_{1,1} ) , 
\]
where $U_j = (2P_j - \one) \tfrac{1}{2}( \one + e )$ for $j = 0,1$. 
For the ungraded index, we consider 
$P^{1,0}\big(i(2P_j - \one) \otimes \rho\big)$ with $j=0,1$ and  
$P^{1,0} =  \tfrac{1}{2}( \one + e\otimes \gamma \rho) \in \End_A(\calH_A) \otimes \Cl_{1,1}$.
We again relate $\Cl_{1,1} \cong M_2(\C)$ and can write 
\begin{align*}
   P^{1,0}\big( i(2P_j - \one)\otimes \rho\big) 
   = \begin{pmatrix} 0 & -i (2P_j - \one) \frac{1}{2}(\one - e) \\  i(2P_j - \one) \frac{1}{2}(\one + e)  & 0 \end{pmatrix} 
   =  \begin{pmatrix} 0 & -iU_j^* \\  iU_j & 0    \end{pmatrix}.
\end{align*}
The factor of $\pm i$ will not affect the complex Van Daele class and therefore
\begin{align*}
    \relind_{1,1}(J_0, J_1) &=  \calM^{\DK}_{\calH_A} \circ \exc_{\calH_A} \left(   \left[  \begin{pmatrix} 0 & -iU_0^* \\  iU_0 & 0    \end{pmatrix} \right] - 
     \left[ \begin{pmatrix} 0 & -iU_1^* \\ iU_1 & 0 \end{pmatrix} \right]  \right) \\
     &=  \wh{ \relind} (2P_0-\one, 2P_1-\one).   \qedhere
\end{align*}
\end{proof}

In the setting of Lemma \ref{lem:Rel_ind_graded_to_ungraded}  we can also 
apply the isomorphism $\Upsilon_{A}^{1,1}: \DK(A \otimes \Cl_{1,1} ) \xrightarrow{\simeq} K_1(A)$, where by Lemma \ref{lemma:Upsilon_Morita_commute_complex}
\begin{align}  \label{eq:Rel_index_graded_ungraded_complex}
   \Upsilon_{A}^{1,1}\big(  \relind_{1,1}( J_0, J_1) \big) &=   \Upsilon_{A}^{1,1}\big(   \wh{ \relind} (2P_0-\one, 2P_1-\one) \big) \nonumber \\
   &= \calM^{K}_{\calH_A} \circ \exc_{\calH_A} \big( [U_0] - [U_1] \big) \in K_1(A).
\end{align}

\subsection{Comparison with complex Hilbert modules} \label{subsec:complex_J0J1_index}

We now make a comparison with the relative index on the standard complex $C^*$-module $\calH_A$ considered by 
Wahl and Ng--Sutradhar--Wang (\cite[\S3, \S8]{Wahl07} and \cite[\S4]{NSW2}). 
If $(P_0, P_1)$ are a Fredholm pair of projections, $\big\| q(P_0) - q(P_1) \big\|_{\calQ_A(\calH_A) } < 1$, then 
the even relative index in complex $K$-theory is given in Eq. \eqref{eq:relind_in_relK} by 
\[
\relind_0^K(P_0,P_1) = \calM_{\calH_A}^K \circ \exc_{\calH_A} \big( [P_0] - [P_1] \big) \in K_0(A) ,
\]
If $\hat{\calH}_A \simeq \calH_A \otimes \C^2$ is a $\Z_2$-graded module  such that $2P_0 - \one$ and $2P_1 - \one$ are odd, then 
 we may define the \emph{odd relative index} in complex $K$-theory by 
\[
\relind_1^K(P_0,P_1) := \calM_{\calH_A}^K \circ \exc_{\calH_A} \big( [U_0] - [U_1] \big) \in K_1(A) , \qquad 
 (2P_j - \one) = \begin{pmatrix} 0 & U_j^* \\ U_j & 0 \end{pmatrix}, \ j= 0,1.
\]
If $U_0-U_1 \in \End_A^0(\calH_A)$, then $\exc_{\calH_A} \big( [U_0] - [U_1] \big) = [U_0 U_1^*] \in K_1\big( \End_A^0(\calH_A) \big)$. 
This definition of the odd relative index therefore agrees with \cite[\S8.1]{Wahl07} (see also \cite[Appendix A.2]{vdDungen24}).

\begin{prop} \label{prop:complex_relative_index}
Let $A$ be an ungraded complex $C^*$-algebra, $\calH_A$  the standard ungraded complex 
$C^*$-module over $A$, and $\hat{\calH}_A$ the standard graded complex $C^*$-module over $A$.
\begin{enumerate}
\item Let $(J_0,J_1)$ be a Fredholm pair of skew-adjoint unitaries in $\End_A(\calH_A)$ 
and $P_k \in \End_A(\calH_A)$ such that $J_k = i(2P_k - \one)$, $k=0,1$. 
Then, under the isomorphism $\Upsilon_A^{1,0}: \DK(A \otimes \Cl_{1+0} ) \xrightarrow{\simeq} K_0(A)$ from Proposition \ref{prop:iso_DK_even}, we have 
\[
\relind_{0,1}(J_0,J_1) \xmapsto{\Upsilon_A^{1,0}} \relind_0^K(P_0,P_1) \in K_0(A) .
\]
\item Let $(G_0, G_1) = (2P_0 - \one, 2P_1 - \one)$ be a Fredholm pair of odd self-adjoint unitaries in 
$\End_A(\hat{\calH}_A)$. 
Then the isomorphism $\Upsilon_A^{1,1}: \DK(A \otimes \Cl_{1+1} ) \to K_1(A)$ is such that 
\[
    \Upsilon_A^{1,1} \big( \wh{\relind}( 2P_0 - \one, 2P_1 - \one ) \big) = \Upsilon_{A}^{1,1}\big( \relind_{1,1}( i(2P_0 - \one), i(2P_1 - \one) ) \big) 
    = \relind_1^K(P_0,P_1).
\]
\end{enumerate}
\end{prop}
\begin{proof}
(1) We first note that $\| q(P_1) - q(P_0) \|_{\calQ_A(\calH_A)} = \tfrac{1}{2} \| q(J_1) - q(J_0) \|_{\calQ_A(\calH_A)} < 1$, which means that 
the class $\exc_{\calH_A}\big( [P_1] - [P_0] \big) \in K_0\big( \End_A^0(\calH_A) \big)$ is well-defined.
Because the space is complex, we can write 
\[
 J_k \otimes \rho =   i(2P_k - \one) \otimes \rho = (2P_k - \one) \otimes i \rho \in  \End_A(\calH_A) \otimes \Cl_{1+0}, 
 \quad k= 0,1,
\]
with $i\rho$ a self-adjoint Clifford generator of $\Cl_{1+0}$. The isomorphism 
\[
\Upsilon_{\End_A^0(\calH_A)}^{1,0}: \DK\big( \End_A^0(\calH_A) \otimes \Cl_{1+0}\big) \xrightarrow{\simeq} K_0 \big( \End_A^0(\calH_A)\big)
\] 
is such that 
\[
   \Upsilon_{\End_A^0(\calH_A)}^{1,0}\circ  \exc_{\calH_A} \big( [(2P_1 - \one) \otimes i\rho] - [(2P_0 - \one)  \otimes i\rho] \big) =
    \exc_{\calH_A}\big( [P_1] - [P_0] \big) \in K_0 \big( \End_A^0(\calH_A) \big).
\]
The result follows as the isomorphism $\Upsilon^{1,0}$ is compatible with Morita invariance in $K$-theory (Lemma \ref{lemma:Upsilon_Morita_commute_complex}).
Part (2) is a restatement of Eq. \eqref{eq:Rel_index_graded_ungraded_complex}.
\end{proof}

\subsection{Comparison with real Hilbert spaces} \label{subsec:HSpace_J0J1_index}

We now show that our $\DK$-theoretic relative index of a Fredholm pair of skew-adjoint unitaries 
agrees with the index previously studied in~\cite[\S4]{BCLR} in the case of a real Hilbert space.
Let us therefore fix an ungraded Hilbert space $\calH$ with real structure $\frakr$ and an ample and ungraded Real $\Cl_{r,s}$-representation 
(we assume $s \geq 1$ by adding more generators if necessary). 
All operators of interest will be invariant under $\frakr$ and so our results apply to real Hilbert spaces 
as considered in~\cite{BCLR}.
We denote by $\calB^{r,s}(\calH)$ the bounded operators on $\calH$ anti-commuting with 
the generators of this $\Cl_{r,s}$-representation and $\calJ^{r,s}(\calH)$ the (Real) skew-adjoint unitaries in $\calB^{r,s}(\calH)$.

If $J_0, J_1 \in \calJ^{r,s}(\calH)$ are a Fredholm pair, 
 the operator $F = \frac{1}{2}(J_0+J_1)$ is skew-adjoint 
and Fredholm: 
indeed, $\| \one + q(F)^2 \|_{\calQ(\calH)}= \frac{1}{4} \| q(J_0) - q(J_1)\|_{\calQ(\calH)}^2 < 1$ so $q(F) \in \calQ(\calH)$ is invertible. 
Hence $\Ker( J_0 + J_1)$ 
is a finite-dimensional subspace that is invariant under the real structure on $\calH$. The relation $J_1( J_0 + J_1) = (J_0 + J_1)J_0$ 
implies that $J_0 \cdot \Ker(J_0 + J_1) \subset \Ker(J_0 + J_1)$, and so $\Ker(J_0+J_1)$ is a $\Cl_{r,s+1}$-module with 
generators $\{e_1,\ldots, e_r, f_1,\dots, f_{s}, J_0\}$. Then recalling that $\calM_{r,s+1}$ denotes the 
Grothendieck group of ungraded $\Cl_{r,s+1}$-modules,
\[
   \Ind_{r,s+2}^\calH( J_0, J_1) = \big[ \Ker(J_0 + J_1) \big] \in \calM_{r,s+1}/ \calM_{r,s+2} \cong \KO_{2+s-r}(\R),
\]
is the relative index of a Fredholm pair of skew-adjoint unitaries on Hilbert spaces~\cite[Theorem 4.5]{BCLR}.
We now show the compatibility of $\Ind_{r,s+2}^\calH( J_0, J_1)$ with 
$\relind_{r,s+1}(J_0, J_1) \in \DK(\Cl_{r,s+1})$.

\begin{prop} \label{prop:Fred_pair_index_agrees_with_hspace}
Let $J_0, J_1 \in \calJ^{r,s}(\calH)$ be a Fredholm pair of skew-adjoint unitaries in $\calH$. 
Then there is an  isomorphism $\Upsilon^{r,s+1}_{\C} \colon \DK(  \Cl_{r,s+1}) \xrightarrow{\simeq} \KO_{s+2-r}(\R)$ such that  
\[
   \relind_{r,s+1}(J_0, J_1) \xmapsto{\Upsilon^{r,s+1}_{\C}} \big[ \Ker(J_0+J_1) \big] = \Ind_{r,s+2}^\calH( J_0, J_1)  \in \calM_{r,s+1}/ \calM_{r,s+2}.
\]
\end{prop}

\begin{proof}
If $(J_0, J_1)$ is a Fredholm pair, then $\relind_{r,s+1}(J_0, J_1) \in \DK(  \Cl_{r,s+1}) $ 
and $ \big[ \Ker(J_0+J_1) \big] \in \calM_{r,s+1}/ \calM_{r,s+2}$ are well-defined. Furthermore, 
any class in $\KO_{s+2-r}(\R)$ can be represented by $ \big[ \Ker(J_0+J_1) \big]$ for some 
Fredholm pair \cite[\S4]{BCLR}.  
We just have to show 
that the induced map  $\Upsilon_\C^{r,s+1}\big( \relind_{r,s+1}(J_0, J_1) \big) = \big[ \Ker(J_0+J_1) \big]$ is well-defined 
and injective.

 If $\relind_{r,s+1}(J_0, J_1)  = \relind_{r,s+1}(J_0', J_1')  \in \DK(  \Cl_{r,s+1})$, 
 then  there are homotopies in $\calJ^{r,s}(\calH)$ from $J_0$ to $J_0'$ and 
 $J_1$ to $J_1'$ such that 
 $(J_{0,t}, J_{1,t})$ is a Fredholm pair for all $t \in [0,1]$ (and we have absorbed the 
 matrix degrees of freedom into $\calH$). We can write 
 \[
   \relind_{r,s+1}(J_0', J_1') = \relind_{r,s+1}(J_0', J_0) + \relind_{r,s+1}(J_0, J_1) + \relind_{r,s+1}(J_1, J_1') 
 \]
 with $\relind_{r,s+1}(J_0', J_0)$ and $ \relind_{r,s+1}(J_1, J_1')$ trivial. Similarly, by the homotopy invariance 
 and addition property of the relative index on Hilbert spaces, 
 \[
   \Ind_{r,s+2}^\calH( J_0', J_1') = \Ind_{r,s+2}^\calH( J_0', J_0)   + \Ind_{r,s+2}^\calH( J_0, J_1)  + \Ind_{r,s+2}^\calH( J_1, J_1' ) = \Ind_{r,s+2}^\calH( J_0, J_1)  
 \]
and the map is well defined.

If $\big[\Ker(J_0+J_1) \big] \in \calM_{r,s+1}/ \calM_{r,s+2}$ is trivial, then by~\cite[Theorem 4.5]{BCLR},
\[
  \phi_{J_0, J_1} :[0,1] \to  \big\{ J \in \calJ^{r,s-1}(\calH) \, \big| \,  \big\| q(J) - q(f_s) \big\|_{\calQ(\calH)} < 2 \big\}, \quad 
  \phi_{J_0, J_1}(t) = -f_s \exp\big( \pi( (1-t)J_0 + tJ_1) f_s\big)
\]
 is a contractible loop.
This then implies that the path of skew-adjoint 
Fredholm operators $F_t := (1-t)J_0 + t J_1$ will also be topologically trivial and so $[F_\bullet\otimes\rho] \in \KKR(\Cl_{s+1,r}, SA)$ is trivial. 
By Proposition \ref{prop:KK_DK_ComplexStrucIndex_agrees},  $\relind_{r,s+1}(J_1, J_0)$ and 
hence $\relind_{r,s+1}(J_0, J_1)$ will also trivial.  Therefore the map $\Upsilon^{r,s+1}_{\C}$ is injective.
\end{proof}

\begin{remark}
Using the results from~\cite[\S4.2]{BCLR}, it follows that our relative index is compatible with the usual relative index $\Ind( P_0, P_1)$ of a pair of projections on a real Hilbert space:
\[
   \Upsilon^{2,1}_\C \big( \relind_{2,1}(J_0, J_1) \big) = \Ind( P_0, P_1) \in \KO_0(\R), \qquad J_l = \begin{pmatrix} 0 & -(2P_l - \one) \\ 2P_l - \one & 0 \end{pmatrix}, \ \   l=0,1.
\]
Similarly, in the complex case,
\[
   \Upsilon^{0,1}_\C \big( \relind_{0+1}(J_0, J_1) \big) = \Ind( P_0, P_1) \in K_0(\C), \qquad J_l = i(2P_l - \one), \ \  l = 0,1.
\]
\end{remark}

\section{Families of Fredholm operators} \label{sec:families}

As in \S\ref{sec:Fredholm_Real}, $Y_B$ is a (possibly $\Z_2$-graded) countably generated $C^*$-module over a (possibly $\Z_2$-graded) $C^*$-algebra.
In this section, we will consider families of operators parametrised by a compact Hausdorff space $\Omega$. 
If $\{S_\omega\}_{\omega\in\Omega}$ is such a family of operators on $Y_B$, then we define the \emph{family operator} $S_\bullet$ on $C(\Omega,Y_B)_{C(\Omega,B)}$ 
by $(S_\bullet y)(\omega) := S_\omega y(\omega)$ for $y\in C(\Omega,Y_B)$. 
To ensure that $S_\bullet$ is a well-defined adjointable/compact operator (in the bounded case) or regular operator (in the unbounded case), we need the family $\{S_\omega\}_{\omega\in\Omega}$ to be suitably continuous, in the following precise sense. 

\begin{lemma} \label{lemma:strong_cts_regular}
Let $\Omega$ be a compact Hausdorff space. 
\begin{enumerate}
 \item A family $\{S_\omega\}_{\omega \in \Omega} \subset \End_B(Y)$ defines an adjointable operator 
  $S_\bullet \in \End_{C(\Omega, B)}\big( C(\Omega, Y_B)\big)$ 
  if and only if $\omega \mapsto S_\omega$ and 
  $\omega \mapsto S_\omega^*$ are strongly continuous.
 \item A family $\{S_\omega\}_{\omega \in \Omega} \subset \End_B^0(Y)$ defines a compact operator 
  $S_\bullet \in \End_{C(\Omega, B)}^0\big( C(\Omega, Y_B)\big)$ 
  if and only if $\omega \mapsto S_\omega$ is norm-continuous.
 \item A self-adjoint family $\{D_\omega\}_{\omega\in\Omega} \subset \Reg_B(Y)$ defines 
  a regular self-adjoint operator $D_\bullet \in \Reg_{C(\Omega,B)}\big( C(\Omega,Y_B) \big)$ on 
  \[
  \Dom(D_\bullet) 
  := \big\{ y \in C(\Omega,Y_B) : y(\omega)\in\Dom(D_\omega) \text{ and } D_\bullet y \in C(\Omega,Y_B) \big\} 
  \subset C(\Omega, Y_B)_{C(\Omega, B)}
  \]
  if and only if $\omega\mapsto ( i \pm D_\omega)^{-1}$ is strongly continuous.
 \item A skew-adjoint family $\{D_\omega\}_{\omega\in\Omega} \subset \Reg_B(Y)$ defines 
  a regular skew-adjoint operator $T_\bullet \in \Reg_{C(\Omega,B)}\big( C(\Omega,Y_B) \big)$ on 
  \[
  \Dom(T_\bullet) 
  := \big\{ y \in C(\Omega,Y_B) : y(\omega)\in\Dom(T_\omega) \text{ and } T_\bullet y \in C(\Omega,Y_B) \big\} 
  \subset C(\Omega, Y_B)_{C(\Omega, B)}
  \]
  if and only if $\omega\mapsto ( \one \pm T_\omega)^{-1}$ is strongly continuous.
\end{enumerate}
\end{lemma}
\begin{proof}
Part (1) is immediate from the definition of the strong/strict topology for adjointable operators on $C^*$-modules. 
Part (2) follows from the observation that a compact operator is a norm-limit of finite-rank operators. 
Part (3) is proven in \cite[Proposition 2.5]{Wahl07}, and part (4) is an easy modification of part (3).
\end{proof}

Given a family $\{F_\omega\}_{\omega \in \Omega} \subset \End_B(Y)$ of \emph{Fredholm} operators, such that 
$\omega \mapsto F_\omega$ and $\omega \mapsto F_\omega^*$ are strongly continuous, 
it is not guaranteed that the family operator $F_\bullet \in \End_{C(\Omega, B)}\big( C(\Omega, Y_B)\big)$ is also Fredholm 
(a counterexample can be found in \cite[page 10]{Wahl07}).  
Hence some care needs to be taken to ensure that the family operator is again Fredholm. 

Combining Lemma \ref{lemma:strong_cts_regular} with Lemma \ref{lemma:bdd_unbdd_Fred_condition}, we obtain the following characterisation for the regularity and Fredholmness of the family operator.
\begin{prop}[cf.\ {\cite[Lemma 2.7]{Wahl07}}] \label{prop:Family_Fredholm_condition}
A family $\{T_\omega\}_{\omega\in\Omega} \subset \Reg_B(Y)$ of regular skew-adjoint Fredholm operators on $Y_B$ defines 
a regular skew-adjoint Fredholm operator $T_\bullet$ on $C(\Omega, Y_B)_{C(\Omega, B)}$ if and only if 
\[
   \omega \mapsto (\one \pm T_\omega)^{-1} \,\, \text{ is strongly continuous, and}\quad 
   \big\|  q\big( (\one + T_\bullet)^{-1} \big) \big\|_{\calQ_{C(\Omega, B)} (C(\Omega, Y_B))} < 1.
\]
An analogous statement holds for a family of self-adjoint regular Fredholm operators.
\end{prop}

\subsection{The Wahl topology on a Hilbert module} \label{subsec:Wahl_topology}

In {\S}\ref{sec:spectral_flow}, we will define the spectral flow for a family of Real skew-adjoint Fredholm operators $\{T_t\}_{t\in[0,1]} \subset \Reg^{r,s}_A(X)$ with invertible endpoints. 
For this purpose, we will need the family to be ``sufficiently continuous'', such that the family operator $T_\bullet \in \Reg^{r,s}_{SA}(SX)$ is also regular and Fredholm (note that $T_\bullet$ is automatically Real and skew-adjoint). 
We have already seen in Lemma \ref{lemma:strong_cts_regular} that the regularity of $T_\bullet$ requires the family $\{T_t\}_{t\in[0,1]}$ to have strongly continuous resolvents. 

The Fredholmness of $T_\bullet$ additionally requires that we can find a ``suitably continuous'' parametrix $Q_\bullet = \{Q_t\}_{t\in[0,1]}$ for $T_\bullet$. 
We will see that, to obtain the Fredholm property, the appropriate notion of continuity is given by the \emph{Wahl topology}. 
This topology was introduced by Wahl \cite{Wahl08} for self-adjoint Fredholm operators on Hilbert spaces. 
As remarked in~\cite{Wahl08}, this topology and its basic properties can be naturally extended to operators on Hilbert $C^*$-modules. 
Here, we will review the construction of the Wahl topology for both self-adjoint and skew-adjoint Fredholm operators on $C^*$-modules.
(If $Y_B$ is a Real $C^*$-module, we will silently assume all operators to be Real as well.)

Consider a family of regular self-adjoint (resp.\ skew-adjoint) Fredholm operators $\{D_\omega\}_{\omega\in\Omega}$, such that the resolvents are strongly continuous (cf.\ Lemma \ref{lemma:strong_cts_regular}). 
From Proposition \ref{prop:Fredholm_properties} we know that, for each $\omega\in\Omega$, 
there is an $\varepsilon >0$ such that for all $\varphi \in C_c(-\varepsilon, \varepsilon)$ (resp.\ $\varphi \in C_c(-i\varepsilon, i\varepsilon)$), $\varphi(D_\omega)$ is compact. 
The basic idea of the Wahl topology is then to ensure that we can choose such an $\varepsilon$ \emph{uniformly}, such that $\omega \mapsto \varphi(D_\omega)$ is \emph{norm-continuous}. 

We consider the sets 
\begin{align*}
   \Reg_B^{\sa}(Y) &= \big\{ \text{self-adjoint } D \in \Reg_B(Y) \big\} , &
   \Reg_B^{\sk}(Y) &= \big\{ \text{skew-adjoint } T \in \Reg_B(Y) \big\}.
\end{align*}

\begin{defn}[Wahl topology]
Let $\phi \in C_c^\infty(\R)$ be a non-negative even function with $\supp(\phi) = [-1,1]$, and $\phi'(x) > 0$ for $x \in (-1,0)$. 
Define $\phi_n\in C_c^\infty(\R)$ by $\phi_n(x) := \phi(nx)$ for $1\leq n\in\N$. 
We also define $\psi,\psi_n \in C_c^\infty( i\R)$ such that $\psi(ix) = i\phi(x)$  and $\psi_n(ix) = i\phi_n(x)$. 
Let $\mathfrak{S}^{\sa}_n(Y_B)$ and $\mathfrak{S}_n^{\sk}(Y_B)$ denote the sets 
$\Reg_B^{\sa}(Y)$ and $\Reg_B^{\sk}(Y)$, respectively, endowed with the 
weakest topology such that, for all $y \in Y_B$, the following maps are continuous:
\begin{align*}
   &\Reg_B^{\sa}(Y) \to Y_B, \quad D \mapsto (i \pm D)^{-1} y, 
   &&\Reg_B^{\sa}(Y) \to \End_B(Y), \quad D \mapsto \phi_n(D), \\
   &\Reg_B^{\sk}(Y) \to Y_B, \quad T \mapsto (\one \pm T)^{-1} y, 
   &&\Reg_B^{\sk}(Y) \to \End_B(Y), \quad T \mapsto \psi_n(T).
\end{align*}
For $m\leq n$, we have continuous inclusions $\mathfrak{S}^{\sa}_m(Y_B) \hookrightarrow \mathfrak{S}^{\sa}_n(Y_B)$ 
and $\mathfrak{S}^{\sk}_m(Y_B) \hookrightarrow \mathfrak{S}^{\sk}_n(Y_B)$. 
We then define $\mathfrak{S}^{\sa}_B(Y)$ and $\mathfrak{S}^{\sk}_B(Y)$ 
to be the sets $\Reg_B^{\sa}(Y)$ and $\Reg_B^{\sk}(Y)$, respectively, 
endowed with the direct limit topology.
We will call this direct limit topology on $\Reg_B^{\sa}(Y)$ and $\Reg_B^{\sk}(Y)$ the \emph{Wahl topology}. 

If $\Omega$ is a compact Hausdorff space, 
then a family $\{D_\omega\}_{\omega\in\Omega} \subset \Reg_B^{\sa}(Y)$ (resp.\ $\{T_\omega\}_{\omega\in\Omega} \subset \Reg_B^{\sk}(Y)$) is called \emph{Wahl-continuous}, 
if the map $\omega \mapsto D_\omega \in \mathfrak{S}^{\sa}_B(Y)$ (resp.\ $\omega \mapsto T_\omega \in \mathfrak{S}^{\sk}_B(Y)$) is continuous. 
By definition of the Wahl topology, this means that the resolvents $\omega \mapsto (i \pm D_\omega)^{-1}$ (resp.\ $\omega \mapsto (\one \pm T_\omega)^{-1}$) are strongly continuous, and there exists an $n\in\N$ such that $\omega \mapsto \phi_n(D_\omega)$ (resp.\ $\omega \mapsto \phi_n(T_\omega)$) is norm-continuous. 
\end{defn}

\begin{defn}[Wahl topology for Fredholm operators]
We define the spaces $\mathfrak{F}_n^{\sa}(Y_B)$ and $\mathfrak{F}_n^{\sk}(Y_B)$ to be the sets 
\begin{align*}
  &\mathfrak{F}_n^{\sa}(Y_B) = \big\{ D \in \Reg_B^{\sa}(Y) \mid \phi_n(D) \in \End_B^0(Y) \big\}, 
  &&\mathfrak{F}_n^{\sk}(Y_B) = \big\{ T \in \Reg_B^{\sk}(Y) \mid \psi_n(T) \in \End_B^0(Y) \big\}
\end{align*}
with the subspace topology of $\mathfrak{S}^{\sa}_n(Y_B)$ and $\mathfrak{S}^{\sk}_n(Y_B)$, respectively.
We define the direct limits $\mathfrak{F}^{\sa}_B(Y) = \varinjlim  \mathfrak{F}_n^{\sa}(Y_B)$ and $\mathfrak{F}^{\sk}_B(Y) = \varinjlim  \mathfrak{F}_n^{\sk}(Y_B)$. 
\end{defn}

Note that the spaces $\mathfrak{F}^{\sa}_B(Y)$ and $\mathfrak{F}^{\sk}_B(Y)$ consist precisely of the Fredholm operators in $\Reg_B^{\sa}(Y)$ and $\Reg_B^{\sk}(Y)$, respectively, equipped with the Wahl topology. 
We highlight some useful properties of the Wahl topology. The proofs in the skew-adjoint case are 
simple modifications of the results in~\cite{Wahl08}.

Let $n \in \N$,  $D_0 \in \mathfrak{F}_n^{\sa}(Y_B)$ and $T_0 \in \mathfrak{F}_n^{\sk}(Y_B)$. Then for any 
 $\varepsilon >0$, the sets 
\begin{align*}
  U(n, \varepsilon, D_0) &= \big\{ D \in \Reg_B^{\sa}(Y) \, \big| \, \| \phi_n(D) - \phi_n(D_0) \| < \varepsilon \big\}, \\
  U(n, \varepsilon, T_0) &= \big\{ D \in \Reg_B^{\sk}(Y) \, \big| \, \| \psi_n(T) - \psi_n(T_0) \| < \varepsilon \big\}
\end{align*}
are  open neighbourhoods of $D_0$ and $T_0$ respectively.

\begin{lemma} \label{lemma:Wahl_invertibles_open}
The invertible operators in $\Reg_B^{\sa}(Y)$ (resp.\ $\Reg_B^{\sk}(Y)$) form an open subset of $\mathfrak{F}^{\sa}_B(Y)$ (resp.\ $\mathfrak{F}^{\sk}_B(Y)$). 
\end{lemma}
\begin{proof}
Let $D_0 \in \Reg_B^{\sa}(Y)$ be invertible (the proof for $T_0 \in \Reg_B^{\sk}(Y)$ is similar). Then there exists a $\delta>0$ such that for all $\varphi\in C_c^\infty(-\delta,\delta)$ we have $\varphi(D_0) = 0$. 
Choose $n\in\N$ such that $\frac1n<\delta$, and choose $\varepsilon>0$ such that $\varepsilon < \frac12\phi_n(0)$. 
Then, for all $D \in U(n, \varepsilon, D_0)$ we have 
\[
\| \phi_n(D) \| = \| \phi_n(D) - \phi_n(D_0) \| < \varepsilon < \tfrac12 \phi_n(0) .
\]
This implies that $0$ is not in the spectrum of $D$, since otherwise the spectral theorem yields the contradiction 
$\phi_n(0) \leq \sup_{\lambda\in\spec(D)} |\phi_n(\lambda)| = \|\phi_n(D)\| < \frac12 \phi_n(0)$. 
\end{proof}

\begin{thm} \label{thm:Wahl_cts_iff_Fredholm}
Let $\Omega$ be a compact Hausdorff space and let $\{D_\omega\}_{\omega\in\Omega}$
be a family of regular self-adjoint (or skew-adjoint) Fredholm operators on $Y_B$. 
Then the following are equivalent:
\begin{enumerate}[label=(\roman*)]
\item The family $\{D_\omega\}_{\omega\in\Omega}$ 
is Wahl-continuous.
\item The family operator $D_\bullet \colon \Dom(D_\bullet) \to C(\Omega, Y_B)_{C(\Omega, B)}$ is regular and Fredholm. 
\end{enumerate}
\end{thm}
\begin{proof}
Assume $\{D_\omega\}_{\omega\in\Omega}$ is Wahl-continuous. 
Since Wahl-continuous families have strongly continuous resolvents, the regularity of $D_\bullet$ is given by Lemma \ref{lemma:strong_cts_regular}. 
Furthermore, there exists an $n\in\N$ such that $\omega \mapsto \phi_n(D_\omega) \in \mathfrak{F}^{\sa}_n(Y_B)$ (resp.\ $\in \mathfrak{F}^{\sk}_n(Y_B)$) is compact and norm-continuous. 
In particular, $\phi_n(D_\bullet)$ is compact, and it follows from Proposition \ref{prop:Fredholm_properties} that $D_\bullet$ is Fredholm. 

Conversely, assume $D_\bullet$ is regular and Fredholm. Then the regularity implies the strong continuity of the resolvents (Lemma \ref{lemma:strong_cts_regular}). 
Furthermore, by Proposition \ref{prop:Fredholm_properties}, 
the Fredholm property of $D_\bullet$ implies that there exists an $n\in\N$ such that $\phi_n(D_\bullet)$ is compact, 
which means that $\phi_n(D_\omega)$ is compact and depends norm-continuously on $\omega$. 
Thus $\{D_\omega\}_{\omega\in\Omega}$ is Wahl-continuous. 
\end{proof}

For our construction of the spectral flow for a \emph{path} of Fredholm operators (i.e., for $\Omega=[0,1]$) in {\S}\ref{sec:spectral_flow}, we will need to consider paths with invertible endpoints, in which case we have: 

\begin{cor} \label{cor:Wahl_cts_path_Fredholm}
Let $\{D_t\}_{t\in[0,1]}$ be a family of regular self-adjoint (or skew-adjoint) Fredholm operators on $Y_B$. 
Then the following are equivalent:
\begin{enumerate}[label=(\roman*)]
\item $D_0$ and $D_1$ are invertible, and the family $\{D_t\}_{t\in[0,1]}$ is Wahl-continuous.
\item The family operator $D_\bullet \colon \Dom(D_\bullet) \to SY_{SB}$ is regular and Fredholm. 
\end{enumerate}
\end{cor}
\begin{proof}
We need to adapt the equivalent statements of Theorem \ref{thm:Wahl_cts_iff_Fredholm} from the compact space $\Omega = [0,1]$ to the noncompact space $(0,1)$. 
Assuming Wahl-continuity of $\{D_t\}_{t\in[0,1]}$ and invertibility of $D_0$ and $D_1$, we can now choose $n\in\N$ large enough, such that $\omega \mapsto \phi_n(D_\omega)$ is compact and norm-continuous with $\phi_n(D_0) = \phi_n(D_1) = 0$. This means that $\phi_n(D_\bullet)$ is compact on $SY_{SB}$. 
Conversely, if $\phi_n(D_\bullet)$ is compact on $SY_{SB}$, then this in particular means $\phi_n(D_0)=\phi_n(D_1)=0$, which shows that $D_0$ and $D_1$ are invertible. 
\end{proof}

For completeness, we also examine Wahl-continuous families of Fredholm operators at the level of Kasparov modules and $\KK$-theory.

\begin{cor} \label{cor:Wahl_family_to_KK_htpy}
Let $\Omega$ be a compact Hausdorff space, $Y_B$ a $\Z_2$-graded Hilbert module 
and $D: \Omega \to \mathfrak{F}^{sa}_B(Y)$ a Wahl-continuous family of 
self-adjoint odd Fredholm operators.  Then:
\begin{enumerate}
  \item The triple $\big( \C, \, C(\Omega, Y_B)_{C(\Omega, B)}, \, \chi(D_\bullet) \big)$ is a 
Kasparov module for any normalising function $\chi$ of $D_\bullet$,
  \item If $\tilde{D}:\Omega \times [0,1]\to \mathfrak{F}^{sa}_B(Y)$ is  a Wahl-continuous family of 
self-adjoint  odd Fredholm operators on $Y_B$, then for any normalising function $\chi$ of $\tilde{D}_\bullet$,
\[
 \big( \C, \, C( \Omega \times [0,1], Y_B)_{C(\Omega \times [0,1], B)} , \, \chi\big( \tilde{D}_\bullet\big) \big)
\]
is a homotopy of Kasparov modules.
\end{enumerate}
\end{cor}

\begin{remark}[The Kasparov topology of essential unitaries] \label{rk:Wahl_and_Kas_topology}
We say that $F \in \End_B(Y)$ is an essential unitary if $\one - F^*F,\, \one-FF^* \in \End_B^0(Y)$. Hence 
essential unitaries are precisely the Fredholm operators that give rise to Kasparov modules.
In the papers~\cite{BJS03, Joachim}, Bunke, Joachim, and Stolz consider a topology on the space of essentially unitary operators 
on a $C^*$-module $Y_B$, which has many similarities with the Wahl topology. 
Namely, they define the Kasparov topology as the weakest topology such that, 
for all $y \in Y_B$, the maps 
\begin{align*}
   &F\mapsto \one - F^*F \in \End_B^0(Y), &&F\mapsto \one - FF^* \in \End_B^0(Y), 
   &F\mapsto Fy \in Y_B, &&F\mapsto F^*y \in Y_B, 
\end{align*}
are continuous.

Let $KC^{\sa}_B(Y)$ and $KC^{\sk}_B(Y)$ be the space of essentially unitary and self-adjoint/skew-adjoint operators on $Y_B$ equipped with the Kasparov topology. 
Then there are continuous inclusions
\[
    KC^{\sa}_B(Y)  \to \mathfrak{F}^{\sa}_B(Y), \qquad \qquad KC^{\sk}_B(Y)  \to \mathfrak{F}^{\sk}_B(Y).
\]
In particular, the invertible operators in $KC^{\sa}_B(Y)$ and $KC^{\sk}_B(Y)$ are open by Lemma \ref{lemma:Wahl_invertibles_open}.
Conversely, given any Wahl-continuous family $\{T_\omega\}_{\omega\in\Omega} \subset \mathfrak{F}^{\sk}_B(Y)$ and any normalising function $\chi$ for $T_\bullet$, 
we see that $\Omega \ni \omega \mapsto \chi(T_\omega) \in  KC^{\sk}_B(Y)$ is continuous in the Kasparov topology 
(an analogous result holds for self-adjoint families). 

A thorough analysis of the various topologies on Fredholm operators in the Hilbert space setting can be found in the book~\cite{SpecFlowBook}.
\end{remark}

\section{Spectral flow on Real Hilbert modules} \label{sec:spectral_flow}

In this section, we will introduce the spectral flow for paths of skew-adjoint Fredholm 
operators on Real Hilbert $C^*$-modules with invertible endpoints. 
We provide an abstract definition of $\DK$-valued spectral flow by simply combining 
our $\DK$-valued Fredholm index with Bott periodicity. 
By observing that (like the Fredholm index) the spectral flow factors through $\KK$-theory, 
we derive several elementary properties of the spectral flow (such as additivity under concatenation of paths). 
We also prove a normalisation property for the spectral flow: 
given a Fredholm pair of skew-adjoint unitaries $(J_0,J_1)$, the spectral flow of the straight line path from $J_0$ to $J_1$
 equals the (inverse of the) relative index from {\S}\ref{sec:relative_index}. 
Finally, in \S\ref{subsec:Sf_complex}, we compare our $\DK$-valued spectral flow with the (even or odd) spectral flow in complex $K$-theory as described by Wahl \cite{Wahl07}.

\subsection{Abstract definition}

As always, let $X_A$ be a full ungraded and countably generated $C^*$-module with an ungraded and ample $\Cl_{r,s}$-representation generated by $\{e_1,\ldots, e_r, f_1, \ldots f_s\}$.
We aim to define the spectral flow for a family of Real skew-adjoint Fredholm operators $\{T_t\}_{t\in[0,1]} \subset \Reg^{r,s}_A(X)$ with invertible endpoints. 
For this purpose, we will need to require that the family operator $T_\bullet$ on $SX_{SA}$ is also regular and Fredholm. 
As we have seen in Corollary \ref{cor:Wahl_cts_path_Fredholm}, this is equivalent to the requirements that the endpoints $T_0$ and $T_1$ are invertible and the family $\{T_t\}_{t\in[0,1]}$ is Wahl-continuous. 

To obtain an abstract definition of the spectral flow of a path $\{T_t\}_{t\in[0,1]} \subset \Reg^{r,s}_A(X)$, we can then simply apply the Fredholm index (taking values in the $\DK$-theory of the suspension $SA$) to the single operator $T_\bullet \in \Reg^{r,s}_{SA}(SX)$, and then compose with Bott periodicity. 
This definition immediately emphasises the close relationship between the Fredholm index and the spectral flow. 

\begin{defn}[Abstract spectral flow] \label{def:DK_spec_flow}
Consider a Wahl-continuous family of Real skew-adjoint Fredholm operators $\{T_t\}_{t\in[0,1]} \subset \Reg^{r,s}_A(X)$ with invertible endpoints. 
Then we define the \emph{spectral flow} to be 
\[
   \Sf_{r,s+1}\big( \{T_t\}_{t\in[0,1]} \big) := \beta_{\DK}^{-1}\big( \Index_{r+1,s+1}( T_\bullet ) \big) 
   \in \DK( A \otimes \Cl_{r,s+1} ) ,
\]
where $\Index_{r+1,s+1}$ is the Fredholm index from Definition \ref{defn:index-of-Fred}, 
and $\beta_{\DK}: \DK(A \otimes \Cl_{r,s+1}) \xrightarrow{\simeq} \DK(SA \otimes \Cl_{r+1,s+1})$ is the Bott periodicity in Van Daele $K$-theory (see Theorem \ref{thm:DK_Bott_isos}).
\end{defn}

We note that the spectral flow is well-defined, since by Corollary \ref{cor:Wahl_cts_path_Fredholm}, the family operator $T_\bullet \in \Reg^{r,s}_{SA}(SX)$ is indeed Fredholm. 
Recall from Corollary \ref{cor:skew_Fred_to_KasMod} that $T_\bullet$ defines a $\KK$-class $\class{T_\bullet} \equiv \big[ \chi(T_\bullet)\otimes\rho \big] \in \KKR( \Cl_{s+1, r}, SA)$ given by the Real Kasparov module 
\[
    \Big( \Cl_{s+1, r}, \  SX_{SA} \otimes \exterior \C, \   \chi(T_\bullet) \otimes \rho  \Big) .
\]
Furthermore,  Theorem \ref{thm:index-of-Fred} shows that the Fredholm index can be expressed as the 
Roe isomorphism $\Roe_A^{r+1,s+1}: \KKR(\Cl_{s+1,r}, A) \xrightarrow{\simeq} \DK( A\otimes \Cl_{r+1,s+1} )$ applied to the $\KK$-class. 
It thus follows that the spectral flow also factors through $\KK$-theory, and can be expressed in terms of the Roe isomorphism and the Kubota isomorphism $\Kubota_A^{r,s+1}: \KKR(\Cl_{s+1,r}, SA) \xrightarrow{\simeq} \DK( A \otimes \Cl_{r,s+1})$ (see Definition \ref{def:Indexed_KK_isomorphisms}), as follows. 

\begin{prop}  \label{prop:KK_and_DK_flow_compatible}
Let $\{T_t\}_{t\in[0,1]} \subset \Reg^{r,s}_A(X)$ be a Wahl-continuous family of Real skew-adjoint Fredholm operators with invertible endpoints. 
Then, for any normalising function $\chi$ for $T_\bullet$, we have 
\begin{align*}
  \Sf_{r,s+1} \big( \{T_t\}_{t\in[0,1]} \big) 
  &= \beta_{\DK}^{-1}\circ \Roe_{SA}^{r+1,s+1} \big( \big[ \chi(T_\bullet)\otimes\rho \big] \big) \\ 
  &= \Kubota_A^{r,s+1} \big( \big[ \chi(T_\bullet)\otimes\rho \big] \big).
\end{align*}
\end{prop}
\begin{proof}
The first equality is Theorem \ref{thm:index-of-Fred} and the second is Theorem \ref{prop:Roe_Kubota_Bott_compatibility}.
\end{proof}

Thus, just as the Roe isomorphism $\Roe$ computes the Fredholm index of a $\KK$-class over $A$,  the Kubota isomorphism $\Kubota$ computes the spectral flow of a $\KK$-class over the suspension $SA$.

\subsection{Properties of the spectral flow}
Because the spectral flow factors through $\KK$-theory by Proposition \ref{prop:KK_and_DK_flow_compatible}, the spectral flow immediately inherits various properties from $\KK$-theory, which we will list here. 

\begin{prop} \label{prop:sf_properties}
The spectral flow satisfies the following properties of the \emph{zero axiom (Z)}, \emph{homotopy invariance (H)}, and \emph{additivity (A)}:
\begin{enumerate}
  \item[(Z)] If $\{T_t\}_{t\in[0,1]} \subset \Reg^{r,s}_A(X)$ is a Wahl-continuous family of Real skew-adjoint \emph{invertible} operators, then $\Sf_{r,s+1}\big( \{T_t\}_{t\in[0,1]} \big) = 0$. 
  \item[(H)] If $\{T_{t,\lambda}\}_{(t,\lambda)\in[0,1]\times[0,1]} \subset \Reg^{r,s}_A(X)$ is a Wahl-continuous family of Real skew-adjoint Fredholm operators, such that $T_{0,\lambda}$ and $T_{1,\lambda}$ are invertible for all $\lambda\in[0,1]$, then $\Sf_{r,s+1}\big( \{T_{t,0}\}_{t\in[0,1]} \big) = \Sf_{r,s+1}\big( \{T_{t,1}\}_{t\in[0,1]} \big)$. 
  \item[(A)] For $j =1,2$, if $\{T_t^{(j)}\}_{t\in[0,1]} \subset \Reg^{r,s}_A(X^{(j)})$ are two Wahl-continuous families of Real skew-adjoint Fredholm operators with invertible endpoints, then $\Sf_{r,s+1}\big( \{T_t^{(1)} \oplus T_t^{(2)}\}_{t\in[0,1]} \big) = \Sf_{r,s+1}\big( \{T_t^{(1)}\}_{t\in[0,1]} \big) + \Sf_{r,s+1}\big( \{T_t^{(2)}\}_{t\in[0,1]} \big)$. 
\end{enumerate}
\end{prop}

Recall from Definition \ref{def:Ind_of_pair_of_complex_structures_DK} that $\relind_{r,s+1}( J_0, J_1)$ denotes the relative index of a Fredholm pair of skew-adjoint unitaries.

\begin{prop} \label{prop:sf_normalisation}
The spectral flow satisfies the following \emph{normalisation} property: 
\begin{enumerate}
\item[(N)] For a Fredholm pair $(J_0,J_1)$ of skew-adjoint unitaries in $\End_A^{r,s}(X)$, we have 
\end{enumerate}
\[
   \Sf_{r,s+1}\big( \big\{ (1-t) J_0 + t J_1 \big\}_{t\in[0,1]} \big) = \relind_{r,s+1}( J_1, J_0) \in \DK( A \otimes \Cl_{r,s+1}).
\]
\end{prop}
\begin{proof}
From Propositions \ref{prop:KK_and_DK_flow_compatible} and \ref{prop:KK_DK_ComplexStrucIndex_agrees} we obtain, for any normalising function $\chi$, 
\[
\Sf_{r,s+1}\big( \big\{ (1-t) J_0 + t J_1 \big\}_{t\in[0,1]} \big) 
= \Kubota_A^{r,s+1} \big( [ \chi(J_\bullet)\otimes\rho ] \big)
= \relind_{r,s+1}( J_1, J_0) .
\qedhere
\]
\end{proof}

\begin{remark}[Comparison with Hilbert space normalisation] \label{Rk:BLCR_sign_error}
In the paper \cite{BCLR} a $KO_{s+2-r}(\R)$-valued spectral flow is defined with the following normalisation: if 
$V$ is a finite-dimensional $\Cl_{r,s+1}$-module with generators $\{e_1,\ldots, e_r, f_1,\ldots, f_s,f_{s+1}\}$, 
then $[0,1]\ni t \mapsto F_t =  (1-2t) f_{s+1}$ is a path of skew-adjoint (Fredholm) operators that anti-commute with the 
$\Cl_{r,s}$-generators. One can then compute a Clifford module valued spectral flow
\[
  \Sf^\calH_{r,s+2}(  F_t ) = \big[ \Ker( F_{1/2} )  \big] = \big[ \Ker\big( f_{s+1} + (-f_{s+1}) \big) \big] = [V] \in \calM_{r,s+1}/\calM_{r,s+2} \cong KO_{s+2-r}(\R),
\]
where $\Ker\big( f_{s+1} + (-f_{s+1}) \big) = V$ is a $\Cl_{r,s+1}$-module with generators $\{e_1,\ldots, e_r, f_1,\ldots, f_s,f_{s+1}\}$. 
More generally, given a Fredholm pair of skew-adjoint unitaries $(J_0, J_1) \in \calJ^{r,s}(\calH)$, the Clifford module valued 
spectral flow is such that 
\[
   \Sf^\calH_{r,s+2}\big( \big\{ (1-t) J_0 + t J_1 \big\}_{t\in[0,1]} \big)  =  \Ind_{r,s+2}^\calH( J_0, J_1) = -  \Ind_{r,s+2}^\calH( J_1, J_0) ,
\]
which is inconsistent with Proposition \ref{prop:sf_normalisation}. While the results in  \cite{BCLR} are internally consistent, they 
do not recover the standard normalisation of spectral flow in complex Hilbert spaces.
To see this, suppose that $V$ is a complex (finite-dimensional)
$\Cl_1$-module with generator $e=2P- \one$. Then using the normalisation from \cite{BCLR} and the isomorphism 
$\calM^\C_1 / \calM^\C_2 \cong \Z$ from \cite[\S2.4.2]{BCLR}, 
\begin{align*}
   \Sf^\calH_{r,s+2}\big( \big\{ (1-2t)(2P-\one) \big\}_{t\in[0,1]} \big) &= \dim\Ker( (2P-\one) - \one) - \dim\Ker\big( (2P-\one) + \one \big) \\
     &= \Tr(P) - \Tr( \one - P ) \in \Z.
\end{align*}
However, this integer counts the spectrum traveling from the $(+1)$-eigenspace to the $(-1)$-eigenspace minus the 
spectrum traveling from the $(-1)$-eigenspace to the $(+1)$-eigenspace, which is the negative of the usual method of computing 
spectral flow. We have therefore used the normalisation of Proposition \ref{prop:sf_normalisation}, which will give the negative 
of the Clifford module spectral flow on real Hilbert spaces (see \S\ref{subsec:Sf_Real_HS}).
\end{remark}

To show that the spectral flow is also additive under concatenation of paths, we need to be able to decompose paths 
$F_\bullet \in \End_{C([0,2],A)}^{r,s}(C( [0,2], X))$ with $F_1$ invertible 
into a sum of paths in $\End_{C([0,1],A)}^{r,s}(C( [0,1], X))$ and $\End_{C([1,2],A)}^{r,s}(C( [1,2], X))$. 
The next result addresses this question.

\begin{lemma}
\label{lemma:MV}
Let $\{F_t\}_{t\in [0,2]} \subset  \End^{r,s}_A(X)$ be a family of Real skew-adjoint Fredholm operators 
with $F_0^2 = F_1^2=F_2^2 =-\one$, 
such that $[0,2] \ni t \mapsto \one + F_t^2 \in \End^{0}_A(X)$ is norm-continuous. 
Let $\cS_1=C_0(0,1)$, $\cS_2=C_0(1,2)$, $\cS=C_0(0,2)$, and $F^{01}_\bullet=F_\bullet|_{[0,1]}$, $F^{12}_\bullet=F_\bullet|_{[1,2]}$. 
Let $\iota=\iota_1+\iota_2:\cS_1\oplus\cS_2\hookrightarrow\cS$ be the inclusion, 
and define
\begin{align*}
\class{F_{\textnormal{gap}}} &= 
 \Big[\Big(\Cl_{s+1,r}, \, (\cS_1\oplus\cS_2) X_{(\cS_1\oplus\cS_2) A} \otimes \exterior \C,  \, F_\bullet \otimes \rho \Big) \Big], \\
 \class{F} &= \Big[ \Big(\Cl_{s+1,r}, \, \cS X_{\cS A}\otimes \exterior \C, \, F_\bullet \otimes \rho \Big) \Big].
\end{align*}
Then $\class{F}=\iota_*(\class{F_{\textnormal{gap}}})=\iota_{1*}(\class{F^{01}_\bullet})\oplus \iota_{2*}(\class{F^{12}_\bullet}) \in \KKR\big( \Cl_{s+1,r} , \cS A \big)$.
\end{lemma}
\begin{proof}
First, applying the inclusion $\iota$ yields
\begin{align*}
& \iota_*(\class{F_{\textnormal{gap}}}) =\Big[\Big(\Cl_{s+1,r}, \, (\cS_1\oplus\cS_2) X_{(\cS_1\oplus\cS_2) A} \otimes \exterior \C , \,  F_\bullet \otimes \rho \Big) \Big] \ox_{(\cS_1\oplus\cS_2) A}
    [( (\cS_1\oplus\cS_2) A, \, \cS A_{\cS A} , \, 0)]\\
&\qquad=\Big[\Big(\Cl_{s+1,r}, \, (\cS_1\oplus\cS_2) X_{\cS A}  \otimes \exterior \C,  \, F_\bullet \otimes \rho \Big) \Big] \\
&\qquad=\Big[\Big( \Cl_{s+1,r},\cS_1 X_{\cS A} \otimes \exterior \C , \, F^{01}_\bullet \otimes \rho) \Big) \Big]
     \oplus \Big[\Big(\Cl_{s+1,r},\cS_2 X_{\cS A}\otimes \exterior \C , \, F^{12}_\bullet \otimes \rho \Big) \Big]  \\
&\qquad=\iota_{1*}(\class{F^{01}_\bullet}) \oplus \iota_{2*}(\class{F^{12}_\bullet}).
\end{align*}
Second, consider the $C^*$-module $Y := \big\{ f \in C([0,1],\cS X_{\cS A}) : f(1)(1)=0 \big\}$ over the $C^*$-algebra $C([0,1],\cS A)$. Then $\ev_0(Y) = \cS X_{\cS A}$ and $\ev_1(Y) = (\cS_1\oplus\cS_2) X_{\cS A}$, so that the Kasparov module 
\[
\Big(\Cl_{s+1,r}, \, Y_{C([0,1],\cS A)}\otimes \exterior \C, \, F_\bullet \otimes \rho \Big)
\]
yields a homotopy between $\class{F}$ and $\iota_*(\class{F_{\textnormal{gap}}})$. 
\end{proof}

\begin{prop} \label{prop:Path_additivity_KK}
The spectral flow satisfies \emph{path additivity under concatenation}:
\begin{enumerate}
  \item[(C)] If $\{T_t\}_{t\in[0,2]} \subset \Reg^{r,s}_A(X)$ is a Wahl-continuous family of Real skew-adjoint Fredholm operators with invertible endpoints $T_0,T_2$ and invertible midpoint $T_1$, then 
  \[
  \Sf_{r,s+1}\big( \{T_t\}_{t\in[0,2]} \big) = \Sf_{r,s+1}\big( \{T_t\}_{t\in[0,1]} \big) + \Sf_{r,s+1}\big( \{T_t\}_{t\in[1,2]} \big) . 
  \]
\end{enumerate}
\end{prop}
\begin{proof}
We can take a normalising function such that $F_\bullet = \chi(T_\bullet)$  will satisfy the hypothesis of Lemma \ref{lemma:MV}, 
and the statement then follows using additivity (A) of the spectral flow from Proposition \ref{prop:KK_and_DK_flow_compatible}.
\end{proof}

\subsection{Spectral flow of odd self-adjoint regular Fredholm operators}

Like we have done for the analytic index and relative index,  we also provide a definition of spectral flow 
for Wahl-continuous paths of odd self-adjoint Fredholm operators $\{D_t\}_{t \in [0,1]} \subset \Reg_B(Y)$ with invertible endpoints.

\begin{defn}[Abstract $\Z_2$-graded spectral flow] \label{def:Odd_DK_spec_flow}
Consider a Wahl-continuous family of Real self-adjoint odd Fredholm operators $\{D_t\}_{t\in[0,1]} \subset \Reg^{r,s}_B(Y)$ with invertible endpoints. 
Then we define the {spectral flow} to be 
\[
   \wh{\Sf}\big( \{D_t\}_{t\in[0,1]} \big) := \beta_{\DK}^{-1}\big( \wh{\Index}_{1,0}( D_\bullet ) \big) 
   \in \DK( B) ,
\]
where $\wh{\Index}_{1,0}$ is the $\Z_2$-graded   index from Theorem \ref{thm:Z_2-graded_Fred_index}
and $\beta_{\DK}: \DK(B ) \xrightarrow{\simeq} \DK(SB\hox \Cl_{1,0})$ 
is the Bott periodicity isomorphism in Van Daele $K$-theory (see Theorem \ref{thm:DK_Bott_isos}).
\end{defn}

Analogously to Proposition \ref{prop:KK_and_DK_flow_compatible}, we can relate the $\Z_2$-graded spectral flow 
to the $\Z_2$-graded Roe and Kubota isomorphisms (Theorem \ref{thm:graded_Roe_and_Kubota}), where Theorem \ref{thm:Z_2-graded_Fred_index} and 
Eq. \eqref{eq:Graded_Bott_compat} give that
\[
  \wh{\Sf}\big( \{D_t\}_{t\in[0,1]} \big) = \beta_{\DK}^{-1} \circ \Roe_{SB} \big( [ \chi(D_\bullet) ] \big) 
    = \Psi_B \big( [ \chi(D_\bullet) ] \big).
\]

\begin{prop}
The $\Z_2$-graded spectral flow satisfies the  \emph{zero axiom (Z)}, \emph{homotopy invariance (H)}, \emph{additivity (A)}, and 
\emph{path additivity under concatenation (C)}. Furthermore, the  spectral flow satisfies the following \emph{normalisation} property: 
\begin{enumerate}
\item[(N)] If $(G_0, G_1)$ are a Fredholm 
pair of odd self-adjoint unitaries in $\End_B(Y)$, then 
\end{enumerate}
\[
   \wh{\Sf} \big( \big\{ (1-t) G_0 + t G_1 \big\}_{t\in[0,1]} \big) = \wh{ \relind}( G_1, G_0) \in \DK( B).
\]
\end{prop}
\begin{proof}
Properties (Z), (H), (A), and (C) follow by the same argument as the Clifford anti-linear setting. 
For the normalisation, we use the $\Z_2$-graded Kubota isomorphism $\Psi_B: \KKR(\C, SB) \xrightarrow{\simeq} B$, where 
\begin{align*}
   \wh{\Sf} \big( \big\{ (1-t) G_0 + t G_1 \big\}_{t\in[0,1]} \big) &= \Psi_B \big( [ G_\bullet ] \big) 
   = \wh{ \relind}( G_1, G_0),
\end{align*}
with the last equality by Proposition \ref{prop:Z_2-rel-index-properties}.
\end{proof}

Recalling Proposition \ref{prop:Graded_ungraded_index_comparison} and Lemma \ref{lem:Rel_ind_graded_to_ungraded}, 
we  can relate the spectral flow of odd self-adjoint 
Fredholm operators with the spectral flow of $\Cl_{1,0}$-anti-linear skew-adjoint Fredholm operators in the complex 
standard $C^*$-module over an ungraded algebra.

\begin{prop} \label{prop:graded_to_ungraded_sf}
Let $\{T_t\}_{t\in[0,1]} = \{iD_t\}_{t\in [0,1]} \subset \Reg^{1,0}_A(\calH_A)$ be a Wahl-continuous family of 
skew-adjoint Fredholm operators with invertible endpoints. Using the $\Cl_{1,0}$-generator, consider 
$\calH_A \cong \calH_A^+ \oplus \calH_A^- \cong \calH_A^+ \otimes \exterior{\C}$ a $\Z_2$-graded $C^*$-module 
with $D_t = \begin{pmatrix} 0 & (D_t)_- \\ (D_t)_+ & 0 \end{pmatrix}$ an odd self-adjoint operator on $\calH_A \otimes \C^2$. 
Then 
\[
 \Sf_{1,1} \big( \{iD_t\}_{t\in[0,1]} \big) = \wh{\Sf}\big( \{D_t\}_{t\in[0,1]} \big) \in \DK(A \otimes \Cl_{1,1} ).
\]
\end{prop}
Unlike Proposition \ref{prop:Graded_ungraded_index_comparison}, the Clifford stability map 
$\calM_{1,1}^{\DK} = \DK(B) \xrightarrow{\simeq} \DK(B \hox \Cl_{1,1})$ is not needed here as 
$\DK(A) := \DK(A \otimes \Cl_{1,1})$ for $A$ ungraded.
\begin{proof}
The $\Z_2$-graded spectral flow of $\{D_t\}$ can be written as the composition 
\[
 \beta_{DK}^{-1} \circ \calM_{\calH_{SA}}^{\DK} \circ \delta \left(   \left[  \begin{pmatrix} 0 & q\big(\chi(D_\bullet)\big)_- \\  q\big(\chi(D_\bullet)\big)_+  & 0    \end{pmatrix} \right] 
 -  \left[  \begin{pmatrix} 0 & \one \\  \one & 0    \end{pmatrix} \right] \right)
\]
with $ \big[ q(\chi(D_\bullet)) \big] - [ \one \otimes \gamma] \in \DK(\calQ_A(\calH_A) \otimes \Cl_{1,1} )$ and we identify 
$M_2(\C) \cong \Cl_{1,1}$. Arguing analogously to the proof of   Lemma \ref{lem:Rel_ind_graded_to_ungraded}, 
the Clifford anti-linear spectral flow is the composition 
\begin{align*}
    &\beta_{DK}^{-1} \circ \calM_{\calH_{SA}}^{\DK} \circ \delta 
      \left(   \left[  \begin{pmatrix} 0 & -i \big( \chi(D_\bullet)\big) \tfrac{1}{2}(\one - e) \\   i  q\big(\chi(D_\bullet)\big) \tfrac{1}{2}(\one+ e)  & 0    \end{pmatrix} \right] 
 -  \left[  \begin{pmatrix} 0 & \one \\  \one & 0    \end{pmatrix} \right] \right)  \\
    &\quad = \beta_{DK}^{-1} \circ \calM_{\calH_{SA}}^{\DK} \circ \delta \left(   \left[  \begin{pmatrix} 0 & q\big( \chi(D_\bullet) \big)_- \\  q\big(\chi(D_\bullet)\big)_+  & 0    \end{pmatrix} \right] 
 -  \left[  \begin{pmatrix} 0 & \one \\  \one & 0    \end{pmatrix} \right] \right),
\end{align*}
where $e\in \End_A(\calH_A)$ is the $\Cl_{1,0}$-generator. The result follows. 
\end{proof}

\subsection{Comparison with complex Hilbert modules} \label{subsec:Sf_complex}

We finish our discussion of the abstract spectral flow by comparing our $\DK$-valued spectral flow to the $K_\ast(A)$-valued spectral flow 
of \emph{self-adjoint} Fredholm operators on \emph{complex} $C^*$-modules as described by Wahl~\cite{Wahl07} 
(as in \cite{Wahl07}, we restrict our attention to the standard $C^*$-module).
We will again make use of the isomorphisms $\Upsilon_{A}^{r,s}: \DK(A \otimes \Cl_{r+s} ) \to K_{1+s-r}(A)$ from 
Appendix \ref{subsec:appendix_DK_KR_isos}.

\begin{prop} \label{prop:Wahl_sf_even_odd}
Let $A$ be an ungraded complex $C^*$-algebra, $\calH_A$  the standard ungraded complex 
$C^*$-module over $A$, and $\hat{\calH}_A$ the standard graded complex $C^*$-module over $A$.
\begin{enumerate}
\item 
Let $\{D_t\}_{t\in[0,1]} \subset \Reg_A(\calH_A)$ be a Wahl-continuous family of self-adjoint Fredholm operators with invertible endpoints. 
Then, under the isomorphism $\Upsilon_A^{0,1}: \DK(A \otimes \Cl_{0+1}) \xrightarrow{\simeq} K_2(A) \cong K_0(A)$, 
\[
  \Upsilon_A^{0,1}\big(  \Sf_{0,1}( \{iD_t\}_{t\in[0,1]} ) \big) = \Sf^{W}( \{D_t\}_{t\in[0,1]} ) \in K_0(A),
\] 
where $\Sf^{W} \in K_0(A)$ is the spectral flow due to Wahl \cite[\S4]{Wahl07}.
\item 
Let $\{D_t\}_{t\in[0,1]} \subset \Reg_A(\hat{\calH}_A)$ be a Wahl-continuous family of odd self-adjoint Fredholm operators with invertible endpoints. 
Then under the isomorphism $\Upsilon_A^{1,1}: \DK(A \otimes \Cl_{1+1}) \xrightarrow{\simeq} K_1(A)$,
\[
  \Upsilon_A^{1,1} \big( \Sf_{1,1}( \{iD_t\}_{t\in[0,1]} ) \big) =  \Upsilon_A^{1,1} \big( \wh{\Sf}( \{D_t\}_{t\in[0,1]} ) \big)  =  \Sf_\mathrm{odd}^{W}( \{D_t\}_{t\in[0,1]} ) \in K_1(A),
\] 
where $\Sf^{W}_\mathrm{odd} \in K_1(A)$ is the odd spectral flow due to Wahl \cite[\S8]{Wahl07}.
\end{enumerate}
\end{prop}
\begin{proof}
(1) By~\cite[Proposition 4.2]{Wahl07}, $\Sf^{W}( \{D_t\}_{t\in[0,1]} )$ can be realised by applying the odd index to the 
Kasparov module $\big[ \big( \Cl_{1}, \, S\calH_{SA} \otimes \exterior \C, \, \chi(D_\bullet) \otimes \gamma \big) \big]$ and then composing with the inverse Bott isomorphism:
\[
\KK(\Cl_{1}, SA) \xrightarrow{\Index_1} K_1(SA) \xrightarrow{\beta_K^{-1}} K_0(A) .
\]
We can write the Clifford generator $\gamma = i \rho$ and  $\chi(D_\bullet) \otimes \gamma = \chi( i D_\bullet) \otimes \rho$. Then 
using the compatibility of the isomorphisms $\Upsilon_A^{r,s}: \DK(A \otimes \Cl_{r,s} ) \to K_{1+s-r}(A)$ with the Bott 
map (Lemma \ref{lemma:DK-complexK-respects-Bott}),
\begin{align*}
  \Upsilon_A^{0,1}\big( \Sf_{0,1}( \{iD_t\}_{t\in[0,1]} ) \big) &= \Upsilon_A^{0,1} \circ \beta_{DK}^{-1} \big( \Index_{1,1}( iD_\bullet )  \big) \\
    &= \beta_K^{-1} \circ \Upsilon_{SA}^{1,1}  \big( \Index_{1,1}( iD_\bullet )  \big) \\
    &= \beta_K^{-1}\big(  \Index_1( D_\bullet ) \big) = \Sf^{W}( \{D_t\}_{t\in[0,1]} ) \in K_0(A),
\end{align*}
where we have used part (2) of Proposition \ref{prop:complex_index}.

(2) Because $\hat{\calH}_A$ is graded with $D_t$ odd, to define $\Sf_{1,1}$ (which does not consider graded operators),  
we  consider $\{iD_t\}_{t\in [0,1]} \subset \Reg_A^{1,0}(\hat{\calH}_A)$ with 
$\Cl_{1,0}$ generator given by the grading operator. Then the equality $ \Sf_{1,1}( \{iD_t\}_{t\in[0,1]} )  = \wh{\Sf}( \{D_t\}_{t\in[0,1]} )$ 
was shown in Proposition \ref{prop:graded_to_ungraded_sf}. 
The odd spectral flow (\cite[Definition 8.4]{Wahl07}) is defined by 
applying the map $\Index_1 \circ \altBott_{\KK}^{-1} = \beta_K^{-1} \circ \Index_0$ to the 
$\KK$-class $[D_\bullet] := \big[ \big( \C, \, S\hat{\calH}_{SA} , \, \chi( D_\bullet) \big) \big] \in \KK(\C, SA)$. 
Because $A$ is ungraded,
\[
   S\hat{\calH}_{SA} \cong \calH_{SA} \otimes \C^2, \qquad D_\bullet = \begin{pmatrix} 0 & (D_\bullet)_- \\ (D_\bullet)_+ & 0 \end{pmatrix}
\]
where $(D_\bullet)_+$ is a Fredholm operator on $\calH_{SA}$. Arguing analogously to part (1), 
\begin{align*}
   \Upsilon_{A}^{1,1} \big(  \Sf_{1,1}( \{iD_t\}_{t\in[0,1]} ) \big) &= \Upsilon_A^{1,1} \circ \beta_{DK}^{-1} \big( \Index_{2,1}( iD_\bullet )  \big) \\
    &= \beta_K^{-1} \circ \Upsilon_{SA}^{2,1}  \big( \Index_{2,1}( iD_\bullet )  \big) \\
    &=  \beta_K^{-1}\big(  \Index_0( (D_\bullet)_+ ) \big) =\Sf_\mathrm{odd}^{W}( \{D_t\}_{t\in[0,1]} ) \in K_1(A),
\end{align*}
where we have used part (1) of Proposition \ref{prop:complex_index}.
\end{proof}

\section{Analytic spectral flow} \label{sec:analytic_spectral_flow}

In this section, we provide a more concrete description of spectral flow that is closer in spirit to the idea of counting eigenvalue 
crossings through zero for paths of Fredholm operators. 
First, in {\S}\ref{subsec:sf_small_perturbations}, we show that the normalisation property of the spectral flow (relating the spectral flow of the straight line path between two skew-adjoint unitaries to their relative index, see Proposition \ref{prop:sf_normalisation}) can be generalised to norm-continuous paths of ``small perturbations''. 

In {\S}\ref{subsec:Phillips}, 
we then prove a formula for the \emph{analytic spectral flow} in analogy to the analytic approach to spectral flow on complex Hilbert spaces due to J. Phillips~\cite{P2}. 
This analytic spectral flow of a path $\{D_t\}_{t\in[0,1]}$ of self-adjoint Fredholm operators is given by considering a partition $0=t_0<t_1,\ldots<t_n=1$ of the unit interval, taking the positive spectral projections $P_+\big( \chi(D_{t_i})\big)$, and then computing a (finite) sum of relative indices of these projections. 

In the $C^*$-module setting, the construction of this analytic spectral flow is complicated by the absence of Borel functional calculus on $C^*$-modules (in general, positive spectral projections can not always be constructed via continuous functional calculus). 
In {\S}\ref{subsec:Phillips}, we deal with this issue by considering only \emph{Riesz-continuous} paths. 
In our Real setting, we replace projections by skew-adjoint unitaries, and define the analytic spectral flow in analogy with \cite{P2}. We prove that this analytic spectral flow agrees with the (abstract) spectral flow of {\S}\ref{sec:spectral_flow}. 
In the special setting of a Hilbert space, 
our analytic formula for the spectral flow shows that the 
$\DK$-valued spectral flow reduces to the (inverse of the) skew-adjoint spectral flow on real Hilbert spaces considered in~\cite{BCLR} (see {\S}\ref{subsec:Sf_Real_HS}).

In {\S}\ref{subsec:analytic_Wahl}, we adapt the approach of Wahl \cite{Wahl07} to our Real setting, and prove a second formula for the analytic spectral flow. 
Wahl's approach works for Wahl-continuous families of regular (unbounded) Fredholm operators, but requires the additional assumption that 
there exist  \emph{locally trivialising families}.

\subsection{Norm-continuous paths of ``small perturbations''} \label{subsec:sf_small_perturbations}

Let $(J_0,J_1)$ be a Fredholm pair of skew-adjoint unitaries in $\End_A^{r,s}(X)$. 
The normalisation property of the spectral flow (Proposition \ref{prop:sf_normalisation}) shows that the spectral flow of the straight line path from $J_0$ to $J_1$ is given by the relative index of $(J_1,J_0)$. 
We can generalise this result as follows: instead of the straight line path, we may take \emph{any} norm-continuous path of (skew-adjoint) essential unitaries, as long as we stay inside the range of ``small perturbations'' (given by compact perturbations and perturbations which are small in norm). 

\begin{lemma} \label{lemma:sf_small_perturbations}
Let $\{F_t\}_{t\in [0,1]} \subset \End^{r,s}_A(X)$ be a norm-continuous family of Real skew-adjoint operators 
such that $F_0$ and $F_1$ are unitary (i.e., $F_0^2 = F_1^2 = -\one$) 
and $F_t$ is essentially unitary (i.e., $\one + F_t^2 \in \End_A^0(X)$) for all $t\in [0,1]$. 
Suppose also that $\big\| q(F_0) - q(F_t) \big\|_{\calQ_A(X)} < \sqrt{2}$ and $\big\| q(F_1) - q(F_t) \big\|_{\calQ_A(X)} < \sqrt{2}$ for all $t \in [0,1]$. 
Then 
\[
   \Sf_{r,s+1} ( \{F_t\}_{t\in[0,1]} ) =  \Sf_{r,s+1}\big( (1-t) F_0 + t F_1 \big) = \relind_{r,s+1}( F_1, F_0) \in \DK( A \otimes \Cl_{r,s+1}) .
\]
\end{lemma}
\begin{proof}
We consider a homotopy from $\{F_t\}_{t\in[0,1]}$ to the straight line path $\{\hat{F}_t := (1-t) F_0 + t F_1\}_{t\in[0,1]}$ given by 
\begin{align*}
  \{F_{s,t}\}_{s,t \in [0,1]}, \qquad F_{s,t} := (1-s) F_t + s \hat{F}_t .
\end{align*}
In order to check that $F_{s,t}$ is Fredholm for all $s,t\in[0,1]$, we compute (using $q(F_t)^2=-\one$)
\begin{align*}
   \big\| \one + q(F_{s,t})^2 \big\|_{\calQ_A(X)} 
   &= \big\| \one - (1-s)^2 \one + s^2 q(\hat{F}_t)^2 + s(1-s) q(F_t \hat{F}_t + \hat{F}_t F_t) \big\|_{\calQ_A(X)}  \\
   &= \big\| s^2 q(\one + \hat{F}_t^2) + s(1-s) q(2 + F_t \hat{F}_t + \hat{F}_t F_t) \big\|_{\calQ_A(X)}  \\
   &= \big\| s^2 q(\one + \hat{F}_t^2) + s(1-s) q\big( \one + \hat{F}_t^2 - (F_t-\hat{F}_t)^2 \big) \big\|_{\calQ_A(X)}  \\
   &= \big\| s q(\one + \hat{F}_t^2) - s(1-s) q(F_t-\hat{F}_t)^2 \big\|_{\calQ_A(X)}  .
\end{align*}
We observe that 
\[
\| q(F_t) - q(\hat{F}_t) \|_{\calQ_A(X)}  
\leq (1-t) \|q(F_t)-q(F_0)\|_{\calQ_A(X)}  + t \|q(F_t)-q(F_1)\|_{\calQ_A(X)} 
< \sqrt{2} .
\]
Moreover, we can rewrite 
\[
q(\one + \hat{F}_t^2) 
= t(t-1) q\big( 2 + F_0 F_1 + F_1 F_0 \big) 
= - t(t-1) q(F_0-F_1)^2 .
\]
Since $\sup_{s\in[0,1]} \big| s(1-s) \big| = \frac14$, we can therefore estimate 
\begin{align*}
\big\| \one + q(F_{s,t})^2 \big\|_{\calQ_A(X)} 
&\leq \frac14 \big\| q(F_0)-q(F_1) \big\|_{\calQ_A(X)}^2 + \frac14 \| q(F_t) - q(\hat{F}_t) \|_{\calQ_A(X)}^2 
< \frac14 \cdot \sqrt{2}^2 + \frac14 \cdot \sqrt{2}^2 
= 1 ,
\end{align*}
so $q(F_{s,t})$ is indeed invertible for all $s,t \in [0,1]$. 
From the homotopy invariance of the spectral flow (Proposition \ref{prop:sf_properties}), it follows that 
\[
    \Sf_{r,s+1}\big( \{F_t\}_{t\in[0,1]} \big) = \Sf_{r,s+1} \big( (1-t) F_0 + t F_1\big).
\]
Since $F_0$ and $F_1$ are skew-adjoint unitaries with $\big\| q(F_0) - q(F_1) \big\|_{\calQ_A(X)} < \sqrt{2}$, 
the second equality in the statement is given by Proposition \ref{prop:sf_normalisation}.
\end{proof}

\subsection{The Phillips formula for Riesz-continuous paths}
\label{subsec:Phillips}

Recall that the Riesz topology on the 
set $\Reg_A^{\sk}(X)$ 
is defined to be the weakest topology such that the following map is continuous:
\begin{align*}
\Reg_A^{\sk}(X) &\to \End_A(X) , \qquad T \mapsto F_T := T(\one -T^2)^{-\frac12} .
\end{align*}
It then follows that also the map $T \mapsto (\one-T^2)^{-1}$ is continuous. 
Since the $C^*$-algebra 
\[
C_\infty(i\R) := \big\{ f \in C(i\R) : \lim_{x\to\pm i\infty} f(x) \text{ exists} \big\} 
\]
is generated by the functions $x \mapsto i$, $x \mapsto x(1-x^2)^{-\frac12}$, and $x \mapsto (1-x^2)^{-1}$, 
we see that $T \mapsto f(T)$ is continuous for any $f \in C_\infty(i\R)$. 
In particular, the map 
$T \mapsto \chi(T)$ 
is continuous for any normalising function $\chi$. 

In this subsection, we consider a Riesz-continuous path $\{T_t\}_{t\in [0,1]}$ of Real skew-adjoint Fredholm operators in $\Reg_A^{r,s}(X)$ with invertible endpoints. 
We also assume that a normalising function $\chi$ is chosen such that $\chi(T_0)$ and $\chi(T_1)$ are skew-adjoint unitaries. 
We then consider the norm-continuous path $\{F_t\}_{t\in[0,1]}$ given by $F_t := \chi(T_t)$. 

The starting point for the construction of the analytic spectral flow is the following observation, which is based on the lifting results for skew-adjoint unitaries from Lemma \ref{lemma:homotopic_to_lifting}. 

\begin{lemma} \label{lemma:pointwise_skew_unitary}
For every $t \in [0,1]$ there exists a skew-adjoint unitary $J_t \in \End_A^{r,s}(X)$ such that $J_t - F_t \in \End_A^0(X)$.
\end{lemma}
\begin{proof}
We have a norm-continous path $t\mapsto q(F_t) \in \calQ_A(X)$ of skew-adjoint unitaries. 
For any $t \in [0,1]$, $q(F_t)$ is therefore homotopic (in $\calQ_A(X)$) to the skew-adjoint unitary $q(F_0)$ which has a skew-adjoint unitary lift $F_0 \in \End_A^{r,s}(X)$. 
Applying Lemma \ref{lemma:homotopic_to_lifting}, there exists a skew-adjoint unitary $J_t \in \End_A^{r,s}(X)$ such that $q(F_t) = q(J_t)$. 
\end{proof}

\begin{defn}[Analytic spectral flow I] \label{defn:spectral_flow_Phillips}
Given the Riesz-continuous path $\{T_t\}_{t\in [0,1]}$ and writing $F_t := \chi(T_t)$, we choose a partition $0 = t_0 < t_1 < \dots < t_n = 1$ such that 
$\big\| q(F_t) - q(F_{t_j}) \big\|_{\calQ_A(X)} < 1$ for all $t \in [t_{j},t_{j+1}]$ and $j=0,\ldots,n-1$. 
For each $j=1,\ldots,n-1$, we choose a skew-adjoint unitary $J_j \in \End_A^{r,s}(X)$ such that $q(F_{t_j}) = q(J_j)$, and we set $J_0 := F_0$ and $J_n := F_1$. 
Then we define 
\[
\Sf_{r,s+1}^\mathrm{Ph}\big( \{T_t\}_{t\in[0,1]} \big) 
:= \bigoplus_{j=0}^{n-1} \relind_{r,s+1}(J_{j+1},J_{j}) .
\]
\end{defn}
\begin{lemma}
The analytic spectral flow $\Sf_{r,s+1}^\mathrm{Ph}$ is well-defined. 
\end{lemma}
\begin{proof}
We note that such a partition exists because $t \mapsto q(F_t)$ is norm-continuous, and that such skew-adjoint unitaries exist by Lemma \ref{lemma:pointwise_skew_unitary}. 
Furthermore, by our choice of normalising function $\chi$, $J_0=F_0$ and $J_n=F_1$ are also skew-adjoint unitaries. 
Since $q(J_j)=q(F_{t_j})$ for each $j$, we have $\big\| q(J_j) - q(J_{j+1}) \big\| = \big\| q(F_{t_j}) - q(F_{t_{j+1}}) \big\| < 1$ by assumption. 
Hence each relative index $\relind_{r,s+1}(J_{j+1},J_{j})$ is well-defined (see Definition \ref{def:Ind_of_pair_of_complex_structures_DK}). 

To show that the definition is independent of the choice of partition, it suffices to show that we can always add an intermediate point $t_j < t' < t_{j+1}$. 
Let $J'$ be a skew-adjoint unitary lift of $q(F_{t'})$. By assumption, we have $\big\| q(J') - q(J_j) \big\| < 1$ as well as $\big\| q(J') - q(J_{j+1}) \big\| < 1$, 
so by additivity of the relative index (Lemma \ref{lemma:addition_for_Ind_J0J1}) we have 
\[
\relind_{r,s+1}(J_{j+1},J_{j}) = \relind_{r,s+1}(J_{j+1},J') + \relind_{r,s+1}(J',J_{j}) .
\]

Finally, we show that the definition is independent of the choices of $J_j$. 
For $j=1,\ldots,n-1$, let $J_j'$ be skew-adjoint unitaries such that $q(J_j') = q(J_j)$ (we set $J_0'=J_0$ and $J_n'=J_n$). 
Then, again using additivity of the relative index,  
\[
\relind_{r,s+1}(J_{j+1}',J_{j}') = \relind_{r,s+1}(J_{j+1}',J_{j+1}) + \relind_{r,s+1}(J_{j+1},J_{j}) + \relind_{r,s+1}(J_{j},J_{j}') .
\]
Since $\relind_{r,s+1}(J_j',J_j) + \relind_{r,s+1}(J_j,J_j') = 0$, we see that the first and third terms on the right hand side disappear after summing over $j$, and we conclude 
\begin{align*}
\bigoplus_{j=0}^{n-1} \relind_{r,s+1}(J_{j+1}',J_{j}') 
&= \bigoplus_{j=0}^{n-1} \relind_{r,s+1}(J_{j+1},J_{j}) .
\qedhere
\end{align*}
\end{proof}

\begin{thm} \label{thm:Phillips_Riesz}
Let $\{T_t\}_{t\in [0,1]}$ be a Riesz-continuous path of Real skew-adjoint Fredholm operators 
in $\Reg_A^{r,s}(X)$ with invertible endpoints. 
Then 
\[
\Sf_{r,s+1}^\mathrm{Ph}\big( \{T_t\}_{t\in[0,1]} \big) = \Sf_{r,s+1}\big( \{T_t\}_{t\in[0,1]} \big) .
\]
\end{thm}
\begin{proof}
We recall that $F_t := \chi(T_t)$. 
Choose a partition $0 = t_0 < t_1 < \dots < t_n = 1$ such that 
$\big\| q(F_t) - q(F_{t_j}) \big\| < 1$ for all $t \in [t_{j},t_{j+1}]$ and $j=0,\ldots,n-1$. 
For each $j=1,\ldots,n-1$, choose a skew-adjoint unitary $J_j \in \End_A^{r,s}(X)$ such that $q(F_{t_j}) = q(J_j)$, and let $J_0 := F_0$ and $J_n := F_1$. 
Define the norm-continuous path $\{\hat{F}_t\}_{t\in[0,1]} \subset \End_A^{r,s}(X)$ of Real skew-adjoint Fredholm operators by 
\[
\hat{F}_t := F_t + \sum_{j=1}^{n-1} \varphi_j(t) (J_j - F_{t_j}) ,
\]
where $\{\varphi_j\}_{j=0}^n$ is a partition of unity on $[0,1]$ with $\varphi_j(t_j)=1$. 
Then $\hat{F}_{t_j} = J_j$ is invertible for all $j=0,\ldots,n$ (and the endpoints $\hat{F}_0=F_0$ and $\hat{F}_1=F_1$ remain unchanged). 
Since $\hat{F}_t - F_t$ is compact, norm-continuous, and vanishes at the endpoints, we see that $\hat{F}_\bullet - F_\bullet$ is compact. 
Hence the $\KKR$-classes $\class{\hat{F}_\bullet}$ and $\class{F_\bullet}$ coincide, and from Proposition \ref{prop:KK_and_DK_flow_compatible} we obtain $\Sf_{r,s+1}\big( \{F_t\}_{t\in[0,1]} \big) = \Sf_{r,s+1}\big( \{\hat{F}_t\}_{t\in[0,1]} \big)$. 
Using the additivity of the spectral flow under concatenation of paths (Proposition \ref{prop:Path_additivity_KK}), we then have 
\[
\Sf_{r,s+1}\big( \{F_t\}_{t\in[0,1]} \big) 
= \Sf_{r,s+1}\big( \{\hat{F}_t\}_{t\in[0,1]} \big) 
= \bigoplus_{j=0}^{n-1} \Sf_{r,s+1}\big( \{\hat{F}_t\}_{t\in[t_j,t_{j+1}]} \big) .
\]
Finally, by construction, each path $\{\hat{F}_t\}_{t\in[t_j,t_{j+1}]}$ consists of ``small perturbations'', so 
from Lemma \ref{lemma:sf_small_perturbations} we know that its spectral flow is given by the relative index of $\hat{F}_{t_{j+1}} = J_{j+1}$ and $\hat{F}_{t_j} = J_j$:
\[
\Sf_{r,s+1}\big( \{\hat{F}_t\}_{t\in[t_j,t_{j+1}]} \big) = \relind_{r,s+1} ( J_{j+1} , J_{j} ) .
\qedhere 
\]
\end{proof}

\subsection{Comparison with real Hilbert spaces} \label{subsec:Sf_Real_HS}

Using our analytic description of the spectral flow, we can now relate our $\DK$-valued spectral flow 
to the spectral flow on real Hilbert spaces considered in \cite[\S5-6]{BCLR}. 
Recall also from \S\ref{subsec:HSpace_J0J1_index} the description of the relative index on real Hilbert spaces in terms of Clifford modules from 
\cite[\S4]{BCLR}.

As discussed in Remark \ref{Rk:BLCR_sign_error}, our $\DK$-valued spectral flow has a normalisation that is 
inconsistent with the Clifford module valued spectral flow from \cite{BCLR}. So while there is a global 
minus sign, this only affects the case where the range is $KO_0(\R)$ or $KO_4(\R)$.

\begin{cor}
Let $\calH$ be an ungraded Real Hilbert space with an ample and ungraded representation of $\Cl_{r,s}$. 
Let $\{T_t\}_{t\in [0,1]}$ be a Riesz-continuous path of Real skew-adjoint Fredholm operators 
in $\Reg_\C^{r,s}(\calH)$ with invertible endpoints. 
Then, under the isomorphism $\Upsilon_\C^{r,s+1} \colon \DK(\Cl_{r,s+1} ) \to  \calM_{r,s+1}/ \calM_{r,s+2} \cong \KO_{s+2-r}(\R)$ from 
Proposition \ref{prop:Fred_pair_index_agrees_with_hspace}, we have 
\[
   \Sf_{r,s+1}\big( \{T_t\}_{t\in[0,1]} \big) \xmapsto{\Upsilon_\C^{r,s+1}} - \Sf^\calH_{r,s+2} \big( \{T_t\}_{t\in[0,1]} \big) \in \calM_{r,s+1}/ \calM_{r,s+2} \cong \KO_{s+2-r}(\R),
\]
where $\Sf^\calH_{r,s+2}$ denotes the Clifford module valued spectral flow from~\cite{BCLR}.
\end{cor}

\begin{proof}
We can choose a sufficiently fine partition $0 = t_0 < \cdots < t_n = 1$ and a normalising function $\chi$ such that, for each $j=0,\ldots,n$, $\chi(T_{t_j})$ is a skew-adjoint unitary on the complement of $\Ker(T_{t_j})$. 
Choose $J_j$ to be skew-adjoint unitaries obtained by complementing $\chi(T_{t_j})$ on $\Ker(T_{t_j})^\perp$ by arbitrary skew-adjoint unitaries in $\End^{r,s}\big( \Ker(T_{t_j}) \big)$. 
From Theorem \ref{thm:Phillips_Riesz}, Proposition \ref{prop:Fred_pair_index_agrees_with_hspace}, and \cite[Theorem 6.4]{BCLR} we then have the equalities 
\begin{align*}
\Upsilon_\C^{r,s+1} \big( \Sf_{r,s+1}\big( \{T_t\}_{t\in[0,1]} \big) \big) 
&= \Upsilon_\C^{r,s+1} \Big( \bigoplus_{j=0}^{n-1} \relind_{r,s+1}(J_{j+1},J_{j}) \Big)  \\
&= \bigoplus_{j=0}^{n-1} \relindH_{r,s+2}(J_{j+1},J_{j})  \\
&= -\bigoplus_{j=0}^{n-1} \relindH_{r,s+2}(J_j,J_{j+1}) 
= - \Sf^\calH_{r,s+2} \big( \{T_t\}_{t\in[0,1]} \big) .
\qedhere 
\end{align*}
\end{proof}

\subsection{The Wahl formula for Wahl-continuous paths} \label{subsec:analytic_Wahl}

In this subsection, we will adapt Wahl's approach to the analytic spectral flow on complex $C^*$-modules \cite[\S3, \S8]{Wahl07} (see also the exposition in \cite[Appendix A]{vdDungen24}) to our setting of Real $C^*$-modules. 

\begin{defn}
Let $T \in \Reg_A^{r,s}(X)$ be a Real skew-adjoint operator. 
A \emph{trivialising operator for $T$} is a (densely-defined) Real skew-symmetric operator $\B \in \Reg_A^{r,s}(X)$ such that $\B$ is relatively $T$-compact 
(i.e., $\Dom(\B) \subset \Dom(T)$ and $\B(T\pm\one)^{-1}$ is compact) and $T+\B$ is invertible. 
\end{defn}

If $T \in \Reg_A^{r,s}(X)$ has a trivialising operator $ \B \in \Reg_A^{r,s}(X)$, if follows from \cite[Proposition A.7]{vdDungen23} that $T+\B$ is also regular and skew-adjoint. 
Moreover, by \cite[Proposition A.11]{vdDungen23}, $T$ is Fredholm and $\class{T} = \class{T+\B} = 0 \in \KKR(\Cl_{s+1,r}, A)$. 

If $\B_0$ and $\B_1$ are two trivialising operators for $T$, then we can choose a normalising function $\chi$ such that $\chi(T+\B_0)$ and $\chi(T+B_1)$ are skew-adjoint unitaries. 
Moreover, by \cite[Proposition A.9]{vdDungen23}, $\chi(T+\B_0) - \chi(T+B_1)$ is compact. 
Hence we can define 
\begin{equation*}
\label{eq:ind}
\ind_{r,s+1}(T,\B_1,\B_0) := \relind_{r,s+1}\big( \chi(T+\B_1) ,\, \chi(T+B_0) \big) .
\end{equation*}

\begin{defn}
Let $\Omega$ be a compact Hausdorff space and  $\{T_\omega\}_{\omega\in\Omega}$ a Wahl-continuous family  of Real skew-adjoint operators in $\Reg_A^{r,s}(X)$. 
A \emph{trivialising family for $\{T_\omega\}_{\omega\in\Omega}$} is a family $\{\B_\omega\}_{\omega\in\Omega}$ of Real skew-symmetric operators in $\End_A^{r,s}(X)$ such that $\B_\bullet$ is a trivialising operator for $T_\bullet$. 

We say \emph{there exist locally trivialising families} for $\{T_\omega\}_{\omega\in\Omega}$ if for each $\omega\in\Omega$ there exist a compact neighbourhood $O_\omega$ of $\omega$ and a trivialising family for $\{T_{\omega'}\}_{\omega' \in O_\omega}$.
\end{defn}
We note that the existence of locally trivialising families for $\{T_\omega\}_{\omega\in\Omega}$ then implies that $T_\bullet$ is Fredholm (using compactness of $\Omega$). 

\begin{defn}[Analytic spectral flow II] \label{defn:spectral_flow_Wahl}
Let $\{T_t\}_{t\in[0,1]}$ be a Wahl-continuous path of Real skew-adjoint Fredholm operators in $\Reg_A^{r,s}(X)$ with invertible endpoints. 
Assume there exist locally trivialising families for $\{T_t\}_{t\in[0,1]}$. 
Let $0 = t_0 < t_1 < \cdots < t_n = 1$ be such that there is a trivialising family $\{\B^j_t\}_{t\in[t_j,t_{j+1}]}$ of $\{T_t\}_{t\in[t_j,t_{j+1}]}$ for each $j=0,\ldots,n-1$. 
Then we define 
\begin{equation*}
\Sf_{r,s+1}^\mathrm{W} \big( \{T_t\}_{t\in[0,1]} \big) 
:= \ind_{r,s+1}\big(T_0,\B^0_0, 0\big) + \sum_{j=1}^{n-1} \ind_{r,s+1}\big(T_{t_j},\B^{j}_{t_j},\B^{j-1}_{t_j}\big) + \ind_{r,s+1}\big(T_1,0, \B^{n-1}_1\big) .
\end{equation*}
\end{defn}
For ease of notation, we will sometimes write $\B^{-1}_0 = \B^n_1 = 0$, so that we can express the analytic spectral flow as a single sum from $j=0$ to $n$:
\begin{equation*}
\Sf_{r,s+1}^\mathrm{W} \big( \{T_t\}_{t\in[0,1]} \big) 
:= \sum_{j=0}^{n} \ind_{r,s+1}\big(T_{t_j},\B^{j}_{t_j},\B^{j-1}_{t_j}\big) .
\end{equation*}

\begin{lemma}
The analytic spectral flow $\Sf_{r,s+1}^\mathrm{W}$ is well-defined. 
\end{lemma}
\begin{proof}
We already observed that, for each $j$, $\chi(T_{t_j}+\B^{j}_{t_j}) - \chi(T_{t_j}+\B^{j-1}_{t_j})$ is compact, so that $\ind_{r,s+1}\big(T_{t_j},\B^{j}_{t_j},\B^{j-1}_{t_j}\big)$ is well-defined. 
To show that the definition is independent of the choice of partition, it suffices to show that we can always add an intermediate point $t_j < t' < t_{j+1}$. 
Choosing the trivialising families on $[t_j,t']$ and $[t',t_{j+1}]$ to be simply obtained from restricting a trivialising family $\{\B^j_t\}_{t\in[t_j,t_{j+1}]}$, we then merely add a vanishing term $\ind_{r,s+1}\big(T_{t'},\B^{j}_{t'},\B^j_{t'}\big)$ to the sum. 

Finally, we show that the definition is independent of the choice of locally trivialising families. 
So suppose we also have trivialising families $\{\A^j_t\}_{t\in[t_j,t_{j+1}]}$ of $\{T_t\}_{t\in[t_j,t_{j+1}]}$. 
By additivity of the relative index we have 
\begin{multline} \label{eq:sf_Wahl_welldef}
\ind_{r,s+1}\big(T_{t_j},\B^{j}_{t_j},\B^{j-1}_{t_j}\big) \\
= \ind_{r,s+1}\big(T_{t_j},\B^{j}_{t_j},\A^{j}_{t_j}\big) + \ind_{r,s+1}\big(T_{t_j},\A^{j}_{t_j},\A^{j-1}_{t_j}\big) + \ind_{r,s+1}\big(T_{t_j},\A^{j-1}_{t_j},\B^{j-1}_{t_j}\big) .
\end{multline}
By the normalisation of the spectral flow (Proposition \ref{prop:sf_normalisation}), 
$\ind_{r,s+1}\big(T_{t_j},\A^j_{t_j},\B^j_{t_j}\big)$ equals the spectral flow of the straight line path from $\chi(T_{t_j}+\B^j_{T_j})$ to $\chi(T_{t_j}+\A^j_{T_j})$. 
Since both $\{T_t+\A^j_t\}_{t\in[t_j,t_{j+1}]}$ and $\{T_t+\B^j_t\}_{t\in[t_j,t_{j+1}]}$ are invertible paths, it then follows from the homotopy invariance of the spectral flow that 
\[
\ind_{r,s+1}\big(T_{t_j},\A^j_{t_j},\B^j_{t_j}\big) = \ind_{r,s+1}\big(T_{t_{j+1}},\A^j_{t_{j+1}},\B^j_{t_{j+1}}\big) .
\]
Hence the first and third terms on the right hand side of Eq.\ \eqref{eq:sf_Wahl_welldef} cancel after summing over $j$, and we conclude 
\begin{align*}
\bigoplus_{j=0}^{n} \ind_{r,s+1}\big(T_{t_j},\B^{j}_{t_j},\B^{j-1}_{t_j}\big)
&= \bigoplus_{j=0}^{n} \ind_{r,s+1}\big(T_{t_j},\A^{j}_{t_j},\A^{j-1}_{t_j}\big) .
\qedhere
\end{align*}
\end{proof}

\begin{thm} \label{thm:sf_Wahl}
Let $\{T_t\}_{t\in[0,1]}$ be a Wahl-continuous path of Real skew-adjoint Fredholm operators in $\Reg_A^{r,s}(X)$ with invertible endpoints. 
Assume there exist locally trivialising families for $\{T_t\}_{t\in[0,1]}$. 
Then 
\[
\Sf_{r,s+1}^\mathrm{W}\big( \{T_t\}_{t\in[0,1]} \big) = \Sf_{r,s+1}\big( \{T_t\}_{t\in[0,1]} \big) .
\]
\end{thm}
\begin{proof}
Let $0 = t_0 < t_1 < \dots < t_n = 1$ be such that there is a trivialising family $\{\B^j_t\}_{t\in[t_j,t_{j+1}]}$ of $\{T_t\}_{t\in[t_j,t_{j+1}]}$ for each $j=0,\ldots,n-1$. 
On each segment $[t_j,t_{j+1}]$, we can use homotopy invariance of the spectral flow to replace the path $\{T_t\}_{t\in[t_j,t_{j+1}]}$ by the concatenation of the following three paths: first the straight line path from $T_{t_j}$ to $T_{t_j} + \B^j_{t_j}$, then the invertible path $\{T_t+\B^j_t\}_{t\in[t_j,t_{j+1}]}$, and finally the straight line path from $T_{t_{j+1}} + \B^j_{t_{j+1}}$ to $T_{t_{j+1}}$. 
Concatenating all segments together, the new path therefore consists alternatingly of piecewise linear segments from $T_{t_j} + \B^{j-1}_{t_j}$ to $T_{t_j} + \B^j_{t_j}$ and invertible segments $\{T_t+\B^j_t\}_{t\in[t_j,t_{j+1}]}$. 
The spectral flow of the invertible segments vanishes. 
The piecewise linear segments consist of relatively compact perturbations of the operator $T_{t_j}$, so we may again use homotopy invariance to replace these segments by straight line paths. 
Using the path additivity property (Proposition \ref{prop:Path_additivity_KK}), the spectral flow can therefore be written as 
\[
\Sf_{r,s+1}\big( \{T_t\}_{t\in[0,1]} \big) 
= \bigoplus_{j=0}^n \Sf_{r,s+1}\big( (1-t) (T_{t_j} + \B^{j-1}_{t_j}) + t (T_{t_j} + \B^j_{t_j}) \big) .
\]
We may assume that a normalising function $\chi$ is chosen such that $\chi\big( T_{t_j} + \B^{j-1}_{t_j} \big)$ and $\chi\big( T_{t_j} + \B^j_{t_j} \big)$ are skew-adjoint unitaries. 
By the normalisation property (Proposition \ref{prop:sf_normalisation}), the spectral flow between them is then given by the relative index, and we obtain 
\begin{align*}
\Sf_{r,s+1}\big( \{T_t\}_{t\in[0,1]} \big) 
&= \bigoplus_{j=0}^n \relind_{r,s+1}\big( \chi(T_{t_j} + \B^{j}_{t_j}) ,\, \chi(T_{t_j} + \B^{j-1}_{t_j}) \big) \\
&= \bigoplus_{j=0}^n \ind_{r,s+1} \big( T_{t_j} , \B^{j}_{t_j} , \B^{j-1}_{t_j} \big) 
= \Sf_{r,s+1}^\mathrm{W}\big( \{T_t\}_{t\in[0,1]} \big) .
\qedhere 
\end{align*}
\end{proof}

Finally, we compare the above result with the Phillips formula for Riesz-continuous paths from Theorem \ref{thm:Phillips_Riesz}. 
The following proposition shows that, for \emph{Riesz}-continuous paths $\{T_t\}$ with invertible endpoints, locally trivialising families always exist for the \emph{normalised} path $\{\chi(T_t)\}$, 
and both versions 
of the analytic spectral flow agree. 

\begin{prop}
Let $\{T_t\}_{t\in[0,1]}$ be a Riesz-continuous path of Real skew-adjoint Fredholm operators in $\Reg_A^{r,s}(X)$ with invertible endpoints. 
Choose a normalising function $\chi$ such that $\chi(T_0)$ and $\chi(T_1)$ are skew-adjoint unitaries. 
Then there exist locally trivialising families for $\{\chi(T_t)\}_{t\in[0,1]}$, and 
\[
\Sf_{r,s+1}^\mathrm{Ph}\big( \{T_t\}_{t\in[0,1]} \big) = \Sf_{r,s+1}^\mathrm{W}\big( \{\chi(T_t)\}_{t\in[0,1]} \big) .
\]
\end{prop}
\begin{proof}
Since the path $\{\chi(T_t)\}_{t\in[0,1]}$ is norm-continuous, we know from Lemma \ref{lemma:pointwise_skew_unitary} that, for each $t_0\in[0,1]$, there exists a skew-adjoint unitary $J_{t_0}$ such that $\B_{t_0} := J_{t_0} - \chi(T_{t_0})$ is compact. 
By norm-continuity, it follows that $\chi(T_t) + \B_{t_0}$ is invertible on some compact neighbourhood $\Omega_0$ of $t_0$, so that the constant family 
$\{\B_{t_0}\}_{t\in \Omega_0}$ provides a trivialising family on this neighbourhood. 
Hence both Definitions \ref{defn:spectral_flow_Phillips} and \ref{defn:spectral_flow_Wahl} of the analytic spectral flow apply (where the latter is applied to $\{\chi(T_t)\}$), and they are equal by Theorems \ref{thm:Phillips_Riesz} and \ref{thm:sf_Wahl}. 
\end{proof}

\begin{remark}
The above proposition only shows that, for a Riesz-continuous path $\{T_t\}_{t\in[0,1]}$ of Real skew-adjoint Fredholm operators in $\Reg_A^{r,s}(X)$ with invertible endpoints, there exist locally trivialising families for the \emph{normalised} path $\{\chi(T_t)\}$. 
We leave it as an open question whether or not there also exist locally trivialising families for the Riesz-continuous path $\{T_t\}_{t\in[0,1]}$ itself. 
\end{remark}

\section{The Robbin--Salamon theorem} \label{sec:Robbin-Salamon}

The Robbin--Salamon theorem~\cite{RS} relates the spectral flow of a family of (unbounded) self-adjoint Fredholm operators $\{D_t\}$ on a complex Hilbert space 
to the Fredholm index of the operator $\partial_t + D_\bullet$. 
Many generalisations of this result have been 
investigated, particularly from the viewpoint of the Kasparov product in unbounded $\KK$-theory~\cite{KL13, vdDungen19, BCLR, vdDungen23, vdDungen24}. 
In this section, we will present a similar result in the setting of Real $C^*$-modules. 

We start with the following observation. 
Consider a Wahl-continuous family $\{T_t\}_{t\in[0,1]}$ of Real skew-adjoint Fredholm operators in $\Reg^{r,s}_A(X)$ with invertible endpoints. 
Recall from Proposition \ref{prop:KK_and_DK_flow_compatible} that the spectral flow of $\{T_t\}_{t\in[0,1]}$ 
is obtained from the $\KKR$-class of $T_\bullet$ via the Roe isomorphism $\Roe_{SA}^{r+1,s+1}: \KKR(\Cl_{s+1,r}, SA)\xrightarrow{\simeq} \DK( SA \otimes \Cl_{r+1,s+1} )$ 
composed with Bott periodicity: 
\[
\Sf_{r,s+1} ( \{T_t\}_{t\in[0,1]} ) 
= \beta_{\DK}^{-1}\circ \Roe_{SA}^{r+1,s+1} \big( \class{T_\bullet} \big) .
\]
We can use Theorem \ref{prop:Roe_Kubota_Bott_compatibility} to interchange the order of the Roe isomorphism and Bott periodicity, at the price of  two additional Clifford generators. 
Namely $\calM_{1,1}^{\DK} \circ   \Roe^{r+1,s+1}_{SA} \circ \altBott_{\KK} = \beta_{\DK} \circ \Roe^{r+1,s+2}_A$, which is equivalent to the statement 
$\calM_{1,1}^{\DK} \circ \beta_{DK}^{-1} \circ \Roe_{SA}^{r+1,s+1} = \Roe_{A}^{r+1,s+2} \circ \altBott_{\KK}^{-1}$,
where $\altBott_{\KK}$ is the alternative version of Bott periodicity (Corollary \ref{cor:KK_Bott_alternative}) 
and $\mathcal{M}_{1,1}^{\DK} \colon \DK(B) \xrightarrow{\simeq} \DK(B\hox\Cl_{1,1})$ is the Clifford stability isomorphism 
(Lemma \ref{lemma:DK_Clifford_stability}). 
Put another way, the following diagram commutes:
\begin{equation} \label{eq:Roe_Bott}
\xymatrixcolsep{4pc}
\xymatrix{
  \qquad \class{T_\bullet} \qquad \ar@{|->}[r] \ar@{}[d]|{\vin} & \quad \Index_{r+1,s+1}(T_\bullet) \quad \ar@{|->}[r] \ar@{}[d]|{\vin} & \quad \Sf_{r,s+1} ( \{T_t\}_{t\in[0,1]} ) \ar@{}[d]|{\vin} \\
  \KKR\big( \Cl_{s+1,r} , SA \big) \ar[dr]_{\altBott_{\KK}^{-1}} \ar[r]^{\Roe_{SA}^{r+1,s+1}} & \DK\big( SA \otimes \Cl_{r+1,s+1} \big) \ar[r]^{\beta_{\DK}^{-1}} & \quad \DK\big( A \otimes \Cl_{r,s+1} \big) \ar[d]^{\mathcal{M}_{1,1}^{\DK}} \\
  & \KKR\big( \Cl_{s+2,r} , A \big) \ar[r]^{\Roe^{r+1,s+2}_{A}} & \DK\big( A \otimes \Cl_{r+1,s+2} \big). \\
 }
\end{equation}
Since the Roe isomorphism of a $\KKR$-class is given by the Fredholm index (Theorem \ref{thm:index-of-Fred}), 
this means that the element $\mathcal{M}_{1,1}^{\DK} \big( \Sf_{r,s+1} ( \{T_t\}_{t\in[0,1]} ) \big)$ (appearing at the bottom right of the above diagram) can be computed as the Fredholm index of an operator representing the $\KKR$-class $\altBott_{\KK}^{-1}\big( \class{T_\bullet} \big)$. 
Our goal in this section is then to find an \emph{explicit} skew-adjoint Fredholm operator $(T_\bullet+\partial)_+ \in \Reg_A^{r,s+1}\big( L^2(\R,X_A)^{\oplus2} \big)$, 
such that $\class{(T_\bullet+\partial)_+} = \altBott_{\KK}^{-1}\big( \class{T_\bullet} \big)$, and therefore 
\[
   \calM_{1,1}^{\DK} \big( \Sf_{r,s+1}\big( \{T_t\}_{t\in[0,1]} \big) \big) 
   = \Index_{r+1,s+2}\big( ( T_\bullet + \partial)_+ \big) 
   \in \DK( A \otimes \Cl_{r+1,s+2} ).
\]

\begin{remark}
In the complex setting, the appearance of the Clifford stability isomorphism $\calM_{1,1}^{\DK} : \DK(B ) \xrightarrow{\simeq} \DK( B \hox \Cl_{1,1} )$ can be avoided. 
Indeed, for complex Clifford algebras (without real structures) we have the isomorphism 
\(
\Cl_{(s+2)+r} \cong \Cl_{s+r} \hox \Cl_2 \cong \Cl_{s+r} \hox M_2(\C) .
\)
Using stability of $\KK$-theory, we have $\KK\big( \Cl_{(s+2)+r} , A \big) \cong \KK\big( \Cl_{s+r} , A \big)$. 
The Kasparov product $\class{T_\bullet} \otimes_{C_0(\R)} [\slashed{\partial}] \in \KK\big( \Cl_{(s+2)+r} , A \big)$ can then explicitly be described by a Fredholm operator representing a class in $\KK\big( \Cl_{s+r} , A \big)$, whose Fredholm index coincides with the spectral flow of $\{T_t\}_{t\in [0,1]}$. In our Real setting, however, $\Cl_{s+2,r}$ and $\Cl_{s,r}$ are \emph{not} isomorphic (instead, we only have $\Cl_{s+2,r} \cong \Cl_{s+1,r-1}$ if $r\geq1$), and the Clifford stability map $\calM_{1,1}^{\DK}$ is needed to take care of the two additional Clifford generators at a later stage. 
\end{remark}

We recall from Corollary \ref{cor:KK_Bott_alternative} that $\altBott_{\KK}^{-1}$ can be described by taking the Kasparov product over $C_0(\R)$ with the class $[\slashed{\partial}] \in \KKR\big( \C_0(\R) \otimes \Cl_{1,0} , \C \big)$, 
which is represented by the unbounded Kasparov module of the standard Dirac operator on the real line:
\[
   \Big(  C_0^1(\R)\otimes \Cl_{1,0}, \, L^2(\R) \otimes \exterior \C, \, \partial_t \otimes \rho \Big) . 
\]
Our first task is therefore to explicitly compute the Kasparov product $\altBott_{\KK}^{-1}\big( \class{T_\bullet} \big) = \class{T_\bullet} \otimes_{C_0(\R)} [\slashed{\partial}]$.

\subsection{The Kasparov product}

In order to ensure we can explicitly compute the Kasparov product $\altBott_{\KK}^{-1}\big( \class{T_\bullet} \big) = \class{T_\bullet} \otimes_{C_0(\R)} [\slashed{\partial}]$, we will need to impose additional assumptions on the path $\{T_t\}_{t\in[0,1]}$, so that we may adapt the approach of \cite{KL13,vdDungen19}. 

\begin{assump}\label{RS_assumptions}
Let $X_A$ be a full $C^*$-module with an ungraded and ample $\Cl_{r,s}$-representation 
generated by $\{e_1,\ldots, e_r, f_1, \ldots f_s\}$.
Let $\{T_t\}_{t\in [0,1]}$ be a Wahl-continuous path of (bounded or unbounded) Real skew-adjoint Fredholm operators in $\Reg_A^{r,s}(X)$ with invertible endpoints. 
We assume: 
\begin{enumerate}
\item 
The domain $W := \Dom(T_t)$ is independent of $t\in[0,1]$, and the inclusion $W_A \hookrightarrow X_A$ is compact (where $W_A$ is equipped with the graph norm of $T_{t_0}$, for some $t_0\in[0,1]$). 
\item 
The resolvents $(\one + T_t)^{-1} \in \End_A^0(X)$ depend norm-continuously on $t\in[0,1]$. 
\end{enumerate}
\end{assump}

We extend the path $\{T_t\}_{t\in[0,1]}$ from the unit interval $[0,1]$ to the real line $\R$ 
by setting $T_t := T_1$ for $t\geq 1$ and $T_t := T_0$ for $t \leq 0$. 
We then denote by $T_\bullet$ the corresponding family operator on the $C^*$-module $C_0(\R,X)_{C_0(\R,A)}$. 

Consider now the $C^*$-module over $A$ given by 
\[
L^2(\R, X)_A  = C_0(\R,X)_{C_0(\R,A)} \otimes_{C_0(\R)} L^2(\R) .
\]
Since $\{T_t\}$ is Wahl-continuous, we obtain a regular skew-adjoint operator $T_\bullet \equiv T_\bullet \otimes \one$ on $L^2(\R, X)_A$. 
Under the identification $L^2(\R, X)_A \cong X_A \otimes L^2(\R)$, we also obtain a regular skew-adjoint operator $\partial_t \equiv \one \otimes \partial_t$ on $L^2(\R, X)_A$. 

Recall that $\gamma_1,\gamma_2,\rho_1,\rho_2$ denote the generators of the graded Clifford algebra $\Cl_{2,2} \simeq \End\big(\exterior \C^2\big)$.
On $L^2(\R, X)_A \otimes \exterior \C^2$, we then consider the representation of $\Cl_{s+2,r}$ given by the generators
\[
   \big\{ \one \otimes \gamma_1, \one\otimes \gamma_2, f_1\otimes \rho_1, \ldots, f_s \otimes \rho_1, e_1\otimes \rho_1, \ldots, e_r\otimes \rho_1 \big\}.
\]

\begin{prop}[cf.\ \cite{vdDungen19}] \label{prop:Dirac_product}
Let $\{T_t\}_{t\in [0,1]}$ satisfy Assumption \ref{RS_assumptions}. 
Then the operator $T_\bullet \otimes \rho_1 + \partial_t \otimes \rho_2$ is a Real odd self-adjoint Fredholm operator in $\Reg_A^{s+2,r}\big( L^2(\R, X)_A \otimes \exterior\C^2 \big)$, and 
\[
\big[ T_\bullet \otimes \rho_1 + \partial_t \otimes \rho_2 \big] 
= \altBott_{\KK}^{-1} \big( \class{T_\bullet} \big) 
\in \KKR(\Cl_{s+2,r},A).
\]
\end{prop}
\begin{proof}
We easily see that the Real regular self-adjoint operators $T_\bullet \otimes \rho_1$ and $\partial_t \otimes \rho_2$ anti-commute with the generators of $\Cl_{s+2,r}$. 
In the complex setting (without real structures), it was shown in \cite[Proposition 3.16]{vdDungen19} that the operator
\[
T_\bullet \otimes \begin{pmatrix} 0 & -1 \\ 1 & 0 \end{pmatrix} + \partial_t \otimes \begin{pmatrix} 0 & -i \\ -i & 0 \end{pmatrix} 
\]
is regular and self-adjoint. 
In our Real setting, the same argument applies, if we replace the two matrices by two (anti-commuting) Real skew-adjoint Clifford generators $\rho_1$ and $\rho_2$. 
Thus $T_\bullet \otimes \rho_1 + \partial_t \otimes \rho_2$ is regular and self-adjoint. 
Moreover, this operator is also Fredholm, as we may similarly construct a (left or right) parametrix via the same argument as in \cite[Theorem 4.3]{vdDungen19} (see also \cite[Lemma 2.3]{AW11} for the case of the real line). 
Finally, it follows from the same argument as in \cite[Theorem 5.15]{vdDungen19} that 
the $\KK$-class $\big[ T_\bullet \otimes \rho_1 + \partial_t \otimes \rho_2 \big]$ 
represents the internal Kasparov product over $C_0(\R)$ of $\class{T_\bullet}$ with $[\slashed{\partial}]$:
\[
\big[ T_\bullet \otimes \rho_1 + \partial_t \otimes \rho_2 \big] 
= \class{T_\bullet} \otimes_{C_0(\R)} [\slashed{\partial}] 
= \altBott_{\KK}^{-1} \big( \class{T_\bullet} \big) .
\qedhere 
\]
\end{proof}

\subsection{Spectral flow and the Fredholm index}

Thus far, we have found an explicit representative $T_\bullet \otimes \rho_1 + \partial_t \otimes \rho_2$ for the $\KKR$-class $\altBott_{\KK}^{-1} \big( \class{T_\bullet} \big)$. 
It remains to describe this class in terms of our $\DK$-valued Fredholm index. 
For this purpose, we need to rewrite the operator in the form 
\[
T_\bullet \otimes \rho_1 + \partial_t \otimes \rho_2 = (T_\bullet+\partial)_+ \otimes \rho ,
\]
where $(T_\bullet+\partial)_+$ is an \emph{ungraded} skew-adjoint Fredholm operator which anti-commutes with the generators of an \emph{ungraded} Clifford algebra $\Cl_{r,s+1}$. 
It will then follow that the $\KKR$-class of $T_\bullet \otimes \rho_1 + \partial_t \otimes \rho_2$ coincides (under the Roe isomorphism) with the Fredholm index of $(T_\bullet+\partial)_+$:
\[
\Index_{r+1,s+2}\big( (T_\bullet+\partial)_+ \big) 
= \Roe_A^{r+1,s+2} \big( \class{(T_\bullet+\partial)_+} \big) 
= \Roe_A^{r+1,s+2} \big( [ T_\bullet \otimes \rho_1 + \partial_t \otimes \rho_2 ] \big) .
\]

\begin{defn}
Let $\{T_t\}_{t\in [0,1]}$ satisfy Assumption \ref{RS_assumptions}. 
The (densely-defined) operator $( T_\bullet + \partial)_+$ on the ungraded $C^*$-module $L^2(\R, X_A)^{\oplus 2}$ is defined by 
\[
  ( T_\bullet + \partial)_+ := \begin{pmatrix}  \partial_t & T_\bullet  \\  T_\bullet & -\partial_t  \end{pmatrix} 
   = T_\bullet \otimes \sigma_1 + \partial_t \otimes \sigma_3 .
\]
\end{defn}
The operator $( T_\bullet + \partial)_+$ anti-commutes with an ungraded $\Cl_{r,s+1}$-representation generated by 
\[
  \big\{ e_1 \otimes \sigma_1, \ldots, e_r \otimes \sigma_1, f_1 \otimes \sigma_1, \ldots, f_s \otimes \sigma_1, \one \otimes (-i\sigma_2) \big\}.
\]
Furthermore, via a similar argument as in the proof of Proposition \ref{prop:Dirac_product}, we see that $( T_\bullet + \partial)_+$ is Real, regular, skew-adjoint, and Fredholm. 
In particular, we have a well-defined Fredholm index 
\(
    \Index_{r+1,s+2}\big( ( T_\bullet + \partial)_+ \big) \in \DK(A \otimes \Cl_{r+1,s+2}) .
\)

\begin{thm} \label{thm:RS}
Let $\{T_t\}_{t\in [0,1]}$ satisfy Assumption \ref{RS_assumptions}. Then 
\[
   \calM_{1,1}^{\DK} \big( \Sf_{r,s+1}\big( \{T_t\}_{t\in[0,1]} \big) \big) 
   = \Index_{r+1,s+2}\big( ( T_\bullet + \partial)_+ \big) 
   \in \DK( A \otimes \Cl_{r+1,s+2} ).
\]
\end{thm}
\begin{proof}
We claim that the unbounded Kasparov module 
\[
   \Big( \Cl_{s+2, r}, \, L^2(\R, X_A)^{\oplus 2} \otimes \exterior \C, \, ( T_\bullet + \partial)_+   \otimes \rho \Big)
\]
represents the same class in $\KKR(\Cl_{s+2,r}, A)$ as 
\[
  \Big( \Cl_{s+2, r}, \, L^2(\R, X_A) \otimes \exterior \C^2, \,  T_\bullet \otimes \rho_1 + \partial_t \otimes \rho_2  \Big) .
\]
We can identify $\exterior \C^2 \cong \C^2 \otimes \exterior \C$ (where $\C^2$ is ungraded and $\exterior \C$ has the standard grading) and correspondingly obtain $\Cl_{2,2} \cong M_2(\C) \otimes \Cl_{1,1}$ via 
\[
\gamma_1 \cong \one \otimes \gamma , \quad 
\gamma_2 \cong -i\sigma_2 \otimes \rho , \quad 
\rho_1 \cong \sigma_1 \otimes \rho , \quad 
\rho_2 \cong \sigma_3 \otimes \rho . 
\]
(Note that the computation $\gamma_1 \gamma_2 \rho_1 \rho_2 \cong \Id \otimes \gamma\rho$ shows that the $\Z_2$-gradings on both sides are indeed compatible under this identification.) 
Under this identification, we may rewrite 
\[
T_\bullet \otimes \rho_1 + \partial_t \otimes \rho_2 
\cong \big( T_\bullet \otimes \sigma_1 + \partial_t \otimes \sigma_3 \big) \otimes \rho 
= ( T_\bullet + \partial)_+ \otimes \rho .
\]
Hence from Proposition \ref{prop:Dirac_product} we have that 
\[
   \class{( T_\bullet + \partial)_+} 
   = \big[ T_\bullet \otimes \rho_1 + \partial_t \otimes \rho_2 \big] 
   = \altBott_{\KK}^{-1}\big( \class{T_\bullet} \big) \in \KKR(\Cl_{s+2,r}, A).
\]
From Theorem \ref{thm:index-of-Fred}, commutativity of  the diagram \eqref{eq:Roe_Bott} and Definition \ref{def:DK_spec_flow}, 
we conclude that 
\begin{align*}
   \Index_{r+1,s+2}\big( ( T_\bullet + \partial)_+ \big)  
   &= \Roe_A^{r+1,s+2}\big( \class{( T_\bullet + \partial)_+} \big)
   = \Roe_A^{r+1,s+2} \circ \altBott_{\KK}^{-1}\big( \class{T_\bullet} \big) \\
   &= \mathcal{M}_{1,1}^{\DK} \circ \beta_{\DK}^{-1} \circ \Roe_{SA}^{r+1,s+1} \big( \class{T_\bullet} \big) 
   = \mathcal{M}_{1,1}^{\DK} \big( \Sf_{r,s+1}\big( \{T_t\}_{t\in[0,1]} \big) \big) .
   \qedhere 
\end{align*}
\end{proof}

\section{Examples from physics} \label{sec:Physics_examples}

\subsection{Free fermion systems with Altland--Zirnbauer  symmetries} \label{subsec:AMZ}

We follow a description of free fermionic Hamiltonians and ground states used by Zirnbauer et al.~\cite{KZ16, AMZ}, 
which is based on the Hartree--Fock--Bogoliubov  mean-field approximation. 
We are typically interested in models of  (superconducting) electrons, more generally fermions, 
which we model mathematically via vectors  in 
a complex Hilbert space $\calV$. The anti-particles (electron holes) are modelled 
via the dual space $\calV^*$. Key properties of superconducting systems can be approximated 
via the dynamics of a self-adjoint Hamiltonian $H$ on the Hilbert space $\calH = \calV \oplus \calV^*$. 
The complex Hilbert space $\calH$  is a Real Hilbert space via the real structure $\psi^\rs = \Gamma \psi$, 
where $\Gamma = \begin{pmatrix} 0 & \mathcal{R}^{-1} \\ \mathcal{R} & 0 \end{pmatrix}$ is 
a self-adjoint antiunitary and $\mathcal{R}: \calV \to \calV^*$ is the 
Riesz isomorphism. We therefore obtain a real structure on $\calB(\calH)$ via 
$\mathrm{Ad}_\Gamma$. 

The  dynamics of a physical system is modelled by a Hamiltonian $H=H^*$ on $\calH$. In particular, 
$H$ gives rise to a time evolution unitary $U_t = e^{itH}$, which should be 
invariant under the real structure $\Gamma$. The relation $\Gamma U_t \Gamma = U_t$ for all $t$ 
is equivalent to the relation $\Gamma H \Gamma = -H$. We call any 
Hamiltonian satisfying $\Gamma H \Gamma = -H$ a Bogoliubov--de Gennes (BdG) Hamiltonian. 
A BdG Hamiltonian then induces a dynamics on  the anti-symmetric Fock space $\exterior \calV$, 
which gives a mathematical description of the many-body  system and ground state. Because 
the dynamics on $\exterior \calV$ are determined by $H$ on $\calH$, the model is a free fermionic 
approximation of many-body phenomena.

Altland and Zirnbauer described fundamental symmetries of such free fermionic systems~\cite{AZ97}, which are described by unitary or 
anti-unitary operators that commute with the BdG Hamiltonian $H$. A simple mathematical description of these symmetries was then 
noted by Kennedy and Zirnbauer.

\begin{prop}[{\cite[{\S}2]{KZ16}, \cite[{\S}3.1--3.2]{AMZ}}] \label{prop:KZ_symmetries}
Let $H$ be a BdG Hamiltonian on $(\calH,\Gamma)$.
If $H$ has Altland--Zirnbauer symmetries, then there are 
Real mutually anti-commuting skew-adjoint unitaries $\{\kappa_j\}_{j=1}^n \subset \calB(\calH)$,
such that $\kappa_j iH  = -  iH \kappa_j$ for all $j=1,\ldots, n$. The 
 integer  $n \in \{0,\ldots, 7\}$ is determined by the free fermionic symmetry.
\end{prop}

There is considerable mathematical and physical interest in studying the topological properties 
of many-body ground states with a spectral gap property, which we call a gapped ground state,  
see~\cite{Kubota25} for recent developments.
In our free fermionic model with BdG Hamiltonian $H$, the ground state  on $\exterior \calV$  
will be gapped if $0 \notin\spec(H)$. If $H$ is a 
gapped (invertible) BdG Hamiltonian, $iH$ is a Real, skew-adjoint, and invertible operator on $\calH$ that 
anti-commutes with ungraded Clifford generators $\{\kappa_1, \ldots, \kappa_n\}$.
We can therefore describe a topological obstruction between pairs of symmetric and gapped BdG Hamiltonians 
(with the same symmetry) via the $\DK$-valued relative index.

\begin{assump} \label{assump:BdG_assump}
Let $H_0$ and $H_1$ be invertible BdG Hamiltonians on $(\calH, \Gamma)$ with the same free fermionic symmetry. 
We assume that there is a $C^*$-algebra of observables $A \subset \Mult(A) \subset \calB(\calH)$
such that the 
ungraded Clifford generators 
$\{\kappa_1,\ldots, \kappa_n\} \subset \Mult(A)$ and $\mathrm{Ad}_{\Gamma}$ 
gives a well-defined real structure on $\Mult(A)$. Furthermore, we assume that $H_0$ and $H_1$ are $A$-comparable, 
meaning that $J_{iH_0} = iH_0 |iH_0|^{-1}, J_{iH_1} =iH_1 |iH_1|^{-1} \in \Mult(A)$ with 
$\big\| q(J_{iH_0}) - q(J_{iH_1}) \big\|_{\calQ_A} < 2$.
\end{assump}

The following is then immediate by the definition of the relative index.

\begin{prop}[cf. \cite{Scaglione}]
Let $H_0$ and $H_1$ be invertible BdG Hamiltonians on $(\calH, \Gamma)$ satisfying Assumption \ref{assump:BdG_assump}. 
Then 
the relative topological phase
\[
   \Ind(H_0, H_1) =  \relind_{0,n+1}\big( J_{iH_0}, J_{iH_1} \big) \in \DK( A \otimes \Cl_{0,n+1} )
\]
is well-defined.
\end{prop}

If $H_0, H_1 \in \Mult(A)$ satisfy Assumption \ref{assump:BdG_assump}, the straight-line path $iH_t = (1-t) iH_0 + t iH_1$ is a norm-continuous path of skew-adjoint 
Fredholm Hamiltonians that have the same free fermionic symmetry. 
If we further assume that $H_0^2 = H_1^2 = \one$,  $\one - H_t^2 \in A$ for all $t$ and $\| q(H_0) - q(H_1) \|_{\calQ_A} < \sqrt{2}$, 
then 
\[
    \Ind(H_0, H_1) = \Sf_{r,s+1}\big( (1-t)iH_0 + t iH_1 \big)
\]
by Lemma \ref{lemma:sf_small_perturbations}.
If $ \Ind(H_0, H_1)$ is non-trivial, then   the path $H_t \in \Mult(A)$ must fail to be invertible  at one or more points. Thus the $\DK$-spectral 
flow gives a precise   meaning to the physical idea that if $H_0$ and $H_1$ have differing topological phases (as measured by 
the relative index $\Ind(H_0, H_1)$), then the path $H_t$ connecting 
the two gapped systems will have a ``topological gap closing''.

\subsection{Spectral flow and ``topological gap filling'' on systems with defects}

We first consider a discrete half-space system $\calH = \ell^2(\Z^{d-1} \times \N, \C^m)$, where 
$\ell^2(\Z^{d-1} \times \{0\}, \C^m)$ describes the codimension-$1$ boundary.  Under mild assumptions, we can assume that the boundary 
system can be modelled by a $C^*$-algebra $C$, and the half-space Hamiltonian $H \in \calB\big( \ell^2(\Z^{d-1}\times \N, \C^m) \big)$ 
is an  adjointable operator on  the Hilbert 
$C^*$-module $\ell^2(\N, C)_{C}$, where  matrix degrees of freedom are absorbed into the algebra $C$. The 
compact endomorphisms $\End_{C}^0\big( \ell^2(\N, C)) \big) \cong C\otimes \calK\big(\ell^2(\N)\big)$ describe 
operators localised near the   boundary, and the quotient 
$q(H) \in \calQ_{C}\big( \ell^2(\N, C ) \big)$ can often 
be faithfully represented on the boundary-free (bulk) Hilbert space $\ell^2(\Z^d, \C^m)$.

More generally, we model a $d$-dimensional system with a 
 boundary or defect by a $C^*$-algebra $C$ using a countably generated Hilbert $C^*$-module $E_C$, 
 where the defect Hamiltonian $H \in \End_C(E)$. See~\cite{ProdanGpoid, BInterface} for a more comprehensive framework and examples.

We assume that  $E_C$ has a real structure and the defect Hamiltonian ${H} = {H}^* = -H^\rs \in \End_C(E)$ 
is such that  $iH$ anti-commutes with the 
generators $\{\kappa_1,\ldots, \kappa_n\} \subset \End_C(E)$ of an ample and ungraded $\Cl_{0,n}$-representation 
(cf.  Proposition \ref{prop:KZ_symmetries}).
The Hamiltonian of the bulk system without defect is given by $q({H}) \in \calQ_C(E)$, which we assume is 
gapped/invertible. 
We denote by $J_{q(iH)} = q({iH})| q({iH})|^{-1} \in \calQ_C(E)$ the corresponding skew-adjoint unitary in the quotient.
The adaptation of Lemma  \ref{lemma:bdry_cayley_nice} in the appendix to the present setting is the following.

\begin{prop}[{cf. \cite[{\S}5]{AMZ}, \cite[{\S}6]{BKR}}] \label{prop:bulk-boundary-DK-Fredholm}
Let  ${H}_0$ be a basepoint \emph{invertible} Hamiltonian in $\End_C(E)$ with the same free fermionic symmetries as ${H}$. 
Then under the composition 
\begin{align*}
   &\hat{\delta}:\DK(\calQ_C(E) \otimes \Cl_{0,n+1}) \xrightarrow{\delta} \DK( \End_C^0(E) \otimes \Cl_{1,n+1} ) 
   \xrightarrow{\simeq} \KKR( \Cl_{n+1,0}, C), \\
  &\hat{\delta}\big( [J_{q(iH)} \otimes \rho] - [ J_{q(iH_0)} \otimes \rho] \big) = 
  \class{ iH } = \big[ \big( \Cl_{n+1, 0}, \, E_C \otimes \exterior \C, \, \chi\big(i {H} \big) \otimes \rho \big) \big]
\end{align*}
with Clifford generators $\{\kappa_1 \otimes \rho, \ldots, \kappa_n \otimes \rho, \one \otimes \gamma\}$. 
\end{prop}

Note that for ease of reading we have suppressed the Clifford shuffle map, 
\[
[J_{q(iH)} \otimes \rho] - [ J_{q(iH_0)} \otimes \rho] = \big[ \sigma_{J_{q(iH_0)}\otimes \rho}^{0,n}( J_{q(iH)} \otimes \rho) \big] - [ J_{q(iH_0)}\otimes \rho] 
\in \DK(\calQ_C(E) \otimes \Cl_{0,n+1}).
\]
In many cases of interest (such as $E_C = \calH_C$, the standard $C^*$-module), the 
boundary map $\delta: \DK(\calQ_C(E) \otimes \Cl_{0,n+1}) \to \DK( \End_C^0(E) \otimes \Cl_{1,n+1} ) $ 
is an isomorphism.

We now consider a spectral flow interpretation of the class $\class{i{H}} \in \KKR( \Cl_{n+1,0}, C ) $ when 
$C \cong C^*(\Z) \otimes A$ and   $E_C \cong (C^*(\Z) \otimes X)_{C^*(\Z) \otimes A}$, which 
occurs when the system is translation invariant in a direction along the defect.
As a Real algebra, the Fourier transform  \cite{AKR} gives
\[
   C^*(\Z) \otimes A \cong C(i \mathbb{T}, A), \qquad 
   C(i \mathbb{T}, A)^\rs = \big\{ f \in C(\mathbb{T}, A) \, \big|\, f^{\rs}(k) = f(-k)^{\rs_A} \big\}
\]
and we consider $E_C \cong C( i \mathbb{T}, X)_{C(i\mathbb{T}, A)}$ with defect Hamiltonian 
$H \cong \{H_k\}_{k \in \mathbb{T}} \in \End_C(E) \cong \End_{C( i \mathbb{T}, A)}\big( C( i \mathbb{T}, X) \big)$. 
If $q(H) \in \calQ_{C( i \mathbb{T}, A)}\big( C( i \mathbb{T}, X) \big)$ is invertible and 
 $\mathbb{T} \ni k\mapsto H_k \in \End_A(X)$ is Wahl-continuous with invertible endpoints, 
  we can consider the spectral flow 
 of the loop of skew-adjoint Fredholm operators 
$\{iH_k\}_{k \in \mathbb{T}} \subset  \End_A^{0,n}( X )$.

We can write the $\KK$-class from Proposition \ref{prop:bulk-boundary-DK-Fredholm} 
as
\[
    \class{iH_\bullet} 
  = \Big[ \Big( \Cl_{n+1, 0}, \, C\big( i \mathbb{T} ,  X_A \big)_{C( i \mathbb{T}, A) }   \otimes \exterior \C, \, \chi\big(i {H_\bullet} \big) \otimes \rho \Big) \Big] \in \KKR( \Cl_{n+1, 0}, C(i \mathbb{T}, A) ).
\]
By identifying $\mathbb{T}$ with the periodic interval $[-\pi,\pi]$, we can consider the path $\R \ni t \mapsto \tilde{H}_t$ on $\R$ by taking 
$\tilde{H}_t = H_{-\pi} = H_\pi$ for all $t \in \R \setminus (-\pi,\pi)$. However, the identity $(i\tilde{H}_t)^\rs = i\tilde{H}_{-t}$ implies 
that  $\tilde{H}_\bullet $ is a Real  operator on the \emph{dual} suspension, $S^{0,1}X_{S^{0,1}A}$, 
where $S^{0,1}Y = C_0(\R, Y)$ as a complex space with real structure $f^{\rs_{0,1}}(t) = f(-t)^{\rs_Y}$.
So we have $\class{i\tilde{H}_\bullet} \in \KKR( \Cl_{n+1, 0}, S^{0,1}A)$ and instead define the \emph{dual} spectral flow 
\[
  {\Sf}^\text{dual}_{2,n+1} ( iH_\bullet) := \beta_{DK} \big( \Index_{1,n+1} ( i\tilde{H}_\bullet  ) \big)
  \in \DK( S^{1,1} A \otimes \Cl_{2,n+1} ) \cong \DK( A \otimes \Cl_{2,n+1} ),
\]
where we have used Theorem \ref{thm:DK_Bott_isos}.

\begin{cor} \label{cor:gap_filling_real}
Let  $H_\bullet, H_\bullet^\mathrm{triv} \in C( i\mathbb{T}, \End_A(X) )$  be 
Fredholm Hamiltonians  of the same symmetry type such that $H_\bullet^\mathrm{triv}$  is invertible. 
If ${\Sf}^\text{dual}_{2,n+1} ( iH_\bullet) \in \DK( A \otimes \Cl_{2,n+1} )$ is non-trivial, then 
the bulk $K$-theory class $[J_{q(iH)} \otimes \rho] - [ J_{q(iH^\mathrm{triv})}\otimes \rho] \in  \DK\big( \calQ_{C(i \mathbb{T}, A)}(C( i \mathbb{T}, X)) \otimes \Cl_{0,n+1}\big)$ 
is non-trivial.
\end{cor}
For ease of reading we have suppressed the Clifford shuffle map, 
\[
[J_{q(iH)} \otimes \rho] - [ J_{q(iH^\mathrm{triv})} \otimes \rho] = \big[ \sigma_{J_{q(iH^\mathrm{triv})}\otimes \rho}^{0,n}( J_{q(iH)} \otimes \rho) \big] - [ J_{q(iH^\mathrm{triv})}\otimes \rho].
\]
\begin{proof}
Triviality of $[J_{q(H)} \otimes \rho] - [ J_{q(H^\mathrm{triv})} \otimes \rho]$ implies that $\class{iH_\bullet}$ and 
therefore $\class{i\tilde{H}_\bullet}$ is trivial by Proposition  \ref{prop:bulk-boundary-DK-Fredholm}. The contrapositive then gives the result.
\end{proof}

In many cases of interest, the map 
\[
[J_{q(iH)} \otimes \rho] - [ J_{q(iH^\mathrm{triv})})\otimes \rho] \mapsto \hat{\delta}\big( [J_{q(iH)} \otimes \rho] - [ J_{q(iH^\mathrm{triv})})\otimes \rho] \big) = \class{ i H_\bullet}
\]
implements an isomorphism of groups. In such a setting with $0 \notin \spec(H_k)$ for all $k \in \mathbb{T}$, 
$\class{i H_\bullet}$, ${\Sf}^\text{dual}_{2,n+1} ( iH_\bullet)$  and the bulk $\DK$-class will be trivial. Put another way, 
a non-trivial bulk class implies that the defect Hamiltonian  
must have a  ``topological gap filling'' in its spectrum for $\class{i H_\bullet}$ to be non-trivial.
Such a property is an example of the $K$-theoretic 
bulk-defect correspondence of  topological materials with Altland--Zirnbauer symmetries.

\begin{example}[Complex algebras and spectral flow]
Let us briefly consider the complex setting, where $S^{1,0}A \cong S^{0,1}A \cong SA$ and we do not need to distinguish 
between the dual suspension and spectral flow. 
We take a defect Hamiltonian $H_\bullet \in \End_{C(\mathbb{T}, A)}\big( C(\mathbb{T}, X) \big)$ with $X_A$ a countably 
generated $C^*$-module. 
If $d$ is even and $q(H_\bullet)$ is invertible, the 
  projection $P_\text{bulk} = P_{(-\infty,0)} (q(H_\bullet) )$ defines the 
bulk $K$-theory class $\big[P_\text{bulk} \big] \in K_0\big( C(\mathbb{T}, \calQ_A(X) ) \big)$,  where 
\[
   \hat{\delta}([P_\text{bulk} ]) = - \big[ H_\bullet \big] =  -\Big[ \Big( \Cl_{0,1}, \, C\big( \mathbb{T}, X_A \big)_{C(\mathbb{T}, A)} \otimes \exterior \C, \, 
    \chi\big( H_\bullet \big) \otimes \gamma \Big) \Big] \in  \KK(\Cl_{0,1}, C(\mathbb{T}, A) ) 
\]
and $\hat{\delta} : K_0\big( C( \mathbb{T}, \calQ_A(X) )\big) \xrightarrow{\delta} K_{-1}\big( C( \mathbb{T}, \End_A^0(X) ) \big) 
\xrightarrow{\simeq} \KK (\Cl_{0,1},  C(\mathbb{T}, A) )$. 
We assume that $k\mapsto H_k$ is Wahl-continuous and extend 
$H_\bullet$ on $C(\mathbb{T}, \End_A(X))$ to an operator $\tilde{H}_\bullet$ on $SX_{SA}$. As 
a path of self-adjoint Fredholm operators, the complex spectral flow
$\Sf( \tilde{H}_\bullet ) \in \DK(A \otimes \Cl_{0+1} ) \cong K_0(A)$ is well-defined (cf. \S\ref{subsec:Sf_complex}).
If $0 \notin \spec(H_k)$ for all $k \in \mathbb{T}$, then $[H_\bullet]$ and $\Sf(\tilde{H}_\bullet)$ are trivial. If $\End_A^0(X)$ is a stable $C^*$-algebra, 
$\hat{\delta}$ is an isomorphism and the bulk $K$-theory class $[P_\text{bulk} ] \in K_0\big( C(\mathbb{T}, \calQ_A(X) ) \big)$ is also trivial. 
So we also have a $K$-theoretic bulk-defect correspondence in the complex setting.
\end{example}

We have worked with discrete systems in this section, though expect 
analogous results to also hold for continuous models following the framework of~\cite{KelStoiber}.

\appendix

\section{Problems with extending the ABS isomorphism}   \label{sec:ABS}

The Atiyah--Bott--Shapiro (ABS) isomorphism $\calM_{r,s}/\calM_{r,s+1} \cong \KO_{s+1-r}(\R)$ provides an elegant description of 
the real $K$-theory of a point~\cite{ABS}, where $\calM_{r,s}$ denotes the Grothendieck group of 
the semigroup of unitary equivalence classes of  (ungraded) finite-dimensional representations of $\Cl_{r,s}$. 
The ABS isomorphism also gives a natural generalisation of the analytic Fredholm index 
to skew-adjoint Fredholm operators on a Hilbert space that anti-commute with the generators of a $\Cl_{r,s}$-representation~\cite{AS69}.

The isomorphism $\KKR(\C, A) \cong \KO_0(A^\rs)$ for ungraded Real $C^*$-algebras $A$ gives a description of real $K$-theory in degree $0$ via 
finitely generated and projective $A$-modules~\cite[{\S}6, Theorem 3]{Kasparov80}. 
For higher degrees of real $K$-theory,  it is tempting to try and extend the Atiyah--Bott--Shapiro approach on Hilbert spaces to give 
a presentation of $\KO_{s+1-r}(A^\rs)$ via finite projective $A$-modules with a Clifford action. 
In this short section we explain that the naive approach towards such a generalisation unfortunately fails. 

Suppose that $A$ is a unital Real $C^*$-algebra.  Let $E$ and $F$ be finite
projective right $A$-modules with adjointable left actions
$\varphi_E \colon \Cl_{r,s}\to \End_A(E)$,
$\varphi_F \colon \Cl_{r,s}\to \End_A(F)$ of an ungraded Clifford algebra $\Cl_{r,s}$. 
We say that $E$ and $F$ are equivalent if the modules
$\overline{\varphi_E(1_{\Cl_{r,s}})E}$ and $\overline{\varphi_F(1_{\Cl_{r,s}})F}$ 
are $\Cl_{r,s}$-equivariantly unitarily equivalent.

\begin{defn*}
\label{defn:grotty-group}
For a unital Real $C^*$-algebra $A$, 
we define $\calM_{r,s}(A^\rs)$ to be the Grothendieck group of the 
abelian semigroup of $\Cl_{r,s}$-equivariant unitary equivalence
classes of (ungraded) finite projective Real right $A$-modules with a left (ungraded)
$\Cl_{r,s}$-action. 

If $A$ is a nonunital algebra with minimal unitization $A^\sim$, we define 
\[
\widetilde{\calM}_{r,s}(A^\rs) := \Ker \big(q:\calM_{r,s}((A^\rs)^\sim)\to\calM_{r,s}(\R) \big),
\] 
where $q([E])=[E/E\cdot A]$.
\end{defn*}

Observe that, for $A^\rs=\R$, we obtain $\widetilde{\calM}_{r,s}(\R) = \calM_{r,s}(\R) = \calM_{r,s}$. 
In analogy with the ABS isomorphism $\calM_{r,s}/\calM_{r,s+1} \cong \KO_{s+1-r}(\R)$, we can then define 
\[
\widetilde{\KO}_{s+1-r}(A^\rs) := \widetilde{\calM}_{r,s}(A^\rs)/ \widetilde{\calM}_{r,s+1}(A^\rs).
\]
Unfortunately, this definition of $\widetilde{\KO}_{s+1-r}(A^\rs)$ does not agree with ${\KO}_{s+1-r}(A^\rs)$ for arbitrary Real $C^*$-algebras. 

\begin{example*}
Let $r=0$ and $s=2$. 
Then, for two anti-commuting skew-adjoint Clifford generators $f_1,f_2$, we immediately obtain a third $f_3=f_1f_2$. 
So for every $C^*$-algebra $A$ and every $A$-module $X$ carrying a $\Cl_{0,2}$ representation, $X$ also carries a $\Cl_{0,3}$ representation. 
Hence 
\[
\widetilde{\KO}_{3}(A^\rs) = \widetilde{\calM}_{0,2}(A^\rs)/ \widetilde{\calM}_{0,3}(A^\rs) = 0 \quad\text{for every Real $C^*$-algebra $A$.}
\]
But of course, there exist Real $C^*$-algebras $A$ (e.g., $A=C_0(\R)$) with $\KO_3(A^\rs) \neq 0$. 
\end{example*}

This example shows that, while the ABS picture of real $K$-theory is valid for the scalar algebra $\R$ (i.e., for the $K$-theory of a point), it is not valid in general.
As far as the authors know, the only way to salvage the above approach is by bringing in additional information 
in the form of an operator $F \in \Mult(A \otimes \calK)$ such that $q(F) \in \calQ_{A\otimes \calK}$ is invertible and 
anti-commutes with the $\Cl_{r,s}$-generators, as in \cite[Appendices]{HR}. 
But this puts us in the realm of Kasparov's $\KK$-theory, and 
so we do not obtain a `simpler' picture of real $K$-theory via the ABS approach.

\section{Kasparov theory with real structures} \label{sec:RealKK}

We give a brief overview of Kasparov's $\KKR$-theory and $\KKR$-groups for $C^*$-algebras and $C^*$-modules 
with a real structure. Further details can be found in~\cite{Blackadar,Kasparov80}.

Let $(B, \frakr_B)$ be a Real $C^*$-algebra. 
A (possibly $\Z_2$-graded) complex Hilbert $C^*$-module $Y_B$ is called \emph{Real}, if there is an anti-linear map $\mathfrak{r}_Y:Y_B\to Y_B$,
called the \emph{real structure}, 
such that 
\begin{align*}
(y^{\rs_Y})^{\rs_Y} &= y , & 
y^{\mathfrak{r}_Y} \cdot b^{\mathfrak{r}_B} &= (y \cdot b)^{\mathfrak{r}_Y} , & 
& \text{and} & 
( y_1^{\mathfrak{r}_Y} \mid y_2^{\mathfrak{r}_Y} )_B &= \big(( y_1 \mid y_2 )_B\big)^{\mathfrak{r}_B} .
\end{align*}
The real structure on the $C^*$-module induces 
a real structure $\mathfrak{r}$ on $\End_B(Y)$ 
via $S^\mathfrak{r} y = \big( S(y^{\mathfrak{r}_Y} ) \big)^{\mathfrak{r}_Y}$. 
Representations of Real algebras 
$\pi: A \to \End_B(Y)$ should be compatible with this real 
structure,  $\pi(a^{\mathfrak{r}_A}) = \pi(a)^\mathfrak{r}$ for all $a \in A$.
When it is clear on which space/algebra a real structure is acting, we will often omit the subscript and write $y^{\rs}$ instead of $y^{\rs_Y}$.

\begin{defn}
\label{defn:real-kasmod}
Let $A$ and $B$ be $\Z_2$-graded Real $C^*$-algebras.
A Real  Kasparov module $(A, {}_\pi{Y}_B, F)$ 
consists of
\begin{enumerate}
\item a Real and
$\Z_2$-graded $C^*$-module ${Y}_B$, 
\item a Real and $\Z_2$-graded $*$-homomorphism $\pi:A \to \End_B(Y)$, 
\item an odd element $F = F^* = F^\rs \in \End_B(Y)$ such that, for all $a \in A$,
\begin{equation}
    [\pi(a), F]_\pm \in \End_B^0(Y) \quad \text{and} \quad  \pi(a)(\one - F^2)  \in \End_B^0(Y).
\label{eq:compact-zero}
\end{equation}
Here the graded commutator is given for $a \in A$ with homogeneous degree $\mathrm{deg}(a) \in \{0,1\}$ by $[\pi(a), F]_\pm = \pi(a) F - (-1)^{\mathrm{deg}(a)} F \pi(a)$. 
\end{enumerate}
If all the operators in \eqref{eq:compact-zero} are zero, then the Kasparov module is called degenerate.
\end{defn}
We will often omit the representation $\pi:A\to \End_B(Y)$ if the context is clear. 

Two Real Kasparov modules $(A, {}_{\pi_0}Y^{(0)}_B, F_0)$ and $(A, {}_{\pi_1}Y^{(1)}_B, F_1)$ are unitarily equivalent if there is a Real 
even unitary  $U: Y^{(0)}_B \to Y^{(1)}_B$ 
such that  $UF_0 U^* = F_1$ and $U\pi_0(a) U^* = \pi_1(a)$ for all $a\in A$. 
Two Real Kasparov modules are homotopic if there is a Real  Kasparov module 
$(A, \tilde{Y}_{B \otimes C([0,1])}, F)$ such that the evaluation at $0$ and $1$ yields Real Kasparov modules 
that are unitarily equivalent to $(A, Y^{(0)}_B, F_0)$ and $(A, Y^{(1)}_B, F_1)$ respectively.
Homotopy classes  of Real Kasparov modules yields an abelian group, $\KKR(A,B)$, 
where the  group operation is by direct sum and the zero class is represented by degenerate Kasparov modules~\cite{Blackadar,Kasparov80}.
Note that  $\KKR(A, B)$ will depend on the choice of real structure for $A$ and $B$.

Let $(A, Y_B, F)$ be a Real  Kasparov module. 
If we ignore the real structures, we obtain a complex  Kasparov module. 
If we restrict the Real $C^*$-module $Y_B$ to the 
elements fixed under $\rs$, we obtain a real (small r) $C^*$-module $Y_{B^{\rs}}^{\rs}$. 
Similarly, the Real left action of $A$ becomes a real left action 
$\pi: A^{\rs} \to \End_{B^{\rs}}( Y^{\rs} )$. 
We do not lose any information by restricting Real Kasparov modules to real $C^*$-modules 
and algebras. Indeed, 
real Kasparov modules can be complexified to obtain Real Kasparov modules and 
so, for fixed real structures on $A$ and $B$, $\KKR(A, B) \cong K\KO(A^{\rs}, B^{\rs})$. 
If the algebra $B$ is trivially graded, we can also 
consider real $K$-theory, where $\KKR(\Cl_{r,s}, B) \cong \KKO(Cl_{r,s}, B^{\rs}) \cong \KO_{r-s}(B^{\rs})$.

We will also consider unbounded representatives of Kasparov modules, which are constructed from 
 a  densely-defined, closed, and right $A$-linear operator $D: \Dom(D) \subset Y_B \to Y_B$. 
We say $D$ is Real and write $D^\rs = D$ if $(\Dom (D))^{\rs_Y} = \Dom (D)$ and 
$(Dy^{\rs_Y})^{\rs_Y} = Dy$ for all $y \in \Dom (D)$.
We say that $D$ 
is \emph{regular} if $D^*$ is densely-defined and the operator $\one +D^*D:\Dom (D^*D) \to Y_B$ has dense range.  
Self-adjoint and regular operators admit a continuous functional calculus. 
If $D$ is regular, $\one +D^*D$ has a bounded and positive inverse in $\End_B(Y)$ with 
$\Ran\big( (\one +D^*D)^{-1} \big) \subset \Dom(D)$ and dense in $Y_B$~\cite[Lemma 9.2]{Lance}.

\begin{defn}
Let $A$ and $B$ be $\Z_2$-graded Real $C^*$-algebras.
An unbounded Real Kasparov module $(\calA, {}_\pi{Y}_B, D)$ 
consists of
\begin{enumerate}
\item a Real and $\Z_2$-graded $C^*$-module ${Y}_B$, 
\item a Real and $\Z_2$-graded $*$-homomorphism $\pi:A \to \End_B(Y)$, 
\item a (densely-defined) odd regular self-adjoint operator $D=D^\rs$ and a dense $*$-subalgebra 
$\calA\subset A$ such that, for all $a\in \calA\subset A$, 
it holds that $\pi(a)\Dom(D)\subset\Dom(D)$ and we have
\begin{align*}
  & [D,\pi(a)]_\pm \,\in\, \End_B(Y)\quad\mbox{and}\quad \pi(a)(1+D^2)^{-1/2}\,\in\, \End_B^0(Y).
\end{align*}
\end{enumerate}
\end{defn}

If $(\calA, {}_\pi{Y}_B, D)$  is an unbounded Real Kasparov module, then 
$(A, X_B, D(1+D^2)^{-1/2})$ is a Real Kasparov module~\cite{BJ}.  
In {\S}\ref{sec:Fredholm_Real} we consider more general methods to obtain Kasparov modules 
from bounded or unbounded Fredholm operators via a normalising function.

\paragraph{Kasparov products}
The internal and external Kasparov products are bilinear and functorial maps
\begin{align*}
  \KKR(A,B)\times \KKR(B,C) &\to \KKR(A,C), & (x,y) &\mapsto x \otimes_B y , \\
  \KKR(A,B)\times \KKR(C,D) &\to \KKR(A\hox C,B\hox D), & (w,z) &\mapsto w \otimes z .
\end{align*}
The group $\KKR(A,A)$ becomes a ring under the internal Kasparov product, with the identity element given by $\Id_A=[(A,A_A,0)]$. 
The external Kasparov product with $\Id_A$ yields a map 
\[
\tau_A \colon \KKR(B,C) \to \KKR(B\hox A,C\hox A) , 
\qquad 
x \mapsto x \otimes \Id_A .
\]
This map allows us to define a more general pairing, referred to as \emph{the internal product over $C$}: 
\begin{align*}
\KKR(A_1,B_1\hox C) \times \KKR(C\hox A_2,B_2) &\to \KKR(A_1\hox A_2,B_1\hox B_2) , \\
(x,y) &\mapsto x \otimes_C y := ( x \otimes \Id_{A_2} ) \otimes_{B_1\hox C\hox A_2} ( {\Id_{B_1}} \otimes y ) .
\end{align*}

\paragraph{Stability and Morita invariance}
Let $\calK$ denote the algebra of compact operators on a Real separable Hilbert space. 
We recall that Real $\KK$-theory is \emph{stable}:
\[
\KKR(A,B) \cong \KKR(A \hox \calK,B) \cong \KKR(A,B \hox \calK) .
\]
In the special case of the  finite-dimensional Real vector space 
$\exterior\C^n$, where  $\Cl_{n,n} \cong \End\big( \exterior\C^n \big)$, we use the special notation 
\begin{equation}
\label{eq:Clifford_Morita_KK}
\calM_{n,n}^{\KK} \colon \KKR(A,B) \xrightarrow{\simeq} \KKR(A,B\hox\Cl_{n,n})  ,
\end{equation}
which we will sometimes refer to as \emph{Clifford stability}. 
Since $\Id_{\Cl_{r,s}} \hat{\otimes} \Id_{\Cl_{s,r}} = \Id_{\Cl_{r+s,r+s}}$, it follows from Clifford stability that also 
\[
\tau_{\Cl_{r,s}} \colon \KKR(A,B) \xrightarrow{\simeq} \KKR(A\hox\Cl_{r,s},B\hox\Cl_{r,s}) , 
\qquad 
x \mapsto x \otimes \Id_{\Cl_{r,s}}
\]
is an isomorphism. 
Furthermore, for any $r,s\in\N$, we may define the \emph{Clifford shuffle isomorphisms} 
$\shuffle^{r,s}_{\KK} \colon \KKR( A \hox \Cl_{s,r}, B)  \xrightarrow{\simeq} \KKR( A, B \hox \Cl_{r,s} )$ 
via the composition 
\begin{multline}   \label{eq:Clifford_shuffle}
  \shuffle^{r,s}_{\KK} \colon 
  \KKR( A \hox \Cl_{s,r}, B) \xrightarrow{ \tau_{\Cl_{r,s}} } 
  \KKR( A \hox \Cl_{s,r}  \hox \Cl_{r,s}, B \hox \Cl_{r,s} ) \\
    \xrightarrow{\simeq} \KKR( A \hox \Cl_{r+s,r+s}, B \hox \Cl_{r,s} ) \xrightarrow{\simeq} %\xrightarrow{ \calM_{r+s,r+s}^{\KK} } 
    \KKR( A, B \hox \Cl_{r,s} ) .
\end{multline}
The maps $\shuffle^{r,s}_{\KK}$ can be computed explicitly on cycles, see Lemma \ref{lem:cliff-shuffle}.

Now let $Y_B$ be a full right $B$-module. 
Then the element 
\[
\big[ \big( \End^0_B(Y), \, Y_B, \, 0 \big) \big] \in \KKR\big( \End^0_B(Y) , B \big) ,
\]
is a \emph{$\KK$-equivalence} (i.e., it is invertible under the internal Kasparov product). 
The \emph{Morita invariance isomorphism} $\calM_Y^{\KK}$ is implemented by taking the Kasparov product with the above element:
\begin{equation} \label{eq:Morita_invariance}
\calM_Y^{\KK} \colon \KKR(A, \End_B^0(Y)) \xrightarrow{\simeq} \KKR(A,B) , 
\quad 
\calM_Y^{\KK} (x)  := x \otimes_{\End_B^0(Y)} \big[ \big( \End^0_B(Y), Y_B, 0 \big) \big] .
\end{equation}

\paragraph{Bott periodicity}
The Bott periodicity isomorphism in Real $\KK$-theory can also be implemented by taking Kasparov products. 
\begin{prop}[{\cite[{\S}5]{Kasparov80}}] \label{prop:KK_Bott}
The unbounded Kasparov modules
\begin{align*}
    \beta &= \Big( \C, \, \big(C_0(\R) \otimes \Cl_{1,0}\big)_{C_0(\R) \otimes \Cl_{1,0}} , \, x \otimes \gamma \Big), \quad 
  (x f)(x) = xf(x), \quad [\beta] \in \KKR(\C, C_0(\R)\otimes \Cl_{1,0}) \nonumber \\
       \slashed{\partial} &= \Big( C^1_0(\R) \otimes \Cl_{1,0}, \, L^2(\R) \otimes \exterior{\C}, \, \partial_x \otimes \rho \Big), \quad 
   [\slashed{\partial}] \in \KKR(C_0(\R) \otimes \Cl_{1,0}, \C),  
\end{align*}
are mutually inverse $\KK$-equivalences. We therefore have group isomorphisms 
\begin{align*}
  &\beta_{\KK}: \KKR(A, B) \to \KKR( A, C_0(\R) \hox B \hox \Cl_{1,0}),  &&\beta_{\KK}(\cdot) = (\cdot) \, \hat\otimes \, [\beta], \\
  &\beta_{\KK}^{-1}: \KKR(A , C_0(\R) \hox B \hox \Cl_{1,0}) \to \KKR( A , B),  &&\beta_{\KK}^{-1}(\cdot) = (\cdot) \, \hat\otimes_{C_0(\R)\hox\Cl_{1,0}} \,  [\slashed{\partial}].
\end{align*}
The map $\beta_{\KK}$ is called the \emph{Bott isomorphism}. 
\end{prop}

We will also find the following alternative version of Bott periodicity useful. 

\begin{cor} \label{cor:KK_Bott_alternative}
We have group isomorphisms 
\begin{align*}
\altBott_{\KK} &\colon \KKR(A\hox \Cl_{1,0}, B) \to \KKR( A, C_0(\R) \hox B ) , \\
\altBott_{\KK}^{-1} &\colon \KKR(A , C_0(\R) \hox B) \to \KKR( A \hox \Cl_{1,0}, B) 
\end{align*}
satisfying 
\[
\altBott_{\KK}^{-1}(\cdot) 
= \beta_{\KK}^{-1} \circ \tau_{\Cl_{1,0}} (\cdot) 
= (\cdot) \otimes_{C_0(\R)} [\slashed{\partial}] .
\]
\end{cor}

\section{Van Daele \texorpdfstring{$K$}{K}-theory} \label{sec:appendix_DK}

We review Van Daele $K$-theory as it was 
first considered in~\cite{vanDaele1, vanDaele2} and 
then further developed in~\cite{BKR,  Kellendonk15, Kellendonk19, Kubota15a, JosephMeyer, Roe04}.

\subsection{Basic definition and properties}

\begin{defn}
Let $B$ be a complex $C^*$-algebra. We say that $B$ has a balanced $\Z_2$-grading 
if $B$ contains an odd self-adjoint unitary (OSU). That is, there is an odd element $e$ satisfying $e=e^* = e^{-1}$ 
(in particular, $B$ must be unital). 
If $B$ has a real structure $\rs$, we also require $e^{\rs} = e$.
\end{defn}

If $B$ is unital and not balanced graded, we replace $B$ by $B \hox \Cl_{1,1}$, which has the OSU $\one \otimes \gamma$.

Extending the grading and real structure of $B$ to $M_n(B)$ entrywise, we 
let $V(B)$ denote the disjoint union $\bigsqcup_k\pi_0\big(\mbox{\rm OSU}(M_k(B))\big)$ 
with $\mbox{\rm OSU}(M_k(B))$ the set of odd self-adjoint unitaries in $M_k(B)$. 
Then $V(B)$  is an abelian semigroup under direct summation,
$[x]+[y] = [x\oplus y]$, and we denote by $GV(B)$ the Grothendieck completion. 
The semigroup homomorphism $d:V(B)\to \N$ taking the 
value $k$ on $M_k(B)$ induces a group
homomorphism $d:GV(B)\to\Z$.

\begin{defn} 
\label{defn:vD-K}
For a balanced $C^*$-algebra $B$, we define
the Van Daele $K$-theory group $\DK(B):=\Ker(d:GV(B)\to\Z)$. 

If $B$ is not unital then we set
$\DK(B)=\Ker(q_*:\DK(B^\sim)\to \DK(\C))$ where $q:B^\sim\to\C$ quotients the minimal unitisation $B^\sim$ by the ideal $B$.
\end{defn}

We denote elements of $\DK(A)$ as formal differences of odd self-adjoint unitaries, $[x]-[y]$. 
It is easy to see from the definition that  $\DK(B) \cong DK\big( M_n(B) \big)$ for any $n \geq 1$.  We 
will review the more general Morita invariance of Van Daele $K$-theory in {\S}\ref{sec:DK_Morita}.

\paragraph{Clifford stability}

For a balanced graded algebra $B$, we could also consider the semigroup $V(B \hox \Cl_{1,1})$. 
The following shows that this leads to a consistent definition and gives a basic Clifford stability of $\DK$-theory.

\begin{lemma}[{\cite[Lemma 2.3]{BKR}}] \label{lemma:DK_Clifford_stability}
Let $B$ be balanced graded.
The map 
\begin{equation*} \label{eq:DK_clifford_stability}
  \DK(B) \ni [ x] - [y] \xmapsto{\calM_{1,1}^{\DK}} \big[ \tfrac{1}{2}( x+y) \hox \one + \tfrac{1}{2}(x-y) \hox \gamma \rho ] - [ \one \hox \gamma ] 
     \in \DK(B \hox \Cl_{1,1})
\end{equation*}
furnishes a natural  isomorphism $\calM_{1,1}^{\DK}: \DK(B) \xrightarrow{\simeq} \DK( B \hox \Cl_{1,1})$. 
The $n$-fold composition gives a natural isomorphism $\calM_{n,n}^{\DK}: \DK(B) \xrightarrow{\simeq} \DK(B \hox \Cl_{n,n} )$.
\end{lemma}

We remark that the isomorphism from $\psi_e: B\hox \Cl_{1,1} \xrightarrow{\simeq} M_2(B)$ with entrywise grading on $M_2(B)$ requires a choice of OSU in $e \in B$, 
where
  \begin{equation}\label{eq-psi}
\psi_e(x\hox 1_{1,1}) = \begin{pmatrix} x & 0 \\ 0 & (-1)^{\mathrm{deg}(x)}exe \end{pmatrix} ,\quad
\psi_e(\one \hox \gamma) = \begin{pmatrix} 0 & e \\ e & 0 \end{pmatrix} ,\quad
\psi_e(\one \hox \rho) = \begin{pmatrix} 0 & e \\ -e & 0 \end{pmatrix} .
\end{equation}
Identifying  $\gamma \cong \sigma_1$ and $\rho \cong -i\sigma_2$, we can therefore  write
\begin{equation} \label{eq:M_clifford11_notation}
   \calM_{1,1}^{\DK}( [x] - [y] ) = \big[ \tfrac{1}{2}( x+y) \hox \one + \tfrac{1}{2}(x-y) \hox \gamma \rho ] - [ \one \hox \gamma ] =
    \left[\begin{pmatrix} x & 0\\ 0 & y\end{pmatrix}\right]-\left[\begin{pmatrix} 0 & \one \\ \one & 0\end{pmatrix}\right] 
   \in \DK(B \hox \Cl_{1,1}).
\end{equation}

\paragraph{Relative \texorpdfstring{$\DK$}{DK}-theory}
It is often more natural to 
consider larger unitisations of non-unital algebras and work with relative $K$-theory. 
For  a balanced graded algebra $B$ with a closed, two-sided, and graded ideal $I$ we define
the relative Van Daele group, using homotopy classes in $\mbox{\rm OSU}(M_n(B))$, by
\[
\DK(B,\, B/I):=\{[x]-[y]:\,\, x,y\in \mbox{\rm OSU}(M_n(B)),\ \
x-y\in M_n(I)\}.
\]
By \cite[Proposition 2.4]{BKR}, there is an excision isomorphism 
\begin{equation} \label{eq:DK_excision}
\exc \colon \DK(B, B/I) \xrightarrow{\simeq} \DK(I) .
\end{equation}
We also note the following that we use frequently.

\begin{lemma}
\label{lem:relative_DK}
Let $I$ be an ideal in a balanced graded $C^*$-algebra $B$, and $x,y\in M_n(B)$ odd self-adjoint unitaries with $\Vert q(x)-q(y)\Vert_{B/I}<2$ (where $q\colon B \to B/I$ denotes the quotient map). 
Then there is a well-defined relative class $[x]-[y]\in \DK(B,B/I)$.
\end{lemma}
\begin{proof}
We ignore matrices for the proof.
Analogous to the proof of Lemma  \ref{lemma:close_J_are_unitary_equiv}, we take a path of self-adjoint odd operators 
$[0,1] \ni t \mapsto Z_t = q(x) + \frac{t}{2}(q(y) -q(x) ) \in B / I$, which is such that 
\[
  \big\| \one - Z_t^2 \big\|_{B/I} =  \big| \tfrac{t}{2}( 1 - \tfrac{t}{2} ) \big|  \, \big\| (q(x) - q(y))^2 \big\|_{B/I} < 1
\]
and so is invertible for all $t \in [0,1]$. We further note that $Z_0 = q(x)$  and $Z_1 =\frac{1}{2}(q(x)+q(y)) \in B / I$ is such 
that 
\[
  Z_1 q(x)  Z_1^{-1} =  q(y),  \qquad [Z_t^2, q(x)] = [Z_t^2, q(y) ] = 0.
\]
Therefore there is a  path of OSUs $[0,1]\ni t \mapsto w_t q(x) w_t^* \in B/I$ with $w_t = \mathrm{sgn}(Z_t)$ and 
such that $w_0 q(x) w_0^* = q(x)$ and $w_1 q(x) w_1^* = q(y)$. 

We now let $v_t = w_t q(x)$, which is a path of even unitaries from $v_0 = q(x)^2 = \one$ to $v_1 =w_1 q(x)$. 
Hence $v_1 \in B/ I$ is in the connected component of $\one$ in $B/ I$ and so lifts to an even unitary $\tilde{v} \in B$ 
that is homotopic to $\one \in B$ via a path of even unitaries in $B$. 
We let $\tilde{y} = \tilde{v}^* y \tilde{v}$, which is homotopic to $y$ by a path of OSUs in $B$. We further note that 
\[
   q( \tilde{y} ) = v_1^* q(y) v_1 = ( w_1 q(x) )^* q(y) w_1 q(x) = q(x) w_1^* q(y) w_1 q(x) = q(x)
\]
and so $\tilde{y} - x \in I$. Hence we have a well-defined class $[\tilde{y} ] - [x]  = [y] - [x] \in \DK(B, B/I)$ 
as $\tilde{y}$ is homotopic to $y$.
\end{proof}

Analogous to  Lemma \ref{lemma:homotopic_to_lifting}, we also have lifting results for odd self-adjoint unitaries.

\begin{lemma}[{\cite[Proposition 4.7]{vanDaele1}}]  \label{lemma:OSU_lifting}
Let $x_0, x_1 \in \mbox{\rm OSU}(B/ I)$  and 
suppose that $x_0 \in A/ I$ can be lifted to an odd self-adjoint unitary $\wt{x}_0 \in \mbox{\rm OSU}(B)$. 
\begin{enumerate}
  \item If $\|x_0 - x_1\|_{A/I} < 2$, then $x_1$ can also be lifted to an odd self-adjoint unitary $\wt{x}_1 \in \mbox{\rm OSU}(B)$.
  \item If $x_0 \sim_h x_1$ in $ \mbox{\rm OSU}(B/I)$, then there is a path $[0,1] \ni t \mapsto \wt{x}_t \in \mbox{\rm OSU}(B)$ 
  such that $q(\wt{x}_t) = x_t$.
\end{enumerate}
\end{lemma}

\paragraph{Basepointed \texorpdfstring{$\DK$}{DK}-theory}
We will occasionally consider Van Daele $K$-theory relative to a choice of basepoint~\cite{vanDaele1}.
For a balanced graded algebra $B$ and $e \in B$ an odd self-adjoint unitary, 
we let $V_e(B) =  \bigcup_{k}\pi_0\big(\mbox{\rm OSU}(M_{k}(B))\big)$ where we embed 
$M_k(B)$ into $M_{k+1}(B)$ via $x\mapsto x\oplus e$. 
Van Daele's $K$-theory group with a basepoint is defined as the Grothendieck
group $\DK_e(B) = GV_e(B)$. If $e$ is homotopic in ${\rm OSU}(B)$ to $-e$, then 
$V_e(B)$ is already an abelian group with $-[x] = [-exe]$. So in this setting 
(which can always be achieved by passing to matrices), the Grothendieck completion is unnecessary.

The group $\DK_e(B)$ does not depend on the choice of $e$ up to isomorphism~\cite[Proposition 2.12]{vanDaele1}.  
Furthermore, for any choice of basepoint $e$, $\DK_e(B)\cong \DK(B)$~\cite[{\S}2.1.1]{BKR}, and  we can 
pass back and forth between the relative and basepointed description of Van Daele $K$-theory.

\subsection{Bott periodicity}

The Bott map in Van Daele $K$-theory can be expressed in terms of suspensions and Clifford algebras. 
Recall that the suspension of a $C^*$-algebra $B$ is given by 
\[
   SB := C_0\big( (0,1) , B \big) = \big\{ f \in C([0,1], B) \mid f(0) = f(1) = 0 \big\} ,
\]
with the grading and real structure inherited from $B$. 
We will freely make use of the isomorphism
  $S(B \hox \Cl_{r,s} ) \cong SB \hox \Cl_{r,s}$.

We also consider the general suspension $S^{p,q} B := C_0(\R^{p,q},  B)$, which 
denotes the complex $C^*$-algebra $C_0(\R^{p+q}, B)$ with real structure 
$f^{{\rs}_{p,q}}(x,y) = f(x, -y)^{\rs_B}$ for $x \in \R^p$, $y \in \R^q$. 
We note that $SB \cong S^{1,0}B$. 

\begin{thm}[{\cite[{\S}2]{vanDaele2}, \cite[Proposition 5.8]{Kubota15a}}]  \label{thm:DK_Bott_isos}
Let $B$ be a $\Z_2$-graded and $\sigma$-unital $C^*$-algebra $B$. Then there are natural isomorphisms,
\[
  \beta_{\DK}: \DK_e( B) \xrightarrow{\simeq} \DK_{\one \hox \gamma}( SB \hox \Cl_{1,0} ), \qquad 
  \beta^S_{\DK}:\DK(B)\to \DK(S^{1,1} B).
\]
\end{thm}

We will primarily consider $\beta_{\DK}$ and describe this map in the relative setting with $B$ balanced graded.
Then for $[x]-[e]\in \DK( B)$ we have
\begin{align*}
  \beta_{\DK} ( [x] - [e] ) &= [ z(x, e, t) ] - [ \one \hox \gamma ] \in \DK( SB \hox \Cl_{1,0} ), \\
   z(x,e,t) &=  \frac{1}{2} \mathrm{Ad}_{( \one + e\hox \gamma )} \circ \mathrm{Ad}_{ \nu( x, t)} \circ \mathrm{Ad}_{ \nu( e, -t)} (\one \hox \gamma),  
   \quad \nu(y, t) = \cos\big( \tfrac{\pi}{2} t \big) + (y \hox \gamma)\sin \big( \tfrac{\pi}{2} t \big).
\end{align*}
Note that any odd self-adjoint unitary $y \hox \one \in B \hox \Cl_{1,0}$ anti-commutes with $\one \hox \gamma$, which implies
\begin{align*}
   \mathrm{Ad}_{\nu(y\hox \one, t)}( \one \hox \gamma) 
   &= (\one \hox \gamma) \big(\cos( \tfrac{\pi}{2} t \big) - (y \hox \gamma)\sin (\tfrac{\pi}{2} t) \big)  \big( \cos ( \tfrac{\pi}{2} t  ) - (y \hox \gamma) \sin ( \tfrac{\pi}{2} t)  \big) \\
   &= (\one \hox \gamma) \exp\big( -\pi t ( y \hox \gamma)  \big).
\end{align*}
Hence we can rewrite 
\begin{align} \label{eq:DK_Bott_graded_comp}
    \beta_{\DK} ( [x] - [e] ) &= \big[ \tfrac{1}{2} \mathrm{Ad}_{( \one + e\hox \gamma )} (\one \hox \gamma) \exp( -\pi t(x \hox \gamma)) \exp( \pi t( e \hox \gamma)) \big] 
      - [ \one \hox \gamma]    \nonumber \\
      &= \big[ ( e \hox \one) \exp( -\pi t(x \hox \gamma)) \exp( \pi t( e \hox \gamma)) \big] - [ e \hox \one].  
\end{align}
as $e\hox \one$ and $\one \hox \gamma$ are homotopic as OSUs with $\tfrac{1}{2} \mathrm{Ad}_{( \one + e\hox \gamma )}$ a map that 
switches basepoints.

\subsection{Boundary maps} \label{subsec:DK_bdry}

We also review some results concerning extensions and boundary maps. 
Let $I$ be a graded ideal in a balanced graded algebra $B$ and consider the short exact sequence,
\begin{equation} \label{eq:graded_generic_SES}
    0 \to I \to B \to B/ I \to 0.
\end{equation}
As considered in~\cite[{\S}4]{vanDaele1}, there is a boundary  map 
$\partial: \DK\big( S(B/I) \big) \to \DK( I )$. 

\begin{prop}[{cf. \cite[Proposition 4.8]{vanDaele1}}]  \label{prop:suspension_bdry_graded}
Let $\{{x}_t\}_{t \in [0,1]}$ be a path of odd self-adjoint elements in $B$ such that 
${x}_0$ and $ {x}_1$ are odd self-adjoint unitaries in $B$, $\{q({x}_t)\}_{t\in [0,1]}$ is a path of 
odd self-adjoint unitaries in $B/I$, and $\big\| q( {x}_0) - q( {x}_t) \big\|_{B/I} < 2$ for all $t \in [0,1]$. 
Then there is a well-defined element 
$[ q({x}_\bullet) ] - [ q(x_0) ] \in \DK\big( S(B/I ) \big)$, and 
the  boundary map $\partial: \DK\big( S(B/I) \big) \to \DK( I )$ is such that 
\[
    \partial\big( [ q(  {x}_\bullet ) ] - [q(x_0)] \big)  = \exc \big( [x_1] - [x_0] \big), \qquad 
    \exc: \DK( B, B/I ) \xrightarrow{\simeq} \DK(I).
\]
\end{prop}

\begin{proof}
We are given a path of OSUs $q(x_t)$ such that $\big\| q( {x}_0) - q( {x}_t) \big\|_{B/I} < 2$ and where $q(x_0)$ and $q(x_1)$ have 
 OSU lifts. Applying Lemma \ref{lemma:OSU_lifting}, 
there is a path $t\mapsto y_t \in \mbox{\rm OSU}(B)$ such that $q(y_t) = q(x_t)$ and where $y_0 = x_0$, $y_1 = x_1$. 
Furthermore,  because $\| q(x_1) - q(x_0) \| < 2$, $y_1$ is homotopic in $\mbox{\rm OSU}(B)$ to 
$\tilde{y}_1$ with $q(x_0) = q( \tilde{y}_1)$ (cf. the proof of Lemma \ref{lem:relative_DK}).  We can therefore consider the path of OSUs 
\[
  \tilde{x}: [0,1] \to \mbox{\rm OSU}(B), \qquad  
  \tilde{x}_t = \begin{cases}  {y}_{2t}, & t \in \big[ 0, \tfrac{1}{2} \big] \\ v_{2t-1}  {y}_1 v_{2t-1}^*, & t \in \big[ \tfrac{1}{2}, t\big] \end{cases}, 
\]
where $t\mapsto v_t y_1 v_t^*$ is the path connecting $y_1$ to $\tilde{y}_1$. Hence we have a path $t\mapsto q(\tilde{x}_t) \in \mbox{\rm OSU}(B/I)$ such that 
$q(\tilde{x}_0) = q( \tilde{x}_1)$ and the class $[q(\tilde{x}_\bullet) ] - [q(x_0)] \in \DK( S (B/I) )$ is a well-defined 
(where $q(\tilde{x}_0)$ is the constant path). The paths $t \mapsto q(\tilde{x}_t)$ and  $t \mapsto \tilde{x}_t$ satisfy the 
hypothesis of~\cite[Proposition 4.8]{vanDaele1}, which we apply to compute the boundary map 
\[
  \partial\big( [ q(  \tilde{x}_\bullet ) ] - [q(x_0)] \big) = \exc \big( [\tilde{x}_1] - [ \tilde{x}_0 ] \big) 
    = \exc \big( [x_1] - [ x_0 ] \big)  .  \qedhere
\]
\end{proof}

By composing Bott periodicity $\beta_{\DK}:\DK(B/I)\to \DK\big(S(B/I)\hox\Cl_{1,0}\big)$ with the 
 boundary map 
$\partial: \DK\big( S(B/I\hox\Cl_{1,0}) \big) \to \DK( I\hox\Cl_{1,0} )$, we obtain an additional boundary map
\begin{equation}
\delta:\DK(B/I)\to \DK( I\hox\Cl_{1,0} ).
\label{eq:additional-bdry}
\end{equation}
Comparing with complex $K$-theory,  $\partial: \DK\big( S(B/I) \big) \to \DK( I )$ is related to the \emph{index map} 
$\partial^K \colon K_1(B/I) \simeq K_0\big(S(B/I)\big) \to K_0(I)$, whereas  $\delta:\DK(B/I)\to \DK( I\hox\Cl_{1,0} )$ 
is more closely akin to the exponential map $\exp^K \colon K_0(B/I) \to K_1(I)$. 

\begin{lemma}[{\cite[Proposition 3.4]{vanDaele2}}]  \label{lemma:delta_formula}
\label{lem:other-bdry}
The boundary map $\delta = \partial \circ \beta_{\DK} \colon  \DK(B/I) \to \DK(I\hat\otimes \Cl_{1,0})$ 
of the short exact sequence $0\to I\to B\to B/I\to 0$ is given by 
\begin{equation*} 
\delta([x_1]-[x_2]) = [Y_1] - [Y_2], \qquad   
Y_i =   - \exp(\pi \tilde{x}_i  \hox  \gamma)(1  \hox \gamma),
\end{equation*}
where $\tilde{x}_i \in B$ is an odd self-adjoint lift of $x_i$ with $\| \tilde{x}_i\|=1$ and $\gamma$ is the odd 
generator of $\Cl_{1,0}$. 
\end{lemma}

For the short exact sequence from Eq. \eqref{eq:graded_generic_SES}, 
naturality of  $\beta_{\DK}$ implies that the diagram
\begin{equation}
  \xymatrix{
     {}\ar[r]&\DK(S^{1,1}B/I)\ar@{.>}[dr]^\delta \ar[r]^{\partial^{1,1}}\ar[d]^{\beta_{\DK}} & \DK(S^{0,1}I) \ar[r]\ar[d]^{\beta_{\DK}}  & \DK(S^{0,1}B) \ar[r] \ar[d]^{\beta_{\DK}} &{} \\
     {} \ar[r]&\DK(S^{2,1}B/I \hox \Cl_{1,0}) \ar[r]^{\partial^{2,1}} & \DK(S^{1,1}I \hox \Cl_{1,0}) \ar[r]  & \DK(S^{1,1}B  \hox \Cl_{1,0}) \ar[r]  &{}
      }
\label{eq:delta-defined}
\end{equation}
commutes,  $\delta=\beta_{\DK}\circ\partial^{1,1}=\partial^{2,1}\circ\beta_{\DK}$. 
An analogous argument applies to the suspension isomorphism $\beta_{\DK}^S$. Hence, 
suitably interpreted, we  have the relations 
\begin{align} \label{eq:Bott_commutator}
    &\partial \circ \beta_{\DK} = \beta_{\DK} \circ \partial, 
    &&\delta \circ \beta_{\DK} = \beta_{\DK} \circ \delta \nonumber  \\
    &\partial \circ \beta_{\DK}^S = \beta_{\DK}^S \circ \partial, 
    &&\delta \circ \beta_{\DK}^S = \beta_{\DK}^S \circ \delta.
\end{align}

\subsection{\texorpdfstring{$C^*$}{C*}-modules and Morita invariance} \label{sec:DK_Morita}
\label{app:Morita_invariance}

To compare Van Daele $K$-theory with $\KK$-theory, we will also consider 
Van Daele $K$-theory in the setting of Hilbert $C^*$-modules.

We let  $Y_B$ be a $\Z_2$-graded and countably generated full 
 (Real) $C^*$-module. 
 In the case that $Y_B = \hat{\calH} \hox B$ is the standard ($\Z_2$-graded) $C^*$-module,  we use the isomorphisms 
  $\End_B^0(\hat{\calH} \hox B) \cong B \hox \calK$, $\End_B(\hat{\calH} \hox B) \cong \Mult( B \hox \calK)$, 
  and $\calQ_B( \hat{\calH} \hox B) \cong \calQ_{B \hox \calK}$, where $\calK = \calK(\hat{\calH})$ denotes the  
  compact operators on a $\Z_2$-graded Hilbert space. 
  We will freely use the isomorphism   $\End_{SB}^0(SY) \cong S \End_B^0(Y)$.
  
Of particular interest to us will be the Van Daele $K$-theory groups and boundary maps associated to the short exact sequence 
\[
   0 \to \End_B^0(Y) \to \End_B(Y) \to \calQ_B(Y) \to 0 ,
\]
with $Y_B$ a $\Z_2$-graded and countably generated $C^*$-module over a $\sigma$-unital  $C^*$-algebra $B$. 
When  $\End_B^0(Y)$ is stable, the boundary map is an isomorphism, which is a consequence of the following.

\begin{lemma}[{\cite[Lemma 5.10, proof of Theorem 5.11]{Kubota15a}}] 
\label{lemma:SQ_and_QS}
Let $Y_B$ be a countably generated and $\Z_2$-graded $C^*$-module such that 
$\End_B^0(Y)$ is stable. Then $\DK( \End_B(Y) ) = \{0\}$ and the 
embedding $\iota: S\calQ_{B}(Y) \hookrightarrow \calQ_{SB}(SY)$ induces an isomorphism 
$\iota_\ast : \DK( S\calQ_{B}(Y) ) \xrightarrow{\simeq} \DK( \calQ_{SB}(SY) )$.
\end{lemma}
\begin{proof}
If $\End_B^0(Y)$ is stable, then $\End_B(Y) \cong \Mult\big( \End_A^0(B) \big)$ is $K$-contractible by~\cite[Lemma 5.10]{Kubota15a}.
There is a canonical identification $\End_{SB}^0(SY) = S \End_B^0(Y)$ and we compare
the short-exact sequences 
\[
  \xymatrix{
     0  \ar[r] & S\End_B^0(Y) \ar[r] \ar@{=}[d]  & S \End_B(Y) \ar[r] \ar[d] & S \calQ_{B}(Y) \ar[r] \ar[d] & 0 \\
     0 \ar[r] & \End_{SB}^0(SY) \ar[r] & \End_{SB}(SY) \ar[r] & \calQ_{SB}(SY) \ar[r] & 0.
 }
\]
Then $\{0\} = \DK( \End_{SB}(SY)) = \DK(S \End_B(Y))$ by the first statement  and the  induced map 
$\iota_\ast : \DK( S \calQ_{B}(Y) ) \to \DK( \calQ_{SB}(SY) )$ is 
an isomorphism.
\end{proof}

Using 
the inclusion $\iota:S \calQ_{B}(Y)\to \calQ_{SB}(SY)$  from Lemma \ref{lemma:SQ_and_QS} 
and the commutativity of $\partial$ with $\beta_{\DK}^S$ from  Eq. \eqref{eq:Bott_commutator}, the diagram 
\[
  \xymatrix{
    \DK( S \calQ_{B}(Y) ) \ar[r]^{\iota_\ast}  \ar[drrrr]_{  \partial } & \DK( \calQ_{SB}(SY) ) \ar[rr]^{ \hspace{-.7cm} \partial \circ \beta_{\DK}^S} &  &
    \DK( S^{0,1} \End_{SB}^0(SY) ) \ar@{=}[r] & \DK( S^{1,1} \End_B^0(Y) ) \ar[d]^{(\beta_{\DK}^{S})^{-1} } \\
    & & &  & \DK( \End_B^0(Y) ).
  }
\]
commutes. When the context is clear and $\End_A^0(Y)$ is stable (e.g. for $Y_B$ the standard module $\hat{\calH}_B$), 
we will simply write 
\begin{align} \label{eq:partial_delta_iota_compat}
      &\partial \circ \iota_\ast = \partial, &&\delta \circ \iota_\ast = \delta
\end{align}
as the maps from $\DK( S\calQ_B(Y)) \cong \DK( \calQ_{SB}(SY) )$ to  $\DK( \End^0_B(Y))$ and 
$\DK( S \End_B^0(Y) \hox \Cl_{1,0})$ respectively.

If $Y_B$ is full, then by \cite[Lemma 2.6]{BKR} there exists a  Morita invariance isomorphism 
\begin{equation} \label{eq:DK_Morita}
\calM_Y^{\DK} \colon \DK( \End_B^0(Y) ) \xrightarrow{\simeq} \DK( B ) 
\end{equation}
For our applications to the Fredholm index  and spectral flow, we require compatibility of the Morita invariance with the Bott map 
and other structures in Van Daele $K$-theory. 

\begin{lemma}[{cf. \cite[Lemma 2.6]{BKR}}]   \label{lemma:DK_Morita_compat_new}
Let $Y_B$ be a countably generated and full $\Z_2$-graded $C^*$-module. 
\begin{enumerate}
 \item For any stabilisation isometry $\nu: Y_B \to \hat{\calH}_B$, the induced homomorphism 
$\mathrm{Ad}_\nu : \End_B^0(B) \to B \hox \calK$ is such that 
\[
    \calM^{\DK}_{\hat{\calH}_B} \circ \big(\mathrm{Ad}_\nu \big)_\ast = \calM^{\DK}_{Y} 
    : \DK\big( \End_B^0(Y) \big) \xrightarrow{\simeq} \DK( B).
\]
 \item Using the identification $S \End_B^0(Y) \hox \Cl_{1,0} \cong \End_{SB}^0(SY) \hox \Cl_{1,0}$, 
 \[
     \beta_{\DK} \circ \calM_{Y}^{\DK}  = \calM_{SY}^{\DK} \circ \beta_{\DK} 
     : \DK\big( \End_B^0(Y) \big) \xrightarrow{\simeq} \DK\big( SB \hox \Cl_{1,0} \big) .
 \]
\end{enumerate}
\end{lemma}
\begin{proof}
(1) As $\End_B^0(Y)$ is non-unital, we  
 consider odd self-adjoint unitaries in $\End_B^0(Y)^\sim \hox \Cl_{1,1}$. We therefore take the 
 OSUs 
$x, y \in \End_B^0(Y)^\sim \hox \Cl_{1,1}$ such that $x-y \in \End_B^0(Y)\hox \Cl_{1,1}$. Note that we can consider elements in 
$\End_B(Y) \hox \Cl_{1,1}$ as acting on $Y_B \hox \Cl_{1,1}$.
Using Lemma \ref{lemma:DK_Clifford_stability}, we will apply Clifford stability $\calM_{1,1}^{\DK}$
and show the result for classes in  $\DK( \End_B^0(Y) \hox \Cl_{1,1} \hox \Cl_{1,1})$ where, 
recalling Eq. \eqref{eq:M_clifford11_notation}, we will use the condensed notation
\[
  \calM_{1,1}^{\DK}( [ x] - [y]) =   \left[\begin{pmatrix} x & 0\\ 0 & y\end{pmatrix}\right]-\left[\begin{pmatrix} 0 & \one \\ \one & 0\end{pmatrix}\right] 
     =: [ x \oplus y ] - [ \one \hox \gamma_1 ].
\]
Following~\cite[page 9]{BKR}, we can represent the class $[ x \oplus y ] - [ \one \hox \gamma_1 ] = [ \tilde{x} \oplus \tilde{y} ] - [ \one \hox \gamma_1 ]$ 
 such that $\tilde{x} \oplus \tilde{y} - \one \hox \gamma_1$ is finite-rank, i.e. an element of $M_l(B) \hox \Cl_{2,2}$ for some $l \in \N$. 

To take the Morita map, we consider a stabilisation isometry  $\nu: Y_B \to \hat{\calH}_B$, 
which is equivalent to a choice of  (even) unitary stabilisation isomorphism 
$W: Y_B \hox \Cl_{2,2} \oplus \hat{\calH} \hox B \hox \Cl_{2,2} \xrightarrow{\simeq} \hat{\calH} \hox B \hox \Cl_{2,2}$. The $C^*$-module $\hat{\calH} \hox B \hox \Cl_{2,2}$ 
has a canonical odd self-adjoint unitary $e  \hox \one$, where $e \in \calU(\hat{\calH})$ is the OSU that identifies $e:\hat{\calH}^{0}\xrightarrow{\simeq} \hat{\calH}^{1}$. 
Then following the definition of $\calM^{\DK}_{Y}$ \cite{BKR},
\begin{align*}
   \calM^{\DK}_{Y}\big( [ x \oplus y ] &- [ \one \hox \gamma_1 ] \big)  = 
   \left[ W \begin{pmatrix} \tilde{x}\oplus \tilde{y} & 0 \\ 0 & e\hox \one \end{pmatrix} W^* \right] 
   -    \left[ W\begin{pmatrix} \one \hox \gamma_1 & 0 \\  0 & e \hox \one   \end{pmatrix} W^* \right] \\
   &= \left[ \nu \big( \tilde{x}\oplus \tilde{y} \big) \nu^* +W(0\oplus (e \hox \one))W^* \right] - 
     \left[ \nu ( \one \hox \gamma_1 ) \nu^* +W(0\oplus (e \hox \one))W^*  \right] \\
     &= \left[ \nu \big( \tilde{x}\oplus \tilde{y} \big) \nu^*  \right] - 
     \left[ \nu ( \one \hox \gamma_1 ) \nu^*   \right] \\
   &= \calM^{\DK}_{\hat{\calH}_B }  \circ ( \mathrm{Ad}_\nu )_\ast  \big( [ x \oplus y ] - [ \one \hox \gamma_1 ] \big) ,
\end{align*}
where we have used that $x \mapsto x \oplus e$ is the inductive map that defines $\DK_e(B)$. 

For part (2), we directly check that $\beta_{\DK} \circ \calM_{Y}^{\DK}  = \calM_{SY}^{\DK} \circ \beta_{\DK}$. 
Following the same procedure as part (1) of the proof, 
\begin{align*}
   \calM^{\DK}_Y \big( [ x \oplus y ] - [ \one \hox \gamma_1 ] \big)  &= 
   \left[ W \begin{pmatrix} \tilde{x}\oplus \tilde{y} & 0 \\ 0 & e\hox \one \end{pmatrix} W^* \right] 
   -    \left[ W\begin{pmatrix} \one \hox \gamma_1 & 0 \\  0 & e \hox \one   \end{pmatrix} W^* \right] 
\end{align*}
We now apply the Bott map, where
\begin{align*}
    &\beta_{\DK} \circ \calM_{\DK}\big( [ x \oplus y ] - [ \one \hox \gamma_1 ] \big)  \\
     &\hspace{.2cm}  = \left[ W \begin{pmatrix} \one \hox \gamma_1 & 0 \\  0 & e \hox \one   \end{pmatrix} 
     \exp\left( -\pi t \begin{pmatrix} \tilde{x} \oplus \tilde{y} & 0 \\  0 & e \hox \one   \end{pmatrix} \hox \gamma_2 \right) 
     \exp\left( \pi t \begin{pmatrix} \one \hox \gamma_1 & 0 \\  0 & e \hox \one   \end{pmatrix} \hox \gamma_2 \right) W^* \right] \\
     &\qquad \quad      - \left[ W \begin{pmatrix} \one \hox \gamma_2 & 0 \\  0 & e \hox \one   \end{pmatrix} W^* \right] \\
   &\hspace{.2cm}  = \left[ W \begin{pmatrix} \one \hox \gamma_1 & 0 \\ 0 & e \hox \one \end{pmatrix} 
   \begin{pmatrix} e^{ -\pi t(\tilde{x}\oplus \tilde{y} \hox \gamma_2) } e^{ \pi t(\one \hox \gamma_1\gamma_2) } & \hspace{10pt} 0 \\
    \hspace{-70pt} 0 & \hspace{-60pt} e^{ -\pi t ( e\hox \gamma_2)} e^{ \pi t ( e\hox \gamma_2) }\end{pmatrix} W^* \right] 
        - \left[ W \begin{pmatrix} \one \hox \gamma_2 & 0 \\  0 & e \hox \one   \end{pmatrix} W^* \right] \\
      &\hspace{.2cm}  = \left[ W \begin{pmatrix} \one \hox \gamma_1 & 0 \\ 0 & e \hox \one \end{pmatrix} 
   \begin{pmatrix} e^{ -\pi t(\tilde{x}\oplus \tilde{y} \hox \gamma_2) } e^{ \pi t(\one \hox \gamma_1\gamma_2) } & 0 \\ 0 & \one \end{pmatrix} W^* \right] 
           - \left[ W \begin{pmatrix} \one \hox \gamma_2 & 0 \\  0 & e \hox \one   \end{pmatrix} W^* \right],   
\end{align*}
which gives an element in $\DK( S\End_B^0(Y) \hox \Cl_{3,2} )$. Note that we have extended $W$ such that $\mathrm{Ad}_W$ 
acts trivially on the extra $\Cl_{1,0}$-generator that comes from the Bott map.

To compute $\calM_{SY}^{\DK} \circ \beta_{\DK}$, we take 
$ [ \tilde{x} \oplus \tilde{y} ] - [ \one \hox \gamma_1] \in \DK( \End_B^0(Y) \hox \Cl_{2,2} )$ and apply
the Bott map, 
\begin{align*}
   \beta_{\DK} \big(  [ \tilde{x} \oplus \tilde{y} ] - [ \one \hox \gamma_1] ) &= \big[ z( \tilde{x} \oplus \tilde{y}, \one \hox \gamma_1, t) \big] - [ \one \hox \gamma_2 ] \\
     &= \big[ (\one \hox \gamma_1) \exp\big( -\pi t( \tilde{x}\oplus \tilde{y} \hox \gamma_2 ) \big) \exp\big( \pi t ( \one \hox \gamma_1\gamma_2) \big) \big] 
     - [ \one \hox \gamma_2] .
\end{align*}
The function $z( \tilde{x} \oplus \tilde{y}, \one \hox \gamma_1, t)$ is an OSU in $S \End_{B}^0(Y)^\sim \hox \Cl_{3,2} \cong \End_{SB}^0(SY)^\sim \hox \Cl_{3,2}$. 
Because $\tilde{x} \oplus \tilde{y} - \one \hox \gamma_1$ is finite-rank, $z( \tilde{x} \oplus \tilde{y}, \one \hox \gamma_1, t) - \one \hox \gamma_2$ is also 
finite rank for all $t$. So for a stabilisation unitary $V: SY_{SB} \hox \Cl_{3,2} \oplus \hat{\calH} \hox SB \hox \Cl_{3,2} \xrightarrow{\simeq} \hat{\calH} \hox SB \hox \Cl_{3,2}$, 
\begin{align*}
  \calM^{\DK}_{SY} \circ \beta_{\DK} \big(  [ \tilde{x} \oplus \tilde{y} ] - [ \one \hox \gamma_1] )  
     &= \left[ V\begin{pmatrix} z( \tilde{x} \oplus \tilde{y}, \one \hox \gamma_1, t) & 0 \\ 0 & e\hox \one \end{pmatrix} V^* \right] 
    - \left[ V \begin{pmatrix} \one \hox \gamma_2 & 0 \\ 0 & e \hox \one \end{pmatrix} V^* \right] \\
 &\hspace{-3.5cm}= \left[ V \begin{pmatrix} (\one \hox \gamma_1) e^{ -\pi t( \tilde{x}\oplus \tilde{y} \hox \gamma_2 ) } e^{ \pi t ( \one \hox \gamma_1\gamma_2) }& 0 \\
   0 & e \hox \one \end{pmatrix} V^* \right] 
      - \left[ V \begin{pmatrix} \one \hox \gamma_2 & 0 \\ 0 & e \hox \one \end{pmatrix} V^* \right] \\
 &\hspace{-3.5cm}=  \left[ V \begin{pmatrix} \one \hox \gamma_1 & 0 \\ 0 & e \hox \one \end{pmatrix} 
   \begin{pmatrix} e^{-\pi t( \tilde{x}\oplus \tilde{y} \hox \gamma_2 )} e^{ \pi t ( \one \hox \gamma_1\gamma_2) } & 0 \\ 0 & \one \end{pmatrix} V^* \right]  
    - \left[ V \begin{pmatrix} \one \hox \gamma_2 & 0 \\  0 & e \hox \one   \end{pmatrix} V^* \right].
\end{align*}
Note that we can consider the element $e$ the same in both $\hat{\calH} \hox B \hox \Cl_{2,2}$ and $\hat{\calH} \hox SB \hox \Cl_{3,2}$ as it only 
acts non-trivially on $\hat{\calH}$.  
Finally, because $\End_{SB}^0(SY) \cong S\End_{B}^0(Y)$, the operations $\mathrm{Ad}_W$ and $\mathrm{Ad}_V$ will agree for elements 
in $\End_{SB}^0(SY)^\sim \hox \Cl_{2,2} \cong \big(S\End_{B}^0(Y)\big)^\sim \hox \Cl_{2,2}$. The result then follows.
\end{proof}

\subsection{Isomorphisms with real and complex \texorpdfstring{$K$}{K}-theory} \label{subsec:appendix_DK_KR_isos}

For an ungraded Real $C^*$-algebra $A$, there are isomorphisms $\Upsilon_A^{r,s} \colon \DK( A \otimes \Cl_{r,s} ) \xrightarrow{\simeq} \KR_{1+s-r}(A) \simeq \KO_{1+s-r}(A^\rs)$. 
These isomorphisms are natural and so are compatible with other natural transformations in $K$-theory and $\DK$-theory.
The description of these isomorphisms below can also be found in~\cite{vanDaele1} and~\cite[{\S}5]{Kellendonk15}. 

We will mainly be interested in the isomorphisms of Van Daele $K$-theory with degree-zero and degree-one (real or complex) $K$-theory, and we will present these isomorphisms in detail. 
For completeness, we will also include a description of the isomorphisms with the higher degree $K$-theory groups, where we use the picture of real $K$-theory due to Boersema and Loring \cite{BL}. 

\begin{prop}[\textbf{\boldmath $\KR_0$ and $K_0$}] \label{prop:iso_DK_even}
Let $A$ be a trivially graded (Real) $C^*$-algebra. 
Then we have an isomorphism $\Upsilon_A^{1,0} \colon \DK( A \otimes \Cl_{1,0}) \xrightarrow{\simeq} \KR_0(A)$ given explicitly by 
\[
\big[ (2p-1) \otimes \gamma \big] - \big[ (2q-1) \otimes \gamma \big] \xmapsto{\simeq} [p] - [q] , 
\]
where $p,q \in M_n(A^\sim)$ are (Real) projections such that $p-q \in M_n(A)$. 
In the absence of real structures, we obtain the isomorphism $\Upsilon_A^{1,0} \colon \DK( A \otimes \Cl_{1}) \xrightarrow{\simeq} K_0(A)$ with even complex $K$-theory. 
\end{prop}
\begin{proof}
Let us first assume $A$ is unital. 
Then odd self-adjoint unitaries in $M_n(A) \otimes \Cl_{1,0}$ must be of the 
form $x \otimes \gamma$ with $x = x^* = x^\frakr = x^{-1} \in M_n(A)$. 
Hence $x = 2p-\one$ for some projection $p \in M_n(A)$. 
Since $(2p_0 - \one) \otimes \gamma \sim_h (2p_1 - \one) \otimes \gamma$ if and only $p_0 \sim_h p_1$, 
we obtain a semigroup isomorphism 
\[
   \bigsqcup_k\pi_0\big(\mbox{\rm OSU}(M_k(B))\big) \ni (2p- \one) \otimes \gamma \mapsto p \in   \bigsqcup_k \mbox{\rm Proj}(M_k(B)) / \sim_h .
\]
Taking the Grothendieck completion, we obtain the isomorphism $\Upsilon_A^{1,0} \colon \DK( A \otimes \Cl_{1,0}) \xmapsto{\simeq} \KR_0(A)$. 

If $A$ is non-unital, we apply the above argument to the minimal unitisation $A^\sim$ and use the fact that both $\DK(A \otimes \Cl_{1,0})$ and 
$\KR_0(A)$ are defined via the quotient $A^\sim \to \C$. 

In the absence of real structures, we can assume that the generator of $\Cl_1$ is a self-adjoint unitary, 
and the same construction then yields the isomorphism $\DK( A \otimes \Cl_1 ) \cong K_0(A)$.
\end{proof}

\begin{prop}[\textbf{\boldmath $\KR_{1}$ and $K_1$}] \label{prop:iso_DK_odd}
Let $A$ be a trivially graded (Real) $C^*$-algebra. 
Then we have an isomorphism $\Upsilon_A^{1,1} \colon \DK( A \otimes \Cl_{1,1}) \xrightarrow{\simeq} \KR_1(A)$ given explicitly by 
\[
\big[ \tfrac{1}{2}(u+u^*) \otimes \gamma + \tfrac{1}{2}(u-u^*) \otimes \rho \big] - [ \one \otimes \gamma ] \xmapsto{\simeq} [u] , 
\]
where $u \in M_n(A^\sim)$ is a (Real) unitary. 
In the absence of real structures, we obtain the isomorphism $\Upsilon_A^{1,1} \colon \DK( A \otimes \Cl_{2}) \xrightarrow{\simeq} K_1(A)$ with odd complex $K$-theory. 
\end{prop}
\begin{proof}
We identify $\Cl_{1,1} \cong M_2(\C)$ with the Real generators $\gamma \cong \sigma_1$ and $\rho \cong -i \sigma_2$. 
If $A$ is unital, 
any Real odd self-adjoint unitary in $A \otimes \Cl_{1,1}$ is of the form 
\[
\tfrac{1}{2}(u+u^*) \otimes \gamma + \tfrac{1}{2}(u - u^*) \otimes \rho 
\cong \begin{pmatrix} 0 & u^* \\ u & 0 \end{pmatrix} ,
\]
and the map $\left(\begin{smallmatrix} 0 & u^* \\ u & 0 \end{smallmatrix}\right) \mapsto u \in \calU(A)$ 
induces the isomorphism $\Upsilon_A^{1,1} \colon \DK( A \otimes \Cl_{1,1}) \xrightarrow{\simeq} \KR_1(A)$, 
where $\KR_1(A)$ is given by homotopy classes of Real unitaries. 
The non-unital and complex cases follow as in the proof of Proposition \ref{prop:iso_DK_even}. 
\end{proof}

\paragraph{\boldmath $\KR_{2}$}
We consider $\DK(A \otimes \Cl_{0,1})$, using the presentation of complex Clifford algebras 
$\Cl_1 \cong \C \oplus \C$ with grading $\Gamma( \alpha, \beta) = (\beta, \alpha)$. We obtain the Real Clifford 
algebra $\Cl_{0,1}$ by taking $\Cl_1 \cong \C\oplus \C$ with real structure $(\alpha, \beta)^{\frakr_{0,1}} = (\ol{\beta}, \ol{\alpha})$ 
(in particular, we have $\rho \sim (i, -i)$). We have that odd self-adjoint unitaries in $A \otimes \Cl_1$ are of the form 
$x=(2p - \one) \otimes (1,-1)$. If we impose $x^{\frakr_A \otimes \frakr_{0,1}} =x$, then 
$(2p - \one)^\rs = -(2p- \one)$. 
In~\cite{BL}, Boersema and Loring characterise the group $\KO_{2}(A^\rs)$ via equivalence classes of self-adjoint unitaries 
such that $u^\rs = -u$. In particular, the map $(2p - \one) \otimes (1,-1) \mapsto 2p - \one$ will give an 
isomorphism $DK(A \otimes \Cl_{0,1}) \cong \KO_2(A^\rs)$.

\paragraph{\boldmath $\KR_{-1}$ and $\KR_3$ }
To consider $\DK(A \otimes \Cl_{2,0})$, we take $\Cl_2 \cong M_2(\C)$ with the real structure $\frakr_{2,0} = \mathrm{Ad}_{\sigma_1} \circ \mathfrak{c}$, 
where $\mathfrak{c}$ denotes entrywise complex conjugation. We find that 
$\sigma_1^{\rs_{2,0}} = \sigma_1$, $\sigma_2^{\rs_{2,0}} = \sigma_2$, and $\Cl_{2}^{\rs_{2,0}} \cong C\ell_{2,0}$. 
Any odd element in $A \otimes \Cl_2$ can be written as $x = a \otimes \sigma_1 + b \otimes \sigma_2$. If $x= x^* = x^{-1}$, then 
$a$ and $b$ are self-adjoint and $u = a+ib \in A$ is unitary. 
Imposing the condition $x^{\frakr_{2,0}} = x$, we obtain that $a^\rs = a$ and $b^\rs = b$ and hence 
$u^\rs = u^*$. To summarise, odd self-adjoint unitaries in $A \otimes \Cl_{2,0}$ can be described in terms of 
unitaries $u \in A$ such that $u^\rs = u^*$. Boersema and Loring characterise  $\KO_{-1}(A^\rs)$  by classes of unitaries 
$u \in M_n(A)$ such that $u^* = u^\rs$.\footnote{Note that instead of a real structure Boersema and Loring use an anti-multiplicative 
involution $a\mapsto a^{\tau} \in A$. We obtain this picture by taking $\tau = \ast \circ \frakr$.} 
Hence the map $a \otimes \sigma_1 + b \otimes \sigma_2 \mapsto a+ ib$ will furnish an isomorphism 
$\DK(A \otimes \Cl_{2,0} ) \cong \KO_{-1}(A^\rs)$. 

We can similarly consider the real structure $\frakr_{0,2}= \mathrm{Ad}_{-i\sigma_2} \circ \mathfrak{c}$ on $\Cl_2 \cong M_2(\C)$, which 
has Real elements $i \sigma_1$ and $-i \sigma_2$, whence $\Cl_2^{\frakr_{0,2}} = C\ell_{0,2}$. So if the odd self-adjoint unitary 
$x= a\otimes \sigma_1 + b \otimes \sigma_2$ is invariant under $\frakr_{0,2}$, it follows that $a^\rs = -a$ and $b^\rs = -b$. Therefore 
the unitary $u= a+ib$ is such that $u^\rs = -u^*$. We can therefore describe classes of $\DK(A \otimes \Cl_{0,2})$ via unitaries 
in $M_{2m}(A)$ with $u^\rs = -u^*$. This condition can be equivalently formulated as requiring $u \in M_{2m}(A)$ to be such that 
\[
   u^{\tilde{\rs}} = u^*, \qquad  \begin{pmatrix} a & b \\ c & d \end{pmatrix}^{\tilde{\rs} } = \begin{pmatrix} d^\rs & -c^\rs \\ -b^\rs & a^\rs \end{pmatrix},
\]
which recovers the description of $\KO_3(A^\rs)$ from~\cite{BL}.

\paragraph{\boldmath Higher $\KR$-groups}
To consider groups $\KR_4$ and higher, we can instead use  the K\"{u}nneth formula for real $K$-theory~\cite{Boersema02}, 
where $\KO_n(A^\rs) \cong KO_{n-4}( A^\rs \otimes \mathbb{H})$ and $\mathbb{H} \simeq M_2(\C)^{\frakr_{0,2}}$ is considered 
as an ungraded real $C^*$-algebra. So we may repeat the previous constructions for a description of $\KR_4$, $\KR_5$, and $\KR_6$. 
Indeed, our description of $\KO_3(A^\rs)$ above can be equivalently formulated via $\KO_{-1}(A^\rs \otimes \mathbb{H})$.

\paragraph{\boldmath Compatibility with $K$-theory operations}

We also check that our isomorphisms behave well with respect  to Bott periodicity and Morita invariance in  $K$-theory.

\begin{lemma} \label{lemma:DK-complexK-respects-Bott}
The following diagram commutes,
\[
   \xymatrix{
      \DK(A \otimes \Cl_{1,0} ) \ar[rr]^{\beta_{DK} } \ar[d]_{\Upsilon^{1,0}_A} & &  \DK( SA \otimes \Cl_{2,0} ) \ar[d]^{\Upsilon^{2,0}_{SA}} \\
      \KO_0(A^\rs) \ar[rr]^{\beta^K}  & & KO_{-1}(SA^{\rs}) & \hspace{-1.3cm} \cong KO_0(S^{1,1}A^\rs).
   }
\]
Replacing $A$ with $S^{s,r}A$ shows that $\Upsilon^{r+2,s}_A \circ \beta_{DK}=\beta^K\circ\Upsilon_A^{r+1,s}$ for all $r,s \geq 0$.
\end{lemma}
\begin{proof}
We take $\big[ (2p-\one) \otimes \gamma \big] - \big[ (2q-\one) \otimes \gamma \big] \in \DK ( A \otimes \Cl_{1,0}  )$ 
and use the presentation $\Cl_{2,0} \cong M_2(\C)$ 
with generators $\sigma_1$ and $\sigma_2$ and real structure $\frakr_{2,0}= \mathrm{Ad}_{\sigma_1} \circ \mathfrak{c}$. 
Then using Eq. \eqref{eq:DK_Bott_graded_comp} with $e \sim \sigma_1$ and $\gamma \sim \sigma_2$, we have 
\begin{align*}
   \Upsilon^{2,0}_{SA}\circ \beta_{\DK} \big( \big[ (2p-\one) \otimes \gamma \big] - \big[ (2q-\one) \otimes \gamma\big] \big) 
   &= \Upsilon^{2,0}_{SA} \left( \left[ \begin{pmatrix} 0 & e^{-2\pi it p} \\ e^{2\pi i tp} & 0 \end{pmatrix} \right] - 
    \left[ \begin{pmatrix} 0 & e^{-2\pi it q} \\ e^{2\pi i tq} & 0 \end{pmatrix} \right]  \right) \\
   &= \big[ e^{2\pi i t p}  \big] - \big[ e^{2\pi i t q}  \big] ,
\end{align*}
where $(e^{2\pi i t p} )^* = (e^{2\pi i t p})^\rs$ and so we indeed obtain a class in $KO_{-1}(SA)$. 
On the other hand 
\begin{align*}
   \beta_K \circ \Upsilon_{A}^{1,0} \big( \big[ (2p-\one) \otimes \gamma \big] - \big[ (2q-\one) \otimes \gamma \big] \big) 
   &= \beta_K \big( [p] - [q] \big) 
   =    \big[ e^{2\pi i t p}  \big] - \big[ e^{2\pi i t q}  \big] . \qedhere
\end{align*}
\end{proof}

We also consider compatibility of the isomorphisms $\Upsilon_{A}^{r,s}$ with the Morita isomorphism in $\DK$ and $K$-theory. 
In the real case, Morita invariance in $\KR$-theory is easiest understood in the language of $\KKR$-theory and the map 
$\calM^{\KK}_Y$ from Eq. \eqref{eq:Morita_invariance}. Hence one  instead considers compatibility of 
Morita invariance in $\KKR$ with the isomorphisms between $\DK$ and $\KKR$-theory considered in Appendix \ref{sec:appendix_KK_DK_isos} below. 
These isomorphisms are connected to the graded Cayley  isomorphism $\mathfrak{C}_B: \KKR(\Cl_{1,0}, B) \xrightarrow{\simeq} \DK(B)$ from~\cite{BKR} 
(see Appendix \ref{subsec:Cayley}), where it was shown that $\mathfrak{C}_B \circ \calM_Y^{\KK} = \calM_Y^{\DK} \circ \mathfrak{C}_{\End_B^0(Y)}$ 
in~\cite[Lemma 4.13]{BKR}. 
Therefore we will restrict our attention to the complex setting and take a full ungraded $C^*$-module $X_A$ 
with stabilisation unitary $W_X: X_A \oplus \calH_A \to \calH_A$. 
We take a representative $[a] - [b] \in K_j \big( \End_A^0(X) \big)$ such that $a-b \in \End_A^0(X)$ is finite-rank. 
Then $\calM_X^K: K_j\big( \End_A^0(X) \big) \xrightarrow{\simeq} K_j(A)$ is such that
\begin{equation} \label{eq:Complex_K_Morita}
  \calM_X^K\big( [a] - [b] \big) = \left[ W_X \begin{pmatrix} a & 0 \\ 0 & I_j \end{pmatrix} W_X^* \right] - \left[ W_X \begin{pmatrix} b & 0 \\ 0 & I_j \end{pmatrix} W_X^* \right] , \qquad 
  I_j = \begin{cases} 0, & j = 0, \\ \one, & j=1. \end{cases}
\end{equation}

\begin{lemma} \label{lemma:Upsilon_Morita_commute_complex}
Let $X_A$ be a full and ungraded complex $C^*$-module, then 
\[
     \calM_X^K \circ \Upsilon_{\End_A^0(X)}^{1+0} = \Upsilon_A^{1+0} \circ \calM_{X\otimes \Cl_{1,0}}^{\DK}, \qquad 
      \calM_X^K \circ \Upsilon_{\End_A^0(X)}^{1+1} = \Upsilon_A^{1+1} \circ \calM_{X \otimes \C^2}^{\DK}.
\]
\end{lemma}
\begin{proof}
We take $[x] - [y] \in \DK\big( \End_A^0(X) \otimes \Cl_{1+s} \big)$
such that $x-y \in \End_A^0(X) \otimes \Cl_{1+s}$ is finite-rank and $s \in \{0,1\}$. 
We also fix an ungraded stabilisation unitary $W_X: X_A \oplus \calH_A \to \calH_A$. 

For the even case, we write $x= (2p-\one) \otimes \gamma$, $y = (2q-\one) \otimes \rho$, and take 
the graded stabilisation unitary 
${W}_{1+0}: (X_A \otimes \Cl_1) \oplus (\calH_A \otimes \Cl_{1}) \to  \calH_A \otimes \Cl_1$. 
Using that $\calH_A \otimes \Cl_1$ has a canonical odd self-adjoint unitary $\one \otimes \gamma$ which 
is homotopic to $-\one \otimes \gamma$, 
we can write
\begin{align*}
  &\Upsilon^{1+0}_A \circ \calM_{X \otimes \Cl_{1}}^{\DK} \big( [(2p-\one)\otimes \gamma] - [ (2q-\one) \otimes \gamma ] \big) \\
    &= \Upsilon_{A}^{1+0} \left( \left[ {W}_{1+0} \big( (2p  - 1) \otimes \gamma \oplus (-\one) \otimes \gamma\big)  {W}^*_{1+0}  \right] - 
    \left[  {W}_{1+0} \big( (2q  - 1) \otimes \gamma \oplus (-\one) \otimes \gamma\big)  {W}^*_{1+0}   \right]  \right) \\
    &= \Upsilon_A^{1+0}  \left( \left[ W_X\begin{pmatrix} 2p-\one & 0 \\ 0 & -\one  \end{pmatrix} W_{X}^* \otimes \gamma \right]  
    -     \left[ W_X\begin{pmatrix} 2q-\one & 0 \\ 0 & -\one \end{pmatrix} W_{X}^* \otimes \gamma \right]  \right) \\
  &=   \left( \left[ W_X\begin{pmatrix} p & 0 \\ 0 & 0  \end{pmatrix} W_{X}^*\right]  
    -     \left[ W_X\begin{pmatrix} q & 0 \\ 0 & 0  \end{pmatrix} W_{X}^*\right]  \right) \\
   &= \calM_X^K \circ \Upsilon^{1+0}_{\End_A^0(X)} \big( [(2p-\one)\otimes \gamma] - [ (2q-\one) \otimes \gamma ] \big).
\end{align*}
For the odd case, $\End_A^0(X) \otimes \Cl_{1,1}$ acts on the graded space  $X_A \otimes \C^2$ and we take the graded 
stabilisation unitary $ {W}_{1+1}: (X_A \otimes \C^2) \oplus \hat\calH_A \to \hat\calH_A$ with $\hat\calH_A \cong \calH_A \otimes \C^2$. 
We write 
$[x] - [y]  = \big[ \tfrac{1}{2}(x+x^*) \otimes \gamma + \tfrac{1}{2}(x-x^*) \otimes \rho\big] - \big[ \tfrac{1}{2}(y+y^*) \otimes \gamma + \tfrac{1}{2}(y-y^*) \otimes \rho\big] 
= \left[\left( \begin{smallmatrix} 0 & U^* \\ U & 0 \end{smallmatrix} \right) \right] - \left[ \left(\begin{smallmatrix} 0 & V^* \\ V & 0 \end{smallmatrix} \right) \right] $ 
and argue analogously to the even case, 
\begin{align*}
  &\Upsilon^{1+1}_A \circ \calM_{X \otimes \C^2}^{\DK}\left( \left[ \begin{pmatrix} 0 & U^* \\ U & 0 \end{pmatrix}  \right] 
    - \left[ \begin{pmatrix} 0 & V^* \\ V & 0 \end{pmatrix} \right] \right) \\
    &= \Upsilon_{A}^{1+1} \left( \left[  {W}_{1+1} \left( \begin{pmatrix} 0 & U^* \\ U & 0 \end{pmatrix} \oplus \begin{pmatrix} 0 & \one \\ \one & 0 \end{pmatrix} \right)  {W}^*_{1+1}  \right] - 
    \left[  {W}_{1+1} \left( \begin{pmatrix} 0 & V^* \\ V & 0 \end{pmatrix} \oplus \begin{pmatrix} 0 & \one \\ \one & 0 \end{pmatrix} \right)  {W}^*_{1+1}  \right] \right) \\
    &= \Upsilon_A^{1+1}  \left( \left[ \begin{pmatrix} 0 & W_X( U^* \oplus \one) W_X^* \\ W_X(U \oplus\one)W_X^*  & 0 \end{pmatrix} \right]  
    -   \left[ \begin{pmatrix} 0 & W_X( U^* \oplus \one) W_X^* \\ W_X(U \oplus\one)W_X^*  & 0 \end{pmatrix}  \right]   \right) \\
  &=  \left[ W_X\begin{pmatrix} U & 0 \\ 0 & \one  \end{pmatrix} W_{X}^*\right]  
    -     \left[ W_X\begin{pmatrix} V & 0 \\ 0 & \one  \end{pmatrix} W_{X}^*\right]   \\
   &= \calM_X^K \circ \Upsilon^{1+1}_{\End_A^0(X)} \left( \left[ \begin{pmatrix} 0 & U^* \\ U & 0 \end{pmatrix}  \right] 
    - \left[ \begin{pmatrix} 0 & V^* \\ V & 0 \end{pmatrix} \right] \right).  \qedhere
\end{align*}
\end{proof}

\section{Isomorphisms between \texorpdfstring{$\KKR$}{KKR} and \texorpdfstring{$\DK$}{DK}} \label{sec:appendix_KK_DK_isos}

For a $\sigma$-unital and $\Z_2$-graded Real $C^*$-algebra $B$, the isomorphism $\KKR(\Cl_{1,0}, B) \cong \DK(B)$ has been 
established in various forms~\cite{Roe04, Kubota15a, BKR}. Here we examine some of these maps, their compatibility with each 
other as well as other natural maps in $\KKR$ and $\DK$-theory.

\subsection{The Roe and Kubota isomorphisms} \label{subsec:bdd_KK_to_DK}

Much of the content of this section can also be found in~\cite[{\S}5]{Kubota15a}. 
We first review the $\Z_2$-graded isomorphisms and then consider the setting of 
skew-adjoint and Clifford anti-linear operators.

Using that $\DK\big( \Mult(B \hox\calK) \big) = 0$ by Lemma \ref{lemma:SQ_and_QS}, 
Roe \cite{Roe04} and Kubota \cite{Kubota15a} separately define isomorphisms from $\KKR$ to $\DK$ by composing the 
isomorphism $\Phi_B: \KKR(\C, B) \xrightarrow{\simeq} \DK( \calQ_{B \hox \calK}) $ from  Lemma \ref{lemma:Phi} 
with the Van Daele boundary maps.

\begin{thm}[{\cite[{\S}2]{Roe04}, \cite[Theorem 5.11]{Kubota15a}}] \label{thm:graded_Roe_and_Kubota}
Let $B$ be a $\Z_2$-graded and $\sigma$-unital $C^*$-algebra. There are natural 
isomorphisms $\Roe_B: \KKR(\C, B) \to \DK( B \hox \Cl_{1,0} )$ and 
$\Kubota_B: \KKR( \C, SB) \to \DK( B)$ defined by the composition of isomorphisms, 
\begin{align*}
    &\Roe_B: \KKR(\C, B) \xrightarrow{ \Phi_B } \DK( \calQ_{B \hox \calK} ) \xrightarrow{\delta} \DK( B \hox \calK \hox \Cl_{1,0} ) 
        \xrightarrow{ \calM^{\DK} } \DK( B \hox \Cl_{1,0} ). \\
    &\Kubota_B: \KR( \C, SB) \xrightarrow{ \Phi_{SB}  } \DK( \calQ_{ SB \hox \calK}  ) 
  \xrightarrow{(\iota_\ast)^{-1}} \DK\big( S \calQ_{B \hox \calK} \big)  \xrightarrow{\partial} \DK( B \hox \calK ) 
  \xrightarrow{ \calM^{\DK} }   \DK( B).
\end{align*}
\end{thm}

Next we incorporate skew-adjoint Fredholm operators on ungraded $C^*$-modules $X_A$ that 
anti-commute with the generators of an ungraded and ample $\Cl_{r,s}$-representation. 
Recalling Corollary \ref{cor:skew_Fred_to_KasMod}, this information defines a class 
in $\KKR( \Cl_{s+1, r}, A)$ with $A$ ungraded. Our aim is to understand the isomorphisms 
$\Roe$ and $\Kubota$ in this setting. To this end, we will use the Clifford shuffle isomorphisms 
$\shuffle^{r,s}_{\KK} : \KKR( A \hox \Cl_{s,r}, B)  \xrightarrow{\simeq} \KKR( A, B \hox \Cl_{r,s} )$ from Eq. \eqref{eq:Clifford_shuffle}.

\begin{defn} \label{def:Cl-indexed_Roe_Kubota_defn}
Let $A$ be an ungraded $\sigma$-unital Real $C^*$-algebra.
\begin{enumerate}
 \item We define the Roe isomorphism $\Roe^{r+1,s+1}_A: \KKR( \Cl_{s+1, r}, A) \to \DK( A \otimes \Cl_{r+1, s+1} ) $ as the composition 
 of isomorphisms
 \[
   \Roe_A^{r+1,s+1}:   \KKR( \Cl_{s+1, r}, A) \xrightarrow{\shuffle_{\KK}^{r,s+1} } \KKR(\C, A \otimes \Cl_{r,s+1} ) 
   \xrightarrow{ \Roe_{A \otimes \Cl_{r,s+1} } } \DK( A \otimes \Cl_{r+1,s+1}).
 \]
 \item We define the Kubota isomorphism $\Kubota_A^{r,s+1}: \KKR(\Cl_{r,s+1}, SA) \to \DK(A \otimes \Cl_{r,s+1})$, as the composition 
 of isomorphisms
 \[
   \Kubota_A^{r,s+1}:   \KKR( \Cl_{s+1, r}, SA) \xrightarrow{\shuffle_{\KK}^{r,s+1} } \KKR(\C, SA \otimes \Cl_{r,s+1} ) 
   \xrightarrow{ \Kubota_{A \otimes \Cl_{r,s+1} } } \DK( A \otimes \Cl_{r,s+1}).
 \]
\end{enumerate}
\end{defn}

Taking a minimal even projection $P^{r,s} \in \End_A(X)\otimes \Cl_{r,s+1}$ and using the maps 
\begin{align*}
       &\sigma_{\one \otimes \gamma}^{r,s}: \Reg^{s,r}_A(X \otimes \C^2)   \to \Reg_A(X) \otimes \Cl_{r+1,s+1}, \qquad \text{(Eq. \eqref{eq:sigma_e-rs_defn})},  \\
   &\sigma_{J_\mathrm{ref} \otimes \rho}^{r,s}: \Reg^{r,s}_A(X) \otimes \Cl_{0,1} \to \Reg_A(X) \otimes \Cl_{r,s+1}, \qquad \text{(Eq. \eqref{eq:skew_sigma^{r,s}})} , 
\end{align*}  
 the Roe and Kubota isomorphisms can be made explicit.
\begin{prop} \label{prop:Concrete_Roe_computation}
Let $X_A$ be full, and let $F \in \End_A^{r,s}(X)$ be Real and skew-adjoint such that $\one + F^2 \in \End_A^0(X)$. 
Let $\class{F} := [F\otimes\rho] \in \KKR(\Cl_{s+1,r},A)$ denote the resulting $\KK$-class (cf. Corollary \ref{cor:skew_Fred_to_KasMod}).
Then 
\begin{align*}
  \Roe^{r+1,s+1}_A\big( \class{F} \big) &= \calM_X^{\DK} \big( \big[ \sigma^{r,s}_{\one \otimes \gamma} \big( - \cosh( \pi \chi(T) ) \otimes \gamma - \sinh( \pi \chi(T) ) \otimes \rho \big) \big] - \big[ \one \otimes \gamma \big]  \big) \\
  &\hspace{-2.5cm} =  \calM_{X}^{\DK} \big( \big[ -P^{r,s}(\cosh(\pi \chi(T)) \otimes \gamma + \sinh( \pi \chi(T)) \otimes \rho ) - (\one - P^{r,s})(\one \otimes \gamma) \big] 
   - \big[ \one \otimes \gamma \big] \big)
\end{align*}
where $\calM_X^{\DK} : \DK( \End_A^0(X) \otimes \Cl_{r+1,s+1} ) \to \DK( A \otimes \Cl_{r+1,s+1} )$ is the Morita isomorphism.
\end{prop}
\begin{proof}
This statement was proven in Theorem \ref{thm:index-of-Fred}. 
\end{proof}

\begin{prop} \label{prop:concrete_Kubota_computation}
Let $F_\bullet \in \End_{SA}^{r,s}(SX)$ be skew-adjoint such that $\one + F_\bullet^2 \in \End_{SA}^0(SX)$. 
Assume also that $F_0, F_1 \in\End_A^{r,s}(X)$ are skew-adjoint unitaries 
with $F_0 - F_t \in \End_A^0(X)$ for all $t\in [0,1]$. 
The resulting $\KK$-class $\class{F_\bullet}\in \KK(\Cl_{s+1,r},SA)$ is such that
\begin{align*}
   {\Kubota}^{r,s+1}_A\big( \class{F_\bullet} \big) 
   &= \calM_{X}^{\DK}  \circ \exc_X \big( [\sigma^{r,s}_{F_0 \otimes \rho}(F_1\otimes \rho)] - [ F_0 \otimes \rho  ]\big)  \\
   &\hspace{-0cm}= \calM_{X}^{\DK}  \circ \exc_X \big(  \big[P^{r,s}( F_1 \otimes \rho) + (\one - P^{r,s})( F_0 \otimes \rho ) \big] - [ F_0 \otimes \rho  ]  \in \DK(A \otimes \Cl_{r,s+1} ),
\end{align*}
where $\calM_X^{\DK} : \DK( \End_A^0(X) \otimes \Cl_{r,s+1} ) \to \DK( A \otimes \Cl_{r,s+1} )$ is the Morita isomorphism.
\end{prop}

\begin{proof}
Recalling that $\Kubota_A^{r,s+1}$ uses the map $\Phi_{SA}^{r,s+1}$, we fix a stabilisation unitary $U \colon X_A \oplus \calH_A \to \calH_A$. 
As explained in the proof of Lemma \ref{lemma:Calkin_class}, we can fix a 
basepoint skew-adjoint unitary $J \in \End_{A}^{r,s}(\calH_A)$. 
We also have a skew-adjoint unitary $J_\mathrm{ref} = U ( F_0 \oplus J ) U^* \in \End_A^{r,s}(\calH_A)$, and 
the constant path $t \mapsto J_\mathrm{ref} =  U ( F_0 \oplus J ) U^*$  also gives a basepoint skew-adjoint unitary in $\End_{SA}^{r,s}(\calH_{SA})$.
Using Lemma \ref{lemma:Calkin_class}, we  compute 
\begin{align*}
   \Kubota^{r,s+1}_A\big( \class{F_\bullet} \big)  &= \calM^{\DK} \circ \partial \circ (\iota_\ast)^{-1} \circ \Phi_{SA}^{r,s+1} \big( \class{F_\bullet} \big) \\
     &=  \calM^{\DK} \circ \partial  \big( \big[ (q_{\calH_A} \otimes \one) \circ \sigma_{J_\mathrm{ref} \otimes \rho}^{r,s}  \big( U (F_\bullet \oplus J) U^* \otimes \rho \big)   \big]  
        - \big[ q_{\calH_A} \big(U (F_0 \oplus J) U^*  \big) \otimes \rho \big] \big) \\
     &=  \calM^{\DK}_{\calH_A} \circ \exc_{\calH_{A}} \big( \big[ \sigma^{r,s}_{J_\mathrm{ref}\otimes \rho} \big( U (F_1 \oplus J) U^* \otimes \rho \big)  \big] 
     - \big[    U (F_0 \oplus J) U^*  \otimes \rho \big] \big) \\
     &= \calM^{\DK}_{\calH_A} \circ \exc_{\calH_{A} } \circ (\mathrm{Ad}_{\nu})_\ast \big(  [\sigma^{r,s}_{F_0 \otimes \rho}(F_1\otimes \rho)] - [ F_0 \otimes \rho  ] \big),
\end{align*}
where we have used Lemma \ref{lemma:sigma^rs_commutes_with_boundary}, Proposition \ref{prop:suspension_bdry_graded},  
taken away a trivial direct summand, and where 
$\nu \colon X_A \hookrightarrow \calH_A$ is the isometry corresponding to the stabilisation isomorphism $U \colon X_A\oplus\calH_A \xrightarrow{\simeq} \calH_A$. 
We can then apply part (1) of Lemma \ref{lemma:DK_Morita_compat_new} to obtain 
\begin{align*}
  {\Kubota}^{r,s+1}_A\big( \class{F_\bullet} \big) &= \calM^{\DK}_{\calH_A} \circ \exc_{\calH_{A} } \circ (\mathrm{Ad}_{\nu})_\ast 
  \big(  [\sigma^{r,s}_{F_0 \otimes \rho}(F_1\otimes \rho)] - [ F_0 \otimes \rho  ] \big)  \\
&   = \calM^{\DK}_{X } \circ \exc_{X } \big(  [\sigma^{r,s}_{F_0 \otimes \rho}(F_1\otimes \rho)] - [ F_0 \otimes \rho  ] \big).  \qedhere
\end{align*}
\end{proof}

To relate  $\Roe^{r+1,s+1}_{A}$ with $\Kubota^{r,s+1}_{A}$, we use the Bott map $\beta_{\DK}$ in Van Daele $K$-theory (see Theorem \ref{thm:DK_Bott_isos})
and the alternative Bott map $\altBott_{\KK}$ in $\KK$-theory (see Corollary \ref{cor:KK_Bott_alternative}).

\begin{thm}[{cf. \cite[Page 22]{Kubota15a}}] 
\label{prop:Roe_Kubota_Bott_compatibility}
\begin{enumerate}
  \item The isomorphism $\Phi_B: \KKR(\C, B) \to \DK( \calQ_{B\hox \calK} )$ is 
  compatible with the Bott map in $\KKR$ and $\DK$ in the following sense:
  \[
       \Phi_{SB \hox \Cl_{1,0}}  \circ \beta_{\KK} = \iota_\ast \circ \beta_{\DK} \circ \Phi_B,
  \]
  where $\iota_\ast : \DK_{\one \hox \gamma}(S\calQ_{B \hox \calK} \hox \Cl_{1,0}) \xrightarrow{\simeq} \DK_{\one \hox \gamma}( \calQ_{SB \hox \calK} \hox \Cl_{1,0}  )$.
  \item The isomorphism $\Phi_A^{r,s+1}: \KKR(\Cl_{s+1,r}, A) \to \DK( \calQ_{A\hox \calK}\otimes \Cl_{r,s+1} )$ is 
  compatible with the Bott map in $\KKR$ and $\DK$ in the following sense:
  \[
     \calM_{1,1}^{\DK} \circ  \Phi_{SA}^{r,s}  \circ \altBott_{\KK} = \iota_\ast \circ \beta_{\DK} \circ \Phi_A^{r,s+1},
  \]
  where 
  $\iota_\ast : \DK(S\calQ_{A}(X) \hox \Cl_{r+1,s+1}) \xrightarrow{\simeq} \DK( \calQ_{SA}(SX) \hox \Cl_{r+1,s+1}  )$.

  \item The Roe isomorphism $\Roe^{r+1,s+1}_A: \KKR(\Cl_{s+1,r}, A) \to \DK(A \otimes \Cl_{r+1,s+1})$ is 
  compatible with the Bott map in the following sense:
   \[
     \calM_{1,1}^{\DK} \circ   \Roe^{r+1,s}_{SA} \circ \altBott_{\KK} = \beta_{\DK} \circ \Roe^{r+1,s+1}_A.
   \]

  \item The Kubota isomorphism $\Kubota_A^{r,s+1}: \KKR(\Cl_{s+1,r}, SA) \to \DK(A \otimes \Cl_{r,s+1})$ is compatible with 
  the Bott map and $\Roe^{r+1,s+1}_A$  in the following sense: 
  \[
   \calM_{1,1}^{\DK} \circ  \Kubota_{SA}^{r,s} \circ \altBott_{\KK} = \beta_{\DK} \circ \Kubota^{r,s+1}_A = \Roe^{r+1,s+1}_{SA}.
  \]
\end{enumerate}
\end{thm}
\begin{proof}
(1) The argument is given in~\cite[Page 22]{Kubota15a}, which we review for completeness. 
Given the class $\big[ (\C, \hat{\calH} \hox B, T ) \big]$, the external product with the Bott element can be represented by 
the Kasparov module 
\[
    \big( \C, \, \hat{\calH} \hox SB_{SB} \hox \Cl_{1,0}, \, \cos(\pi t) \hox \gamma +  T \sin(\pi t) \hox 1 \big),
\]
see~\cite[Proposition 9.2]{Wahl07} for example. We will compare $\beta_{\KK}([T])$ with 
$\Phi_{SB \hox \Cl_{1,0}}^{-1} \circ \iota_\ast \circ \beta_{\DK} \circ \Phi_B ([T])$, where 
\[
   \beta_{\DK} \circ \Phi_B([T]) = 
    \big[  \frac{1}{2} \mathrm{Ad}_{( \one + e\hox \gamma )} \circ \mathrm{Ad}_{ \nu( q(T), t)} \circ \mathrm{Ad}_{ \nu( e, -t)} (\one \hox \gamma) \big] 
    \in \DK_{\one \hox \gamma}( S\calQ_{B \hox \calK} \hox \Cl_{1,0} ).
\]
The map $\iota_\ast$ takes us to an element in $\DK_{\one \hox \gamma}( \calQ_{SB \hox \calK} \hox \Cl_{1,0} )$. The inverse map then takes 
us to the $\KK$-class
\begin{align*}
   \Phi_{SB \hox \Cl_{1,0}}^{-1} \circ \iota_\ast \circ \beta_{\DK} \circ \Phi_B ([T]) 
   &= \big[ \big( \C, \, \hat{\calH} \hox SB \hox \Cl_{1,0}, \,  \mathrm{Ad}_{ \nu( T, t)} \circ \mathrm{Ad}_{ \nu( e, -t)} (\one \hox \gamma) \big) \big] \\
   &= \big[ \big( \C, \, \hat{\calH} \hox SB \hox \Cl_{1,0}, \,  \mathrm{Ad}_{ \nu( T, t)} ( (\one \hox \gamma) \exp( \pi t ( e \hox \gamma) ) \big) \big] \\
   &= \big[ \big( \C, \, \hat{\calH} \hox SB \hox \Cl_{1,0}, \,  \mathrm{Ad}_{ \nu( T, t)}  (\one \hox \gamma)  \big) \big] \\
   &= \big[ \big( \C, \, \hat{\calH} \hox SB \hox \Cl_{1,0}, \,  \cos(\pi t) \hox \gamma + \sin(\pi t) T \hox \one \big) \big]  \\
   &= \beta_{\KK}([T]) ,
\end{align*}
as required.

(2)  We recall that $\Phi_{A}^{r,s+1} = \Phi_{A \otimes \Cl_{r,s+1}} \circ \shuffle_{\KK}^{r,s+1}$. 
The isomorphisms $\calM_{1,1}^{\DK}$ and $\calM_{1,1}^{\KK}$ are natural and so
$\calM_{1,1}^{\DK} \circ \Phi_B = \Phi_{B \hox \Cl_{1,1}} \circ \calM_{1,1}^{\KK}$. The Clifford shuffle commutes
 with the Bott map in $\KK$ and so 
 \begin{align*}
     \calM_{1,1}^{\DK} \circ  \Phi_{SA}^{r,s}  \circ \altBott_{\KK} &= 
      \calM_{1,1}^{\DK} \circ \Phi_{SA \otimes \Cl_{r,s} } \circ \shuffle_{\KK}^{r,s} \circ (\tau_{\Cl_{1,0}})^{-1} \circ \beta_{\KK} \\
       &= \Phi_{SA \otimes \Cl_{r+1,s+1}}  \circ \calM_{1,1}^{\KK} \circ \shuffle_{\KK}^{r,s} \circ (\tau_{\Cl_{1,0}})^{-1} \circ \beta_{\KK} \\
       &= \Phi_{SA \otimes \Cl_{r+1,s+1}}  \circ  \shuffle_{\KK}^{r,s+1}   \circ \beta_{\KK} \\
       &=  \Phi_{SA \otimes \Cl_{r+1,s+1}}  \circ \beta_{\KK}  \circ  \shuffle_{\KK}^{r,s+1}  \\
         &= \iota_\ast \circ \beta_{\DK} \circ \Phi_{A \otimes \Cl_{r,s+1} } \circ \shuffle_{\KK}^{r,s+1} \qquad \text{(part (1))} \\
      &= \iota_\ast \circ \beta_{\DK} \circ \Phi_A^{r,s+1},
 \end{align*}
 where the identity $ \calM_{1,1}^{\KK} \circ \shuffle_{\KK}^{r,s} \circ (\tau_{\Cl_{1,0}})^{-1} =  \shuffle_{\KK}^{r,s+1}$ is an easy check.

(3) Using part (2), we compute 
\begin{align*}
   \calM_{1,1}^{\DK} \circ  \Roe_{{SA}}^{r+1,s} \circ \altBott_{\KK}  &=   \calM_{1,1}^{\DK} \circ \calM^{\DK} \circ  \delta \circ \Phi_{SA}^{r,s} \circ \altBott_{\KK}  \\
   &=  \calM^{\DK} \circ  \delta  \circ  \calM_{1,1}^{\DK} \circ  \Phi_{SA}^{r,s} \circ \altBott_{\KK} \\
  &=    \calM^{\DK}\circ   \delta \circ \iota_\ast \circ \beta_{\DK} \circ \Phi_{A}^{r,s+1}  \qquad \text{(part (2))} \\
  &=    \calM^{\DK} \circ  \delta\circ \beta_{\DK} \circ \Phi_{A}^{r,s+1}  \qquad \text{(Eq. \eqref{eq:partial_delta_iota_compat})} \\
  &=    \calM^{\DK} \circ  \beta_{\DK} \circ \delta \circ \Phi_{A}^{r,s+1}  \qquad  \text{(Eq. \eqref{eq:Bott_commutator})} \\
  &=   \beta_{\DK} \circ \calM^{\DK} \circ \delta \circ \Phi_{A}^{r,s+1}  \qquad \text{(Lemma \ref{lemma:DK_Morita_compat_new}, part (2))} \\
  &= \beta_{\DK} \circ \Roe^{r+1,s+1}_{A} . \qquad \text{(Definition of $\Roe$)}.
\end{align*}

(4)  Analogously to part (3), we find that 
\begin{align*}
  \calM_{1,1}^{\DK} \circ  \Kubota_{SA}^{r,s} \circ \altBott_{\KK} &= \calM_{1,1}^{\DK} \circ \calM^{\DK} \circ   \partial \circ (\iota_\ast)^{-1} \circ \Phi_{S^2 A}^{r,s} \circ \altBott_{\KK}  \\
  &= \calM^{\DK} \circ   \partial \circ (\iota_\ast)^{-1}  \circ \calM_{1,1}^{\DK} \circ  \Phi_{S^2 A}^{r,s} \circ \altBott_{\KK} \\
   &= \calM^{\DK} \circ   \partial \circ (\iota_\ast)^{-1} \circ \iota_* \circ \beta_{\DK} \circ \Phi_{SA}^{r,s+1}  \qquad \text{(part (2))} \\
   &= \calM^{\DK} \circ \partial \circ \beta_{\DK} \circ \Phi_{SA}^{r,s+1} \\
   &= \calM^{\DK} \circ \delta \circ \Phi_{SA}^{r,s+1} = \Roe_{SA}^{r+1,s+1} \qquad \text{(Definition of $\delta$ and $\Roe$)},
\end{align*}
which we then compare to 
\begin{align*}
   \beta_{\DK} \circ \Kubota_A^{r, s+1}
   &= \beta_{\DK} \circ \calM^{\DK} \circ  \partial \circ (\iota_\ast)^{-1} \circ \Phi_{S A}^{r,s+1}  \\
   &= \calM^{\DK} \circ \beta_{\DK} \circ \partial \circ (\iota_\ast)^{-1} \circ \Phi_{S A}^{r,s+1}  \qquad \text{(Lemma \ref{lemma:DK_Morita_compat_new}, part (2))} \\
   &= \calM^{\DK} \circ \partial \circ \beta_{\DK}\circ (\iota_\ast)^{-1} \circ \Phi_{S A}^{r,s+1} \qquad \text{(Eq. \eqref{eq:Bott_commutator})} \\
   &= \calM^{\DK} \circ \delta \circ   (\iota_\ast)^{-1} \circ \Phi_{S A}^{r,s+1} = \Roe_{SA}^{r+1,s+1}  
   \qquad \text{(Definition of $\delta$, $\Roe$, and Eq. \eqref{eq:partial_delta_iota_compat})} 
\end{align*}
and we are done.
\end{proof}

Using part (1) of Theorem \ref{prop:Roe_Kubota_Bott_compatibility}, the same argument as in the proof of parts (3) and (4)  (using $\beta_{\KK}$ 
rather than $\altBott_{\KK}$)
also shows 
\begin{equation} \label{eq:Graded_Bott_compat}
    \Roe_{SB \hat\otimes \Cl_{1,0} } \circ \beta_{\KK} =\beta_{\DK} \circ \Roe_{B} , \qquad \qquad 
    \Kubota_{SB \hat\otimes \Cl_{1,0}} \circ \beta_{\KK} = \beta_{\DK} \circ \Kubota_B  = \Roe_{SB}
\end{equation}
for any $\Z_2$-graded algebra $B$.

\subsection{The Cayley isomorphism} \label{subsec:Cayley}

We  complete our discussion by describing an isomorphism between $\KKR$-theory and $\DK$-theory via the Cayley transform studied in~\cite{BKR}. 
This isomorphism is well-adapted to computations with unbounded Kasparov modules. 
We first review the isomorphism for graded algebras.

\begin{thm}[{\cite[Theorem 4.15]{BKR}}] \label{thm:Graded_Cayley_iso}
Let $\big( \Cl_{1,0}, Y_B, D\big)$ be an unbounded Kasparov module such that $D=D^*=D^\rs$ and 
$D$ anti-commutes with the odd self-adjoint unitary $e \in \End_B(Y)$ that generates the $\Cl_{1,0}$-representation. 
There is an isomorphism 
\begin{align*}
  &\mathfrak{C}_B: \KKR(\Cl_{1,0}, B) \to  \DK( B ), 
  \end{align*}
  which on the class of the cycle $\big( \Cl_{1,0}, Y_B, D\big)$ is given by
  \begin{align*}
  &\mathfrak{C}_B\big( [ ( \Cl_{1,0}, \, Y_B, \, D ) ]\big) = \calM_{Y}^{\DK} \circ \exc_Y\big( \big[ e(D+e)(D-e)^{-1} \big] - [e] \big).
\end{align*}
If $B$ is balanced graded and $x, y \in M_n(B)$ are odd self-adjoint unitaries, then the inverse isomorphism 
$\mathfrak{C}_B^{-1}: \DK(B) \to \KKR(\Cl_{1,0}, B)$ is given by 
\[
    \mathfrak{C}_B^{-1}\big( [x] - [y] \big) = \big[ \big( \Cl_{1,0}, \, \ol{(x-y)B_B^n}, \, y( x+y)(x-y)^{-1} \big) \big],	
\]
where the $\Cl_{1,0}$-representation is generated by $y$.
\end{thm}

\begin{remarks} \label{remarks:Cayley_subtleties}
\begin{enumerate}
  \item The inverse map $\mathfrak{C}_B^{-1}$ can also be explicitly written down when $B$ is not balanced graded, though it is slightly more involved, 
see~\cite[Lemma 4.10]{BKR}.
 \item   If the $C^*$-module $Y_B$ is not full, then   $\calM_{Y}^{\DK} \circ \exc_Y\big( \big[ e(D+e)(D-e)^{-1} \big] - [e] \big) \in \DK( I_Y)$, 
 where $I_Y = \mathrm{span}\{ \scal<y_1|y_2> \mid y_1,y_2\in Y \}$ is an ideal in $B$. In such a case, we instead define 
 $\mathfrak{C}_B = \iota_\ast \circ \calM_{Y}^{\DK} \circ \exc_Y\big( \big[ e(D+e)(D-e)^{-1} \big] - [e] \big)$ with $\iota: I_Y \hookrightarrow{B}$. 
 \item  The maps $\calC_e(D) = e(D+e)(D-e)^{-1}$ and $\wt{\calC}_y(x) = y( x+y)(x-y)^{-1}$ are not complete inverses of each other. 
 If $x, y \in B$ are OSUs, then  $\calC_y\big( \wt{\calC}_y( x) \big)$ is the restriction of $x$ to $I = \ol{B(x-y)B}$, the ideal 
 arising from the  $C^*$-module  $\ol{(x-y)B_B}$, which is often not full.
 However, $\iota_\ast \big( [x|_{I} ] - [y|_{I} ] \big) = [ x] - [y] \in \DK(B)$ by~\cite[Lemma 4.9]{BKR} and  we  
 recover all relevant $K$-theoretic information. 
\end{enumerate}
\end{remarks}

We have two aims for the remainder of this appendix. The first is
to define the Cayley isomorphism for Clifford anti-linear Fredholm operators. 
The second is to show that the Roe and Cayley isomorphisms agree.

Let $T \in \Reg_A^{r,s}(X)$ be an unbounded Real skew-adjoint Fredholm operator on a full $C^*$-module $X_A$, with an 
ample $\Cl_{r,s}$-representation generated by $\{e_1,\ldots,e_r,f_1,\ldots,f_s\}$. 
Then Corollary \ref{cor:skew_Fred_to_KasMod} provides the Real Kasparov module 
  \[
     \Big( \Cl_{s+1, r}, \, X_A \otimes \exterior \C, \, \chi(T) \otimes \rho \Big)
  \]
with the $\Cl_{s+1, r}$-representation generated by 
$\big\{ \one \otimes \gamma, f_1 \otimes \rho, \ldots, f_s \otimes \rho, e_1 \otimes \rho, \ldots, e_r \otimes \rho \big\}$. 
We denote the corresponding $\KK$-class by $\class{T} \in \KKR( \Cl_{s+1, r}, A)$. 
If $(T+ \one)^{-1} \in \End_A^0(X)$, then $\big( \Cl_{s+1, r},  \, X_A \otimes \exterior \C, \, T \otimes \rho \big)$ is an unbounded
Kasparov module representing the class $\class{T}$. 
We will adapt the graded Cayley transform of Theorem \ref{thm:Graded_Cayley_iso} to unbounded Kasparov modules of this form. 
Since $A$ is an ungraded $C^*$-algebra, we   replace $A$ by $A \otimes \Cl_{1,1}$ by the 
Clifford stability isomorphism $\calM_{1,1}^{\KK}$ of Eq.\ \eqref{eq:Clifford_Morita_KK}.
We also need to move the Clifford algebra $\Cl_{s,r}$ via the Clifford shuffle isomorphism $\shuffle_{\KK}^{r,s}$ of Eq.\ \eqref{eq:Clifford_shuffle}. 

\begin{defn}
\label{defn:cayley-r-s}
We define $\mathfrak{C}^{r+1,s+1}_A:\KKR(\Cl_{s+1,r},A)\to \DK(A\ox\Cl_{r+1,s+1})$ by the composition 
\[
\mathfrak{C}^{r+1,s+1}_A=\mathfrak{C}_{A\ox\Cl_{r+1,s+1}}\circ \shuffle^{r,s}_{\KK} \circ \calM_{1,1}^{\KK} .
\]
\end{defn}

To explicitly compute  $\mathfrak{C}^{r+1,s+1}_A$, we will again use the map 
\[
   \sigma_e^{r,s}: \Reg^{s,r}_B(Y) \ni x \mapsto P^{r,s}(x \hox 1_{r,s} ) + (\one - P^{r,s})(e \hox 1_{r,s} ) \to \Reg_B(Y) \hox \Cl_{r,s},
\]
which sends odd self-adjoint unitaries to odd self-adjoint unitaries. 

\begin{prop} \label{prop:CliffordCayley_computation}
Let $T \in \Reg_A^{r,s}(X)$ be an unbounded Real skew-adjoint Fredholm operator.
\begin{enumerate}
  \item The operator $U_T = (T-\one)(T+\one)^{-1}$ is a  Real unitary in $\End_A(X)$ such that $e_j U_T = U_T^* e_j$ and 
$f_k U_T = U_T^* f_k$ for all $j=1,\ldots,r,\,k=1,\ldots,s$.
  \item The odd self-adjoint unitaries $\tfrac12( U_T + U_T^* ) \otimes \gamma + \tfrac12( U_T - U_T^* ) \otimes \rho, \, \one \otimes \gamma \in \End_A(X) \otimes \Cl_{1,1}$ 
  anti-commute with the  $\Cl_{s,r}$-generators $\{f_1\otimes \rho, \ldots, e_r \otimes \rho\}$.  The Van Daele class 
  \[
     \calM_{X }^{\DK} \circ \exc_{X } \big(  \big[ \sigma^{r,s}_{1 \otimes \gamma} \big(\tfrac12( U_T + U_T^* ) \otimes \gamma + \tfrac12( U_T - U_T^* ) \otimes \rho\big)  \big] - [ \one \otimes\gamma] \big)  \in \DK(A \otimes \Cl_{r+1,s+1} )
  \]
    is well-defined and, if $T$ has compact resolvent, coincides with $\mathfrak{C}_A^{r+1,s+1}( \class{T} )$.
\end{enumerate}
\end{prop}
\begin{proof}
Part (1) mostly follows from the discussion in \cite[Example 5.1]{BKR}. We  check  the anti-commutation with Clifford generators, where 
\begin{align*}
   f_k ( T - \one)(T+ \one)^{-1} &= (-T - \one)f_k (T + \one)^{-1} = (-T - \one)(-T + \one)^{-1} f_k \\ 
   &= (T+ \one)(T-\one)^{-1} f_k = U_T^* f_k.
\end{align*}
The same argument will show that $e_j U_T = U_T^* e_j$. 

For part (2), the properties $e_j U_T = U_T^* e_j$ and $f_k U_T = U_T^* f_k$ imply that 
the odd self-adjoint unitary  $\tfrac12( U_T + U_T^* ) \otimes \gamma + \tfrac12( U_T - U_T^* ) \otimes \rho \in \End_A(X) \otimes \Cl_{1,1}$ anti-commutes 
with $\{f_1\otimes \rho, \ldots, e_r \otimes \rho\}$, the generators of a graded $\Cl_{s,r}$-representation. 
The identity $\one - U_T  =   (T+ \one - T + \one) (T+\one)^{-1}  = 2  (T+ \one)^{-1}$ implies that 
\[
\big\| q_X \big(\one -  U_T  \big) \big\|_{\calQ_{A}(X)}  = 2  \big\| q_X\big( (T + \one)^{-1} \big) \big\|_{\calQ_{A}(X)} < 2
\]
as $T$ is skew-adjoint and Fredholm (Lemma \ref{lemma:bdd_unbdd_Fred_condition}). This inequality can then be used 
to show that 
\[
  \big\|  \tfrac12 q_X( U_T + U_T^* ) \otimes \gamma + \tfrac12 q_X( U_T - U_T^* ) \otimes \rho  - q_X( \one) \otimes \gamma \big\|_{\calQ_{A}(X) \otimes \Cl_{1,1} } < 2,
\]
which in turn implies that the Van Daele class 
\[
\exc_{X } \big(  \big[ \sigma^{r,s}_{1 \otimes \gamma} \big(\tfrac12( U_T + U_T^* ) \otimes \gamma + \tfrac12( U_T - U_T^* ) \otimes \rho\big)  \big] - [ \one \otimes\gamma] \big)  \in \DK( \End_A^0(X) \otimes \Cl_{r+1,s+1} )
\]
is well-defined by Lemma \ref{lem:relative_DK}.

Suppose now that $T$ has compact resolvent and $\class{T}$ is represented by an unbounded Kasparov module. 
Applying the graded map from Definition \ref{defn:cayley-r-s}, $\mathfrak{C}_A^{r+1,s+1}(\class{T})$ is given by 
\begin{align*}
      \calM_{X }^{\DK} \circ \exc_{X  } \big( \big[ \sigma_{\one \otimes \gamma}^{r,s}\big( (\one \otimes \gamma)( T \otimes \rho + \one \otimes \gamma)( T \otimes \rho - \one \otimes \gamma)^{-1} \big) \big] 
     - [ \one \otimes \gamma ] .
\end{align*}
Using the isomorphism $\Cl_{1,1} \cong M_2(\C)$ with $\gamma \cong \sigma_1$ and $\rho \cong -i \sigma_2$, we have that 
\begin{align}
   (\one \otimes \gamma)(T \otimes \rho + \one \otimes \gamma)(T \otimes \rho - \one \otimes \gamma)^{-1} 
   &= 
      \begin{pmatrix} 0 & \one \\ \one & 0 \end{pmatrix} \begin{pmatrix} 0 & -T + \one \\ T + \one & 0 \end{pmatrix} 
       \begin{pmatrix} 0 & -T - \one \\ T - \one & 0 \end{pmatrix}^{-1}  \nonumber \\
     &= \begin{pmatrix} 0 & (T+\one)(T-\one)^{-1} \\ (T-\one)(T+\one)^{-1} & 0 \end{pmatrix} \nonumber\\
     &= \begin{pmatrix} 0 & U_T^* \\ U_T & 0 \end{pmatrix}  \nonumber \\
     &= \tfrac12( U_T + U_T^* ) \otimes \gamma + \tfrac12( U_T - U_T^* ) \otimes \rho,  \label{eq:form-real}   
\end{align}
which gives the result.
\end{proof}

\begin{remark} \label{remark:Cayley_subtleties_continue}
If $\big( \Cl_{r,s+1}, \, \tilde{X}_A \otimes \exterior{\C} , \, T \otimes \rho\big)$ is an unbounded Kasparov module, where 
$\tilde{X}_A$ is \emph{not} full, then taking the ideal $I_{\tilde{X}} = \mathrm{span}\{ \scal<x_1|x_2> \mid x_1,x_2\in \tilde{X} \}$ and 
$\iota: I_{\tilde{X}} \hookrightarrow A$, we have 
\begin{align*}
  & \mathfrak{C}^{r+1,s+1}_A( \class{T} ) 
  =   \iota_\ast \circ   \calM_{X }^{\DK} \circ \exc_{X } \big(  \big[ \sigma^{r,s}_{1 \otimes \gamma} \big(\tfrac12( U_T + U_T^* ) \otimes \gamma + \tfrac12( U_T - U_T^* ) \otimes \rho\big)  \big] - [ \one \otimes\gamma] \big)  
\end{align*}
cf. Remarks \ref{remarks:Cayley_subtleties}.
\end{remark}

\begin{cor}
Let $A$ be a Real ungraded $C^*$-algebra such that $\{e_1,\ldots, e_r, f_1,\ldots, f_s\} \subset \Mult(A)$ are generators of an ungraded $\Cl_{r,s}$-representation. 
Then every element in $\DK( A \otimes \Cl_{r+1,s+1})$ can be represented by an equivalence class 
\[
  \exc_A \big( \big[ \sigma_{1 \otimes \gamma}^{r,s} \big( \tfrac12( V + V^* ) \otimes \gamma + \tfrac12( V - V^* ) \otimes \rho \big) \big] - \big[ \one \otimes \gamma \big] \big),
\]
where $V= V^\rs \in \Mult(A)$ is unitary, $e_j V = V^* e_j$, and $f_k V = V^* f_k$ for all 
$j=1,\ldots, r, \, k=1,\ldots, s$ and $\exc_A: \DK( \Mult(A) \otimes \Cl_{r+1,s+1}, \calQ_A \otimes \Cl_{r+1,s+1} ) \xrightarrow{\simeq} \DK(A \otimes \Cl_{r+1,s+1} )$ 
is the excision isomorphism.
\end{cor}
\begin{proof}
The map $\mathfrak{C}_A^{r+1,s+1}$ is an isomorphism and its range are equivalence classes of the form given by the  statement. 
By part (2) of Lemma \ref{lem:AMZ-OSU-iso}, any element in $\DK( A \otimes \Cl_{r+1,s+1})$ can be respresented by equivalence classes 
of odd self-adjoint unitaries (OSUs) in $M_k\big( P^{r,s}( \Mult(A) \hox \Cl_{r+1,s+1} )P^{r,s} \big)$. Furthermore, part (1) of 
 Lemma \ref{lem:AMZ-OSU-iso} says that there is a bijective correspondence 
between OSUs in $P^{r,s}( \Mult(A) \hox \Cl_{r+1,s+1} )P^{r,s}$ and $\Cl_{s,r}$-anti-linear OSUs in $\Mult(A) \hox \Cl_{1,1}$. 
Any OSU in $\Mult(A) \hox \Cl_{1,1}$ is of the form 
$x = \tfrac12( V + V^* ) \otimes \gamma + \tfrac12( V - V^* ) \otimes \rho \in \Mult(A) \otimes \Cl_{1,1}$ 
with $V = V^\rs$ unitary. Taking the $\Cl_{s,r}$-generators as $\{f_1 \otimes \rho, \ldots, f_s \otimes \rho, e_1 \otimes \rho, \ldots, e_r \otimes \rho \}$,  
the requirement that $x$ is $\Cl_{s,r}$-anti-linear is equivalent to the condition that $e_j V = V^* e_j$ and $f_k V = V^* f_k$ for all 
$j=1,\ldots, r, \, k=1,\ldots, s$. Applying the map $\sigma_{\one \otimes \gamma}^{r,s}$, the result follows from part (3) of Lemma \ref{lem:AMZ-OSU-iso}.
\end{proof}

We also remark that by considering $M_n(A)$ instead of $A$ for $n$ sufficiently large, we can guarantee that 
$M_n(\Mult(A))$ contains the ungraded $\Cl_{r,s}$-generators $\{e_1,\ldots, e_r, f_1,\ldots, f_s \}$.

\begin{prop} \label{prop:CliffordCayley_inverse_representative}
Let $V=V^\rs$ be a unitary in $\Mult(A)$ such that $e_j V = V^* e_j$ and $f_k V = V^* f_k$ for all 
$j=1,\ldots, r, \, k=1,\ldots, s$. Define the operator 
\[
   \calC^u(V) : \Ran( V- \one) \to A_A, \qquad \calC^u(V) = - (V + \one)(V- \one)^{-1}.
\]
Then $\calC^u(V)$ extends to a regular, Real and skew-adjoint operator on $\ol{(V-\one)A}_A$ that anti-commutes 
with $\{e_1,\ldots, e_r, f_1,\ldots, f_s\}$. If $\one - V \in A$, then 
\[
\Big( \Cl_{s+1,r}, \  \ol{(V-\one)A}_A \otimes \exterior \C, \  \calC^u(V) \otimes \rho \Big) 
\]	
is an unbounded Kasparov module 
whose class $\class{\calC^u(V)} \in \KKR( \Cl_{s+1, r}, A)$ represents the isomorphism 
$(\mathfrak{C}_A^{r+1,s+1})^{-1}\big( \big[ \sigma_{\one \otimes \gamma}^{r,s}\big( \tfrac12(V+V^*) \otimes \gamma + \tfrac12(V-V^*) \otimes \rho \big) \big]  - [\one \otimes \gamma]  \big)$.
\end{prop}
\begin{proof}
The inverse Cayley transform from Theorem \ref{thm:Graded_Cayley_iso} applied to the $\Cl_{s,r}$-anti-linear odd self-adjoint unitaries 
$\tfrac12(V+V^*) \otimes \gamma + \tfrac12(V-V^*) \otimes \rho$ and $\one \otimes \gamma \in \Mult(A) \otimes \Cl_{1,1}$  gives the self-adjoint operator 
\begin{align*}
  &(\one \otimes \gamma) \big( \tfrac12(V+V^*) \otimes \gamma + \tfrac12(V-V^*) \otimes \rho + \one\otimes \gamma\big) 
  \big( \tfrac12(V+V^*) \otimes \gamma + \tfrac12(V-V^*) \otimes \rho - \one\otimes \gamma\big)^{-1} \\
  &\qquad = - (V+\one)(V-\one)^{-1} \otimes \rho.
\end{align*}
Once again, this equality is easiest understood via the isomorphism $\Cl_{1,1} \cong M_2(\C)$ with 
$\one\otimes \gamma \cong \sigma_1$, 
$\one \otimes \rho \cong -i\sigma_2$, and 
\begin{align*}
  &(\one \otimes \gamma) \big( \tfrac12(V+V^*) \otimes \gamma + \tfrac12(V-V^*) \otimes \rho + \one\otimes \gamma\big) 
  \big( \tfrac12(V+V^*) \otimes \gamma + \tfrac12(V-V^*) \otimes \rho - \one\otimes \gamma\big)^{-1} \\
  &\qquad \qquad = \begin{pmatrix} 0 & \one \\ \one & 0 \end{pmatrix} \begin{pmatrix} 0 & V^*+\one \\ V+\one & 0 \end{pmatrix} 
  \begin{pmatrix} 0 & (V-\one)^{-1} \\ (V^*- \one)^{-1} & 0 \end{pmatrix}  \\
  &\qquad \qquad =  \begin{pmatrix} 0 & (V+\one)(V-\one)^{-1} \\ (V^*+\one)(V^*-\one)^{-1} & 0 \end{pmatrix} \\
  &\qquad \qquad =  \begin{pmatrix} 0 & (V+\one)(V-\one)^{-1} \\ -(V+\one)(V-\one)^{-1} & 0 \end{pmatrix} 
  = \calC^u(V) \otimes \rho.
\end{align*}
By~\cite[Lemma 4.5]{BKR}, $\calC^u(V) \otimes \rho$ is self-adjoint and regular, which implies that $\calC^u(V)$ is 
skew-adjoint and regular on $\ol{(V-\one)A}_A$. Because $e_j V = V^* e_j$ and $f_k V = V^* f_k$, the Clifford 
representation $\Cl_{r,s}$ in $\Mult(A) =\End_A(A)$ is also well-defined on $\ol{(V-\one)A}_A$ as 
\[
    f_k (V-\one)\psi = (V^*-\one) f_k \psi = (\one-V) V^* f_k\psi, \quad \psi \in A_A,
\]
and $f_k$ (and $e_j$) preserve $\Ran(V-\one)$. Similarly, 
\begin{align*}
  f_k (V+\one)(V-\one)^{-1} = (V^*+ \one)(V^*-\one)^{-1} f_k = - (V+\one)(V-\one)^{-1} f_k,
\end{align*}
which shows that  $\calC^u(V) \in \Reg^{r,s}_A\big( \ol{(V-\one)A} \big)$.
 If $\one - V \in A$ then 
\begin{align*}
  \big( \one + \calC^u(V) \big)^{-1} = \big( (V-\one -(V+\one))(V-\one)^{-1} \big)^{-1} 
  = - \tfrac{1}{2} (V-\one) \in A
\end{align*}
and $\calC^u(V)$ has compact resolvent. Applying Corollary \ref{cor:skew_Fred_to_KasMod},
$\big(  \Cl_{s+1,r}, \, \ol{(V-\one)A}_A \otimes \exterior \C, \,  \calC^u(V) \otimes \rho \big)$ is an unbounded 
Kasparov module. 
To show $\class{\calC^u(V)}$ represents $(\mathfrak{C}_A^{r+1,s+1})^{-1}$, we apply Proposition 
\ref{prop:CliffordCayley_computation} to this unbounded Kasparov module. However, $\tilde{X}_A = \ol{(V-\one)A}_A$ is 
not full in general, so using Remark \ref{remark:Cayley_subtleties_continue},
\[
     \mathfrak{C}_A^{r+1,s+1} \big( \class{ \calC^u(V) } \big) = 
     \iota_\ast \circ \calM^{\DK}_{\tilde{X}} \circ  \exc_{\tilde{X}} \big( 
     \big[  \sigma_{\one \otimes \gamma}^{r,s}(\tfrac12( U_{\calC^u(V)} + U_{\calC^u(V)}^* ) \otimes \gamma + \tfrac12( U_{\calC^u(V)} - U_{\calC^u(V)}^* ) \otimes \rho) \big] - [\one \otimes \gamma] \big)  
\]
with $\iota: I_{V} = \ol{A(V - \one)A} \hookrightarrow A$.
We can then check directly 
\begin{align} \label{eq:Cayley_of_cayley_identity}
   U_{\calC^u(V)} &= (\calC^u(V) - \one)( \calC^u(V) + \one)^{-1}  \nonumber \\
   &= \big( (-V -\one -  V + \one)(V-\one)^{-1} \big) \big( (-V - \one + V - \one )(V- \one)^{-1} \big)^{-1} 
     =  V|_{\tilde{X}} 
\end{align}
and so 
\begin{align*}
    \mathfrak{C}_A^{r+1,s+1} \big( \class{ \calC^u(V) } \big) 
    &=    \iota_\ast \circ (\sigma_{\one \otimes \gamma}^{r,s})_\ast  \big(   \big[   \tfrac12( V|_{I_V}  + (V|_{I_V})^* ) \otimes \gamma + \tfrac12( V|_{I_V} - (V|_{I_V})^* ) \otimes \rho \big] - [\one \otimes \gamma] \big)  \\
    &= (\sigma_{\one \otimes \gamma}^{r,s})_\ast \circ \iota_\ast  \big(   \big[   \tfrac12( V|_{I_V}  + (V|_{I_V})^* ) \otimes \gamma + \tfrac12( V|_{I_V} - (V|_{I_V})^* ) \otimes \rho ) \big] - [\one \otimes \gamma] \big) \\
    &= \big[ \sigma_{\one \otimes \gamma}^{r,s}(\tfrac12(V+V^*) \otimes \gamma + \tfrac12(V-V^*) \otimes \rho) \big] - [\one \otimes \gamma]
\end{align*}
where the last line is~\cite[Lemma 4.9]{BKR}.
\end{proof}

We will now compare the Cayley isomorphism $\mathfrak{C}_A^{r+1,s+1}$ with the Roe isomorphism 
$\Roe_A^{r+1,s+1}$ from Definition \ref{def:Indexed_KK_isomorphisms}.
As a first step, we recall the following.

\begin{lemma}[{\cite[Proposition 5.9]{BKR}}]  \label{lemma:bdry_cayley_nice}
Let $e \in \End_B(\hat{\calH}_B)$ be an odd self-adjoint unitary, and let $x \in \End_B(\hat{\calH}_B)$ be odd, self-adjoint, and such that
$q(x) \in \calQ_{B \hox \calK}$ is an odd self-adjoint unitary. Then the composition of isomorphisms
\begin{align*}
    &\DK( \calQ_{B \hox \calK}) \xrightarrow{\delta} \DK( B \hox \calK \hox \Cl_{1,0}) \xrightarrow{ \mathfrak{C}^{-1}_{B \hat\otimes \calK \hat\otimes \Cl_{1,0}} } 
   \KKR( \Cl_{1,0}, B \hox \calK \hox \Cl_{1,0} ) \\
    &\qquad  \xrightarrow{\tau_{\Cl_{1,0}}^{-1}} \KKR( \C,  B \hox \calK) \xrightarrow{\calM^{\KK}_{\hat\calH_B}} \KKR(\C, B),
\end{align*}
maps $[q(x)] - [q(e)] \in \DK(\calQ_{B \hox \calK})$ to the equivalence class of the Kasparov module 
$\big( \C, \, \hat{\calH}_B, \, x \big)$. That is, the composition computes $\Phi_{B}^{-1}$ from Lemma \ref{lemma:Phi}.
\end{lemma}
\begin{proof}
The cited result shows that the composition $\tau_{\Cl_{1,0}}^{-1} \circ \mathfrak{C}^{-1}\circ \delta( [q(x)] - [q(e)] )$ is represented 
by the Kasparov module $\big( \C, \, (B \hox \calK)_{B \hox \calK}, \, x \big)$. The Kasparov product with 
$\big( B \hox \calK, \,  \hat{\calH}_B, \, 0 \big)$ implements $\calM^{\KK}$ and gives the result.
\end{proof}

\begin{thm} \label{thm:Roe_and_Cayley_comparison}
The maps $\Roe_{A}^{r+1,s+1}: \KKR(\Cl_{s+1,r}, A) \to \DK(A \otimes \Cl_{r+1, s+1} )$ and 
$\mathfrak{C}_{A}^{r+1,s+1}: \KKR(\Cl_{s+1,r}, A) \to \DK(A \otimes \Cl_{r+1, s+1} )$ are the same isomorphism.
\end{thm}
\begin{proof}
The result will follow if we can show that
\[
   (\mathfrak{C}_{A}^{r+1,s+1})^{-1} \circ  \Roe_A^{r+1,s+1} = \mathrm{Id}, \qquad 
    \Roe_A^{r+1, s+1} \circ (\mathfrak{C}_{A}^{r+1,s+1})^{-1} = \mathrm{Id}
\]
with $\mathrm{Id}$ the identity automorphism. It also suffices to work with the standard module 
and we fix the class 
$\class{F} = \big[ \big( \Cl_{s+1,r}, \, \calH_A \otimes \exterior \C, \, F \otimes \rho \big) \big] \in \KKR(\Cl_{s+1,r}, A)$
with $F=F^\rs = -F^* \in \End^{r,s}_A(\calH_A)$ such that $\one + F^2 \in \End_A^0(\calH_A)$.
Therefore if we fix a basepoint skew-adjoint unitary $J_0 \in \End_A^{r,s}(\calH_A)$, then 
\begin{align*}
   (\mathfrak{C}_{A}^{r+1,s+1})^{-1} \circ  \Roe_A^{r+1,s+1} (\class{F})
     &= (\shuffle_{\KK}^{r,s})^{-1} \circ \big(\calM_{1,1}^{\KK}\big)^{-1} \circ  \mathfrak{C}_{A \otimes \Cl_{r+1,s+1}}^{-1} \circ \calM_{\hat\calH_B}^{\DK} \circ \delta \circ \Phi_A^{r,s+1} (\class{F} ) \\
     &\hspace{-2cm}= (\shuffle_{\KK}^{r,s})^{-1} \circ \big(\calM_{1,1}^{\KK}\big)^{-1}  \circ  \calM_{\hat\calH_B}^{\KK} \circ \mathfrak{C}^{-1}_{A \otimes \calK \otimes \Cl_{r+1,s+1}} \circ \delta 
     \circ \Phi_{A \otimes \Cl_{r,s+1}} \circ \shuffle_{\KK}^{r,s+1} ( \class{F} ) \\
     &\hspace{-2cm}= (\shuffle_{\KK}^{r,s+1})^{-1} \circ \calM_{\hat\calH_B}^{\KK}  \circ \tau_{\Cl_{1,0}}^{-1} \circ \mathfrak{C}^{-1}_{A \otimes \calK \otimes \Cl_{r+1,s+1}} \circ \delta 
     \circ \Phi_{A \otimes \Cl_{r,s+1}} \circ \shuffle_{\KK}^{r,s+1}( \class{F} ) \\
     &\hspace{-2cm}= (\shuffle_{\KK}^{r,s+1})^{-1} \circ \Phi_{A \otimes \Cl_{r,s+1}}^{-1} \circ   \Phi_{A \otimes \Cl_{r,s+1}} \circ \shuffle_{\KK}^{r,s+1}( \class{F} ) \\
     &\hspace{-2cm} = \class{F},
\end{align*}
where we have used that $\calM_{\hat\calH_B}^{\KK} \circ \mathfrak{C}_{B\hox \calK}^{-1} = \mathfrak{C}_B^{-1} \circ \calM_{\hat\calH_B}^{\DK}$ 
from~\cite[Lemma 4.13]{BKR}, 
 $(\shuffle_{\KK}^{r,s})^{-1} \circ \big(\calM_{1,1}^{\KK}\big)^{-1}  =  (\shuffle_{\KK}^{r,s+1})^{-1}  \circ \tau_{\Cl_{1,0}}^{-1}$, 
and Lemma \ref{lemma:bdry_cayley_nice}.

For the other direction, we take $V \in A+ \one$ unitary with 
$e_j V= V^* e_j$ and $f_k V = V^* f_k$ for all $j,k$.  
Applying $(\mathfrak{C}_{A}^{r+1,s+1})^{-1}$ we have the Fredholm operator $\calC^u(V)$ on $\tilde{X}_A = \ol{(V-\one)A}_A$, which is not full in general. 
To apply Proposition \ref{prop:Concrete_Roe_computation}, 
we let $I_V = \ol{A(V- \one)A}$ and $\iota: I_V\hookrightarrow A$ (cf. Remarks \ref{remarks:Cayley_subtleties}). 
Then for an appropriate normalising function $\chi$, Proposition \ref{prop:Concrete_Roe_computation} and Lemma \ref{lem:index-sinh-cosh} give that
\begin{align*}
   &\Roe_A^{r+1, s+1} \circ (\mathfrak{C}_{A}^{r+1,s+1})^{-1} \Big( \big[ \sigma_{\one \otimes \gamma}^{r,s}(\tfrac12(V+V^*) \otimes \gamma + \tfrac12(V-V^*) \otimes \rho) \big] - [\one \otimes \gamma] \Big) \\
   &\hspace{2cm} = 
   \Roe_A^{r+1, s+1} \Big( \big[ \big(  \Cl_{s+1, r}, \, \ol{(V-\one)A}_A \otimes \exterior \C, \, \calC^u(V) \otimes \rho \big) \big] \Big) \\
   &\hspace{2cm} = \iota_\ast \circ \calM^{\DK}_{\tilde{X}} 
    \Big(    \big[ \sigma_{\one \otimes \gamma}^{r,s}(- \cosh\big( \pi \chi(\calC^u(V)) \big) \otimes \gamma - \sinh\big( \pi \chi(\calC^u(V)) \big) \otimes \rho )\big] - [ \one \otimes \gamma] \Big).
\end{align*}
Adapting \cite[Corollary 3.6]{BKR} to the skew-adjoint setting, we have for $ix\in i\R$,
\[
(u+1)(u-1)^{-1}=x,\qquad u=e^{-2i\tan^{-1}(ix)+i\pi}=-e^{-2i\tan^{-1}(ix)}=(x+1)(x-1)^{-1}.
\]
So for any normalising function $\chi:i\R\to[-i,i]$, $\pi\chi(ix)-2i\tan^{-1}(ix)\in C_0(i\R)$ and we have a homotopy
\[
u_t=-e^{(1-t)\pi \chi(ix)-2it\tan^{-1}(ix)},
\]
which gives a homotopy between 
 $-\exp( \pi \chi(\calC^u(V)) )$ and $(\calC^u(V) - \one)( \calC^u(V) + \one)^{-1} \in \End_A^0(\tilde{X})^\sim$.
Recalling Eq. \eqref{eq:Cayley_of_cayley_identity}, $(\calC^u(V) - \one)( \calC^u(V) + \one)^{-1} =V|_{\tilde{X}}$.  
Letting $W_{\calC^u(V)} = e^{ \pi \chi(\calC^u(V)) }$, we find that
\begin{align*}
  & \iota_\ast \circ \calM^{\DK}_{\tilde{X}} \circ (\sigma_{\one \otimes \gamma}^{r,s})_\ast \big( \big[ - \cosh\big( \pi \chi(\calC^u(V)) \big) \otimes \gamma - \sinh\big( \pi \chi(\calC^u(V)) \big) \otimes \rho \big] - [ \one \otimes \gamma] \big)  \\
    &\quad  = (\sigma_{\one \otimes \gamma}^{r,s})_\ast \circ  \iota_\ast  \circ \calM^{\DK}_{\tilde{X}} \big(  \big[ - \tfrac12\big( W_{\calC^u(V)} + W_{\calC^u(V)}^* \big) \otimes \gamma - \tfrac12\big( W_{\calC^u(V)} - W_{\calC^u(V)}^* \big)\otimes \rho \big] - [ \one \otimes \gamma] \big) \\
  &\quad = (\sigma_{\one \otimes \gamma}^{r,s})_\ast\circ   \iota_\ast \big( \big[ \tfrac12( V|_{I_V}  + (V|_{I_V})^* ) \otimes \gamma + \tfrac12( V|_{I_V} - (V|_{I_V})^* ) \otimes \rho \big] - [\one \otimes \gamma] \big) \\
  &\quad = \big[  \sigma_{\one \otimes \gamma}^{r,s}(\tfrac12(V+V^*) \otimes \gamma + \tfrac12(V-V^*) \otimes \rho) \big] - [ \one \otimes \gamma],
\end{align*}
with the last equality by~\cite[Lemma 4.9]{BKR}.
\end{proof}

Recalling Proposition \ref{prop:Roe_Kubota_Bott_compatibility}, we also have that 
$ \calM_{1,1}^{\DK} \circ   \mathfrak{C}^{r+1,s}_{SA} \circ \altBott_{\KK} = \beta_{\DK} \circ \mathfrak{C}^{r+1,s+1}_A$. 
Building upon Eq. \eqref{eq:Graded_Bott_compat}, we lastly state   the compatibility of the $\Z_2$-graded Cayley, Roe, and Kubota isomorphisms.

\begin{prop} \label{prop:Kubota_Roe_Cayley_compat}
Let $B$ be a $\Z_2$-graded Real $C^*$-algebra. Then the Kubota isomorphism $\Kubota_B: \KKR(\C, SB) \to \DK(B)$ is such that 
\[
   \Kubota_B = \beta_{\DK}^{-1} \circ \Roe_{SB} = \mathfrak{C}_{B} \circ \altBott_{\KK}^{-1}.
\]
\end{prop}

%%%%%%%%%%%%%%%%%%%%%%%%%%%%%%%%%%%%%%%%%%%%%%%%%%%%%%%%%%%%%%%%%%%%%
%%%%%%%%%%%%%%%%%%%%%%%%%%%%%%%%%%%%%%%%%%%%%%%%%%%%%%%%%%%%%%%%%%%%%

\section{List of isomorphisms in \texorpdfstring{$\KKR$}{KKR}-theory and \texorpdfstring{$\DK$}{DK}-theory} \label{appendix:isos}

% \begin{table}[t]
% \caption{List of isomorphisms in $\KKR$-theory and $\DK$-theory\label{table:isos}}
\begin{tabular}{|@{$\,$}l|l@{$\colon$}c@{$\;\rightarrow \;$}c|l|}
\hline 
\textbf{Description} & \textbf{Symbol} & \textbf{Source} & \textbf{Target} & \textbf{Reference} \\
\hline 
\multicolumn{5}{|@{$\,$}l|}{\emph{Isomorphisms within $\KKR$-theory}} \\
\hline 
Clifford shuffle & $\shuffle^{r,s}_{\KK}$ & $\KKR(A\hox\Cl_{s,r},B)$ & $\KKR(A,B\hox\Cl_{r,s})$ & Eq.\ \eqref{eq:Clifford_shuffle} \\
%\hline 
Morita invariance & $\calM_Y^{\KK}$ & $\KKR(A,\End_B^0(Y))$ & $\KKR(A,B)$ & Eq.\ \eqref{eq:Morita_invariance} \\
Morita invariance & $\calM_X^{K}$ & $K(\End_A^0(X))$ & $K(A)$ & Eq.\ \eqref{eq:Complex_K_Morita}\\
%\hline 
Clifford stability & $\calM_{n,n}^{\KK}$ & $\KKR(A,B)$ & $\KKR(A,B\hox\Cl_{n,n})$ & Eq.\ \eqref{eq:Clifford_Morita_KK} \\
%\hline 
Bott periodicity & $\beta_{\KK}$ & $\KKR(A, B)$ & $\KKR( A, SB \hox \Cl_{1,0})$ & Proposition \ref{prop:KK_Bott} \\
%\hline 
Alternative Bott & $\altBott_{\KK}$ & $\KKR(A \hox \Cl_{1,0}, B)$ & $\KKR( A, SB)$ & Corollary \ref{cor:KK_Bott_alternative} \\
\hline 
\multicolumn{5}{|@{$\,$}l|}{\emph{Homomorphisms within $\DK$-theory}} \\
\hline 
Clifford shuffle & $(\sigma_e^{r,s})_\ast $ & $\DK_e( B )$ & $\DK_e( B \hox \Cl_{r,s} )$  & Corollary \ref{cor:sigma_also_gives_DK_hom} \\
Boundary map & $\partial$ & $\DK\big( S(B/I) \big)$ & $\DK(I)$ & Proposition \ref{prop:suspension_bdry_graded} \\
%\hline 
Boundary map & $\delta$ & $\DK(B/I)$ & $\DK(I\hox\Cl_{1,0})$ & Lemma \ref{lem:other-bdry} \\
\hline 
\multicolumn{5}{|@{$\,$}l|}{\emph{Isomorphisms within $\DK$-theory}} \\
\hline 
Excision & $\exc_Y$ & $\DK\big( \End_B(Y) , \calQ_B(Y) \big)$ & $\DK\big( \End_B^0(Y) \big)$  & Eq.\ \eqref{eq:DK_excision} \\
%\hline 
Morita invariance & $\calM_Y^{\DK}$ & $\DK\big(\End_B^0(Y)\big)$ & $\DK(B)$ & Eq.\ \eqref{eq:DK_Morita} \\
%\hline 
Clifford stability & $\calM_{n,n}^{\DK}$ & $\DK(B)$ & $\DK(B\hox\Cl_{n,n})$ & Lemma \ref{lemma:DK_Clifford_stability} \\
%\hline
Bott periodicity & $\beta_{\DK}$ & $\DK(B)$ & $\DK( SB \hox \Cl_{1,0})$ & Theorem \ref{thm:DK_Bott_isos} \\
%\hline 
\hline 
\multicolumn{5}{|@{$\,$}l|}{\emph{Isomorphisms between $\KKR$- and $\DK$-theory}} \\
\hline 
Roe & $\Roe^{r+1,s+1}_A$ &  $\KKR( \Cl_{s+1, r}, A)$ & $\DK( A \otimes \Cl_{r+1, s+1} ) $ & Definition \ref{def:Cl-indexed_Roe_Kubota_defn}\\
%\hline 
Kubota & $\Kubota_A^{r,s+1}$ & $\KKR(\Cl_{s+1,r}, SA)$ & $\DK(A \otimes \Cl_{r,s+1})$ & Definition \ref{def:Cl-indexed_Roe_Kubota_defn} \\
Cayley & $\Cayley_A^{r+1,s+1}$ & $\KKR( \Cl_{s+1, r}, A)$ & $\DK(A \otimes \Cl_{r+1,s+1})$ & Definition \ref{defn:cayley-r-s} \\
\hline 
\end{tabular}
% \end{table}

%%%%%%%%%%%%%%%%%%%%%%%%%%%%%%%%%%%%%%%%%%%%%%%%%%%%%%%%%%%%%%%%%%%%%
%%%%%%%%%%%%%%%%%%%%%%%%%%%%%%%%%%%%%%%%%%%%%%%%%%%%%%%%%%%%%%%%%%%%%
%%%%%%%%%%%%%%%%%%%%%%%%%%%%%%%%%%%%%%%%%%%%%%%%%%%%%%%%%%%%%%%%%%%%%

\end{document}